\documentclass[10pt,portrait]{amsart}

\usepackage{latexsym,exscale,enumerate,amsfonts,amssymb, ulem, xparse, mathtools}
\usepackage{amsmath,amsthm,amsfonts,amssymb,amscd, stmaryrd,textcomp,mathscinet}
\usepackage{hyperref}
\usepackage{ulem}
\usepackage{bbold}
\usepackage{tocvsec2}
\usepackage{wasysym}

\usepackage{libertine}
\usepackage[usenames,dvipsnames]{xcolor}
\colorlet{green}{black!30!green} % <------------------------------this redefines green to not be so bright and harsh!

\usepackage{tikz}
\usetikzlibrary{calc}
\usetikzlibrary{decorations.markings}
\usetikzlibrary{decorations.pathreplacing}
\usetikzlibrary{arrows,shapes,positioning}
\usetikzlibrary{patterns}
\tikzstyle directed=[postaction={decorate,decoration={markings,
    mark=at position #1 with {\arrow{>}}}}]
\tikzstyle rdirected=[postaction={decorate,decoration={markings,
    mark=at position #1 with {\arrow{<}}}}]
\tikzset{anchorbase/.style={baseline={([yshift=-0.5ex]current bounding box.center)}}}

\tikzset{
    partial ellipse/.style args={#1:#2:#3}{
        insert path={+ (#1:#3) arc (#1:#2:#3)}
    }
}

\usepackage{tkz-fct}

\theoremstyle{plain}
\newtheorem{theorem}{Theorem}
\newtheorem{theorem*}[theorem]{Theorem*}
\newtheorem{theorem**}[theorem]{Theorem*}
\numberwithin{theorem}{section}

\newtheorem{corollary*}[theorem]{Corollary*}
\newtheorem{corollary**}[theorem]{Corollary*}

\newtheorem{lemma*}[theorem]{Lemma*}
\newtheorem{lemma**}[theorem]{Lemma*}

\newtheorem{notation*}[theorem]{Notation*}

\newtheorem{proposition*}[theorem]{Proposition*}
\newtheorem{proposition**}[theorem]{Proposition*}

\theoremstyle{definition}
\newtheorem{definition}[theorem]{Definition}
\newtheorem{definition*}[theorem]{Definition*}
\newtheorem{definition**}[theorem]{Definition*}

\theoremstyle{remark}
\newtheorem{remark}[theorem]{Remark}

\newcommand{\numroman}{\renewcommand{\labelenumi}{\roman{enumi})}}
\newcommand{\numarabic}{\renewcommand{\labelenumi}{\arabic{enumi})}}
\newcommand{\numAlph}{\renewcommand{\labelenumi}{\Alph{enumi}.}}

\newcommand{\To}{\Rightarrow}
\newcommand{\TO}{\Rrightarrow}
\newcommand{\Hom}{{\rm Hom}}
\newcommand{\HOM}{{\rm HOM}}
\renewcommand{\to}{\rightarrow}
\newcommand{\maps}{\colon}
\newcommand{\op}{{\rm op}}
\newcommand{\co}{{\rm co}}
\newcommand{\iso}{\cong}
\newcommand{\id}{{\rm id}}
\newcommand{\bigb}[1]{
\begin{pspicture}(0,0)
 \rput(0,0){\psframebox[framearc=.5,fillstyle=solid]{\small $#1$}}
\end{pspicture}}
\newcommand{\del}{\partial}
\newcommand{\Res}{{\rm Res}}
\newcommand{\End}{{\rm End}}
\newcommand{\Aut}{{\rm Aut}}
\newcommand{\im}{{\rm im\ }}
\newcommand{\coim}{{\rm coim\ }}
\newcommand{\chr}{{\rm char\ }}
\newcommand{\coker}{{\rm coker\ }}
\newcommand{\Span}{{\rm Span}}
\newcommand{\spann}{{\rm span}}
\newcommand{\rk}{{\rm rk\ }}
\def\bigboxtimes{\mathop{\boxtimes}\limits}

\newcommand{\scs}{\scriptstyle}

\def\mf{\mathfrak}

\numberwithin{equation}{section}

\def\OC#1{\textcolor[rgb]{1.00,0.00,0.00}{[Orientability Comment: #1]}}%
\def\BC#1{\textcolor[rgb]{1.00,0.00,0.00}{[Boerner Relation Comment: #1]}}%
\def\comm#1{}%
\def\new#1{\b #1\e}%

\def\emph#1{{\sl #1\/}}
\def\ie{{\sl i.e. \/}}
\def\eg{{\sl e.g. \/}}
\def\Eg{{\sl E.g.\/}}
\def\etc{{\sl etc.\/}}
\def\cf{{\sl c.f.\/}}
\def\etal{\sl{et al.\/}}%
\def\vs{\sl{vs.\/}}%

\let\hat=\widehat
\let\tilde=\widetilde

\let\phi=\varphi
\let\theta=\vartheta
\let\epsilon=\varepsilon

\usepackage{bbm}
\def\C{{\mathbbm C}}
\def\N{{\mathbbm N}}
\def\R{{\mathbbm R}}
\def\Z{{\mathbbm Z}}
\def\Q{{\mathbbm Q}}
\def\H{{\mathbbm H}}
\def\P{{\mathbbm P}}
\newcommand{\Bb}{\mathbb{B}}
\newcommand{\I}{\mathbb{I}}
\newcommand{\Ss}{\mathbb{S}}
\newcommand{\D}{\mathbb{D}}

\usepackage{lscape}

\usetikzlibrary{knots}

\usepackage{pgfplots}

\pgfplotsset{every axis/.append style={
                    axis x line=middle,    % put the x axis in the middle
                    axis y line=middle,    % put the y axis in the middle
                    axis line style={<->,color=blue}, % arrows on the axis
                    xlabel={$x$},          % default put x on x-axis
                    ylabel={$y$},          % default put y on y-axis
            }}

\begin{document}
\allowdisplaybreaks

% ==============================================================================
\title{Movie moves for framed foams from multijet transversality}

\author{Hoel Queffelec}
\address{IMAG\\ Univ. Montpellier\\ CNRS \\ Montpellier \\ France}
\email{hoel.queffelec@umontpellier.fr}

\author{Kevin Walker}
\address{Microsoft Station Q\\ Santa Barbara, California}
\email{kevin@canyon23.net}

\begin{abstract}
We use multijet transversality techniques to give a presentation by generators and relations of the categories of framed tangled webs and foams.
\end{abstract}

\maketitle
\maxtocdepth{subsection}
\tableofcontents

%%%%%%%%%%%%%%%%%%%%%%%%%%%%%%%%%%%%%%%%%%%%%%%%%%%%%%%%%%%%%%%%%
\section{Introduction}
%%%%%%%%%%%%%%%%%%%%%%%%%%%%%%%%%%%%%%%%%%%%%%%%%%%%%%%%%%%%%%%%%

In a beautiful book, Carter and Saito~\cite{CS_book} describe ways to represent surfaces embedded in higher-dimensional vector spaces, and their isotopies. This is to be understood as a higher dimensional analog of Reidemeister's theorem, stating that generically, a link can be represented by a planar diagram assembled from elementary pieces (namely, pieces of strands and crossings), and that an isotopy of the link translates into a succession of elementary moves. In the case of knotted surfaces, diagrams of links are replaced by ``movies'' (a sequence of link diagrams). 
An isotopy of the knotted surface then translates into a succession of elementary movie isotopies or ``movie moves'', for which Carter and Saito establish an exhaustive list.

The goal of the present paper is to extend this work to certain knotted singular surfaces (namely foams), 
which play a prominent role in most 
developments of quantum knot homologies.

\subsection{Functoriality of Khovanov homology}

Khovanov's early definition of a homological lift of the Jones polynomial~\cite{Kh1} extends to an invariant of knotted surfaces~\cite{Kh2}. 
This is most easily seen under Bar-Natan's reformulation of Khovanov homology~\cite{BN2}. 
Indeed, in this latter version, the definition of the invariant makes use of cobordisms between (crossingless) curves. 
That way, an unknotted surface is naturally assigned a morphism, and elementary cobordism generators between links (namely, Reidemeister moves) are assigned morphisms from the very proof of the invariance. 
One then needs to check that these assignments do not depend on the chosen isotopy representative, and this is where Carter and Saito's movie moves come into play. Considering the images under Khovanov's process of all those moves, it turns out that the process is only independent on the choice up to a global sign. Functoriality holds on $\Z/2\Z$, but not on $\Z$ (see for example~\cite[p.1480]{BN2}).

This functoriality defect was fixed first by Clark, Morrison and the second author~\cite{CMW} and Caprau~\cite{Cap}, using cobordisms carrying special lines, then by Blanchet~\cite{Blan} using singular cobordisms ($\mathfrak{gl}_2$ incarnations of foams).

\subsection{From knot cobordisms to web-tangles and foams}

The latter fix introduces the notion of foams (singular cobordisms between trivalent graphs) as an intermediate object in Khovanov's process, allowing for a sign adjustment. That way, the homology of a knot remains unchanged, but the homomorphism associated to a (classical) cobordism gets adjusted by a sign (inherited from the determinant representation of $\mathfrak{gl}_2$). Checking the Carter-Saito's movie moves yields a functorial theory.

The appearance of foams in the context of knot homologies goes back to Khovanov-Rozansky's categorification~\cite{Kh5,KhR3} (and Mackaay-Stosic-Vaz version of it~\cite{MSV}) of the Reshetikhin-Turaev's $\mathfrak{sl}_n$ invariants~\cite{RT}. In these works, they simply appear as the most natural version of cobordisms between certain trivalent graphs called webs. These webs, as introduced by Kuperberg~\cite{Kup}, can be understood as diagrammatic versions of the categories of representations that are used in the definition of Reshetikhin and Turaev's invariants. The interest in their study was reinforced by Cautis-Kamnitzer-Morrison's proof of a presentation for them~\cite{CKM}. An analogous presentation of the foam category was then provided by the first author with Rose~\cite{QR}.

The question of functoriality of $\mathfrak{sl}_n$ link homologies is next in line, and was successfully addressed by Ehrig, Tubbenhauer and Wedrich~\cite{ETW}.

However, having at hand this notion of webs, it is natural to consider tangled versions of these, and the definition of the Khovanov-Rozansky functors will extend (almost) for free. Such an extension finds further justification in skein approaches~\cite{Prz1,Tur} to generalizations of the invariants to 3-manifolds (see~\cite{QW_SkeinCat} to read more about our own interest in this). However, while the definition of the extended version of the functors is given almost for free, the next question of functoriality confronts an unexpected issue: there does not exist a complete list of movie moves for foams.

In this direction, the main reference is Carter's work~\cite{Carter_foams}, where he presents a list of moves for unframed embedded foams with no preferred vertical direction (see also the book of Carter and Kamada~\cite{CarterKamada}).

The goal of this paper is to fill this gap in the literature, adding both the vertical direction (because we want to be able to project down to link diagrams) and a notion of framing. This notion appears to us as necessary since naive attempts to prove functoriality in an unframed setting quickly run into contradictions (see~\cite[Section 2.3.2.2]{Queff_PhD}). One might imagine that a closer analysis in the spirit of~\cite{Vogel,EST} might avoid framing. In such a case one can easily deduce from our main theorem~\ref{th:mainHT} a complete list of framing-free moves.

\subsection{Multijet transversality}

In order to organize the analysis of all possible moves, we have chosen to use the framework of multijet transversality
(Section \ref{sec:multijet}). 
This entails working in the smooth setting.
A framed foam is represented by a smooth map into $\R^4$.
The space of all such smooth maps has a stratification, with strata corresponding to non-generic situations, such as
double (or triple or quadruple) points in the projection to $\R^2$, framing vectors pointing the vertical direction, and so on.
After a small perturbation the original framed foam (and its associated multijets) can be made transverse to this stratification.
The resulting transverse intersections lead to foam generators (caps, saddles, Reidemeister moves, and so on).
Similarly, a 1-parameter family (isotopy) of foams can be perturbed to be transverse to the stratification, and the resulting 
transverse intersections correspond to movie moves for framed foams.

\subsection{Results}

In Section~\ref{sec:multijet} we briefly recall the basics of multijet transversality.

In Section~\ref{sec:webs} we investigate the situation for framed webs and isotopies of them, and give two versions of a presentation, using either half twists in Theorem~\ref{th:ReidHTwist} or full twists in Theorem~\ref{th:ReidFT}. The first version might appeal more to a topologist, while the second one is designed for knotted web invariants.

In Section~\ref{sec:foams} we extend this analysis to the case of foams and isotopies of them. The main results of the paper are Theorems~\ref{th:mainHT} and~\ref{th:mainFT}, which again are two versions of a full list of movie moves for framed foams between knotted framed web-tangles represented by diagrams.

\settocdepth{section}
\subsection*{Acknowledgements}
We would like to warmly thank Paul Wedrich, who was part of the research project that led to this paper and contributed a lot to the ideas that we present here. H.Q. also would like to thank Scott Carter for generously sharing partial results about movie moves for foams (more than 10 years ago!).

\subsection*{Funding}
H.Q. received partial support from the CNRS-MSI partnership {\it FuMa} and the ANR grants {\it QUANTACT} and {\it CATORE} .

\settocdepth{subsection}

%%%%%%%%%%%%%%%%%%%%%%%%%%%%%%%%%%%%%%%%%%%%%%%%%%%%%%%%%%%%%%%%%%%%%%%%%%%%%%%%%%%%%% 
\section{Multijet transversality} \label{sec:multijet}
%%%%%%%%%%%%%%%%%%%%%%%%%%%%%%%%%%%%%%%%%%%%%%%%%%%%%%%%%%%%%%%%%%%%%%%%%%%%%%%%%%%%%%% 

The situation handled by multijet transversality goes as follows. (We rely on Golubitsky and Guillemin's textbook~\cite{GoGui}, in particular Chapters 2.2 and 4.) Let $X$ and $Y$ be smooth manifolds. Denote $J^k(X,Y)_{p,q}$ the set of equivalence classes for mappings $f:X\to Y$ with $f(p)=q$, where the equivalence relation is that $f\sim_k g$ if $f$ has $k$-th order contact with $g$ at $p$. This property is inductively defined as follows:
\begin{itemize}
\item if $k=1$, $(df)_p=(dg)_p$;
\item if $k>1$, $(df)_p$ and $(dg)_p$ have $(k-1)$-st order contact at every point in $T_pX$.
\end{itemize}
This amounts to asking that all partial derivatives of order up to $k$ agree.

Then one can form
\[
  J^k(X,Y)=\bigcup_{(p,q)\in X\times Y}J^k(X,Y)_{p,q}
\]
the elements of which are called $k$-jets from $X$ to $Y$. 
The set $J^k(X,Y)$ can be given the structure of a finite-dimensional smooth manifold in a natural way.

Given $f:X\to Y$ there is an associated $k$-jet $j^kf:X\to J^k(X,Y)$.

Now, consider
\[
  X^s=X\times \cdots \times X\;\text{and}\; X^{(s)}=\{(x_1,\dots, x_s)\in X^s|\forall i,j,\;x_i\neq x_j\}
\]
One has source maps
\[
  \alpha:J^k(X,Y)\mapsto X,\; \alpha^s:J^k(X,Y)^s\mapsto X^s
\]
and one can form the $s$-fold $k$-jet bundle:
\[
  J^k_s(X,Y)=(\alpha^s)^{-1}(X^{(s)})
\]
Given $f:X\to Y$ there is an associated $s$-fold $k$-jet map $j_s^k f:X^{(s)}\to J^k_s(X,Y)$.
The map $j_s^k f$ describes the behavior of $f$ up to order $k$ at $s$ distinct points of $X$.

Our main tool is the following theorem of John Mather (see~\cite[Proposition 3.3]{MatherV}) generalizing Thom's transversality theorems~\cite{Thom1,Thom2} (see also~\cite{Laudenbach} for a gentle introduction to the topic):
\begin{theorem}[Multijet transversality theorem, Theorem 4.13 in~\cite{GoGui}]
  Let $W$ be a submanifold of $J^k_s(X,Y)$. Let:
  \[
    T_W=\{f\in \mathcal{C}^{\infty}(X,Y)|j_s^kf \bar{\pitchfork} W\}.
    \]
    Then $T_W$ is a residual subset of $\mathcal{C}^{\infty}(X,Y)$. Moreover, if $W$ is compact, then $T_W$ is open.
  \end{theorem}
  Above $\bar{\pitchfork}$ is the notation for transverse intersection, and residual means that it is the countable intersection of open dense subsets. In the case of a Baire space (which $\mathcal{C}^{\infty}(X,Y)$ is), this implies that it is dense.

  The typical situation we will want to address using the above theorem is that of a mapping $f:F\mapsto R^4$ of a (suitably modified) foam into $\R^4$. The submanifold $W$ will be given by some condition we wish to avoid or control, for example having multiple points under the vertical projection. Then we apply the theorem to claim that up to minor adjustment $f$ can be made transverse to $W$, and then we go on to analyze what a local model is.

%%%%%%%%%%%%%%%%%%%%%%%%%%%%%%%%%%%%%%%%%%%%%%%%%%%%%%%%%%%%%%%%%%%%%%%%%%%%%%%%%%%%%%%%%%
\section{Reidemeister theorem for framed webs} \label{sec:webs}
%%%%%%%%%%%%%%%%%%%%%%%%%%%%%%%%%%%%%%%%%%%%%%%%%%%%%%%%%%%%%%%%%%%%%%%%%%%%%%%%%%%%%%%%%%

% ----------------------------------------------------------------------------------------
\subsection{Webs}
% ----------------------------------------------------------------------------------------

  \begin{definition}
    An abstract smooth web is a singular $1$-manifold locally smoothly diffeomorphic to either an interval or smooth realizations of the following trivalent graphs:
    
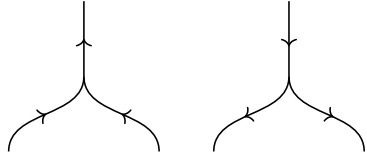
\[
  \begin{tikzpicture}[anchorbase,decoration={markings, mark=at position 0.5 with {\arrow{>}}; }]
    \draw [semithick, postaction={decorate}] (0,0) to [out=90,in=-90] (1,1);
    \draw [semithick, postaction={decorate}] (2,0) to [out=90,in=-90] (1,1);
    \draw [semithick, postaction={decorate}] (1,1) to [out=90,in=-90] (1,2);
  \end{tikzpicture}
  \qquad
  \begin{tikzpicture}[anchorbase,decoration={markings, mark=at position 0.5 with {\arrow{<}}; }]
    \draw [semithick, postaction={decorate}] (0,0) to [out=90,in=-90] (1,1);
    \draw [semithick, postaction={decorate}] (2,0) to [out=90,in=-90] (1,1);
    \draw [semithick, postaction={decorate}] (1,1) to [out=90,in=-90] (1,2);
  \end{tikzpicture}
\]
In other words, restricting the chart maps to the $1$-manifolds formed by selecting one of the two legs and the other strand forms a smooth $1$-manifold.
We write $W_0$ for the union of the $0$-cells, and $W_1$ for the union of the $1$-cells.
  \end{definition}
Then one can turn an abstract web into  a 2-manifold $\tilde{W}$  by considering a $2$-manifold with boundary with preferred web in it, locally smoothly diffeomorphic to one of the following pieces:
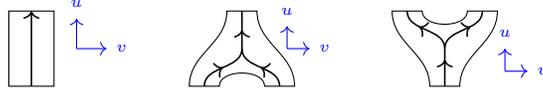
\[
    \begin{tikzpicture}[anchorbase]
      \draw [semithick,->] (0,0) -- (0,1);
      \draw (-.3,0) rectangle (.3,1);
      \draw [blue,->] (.6,.5) -- (1,.5);
      \node [blue] at (1.2,.5) {\tiny $v$};
      \draw [blue,->] (.6,.5) -- (.6,.9);
      \node [blue] at (.6,1.1) {\tiny $u$};
      \node at (.5,-.1) {\tiny \vphantom{$u$}};
    \end{tikzpicture}
\qquad
  \begin{tikzpicture}[anchorbase,decoration={markings, mark=at position 0.5 with {\arrow{>}}; }]
    \draw (-.2,0) -- (.2,0) to [out=90,in=90] (.8,0) -- (1.2,0) to [out=90,in=-90] (.7,1) -- (.3,1) to [out=-90,in=90] (-.2,0);
    \draw [semithick, postaction={decorate}] (0,0) to [out=90,in=-90] (.5,.5);
    \draw [semithick, postaction={decorate}] (1,0) to [out=90,in=-90] (.5,.5);
    \draw [semithick, postaction={decorate}] (.5,.5) to [out=90,in=-90] (.5,1);
      \draw [blue,->] (1.1,.5) -- (1.4,.5);
    \node [blue] at (1.6,.5) {\tiny $v$};
    \draw [blue,->] (1.1,.5) -- (1.1,.8);
    \node [blue] at (1.1,1) {\tiny $u$};
    \node at (1.1,0) {\tiny \vphantom{$u$}};
  \end{tikzpicture}
\qquad
  \begin{tikzpicture}[anchorbase,decoration={markings, mark=at position 0.5 with {\arrow{<}}; }]
    \draw (-.2,1) -- (.2,1) to [out=-90,in=-90] (.8,1) -- (1.2,1) to [out=-90,in=90] (.7,0) -- (.3,0) to [out=90,in=-90] (-.2,1);
    \draw [semithick, postaction={decorate}] (0,1) to [out=-90,in=90] (.5,.5);
    \draw [semithick, postaction={decorate}] (1,1) to [out=-90,in=90] (.5,.5);
    \draw [semithick, postaction={decorate}] (.5,.5) to [out=-90,in=90] (.5,0);
      \draw [blue,->] (1.3,.2) -- (1.6,.2);
    \node [blue] at (1.8,.2) {\tiny $v$};
    \draw [blue,->] (1.3,.2) -- (1.3,.5);
    \node [blue] at (1.3,.7) {\tiny $u$};
  \end{tikzpicture}
\]

\begin{remark}
  Our asymmetric choice for trivalent vertices might sound surprising, but it allows us to reduce the number of moves that will appear, in particular for the tetrahedral vertices that will be introduced in the foam section. For example, the moves involving twists presented in~\cite[Chapter 10, Section 2]{CarterKamada} do not appear in our case, as they involve a symmetry of the trivalent vertex (the closest analogs in our context would be the last two relations in Theorem~\ref{th:mainHT}).
  Another reason for making such a choice is that it allows us to cover a web by smooth segments, or a foam by smooth disks. 
On a related note, one can erase some of the strands (or facets) and the remainder will still be a web (or foam). 
All of these are particularly useful when turning to functoriality proofs for Khovanov or Khovanov-Rozansky homologies (where labels of strands/facets break the symmetry anyway from the beginning).
\end{remark}

We now consider smooth maps:
\[f\colon \tilde{W}\mapsto \R^3\] One recovers the usual notion of a tangled web
by restricting $f$ to $W$. Furthermore, we get a framing on $f(W)$ by looking at
a section of the normal bundle of $W$ in $\tilde{W}$ (given by $\frac{\partial}{\partial v}$ in some parametric version for example).

\begin{remark}
Two notions of framed objects appear in the literature: either one considers a vector field (non-vanishing, with possibly extra conditions), or one considers actual ribbons. In what follows, we consider the first notion. However, $f$ could produce the former or the latter, and several functions $f$ could induce the same web with the same vector field. A deformation of $f$ clearly induces a deformation of the associated framed web. Conversely, given a framed web, one can build an associated function $f$ (extend $f$ from $W$ to $\tilde{W}$ by going along the framing vector, drawing a ribbon of given length; by compactness one can find such a non-zero length so that the image is embedded). Then an isotopy of webs comes from a deformation of functions.
\end{remark}

We will denote by $(x,y,z)$ coordinates in $\R^3$, and we have a preferred projection $\pi$ onto $\R^2=\{(x,y,0)\}$. We require that $f$ is injective on $W$. At the level of jet bundles, consider the union of submanifolds\footnote{As the web is not a manifold, we only get a union of submanifolds by considering all circles that cover the web.}:
\[
  \left\{
      (M_1,N_1,M_2,N_2)\; |\; M_1,M_2\in W,\; N_1=N_2\in \R^3
    \right\}
\]
This is a union of submanifolds of codimension $5$ (twice $1$ because we restrict $M_1$ and $M_2$ to $W$, and then $3$ because we want the three coordinates of $N_2$ to be equal to those of $N_1$. The graph:
\[\left\{
    \left(M_1,f(M_1),M_2,f(M_2)\right)
  \right\}
\]
is of dimension $4$. Thus generically, $f$ is injective on $W$. We also impose that $df$ is full rank on $W$ (so that the framing does not meet the strands). The condition on jet bundles for this condition not to be met is:
\[
  \left\{(M,N,D)\; |\; M\in W,\; N\in \R^3,\; D\in M_{3,2}(\R),\;rank(D)\leq 1\right\}
\]
This is an object of codimension $3$: $1$ for the restriction to $W$, and $2$ for the rank condition. Indeed, one of the two vectors can be chosen freely. Then the other one has to be proportional to it, leaving us with only one choice (the proportionality factor) for three coordinates.
The corresponding graph:
\[
  \left\{(M,f(M),df_M)\; |\; M\in \tilde{W} \right\}\]
is of dimension $2$. Again, the condition on the rank is generic.

\begin{definition}
We say that a function $f$ injective on $W$ and with $df$ full rank on $W$ represents a framed web with generic projection if the following conditions are fulfilled:
\begin{enumerate}
\item \label{c1} $\pi\circ f$ has isolated double points, which both pre-images lying in $W_1$;
\item \label{c2} such double points are transverse: the two vectors $\pi(\frac{\partial f}{\partial u})$ span $\R^2$;
\item \label{c3} $\frac{\partial f}{\partial u}$ is not vertical on $W$ (in other words, $\pi(\frac{\partial f}{\partial u}) \ne 0$ on $W$);
  \item except at isolated ({\it half twist}) points, we have: \begin{equation}\frac{\partial f}{\partial v}\notin \Span_{\R}\left(\frac{\partial f}{\partial u},\begin{pmatrix} 0 \\ 0 \\ 1\end{pmatrix}\right) \label{eq:framing}\end{equation}
  \item \label{c5} the above half twist points are transverse: 
    \[\frac{\partial^2 f}{\partial u \partial v}\notin \Span_{\R}\left(\frac{\partial f}{\partial u},\begin{pmatrix} 0 \\ 0 \\ 1\end{pmatrix}\right) \]
\end{enumerate}
\end{definition}

\begin{remark}
Let us say a word about the framing convention. Above we have considered that points are generic if the projection of the framing vector is not proportional to the tangent vector (and in particular non-zero). That way we isolate points where the framing goes over or under the strand. Later in the argument we will want to turn the framing into a preferred position, namely, lying under $\pi$ on the right hand side of the strands (with respect to orientation). To do so, one can choose a rotation around $\frac{\partial f}{\partial u}$ that brings $\frac{\partial f}{\partial v}$ into the desired position (so that it projects to right hand side, and maximizes the length of the projection). Our preferred choice of rotation is the one that doesn't make $\frac{\partial f}{\partial v}$ cross $\R\frac{\partial f}{\partial u}+\R^{-}\begin{pmatrix} 0 \\ 0 \\ 1\end{pmatrix}$. At the end of the process, we get a framed web with framing on the right except at isolated points where the framing turns around a strand (we call this a twist). This process only requires that we isolate those points where the framing passes under a strand (because of the $\R^-\begin{pmatrix} 0 \\ 0 \\ 1\end{pmatrix}$ condition). However, the process of bringing a framed web with generic projection to a framed web with right-handed framing introduces crossings (at trivalent points). We thus found it easier to isolate the cases where the framing passes under or over a strand (these situations are called half-twists). This produces Theorems~\ref{th:ReidHTwist} and~\ref{th:mainHT}. At the very end, we will apply our preferred rotation and deduce classification results for right-sided framed webs: this yields Theorems~\ref{th:ReidFT} and~\ref{th:mainFT}.
\end{remark}

Let us first find local models for generic points.

\subsubsection{Local model for generic $m\in W_1$}

Let $m\in W_1$ that is not a multiple point under $\pi$, and so that condition~\ref{c3} and Equation~\eqref{eq:framing} above hold.

Fix a local chart on $\tilde{W}$ so that $m=0$, and consider a neighborhood of $m$ of the form $[-u_1,u_1]\times [-v_1,v_1]$. Since $\frac{\partial f}{\partial u}$ is not vertical, one of its $x$ or $y$ coordinates at least is non-zero. Up to rotation around a vertical axis in the target 3-space, one can assume that:
\[
  \frac{\partial f}{\partial u}=\begin{pmatrix} a \\ 0 \\ b\end{pmatrix},\;a>0.
\]
Then upon post-action by a matrix in $GL_2\subset GL_3$ (where $GL_2$ acts on the $x$ and $y$ coordinates), one can reduce $df$ to be of the following kind: 
\[
  df=\begin{pmatrix}
    a & 0\\
    0 & c\\
    b & d\\
  \end{pmatrix}
\]
Because of condition~\eqref{eq:framing}, we have $c\neq 0$.

  Then around $m$, one can write:
\[
f(\epsilon_1,\epsilon_2)=f(m)+\begin{pmatrix} a \epsilon_1 \\ c \epsilon_2 \\ b\epsilon_1+d\epsilon_2 \end{pmatrix}+o(\epsilon_1,\epsilon_2).
\]

Setting $\epsilon_2=0$, this draws a portion of line in the $3$-space that projects into a portion of line on the $(x,y)$ plane (as $a\neq 0$). Now we can look at the framing. Recall $\left(\frac{\partial f}{\partial v}\right)_m=\begin{pmatrix}0 \\ c \\ d \end{pmatrix}$ with $c\neq 0$. Thus:
\[
  \left(\frac{\partial f}{\partial v}\right)_{(\epsilon_1,0)}=\begin{pmatrix} 0 \\ c \\ d \end{pmatrix}+o(\epsilon_1).
\]
This means that around $m$ at first order, the projection of the framing can be simply obtained by transporting the framing at $m$. Depending on the sign of $c$, we get as generator:
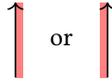
\[
  \begin{tikzpicture}[anchorbase]
    \fill [fill=red,fill opacity=.5] (0,0) -- (0,1) -- (.1,1) -- (.1,0);
    \draw [semithick,->] (0,0) -- (0,1);
  \end{tikzpicture}
  \quad \text{or}\quad
    \begin{tikzpicture}[anchorbase]
    \fill [fill=red,fill opacity=.5] (0,0) -- (0,1) -- (-.1,1) -- (-.1,0);
    \draw [semithick,->] (0,0) -- (0,1);
  \end{tikzpicture}
\]
Above we have indicated the framing by a red ribbon.

\subsubsection{Local model for $m\in W_0$}

Let us now look at $m\in W_0$ and run the same analysis. Conveniently, forgetting one of the two strands that arrive parallel brings us back to the previous situation, so we can make the same simplifications. Depending on the framing, this will bring us to one of the two following situations:
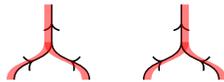
\[
    \begin{tikzpicture}[anchorbase,decoration={markings, mark=at position 0.5 with {\arrow{>}}; }]
      \draw [semithick, postaction={decorate}] (0,0) to [out=90,in=-90] (.5,.5);
      \fill [fill=red, fill opacity=.5] (0,0) to [out=90,in=-90] (.5,.5) -- (.4,.5) to [out=-90,in=90] (-.1,0);
      \draw [semithick, postaction={decorate}] (1,0) to [out=90,in=-90] (.5,.5);
      \fill [fill=red, fill opacity=.5] (1,0) to [out=90,in=-90] (.5,.5) -- (.4,.5) to [out=-90,in=90] (.9,0);
      \draw [semithick, postaction={decorate}] (.5,.5) to [out=90,in=-90] (.5,1);
      \fill [fill=red, fill opacity=.5] (.5,.5) to [out=90,in=-90] (.5,1) -- (.4,1) to [out=-90,in=90] (.4,.5);
    \end{tikzpicture}
    \qquad
    \begin{tikzpicture}[anchorbase,decoration={markings, mark=at position 0.5 with {\arrow{>}}; }]
      \draw [semithick, postaction={decorate}] (0,0) to [out=90,in=-90] (.5,.5);
      \fill [fill=red, fill opacity=.5] (0,0) to [out=90,in=-90] (.5,.5) -- (.6,.5) to [out=-90,in=90] (.1,0);
      \draw [semithick, postaction={decorate}] (1,0) to [out=90,in=-90] (.5,.5);
      \fill [fill=red, fill opacity=.5] (1,0) to [out=90,in=-90] (.5,.5) -- (.6,.5) to [out=-90,in=90] (1.1,0);
      \draw [semithick, postaction={decorate}] (.5,.5) to [out=90,in=-90] (.5,1);
      \fill [fill=red, fill opacity=.5] (.5,.5) to [out=90,in=-90] (.5,1) -- (.6,1) to [out=-90,in=90] (.6,.5);
    \end{tikzpicture}
\]
One can also consider analogous pictures with reversed orientations (split vertices).

\subsubsection{Conditions}

Before looking for local models for double points and places where the assumption from Equation~\ref{eq:framing} fails, let us first check that the assumptions we made on $f$ to represent a framed web are reasonable.

Requiring $\pi\circ f$ to have double points corresponds to the union of submanifolds in $J_{0}^2$:
\[
\left\{
(M_1,N_1,M_2,N_2), M_1\in W,M_2\in W,\pi(N_1)=\pi(N_2)
\right\}
\]
These manifolds are of codimension $4$ (we get $1+1$ by restricting $M_i\in W$, and 2 from the equation $\pi(N_1)=\pi(N_2)$).
The graph:
\[
  \left\{
    (M_1,f(M_1),M_2,f(M_2)
    \right\}
\]
is of dimension $4$, so one can make $f$ transverse to the submanifolds with isolated intersections.

Furthermore, if one requires $M_1$ or $M_2$ to lie on $W_0$, the codimension drops again by one and transverse intersections are empty.

We argue that no higher multiplicities will generically occur. Indeed, in $J_{k}^1$, the union of submanifolds corresponding to $k$ points being sent to the same projection has codimension $k+2(k-1)=3k-2$, while the graph is of dimension $2k$. At $k=2$ we have isolated singularities, at $k=3$ we have a drop of $1$ (so this will show up when passing to isotopies), and starting $k=4$ we fall into empty transverse intersections even when looking at $1$-parameter families of functions.

% Condition \ref{c3}  requiring the tangent vector to not be vertical corresponds to a union of submanifolds of codimension $3$ ($1$ from the restriction to $W$, $2$ for the vanishing of $x$ and $y$ coefficients) with a $2$-dimensional graph: it is generically fulfilled for a web (and we will see isolated points when we turn to time families).

% The condition from Equation \eqref{eq:framing} corresponds to a codimension $2$ object, so one expects isolated points on a web where this is not verified. Condition~\ref{c5} then increases the codimension by $1$: this is generically true for a web, and one will expect isolated points over time.

Let us now go to $J_{1}^2$ and consider non-transverse double points on the projection:
\[
\left\{
(M_1,N_1,D_1,M_2,N_2,D_2), M_1\in W,M_2\in W,\pi(N_1)=\pi(N_2),det(\begin{pmatrix} (D_1)_{1,1} & (D_2)_{1,1} \\ (D_1)_{2,1} & (D_2)_{2,1}\end{pmatrix})=0
\right\}
\]
This is a union of submanifolds of codimensions $5$, while the graph is still of dimension $4$: transverse intersections are empty.

Now, let us work in $J_{1}^{1}$ and consider the following submanifolds, that correspond to the failure of condition~\ref{c3}.
\[
  \left\{ (M,N,D),M\in W,D=\begin{pmatrix}0 & \cdot \\ 0 & \cdot \\ \cdot & \cdot \end{pmatrix} \right\}
\]
This is a union of codimension $3$ submanifolds, and the graph is of dimension $2$. We again have empty transverse intersection.

We now turn our attention to the framing. Failure of condition~\eqref{eq:framing} can be written as:
\[df=\begin{pmatrix}a & \lambda a \\ b & \lambda b \\ c & \lambda c+\mu \end{pmatrix}\]
One gets a union of codimension 2 ($1$ by restriction to $W$, $1$ for the rank condition) submanifolds. So one expects isolated transverse intersections. Notice that these intersections occur on $W_1$ (forcing them to be on $W_0$ drops the codimension by $1$). Condition~\ref{c5} also increases the codimension by $1$: it holds generically for a web, and one will expect isolated points over time where it doesn't hold.

Furthermore, notice that imposing failure of the above condition on top of the double point case also drops the codimension by $1$, and one can thus assume that these two situations arise at distinct places.

We will now go back to the above special situations and identify local models for them.

\subsubsection{Multiple points}

Let us first look at transverse double points in $\pi\circ f$. The situation is very classical : the local picture is determined by both derivatives (and transversality implies that they are not colinear) and looks like (the first one is a positive crossing, the second one a negative crossing):
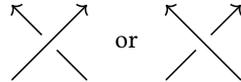
\[
  \begin{tikzpicture}[anchorbase]
    \draw [semithick, ->] (0,0) -- (1,1);
    \draw [semithick] (1,0) -- (.6,.4);
    \draw [semithick, ->] (.4,.6) -- (0,1);
  \end{tikzpicture}
\quad \text{or} \quad
  \begin{tikzpicture}[anchorbase]
    \draw [semithick, ->] (1,0) -- (0,1);
    \draw [semithick] (0,0) -- (.4,.4);
    \draw [semithick, ->] (.6,.6) -- (1,1);
  \end{tikzpicture}
\]
As explained previously, the framing can be assumed to be non-singular under projection at all time, which we can emphasize as follows:
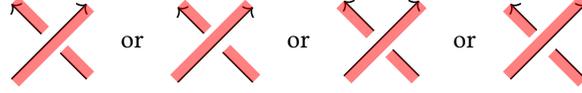
\[
  \begin{tikzpicture}[anchorbase]
    \draw [semithick, ->] (0,0) -- (1,1);
    \fill [fill=red, fill opacity=.5] (0,0) --(.1,-.1) -- (1.1,.9) -- (1,1);
    \draw [semithick] (1,0) -- (.65,.35);
    \fill [fill=red, opacity=.5] (1,0) -- (.65,.35) -- (.75,.45) -- (1.1,.1);
    \draw [semithick, ->] (.4,.6) -- (0,1);
    \fill [fill=red,opacity=.5] (.4,.6) -- (.5,.7) -- (.1,1.1) -- (0,1);
  \end{tikzpicture}
\quad \text{or} \quad
  \begin{tikzpicture}[anchorbase]
    \draw [semithick, ->] (0,0) -- (1,1);
    \fill [fill=red, fill opacity=.5] (0,0) --(-.1,.1) -- (.9,1.1) -- (1,1);
    \draw [semithick] (1,0) -- (.6,.4);
    \fill [fill=red, opacity=.5] (1,0) -- (.6,.4) -- (.7,.5) -- (1.1,.1);
    \draw [semithick, ->] (.35,.65) -- (0,1);
    \fill [fill=red,opacity=.5] (.35,.65) -- (.45,.75) -- (.1,1.1) -- (0,1);
  \end{tikzpicture}
\quad \text{or} \quad
  \begin{tikzpicture}[anchorbase]
    \draw [semithick, ->] (0,0) -- (1,1);
    \fill [fill=red, fill opacity=.5] (0,0) --(.1,-.1) -- (1.1,.9) -- (1,1);
    \draw [semithick] (1,0) -- (.65,.35);
    \fill [fill=red, opacity=.5] (1,0) -- (.65,.35) -- (.55,.25) -- (.9,-.1);
    \draw [semithick, ->] (.4,.6) -- (0,1);
    \fill [fill=red,opacity=.5] (.4,.6) -- (.3,.5) -- (-.1,.9) -- (0,1);
  \end{tikzpicture}
\quad \text{or} \quad
  \begin{tikzpicture}[anchorbase]
    \draw [semithick, ->] (0,0) -- (1,1);
    \fill [fill=red, fill opacity=.5] (0,0) --(-.1,.1) -- (.9,1.1) -- (1,1);
    \draw [semithick] (1,0) -- (.6,.4);
    \fill [fill=red, opacity=.5] (1,0) -- (.6,.4) -- (.5,.3) -- (.9,-.1);
    \draw [semithick, ->] (.35,.65) -- (0,1);
    \fill [fill=red,opacity=.5] (.35,.65) -- (.25,.55) -- (-.1,.9) -- (0,1);
  \end{tikzpicture}
\]
We have a similar list of cases for the negative crossing.

\subsubsection{Framing}

We now turn our attention to the framing issue. Assume that at some point $m\in W_1$ :
\[
  (df)_m=\begin{pmatrix} a & \lambda a \\ b & \lambda b \\ c & \lambda c +\mu \end{pmatrix}\; \mu\neq 0
\]
  
Up to rotation in the $x,y$ plane one can assume that $a>0$ and $b=0$, and up to adding to $v$ a multiple of $u$, one reduces to:
\[
  (df)_m=\begin{pmatrix} a & 0 \\ 0 & 0 \\ c &  \mu \end{pmatrix}\; a>0,\;\mu\neq 0
\]

If we go to $J_2^{1}$ and consider the union of submanifolds corresponding to simultaneously requiring:
\begin{itemize}
\item failure of the framing condition~\eqref{eq:framing};
\item that the $\frac{\partial^2}{\partial u \partial v}$ entry of the matrix corresponding to the second derivative has zero coordinate in the $y$ direction;
\end{itemize}
then the codimension increases again by one. This means that we can assume that the entry $\frac{\partial^2f}{\partial u \partial v}(m)\neq \begin{pmatrix} \cdot \\ 0 \\ \cdot\end{pmatrix}$. Let us write:
\[\frac{\partial^2f}{\partial u \partial v}(m)\neq \begin{pmatrix} \alpha \\ \beta \\ \gamma \end{pmatrix}, \; \beta \neq 0.\]

Up to symmetry, one can assume that $\beta >0$. Then by writing a Taylor expansion for $\frac{\partial f}{\partial v}$, one has:
\[
  \left(\frac{\partial f}{\partial v}\right)_{(\epsilon_1,0)}=\begin{pmatrix} \alpha \epsilon_1 \\ \beta \epsilon_1 \\ \mu+\gamma \epsilon_1 \end{pmatrix}+o(\epsilon_1)
\]
This means that just before $m$ on $W$, the framing was pointing to the right, while just after, it points to the left, passing at $m$ under the strand if $\mu<0$ and over the strand if $\mu>0$. We make this into the first and third of the following generators (for more clarity, we indicate the framing). The other two generators correspond to $\beta <0$.
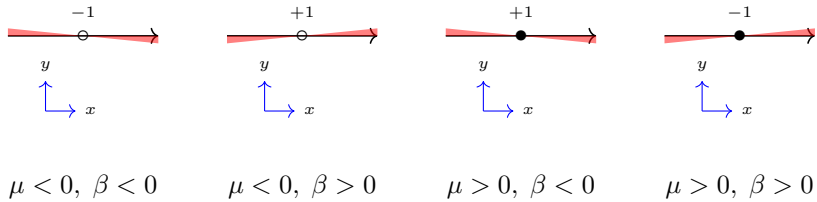
\[
  \begin{tikzpicture}[anchorbase,rotate=-90] 
    \draw [semithick, ->] (0,0) -- (0,2);
    \fill [opacity=.5, fill=red] (-.1,0) -- (.1,2) -- (0,2) -- (0,0);
    \node at (0,1) {$\circ$};
    \node at (-.3,1) {\tiny $-1$};
    \draw [blue,->] (1,.5) --(1,.9);
    \node at (1,1.1) {\tiny $x$};
    \draw [blue,->] (1,.5) --(.6,.5);
    \node at (.4,.5) {\tiny $y$};
    \node at (2,1) {$\mu<0,\; \beta<0$};
  \end{tikzpicture}
  \qquad
  \begin{tikzpicture}[anchorbase,rotate=-90] 
    \draw [semithick, ->] (0,0) -- (0,2);
    \fill [opacity=.5, fill=red] (.1,0) -- (-.1,2) -- (0,2) -- (0,0);
    \node at (0,1) {$\circ$};
    \node at (-.3,1) {\tiny $+1$};
    \draw [blue,->] (1,.5) --(1,.9);
    \node at (1,1.1) {\tiny $x$};
    \draw [blue,->] (1,.5) --(.6,.5);
    \node at (.4,.5) {\tiny $y$};
    \node at (2,1) {$\mu<0,\; \beta>0$};
  \end{tikzpicture}
  \qquad
  \begin{tikzpicture}[anchorbase,rotate=-90] 
    \draw [semithick, ->] (0,0) -- (0,2);
    \fill [opacity=.5, fill=red] (-.1,0) -- (.1,2) -- (0,2) -- (0,0);
    \node at (0,1) {$\bullet$};
    \node at (-.3,1) {\tiny $+1$};
    \draw [blue,->] (1,.5) --(1,.9);
    \node at (1,1.1) {\tiny $x$};
    \draw [blue,->] (1,.5) --(.6,.5);
    \node at (.4,.5) {\tiny $y$};
    \node at (2,1) {$\mu>0,\; \beta<0$};
  \end{tikzpicture}
  \qquad
  \begin{tikzpicture}[anchorbase,rotate=-90] 
    \draw [semithick, ->] (0,0) -- (0,2);
    \fill [opacity=.5, fill=red] (.1,0) -- (-.1,2) -- (0,2) -- (0,0);
    \node at (0,1) {$\bullet$};
    \node at (-.3,1) {\tiny $-1$};
    \draw [blue,->] (1,.5) --(1,.9);
    \node at (1,1.1) {\tiny $x$};
    \draw [blue,->] (1,.5) --(.6,.5);
    \node at (.4,.5) {\tiny $y$};
    \node at (2,1) {$\mu>0,\; \beta>0$};
  \end{tikzpicture}
\]
We emphasize our notation: a $\circ$ sign means that the framing goes under the strand, while a $\bullet$ sign means that it goes over it. The sign indicated next to the framing change indicates whether the vector turns positively or negatively (with respect to the right hand rule).

% -------------------------------------------------------------------------------------
\subsection{Isotopies of webs}
% -------------------------------------------------------------------------------------

Let us now consider isotopies of webs by looking at families:
\[f\colon \tilde{W}\times [0,1]\mapsto \R^3\]
For $t\in [0,1]$, we denote by $f_t$ the corresponding thickened web embedding.

\begin{definition}
  We say that such a family is an isotopy if:
  \begin{enumerate}
  \item $\forall t\in [0,1]$, $f_t$ is injective on $W$;
  \item $\forall t\in [0,1]$, $\left(\frac{\partial f}{\partial u},\frac{\partial f}{\partial v}\right)$ is of rank $2$ on $W$.
  \end{enumerate}
\end{definition}
Both conditions fail at isolated times for generic families of functions. We thus only consider functions that have no such singular points, and notice that they are stable under small perturbations.

Looking again at the previous analysis, the codimensions remain unchanged but the dimension of the graph increases by one, so we have:
\begin{itemize}
\item isolated triple points;
\item a $1$-dimensional set of double points, that generically are transverse and lie on $W_1$; among which (both situations are mutually exclusive):
  \begin{itemize}
  \item isolated non transverse double points;
  \item isolated double points from $W_0\times W_1$;
  \end{itemize}
\item isolated points where $\frac{\partial f}{\partial u}$ is vertical on $W$;
\item a $1$-dimensional set of framing changes controlled by equation~\eqref{eq:framing}; among those (situations are mutually exclusive):
  \begin{itemize}
  \item isolated points where condition \ref{c5} does not hold;
  \item isolated points on $W_0$;
  \item isolated double points.
  \end{itemize}
\end{itemize}

We will analyze these situations in the following paragraphs.

\subsubsection{Triple points and non-transverse double points}

Triple points may only involve points from $W^{1}$ and thus correspond to the third Reidemeister move, as usual, while non-transverse double points also involve only points from $W^{1}$ and yield the second Reidemeister move. We will investigate the other situations a little further.

\subsubsection{Isolated double points involving a trivalent vertex}

Let us start with isolated double points involving $m_1\in W_1$ and $m_2\in W_0$. We consider the case where $m_2$ is a 
split vertex 
(the case of a merge can be run parallel). We can assume that the projections of the derivatives in the $u$ direction make the intersection transverse, so that, up to $GL_2$ action in the $(x,y)$ plane in the target space, one has:
\[
  \left(\frac{\partial f}{\partial u}\right)_{m_1}=\begin{pmatrix}a \\ 0 \\ b \end{pmatrix}\quad a>0 \quad,\quad  \left(\frac{\partial f}{\partial u}\right)_{m_2}=\begin{pmatrix}0 \\ c \\ d \end{pmatrix}\quad  c\neq 0.
\]
We consider the case when $c>0$. The other one is symmetric.
This controls the shape of the intersection as follows:
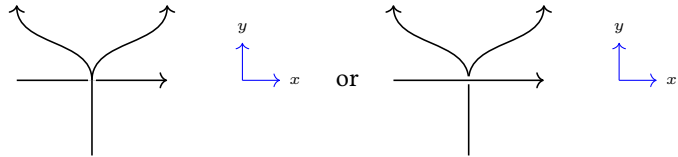
\[
  \begin{tikzpicture}[anchorbase]
    % strand
    \draw [semithick,->] (-1,0) -- (1,0);
    % web
    \draw [line width=3, white] (0,-1) -- (0,0); 
    \draw [line width=3, white] (0,0) to [out=90,in=-90] (1,1);
    \draw [line width=3, white] (0,0) to [out=90,in=-90] (-1,1);
    \draw [semithick] (0,-1) -- (0,0);
    \draw [semithick,->] (0,0) to [out=90,in=-90] (1,1);
    \draw [semithick,->] (0,0) to [out=90,in=-90] (-1,1);
    % chart
    \draw [blue, ->] (2,0) -- (2.5,0);
    \node at (2.7,0) {\tiny $x$};
    \draw [blue, ->] (2,0) -- (2,.5);
    \node at (2,.7) {\tiny $y$};
  \end{tikzpicture}
  \quad \text{or} \quad
  \begin{tikzpicture}[anchorbase]
    % web
    \draw [line width=3, white] (0,-1) -- (0,0); 
    \draw [line width=3, white] (0,0) to [out=90,in=-90] (1,1);
    \draw [line width=3, white] (0,0) to [out=90,in=-90] (-1,1);
    \draw [semithick] (0,-1) -- (0,0);
    \draw [semithick,->] (0,0) to [out=90,in=-90] (1,1);
    \draw [semithick,->] (0,0) to [out=90,in=-90] (-1,1);
    % strand
    \draw [line width=3, white] (-1,0) -- (1,0);
    \draw [semithick,->] (-1,0) -- (1,0);
    % chart
    \draw [blue, ->] (2,0) -- (2.5,0);
    \node at (2.7,0) {\tiny $x$};
    \draw [blue, ->] (2,0) -- (2,.5);
    \node at (2,.7) {\tiny $y$};
  \end{tikzpicture}
\]

The shape just before or just after in the time direction will be controlled by $\frac{\partial f}{\partial t}$. Again by a codimension argument, we can assume that both vectors have non zero $y$ coordinates, denoted $\eta_1$ for $m_1$ and $\eta_2$ for $m_2$. Then:
\begin{align*}
  f\left(m_1+\begin{pmatrix}\epsilon_1 \\ 0 \\ \delta_1 \end{pmatrix}\right)&= f(m_1)+\epsilon_1\begin{pmatrix}a \\ 0 \\ b \end{pmatrix}+\delta_1\begin{pmatrix} \cdot \\ \eta_1 \\ \cdot \end{pmatrix}
\end{align*}
so that one sees a line moving north or south (depending on the sign of $\eta_1$) at pace $|\eta_1|$,
and 
\begin{align*}
  f\left(m_2+\begin{pmatrix}0 \\ 0 \\ \delta_2 \end{pmatrix}\right)&= f(m_1)\delta_2\begin{pmatrix} \cdot \\ \eta_2 \\ \cdot \end{pmatrix}
\end{align*}
so that the trivalent point moves north or south (depending on the sign of $\eta_2$) at pace $|\eta_2|$ (the east/west component of the motion is not relevant). Altogether, depending on the sign of $\eta_2-\eta_1$, one gets a move:
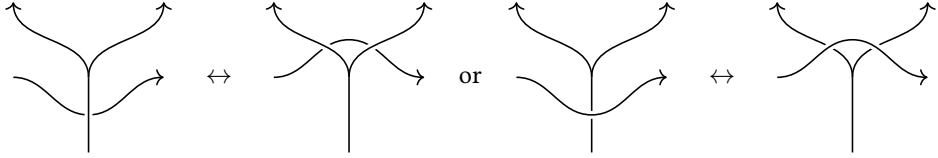
\[
  \begin{tikzpicture}[anchorbase]
    % strand
    \draw [semithick,->] (-1,0) to [out=0,in=180] (0,-.5) to [out=0,in=180]  (1,0);
    % web
    \draw [line width=3, white] (0,-1) -- (0,0); 
    \draw [line width=3, white] (0,0) to [out=90,in=-90] (1,1);
    \draw [line width=3, white] (0,0) to [out=90,in=-90] (-1,1);
    \draw [semithick] (0,-1) -- (0,0);
    \draw [semithick,->] (0,0) to [out=90,in=-90] (1,1);
    \draw [semithick,->] (0,0) to [out=90,in=-90] (-1,1);
  \end{tikzpicture}
  \quad \leftrightarrow
  \quad
    \begin{tikzpicture}[anchorbase]
    % strand
    \draw [semithick,->] (-1,0) to [out=0,in=180] (0,.5) to [out=0,in=180]  (1,0);
    % web
    \draw [line width=3, white] (0,-1) -- (0,0); 
    \draw [line width=3, white] (0,0) to [out=90,in=-90] (1,1);
    \draw [line width=3, white] (0,0) to [out=90,in=-90] (-1,1);
    \draw [semithick] (0,-1) -- (0,0);
    \draw [semithick,->] (0,0) to [out=90,in=-90] (1,1);
    \draw [semithick,->] (0,0) to [out=90,in=-90] (-1,1);
  \end{tikzpicture}
\quad \text{or} \quad
  \begin{tikzpicture}[anchorbase]
    % web
    \draw [line width=3, white] (0,-1) -- (0,0); 
    \draw [line width=3, white] (0,0) to [out=90,in=-90] (1,1);
    \draw [line width=3, white] (0,0) to [out=90,in=-90] (-1,1);
    \draw [semithick] (0,-1) -- (0,0);
    \draw [semithick,->] (0,0) to [out=90,in=-90] (1,1);
    \draw [semithick,->] (0,0) to [out=90,in=-90] (-1,1);
    % strand
    \draw [line width=3, white] (-1,0) to [out=0,in=180] (0,-.5) to [out=0,in=180]  (1,0);
    \draw [semithick,->] (-1,0) to [out=0,in=180] (0,-.5) to [out=0,in=180]  (1,0);
  \end{tikzpicture}
  \quad \leftrightarrow \quad
  \begin{tikzpicture}[anchorbase]
    % web
    \draw [line width=3, white] (0,-1) -- (0,0); 
    \draw [line width=3, white] (0,0) to [out=90,in=-90] (1,1);
    \draw [line width=3, white] (0,0) to [out=90,in=-90] (-1,1);
    \draw [semithick] (0,-1) -- (0,0);
    \draw [semithick,->] (0,0) to [out=90,in=-90] (1,1);
    \draw [semithick,->] (0,0) to [out=90,in=-90] (-1,1);
    % strand
    \draw [line width=3, white] (-1,0) to [out=0,in=180] (0,.5) to [out=0,in=180]  (1,0);
    \draw [semithick,->] (-1,0) to [out=0,in=180] (0,.5) to [out=0,in=180]  (1,0);
  \end{tikzpicture}
\]

We emphasize that no framing changes occur during the move.

\subsubsection{Isolated points with vertical derivative}

Now, we consider the case of an isolated point where the tangent vector is vertical. In the non-framed setting, this yields a Reidemeister 1 move. Here we will have to be more careful about the framing. We consider $(u_0,v_0,t_0)\in W\times \{t_0\} \subset \tilde{W}\times \{t_0\}$ so that $(\frac{\partial f}{\partial u})(u_0,v_0,t_0)=\begin{pmatrix} 0 \\ 0 \\ \alpha \end{pmatrix}$ ($\alpha\neq 0$).

Again by a codimension argument, one can assume that $\pi\left(\frac{\partial^2f}{\partial u^2}\right)$ is non-zero at $(u_0,v_0,t_0)$, and up to rotation of the target space around a vertical axis we can assume that:
\[
  \frac{\partial^2f}{\partial u^2}=\begin{pmatrix}\eta \\ 0 \\ \cdot \end{pmatrix},\quad \eta>0.
\]

Then the $y$-coordinate in $\frac{\partial^3f}{\partial u^3}=\zeta$ can be assumed to be non-zero, as well as the $y$-coordinate $\mu$ in $\frac{\partial^2f }{\partial u \partial t}$.

This brings us to the following situation (up to a drift proportional to $\delta$ that won't change the shape of the picture):
\[
  f(u_0+\epsilon, v_0,t_0+\delta)\sim f(u_0,v_0,t_0)+\begin{pmatrix} \epsilon^2\eta \\ \epsilon \delta \mu+\epsilon^3\zeta  \\ \epsilon \alpha \end{pmatrix}
\]

Assume that $\zeta>0$ (the case of $\zeta<0$ can be treated similarly). Then at $\delta=0$ this draws (we take the projection, and the example drawn below is for $\eta=\zeta=1$):

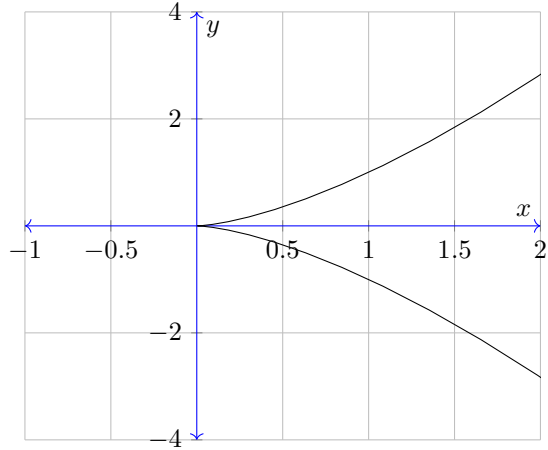
\[
\begin{tikzpicture}
    \begin{axis}[
            xmin=-1,xmax=2,
            ymin=-4,ymax=4,
            grid=both,
            ]
            \addplot [domain=-3:3,samples=50]({x^2},{x^3}); 
    \end{axis}
\end{tikzpicture}
\]

Up to reversal of the time direction, we can assume that $\mu>0$. Then at positive $\delta$, we have an additional component $\epsilon \delta \mu$ in the $y$ direction, that pushes up the part of the curve with $\epsilon>0$ and down the part of the curve with $\epsilon <0$: this has the result of smoothing the curl, as illustrated below ($\mu=1$, $\delta=1$).
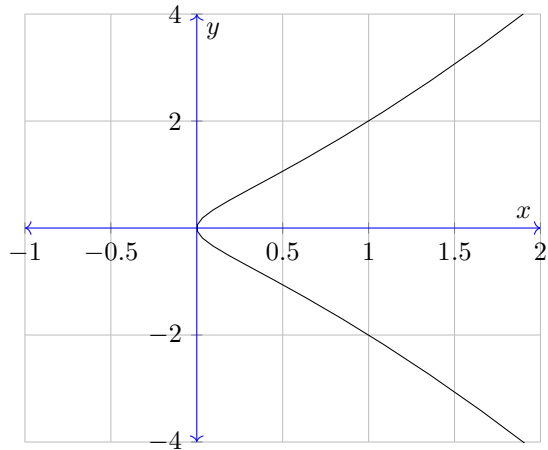
\[
\begin{tikzpicture}
    \begin{axis}[
            xmin=-1,xmax=2,
            ymin=-4,ymax=4,
            grid=both,
            ]
            \addplot [domain=-3:3,samples=50]({x^2},{x^3+x}); 
    \end{axis}
\end{tikzpicture}
\]

On the other hand for $\delta<0$, the upper part of the curve is pushed down and the bottom part is pushed up. For small $\epsilon$, then $\epsilon\delta >> \epsilon^3$ and the part in $\epsilon$ dominates in $y$ coordinate. At larger $\epsilon$ we go back to the first picture, as $\epsilon^3$ dominates $\epsilon\delta$. Below we illustrate the case $\mu=1$, $\delta=-1$.

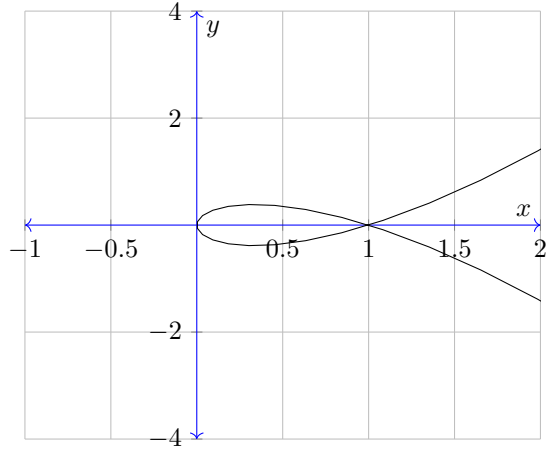
\[
\begin{tikzpicture}
    \begin{axis}[
            xmin=-1,xmax=2,
            ymin=-4,ymax=4,
            grid=both,
            ]
            \addplot [domain=-3:3,samples=50]({x^2},{x^3-x}); 
    \end{axis}
\end{tikzpicture}
\]

Let us now focus on the framing. Recall that at $(u_0,v_0,t_0)$ we have a framing vector that is not vertical. Thus it projects onto a non-zero vector in the plane, and it determines the general shape of the framing around $(u_0,v_0,t_0)$. In other words, the framing around $(u_0,v_0,t_0)$ is always parallel to the one at $(u_0,v_0,t_0)$ up to small perturbation. Assume that the framing vector has non-zero $y$-coordinate. Requiring this assumption to fail is a codimension $1$ condition, and thus yields a non-generic situation. 
Then one easily sees that the condition~\eqref{eq:framing} will fail twice, one at $\delta<0$ and once at $\delta >0$, once passing over the derivative, once under the derivative (depending on the sign of $\alpha$).

We get a move:

\begin{align} \label{eq:frR1}
  \begin{tikzpicture}[anchorbase]
    \fill [fill=red,opacity=.5] (1,-1) to [out=180,in=-90] (0,0) to [out=90,in=180] (1,1) -- (1,1.1) to [out=180,in=90] (0,0) to [out=-90,in=180] (1,-.9);
    \draw[semithick] (1,-1) to [out=180,in=-90] (0,0) to [out=90,in=180] (1,1);
    \node at (0,0) {$\circ$};
    \node at (-.3,0) {\tiny $+$};
  \end{tikzpicture}
  \quad \leftrightarrow
  \quad
    \begin{tikzpicture}[anchorbase]
      \fill [fill=red,opacity=.5] (1,-1) to [out=180,in=-90] (.5,0) to [out=90,in=90] (0,0) to [out=90,in=90] (.6,0) to [out=-90,in=180] (1,-.9);
      \draw[semithick] (1,-1) to [out=180,in=-90] (.5,0) to [out=90,in=90] (0,0);
      \draw [line width=5, white] (0,0) to [out=-90,in=-90] (.5,0) to [out=90,in=180] (1,1);
      \fill [fill=red, opacity=.5] (0,0) to [out=-90,in=-90] (.5,0) to [out=90,in=180] (1,1) -- (1,1.1) to [out=180,in=90] (.4,0) to [out=-90,in=-90] (0,0);
      \draw [semithick] (0,0) to [out=-90,in=-90] (.5,0) to [out=90,in=180] (1,1);
      \node at (0,0) {$\bullet$};
      \node at (-.3,0) {\tiny $-$};
    \end{tikzpicture}
  \qquad,\qquad
  \begin{tikzpicture}[anchorbase]
    \fill [fill=red,opacity=.5] (1,-1) to [out=180,in=-90] (0,0) to [out=90,in=180] (1,1) -- (1,.9) to [out=180,in=90] (0,0) to [out=-90,in=180] (1,-1.1);
    \draw[semithick] (1,-1) to [out=180,in=-90] (0,0) to [out=90,in=180] (1,1);
    \node at (0,0) {$\bullet$};
    \node at (-.3,0) {\tiny $+$};
  \end{tikzpicture}
  \quad \leftrightarrow
  \quad
    \begin{tikzpicture}[anchorbase]
      \fill [fill=red,opacity=.5] (1,-1) to [out=180,in=-90] (.5,0) to [out=90,in=90] (0,0) to [out=90,in=90] (.4,0) to [out=-90,in=180] (1,-1.1);
      \draw[semithick] (1,-1) to [out=180,in=-90] (.5,0) to [out=90,in=90] (0,0);
      \draw [line width=5, white] (0,0) to [out=-90,in=-90] (.5,0) to [out=90,in=180] (1,1);
      \fill [fill=red, opacity=.5] (0,0) to [out=-90,in=-90] (.5,0) to [out=90,in=180] (1,1) -- (1,.9) to [out=180,in=90] (.6,0) to [out=-90,in=-90] (0,0);
      \draw [semithick] (0,0) to [out=-90,in=-90] (.5,0) to [out=90,in=180] (1,1);
      \node at (0,0) {$\circ$};
      \node at (-.3,0) {\tiny $-$};
  \end{tikzpicture} \\
  \begin{tikzpicture}[anchorbase]
    \fill [fill=red,opacity=.5] (1,-1) to [out=180,in=-90] (0,0) to [out=90,in=180] (1,1) -- (1,1.1) to [out=180,in=90] (0,0) to [out=-90,in=180] (1,-.9);
    \draw[semithick] (1,-1) to [out=180,in=-90] (0,0) to [out=90,in=180] (1,1);
    \node at (0,0) {$\bullet$};
    \node at (-.3,0) {\tiny $-$};
  \end{tikzpicture}
  \quad \leftrightarrow
  \quad
    \begin{tikzpicture}[anchorbase]
      \draw [line width=5, white] (0,0) to [out=-90,in=-90] (.5,0) to [out=90,in=180] (1,1);
      \fill [fill=red, opacity=.5] (0,0) to [out=-90,in=-90] (.5,0) to [out=90,in=180] (1,1) -- (1,1.1) to [out=180,in=90] (.4,0) to [out=-90,in=-90] (0,0);
      \draw [semithick] (0,0) to [out=-90,in=-90] (.5,0) to [out=90,in=180] (1,1);
      \draw[line width=5, white] (1,-1) to [out=180,in=-90] (.5,0) to [out=90,in=90] (0,0);
      \fill [fill=red,opacity=.5] (1,-1) to [out=180,in=-90] (.5,0) to [out=90,in=90] (0,0) to [out=90,in=90] (.6,0) to [out=-90,in=180] (1,-.9);
      \draw[semithick] (1,-1) to [out=180,in=-90] (.5,0) to [out=90,in=90] (0,0);
      \node at (0,0) {$\circ$};
      \node at (-.3,0) {\tiny $+$};
    \end{tikzpicture}
  \qquad,\qquad
  \begin{tikzpicture}[anchorbase]
    \fill [fill=red,opacity=.5] (1,-1) to [out=180,in=-90] (0,0) to [out=90,in=180] (1,1) -- (1,.9) to [out=180,in=90] (0,0) to [out=-90,in=180] (1,-1.1);
    \draw[semithick] (1,-1) to [out=180,in=-90] (0,0) to [out=90,in=180] (1,1);
    \node at (0,0) {$\circ$};
    \node at (-.3,0) {\tiny $-$};
  \end{tikzpicture}
  \quad \leftrightarrow
  \quad
    \begin{tikzpicture}[anchorbase]
      \draw [line width=5, white] (0,0) to [out=-90,in=-90] (.5,0) to [out=90,in=180] (1,1);
      \fill [fill=red, opacity=.5] (0,0) to [out=-90,in=-90] (.5,0) to [out=90,in=180] (1,1) -- (1,.9) to [out=180,in=90] (.6,0) to [out=-90,in=-90] (0,0);
      \draw [semithick] (0,0) to [out=-90,in=-90] (.5,0) to [out=90,in=180] (1,1);
      \draw[line width=5, white] (1,-1) to [out=180,in=-90] (.5,0) to [out=90,in=90] (0,0);
      \fill [fill=red,opacity=.5] (1,-1) to [out=180,in=-90] (.5,0) to [out=90,in=90] (0,0) to [out=90,in=90] (.4,0) to [out=-90,in=180] (1,-1.1);
      \draw[semithick] (1,-1) to [out=180,in=-90] (.5,0) to [out=90,in=90] (0,0);
      \node at (0,0) {$\bullet$};
      \node at (-.3,0) {\tiny $+$};
  \end{tikzpicture} 
  \nonumber
\end{align}

\subsubsection{Isolated framing changes with annihilation of second derivative}

Let us now consider an isolated point $m\times \{t_0\}\in W_1\times \{t_0\}$ of coordinates $(u_0,v_0,t_0)$ where:
\[(df)m=\begin{pmatrix}  a & \lambda a &  \cdot \\ b & \lambda b & \cdot \\ c & \lambda c +\mu & \cdot \end{pmatrix},\;\mu \neq 0.\]
As before, we reduce the situation to:
\[(df)m=\begin{pmatrix}  a &0 &  \cdot \\ 0 & 0 & \cdot \\ c & \mu & \cdot \end{pmatrix},\;\mu \neq 0.\]
We have furthermore assumed that $\frac{\partial^2 f}{\partial u\partial v}$ has zero $y$ coordinate. We can assume that $\beta'$, the $y$-coordinate of $\frac{\partial^3 f}{\partial u^2\partial v}$ is non-zero, as the codimension again increases when requiring failure of this property. Let us assume, up to symmetry, that $\beta'>0$. Then one can write:

\[
  \frac{\partial f_{t_0}}{\partial v}(u_0+\epsilon_1,v_0)=f(u_0,v_0,t_0)+ \begin{pmatrix} \alpha \epsilon_1 \\ \beta' \epsilon_1^2 \\ \mu+\gamma \epsilon_1 \end{pmatrix}+o(\epsilon_1)
\]
At $t=t_0$, this draws:
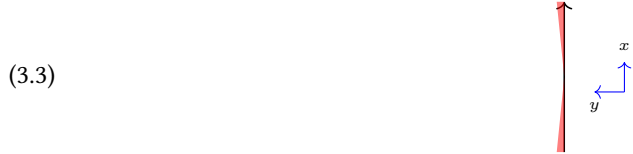
\begin{equation} \label{eq:FrChMorse}
  \begin{tikzpicture}[anchorbase] 
    \draw [semithick, ->] (0,0) -- (0,2);
    \fill [opacity=.5, fill=red] (-.1,0) -- (0,1) -- (-.1,2) -- (0,2) -- (0,0);
    \draw [blue,->] (.8,.8) -- (.8,1.2);
    \node at (.8,1.4) {\tiny $x$};
    \draw [blue,->] (.8,.8) -- (.4,.8);
    \node at (.4,.6) {\tiny $y$};
  \end{tikzpicture}
\end{equation}

Let us now look at the time direction. We claim that $\nu=\frac{\partial^2 f}{\partial v\partial t}$ can be assumed to be non-zero, again by codimension considerations. Thus one sees (up to a time drift in the $x$ and $z$ directions we haven't included):
\[
  \frac{\partial f}{\partial v}(u_0+\epsilon_1,v_0,t_0+\delta)\simeq f(u_0,v_0,t_0)+\begin{pmatrix} \alpha \epsilon_1  \\ \beta' \epsilon_1^2+\nu \delta \\ \mu+\gamma \epsilon_1 \end{pmatrix}
\]
Depending on the sign of $\delta$, one gets one or the other direction of the following moves (depending on the sign of $\mu$):
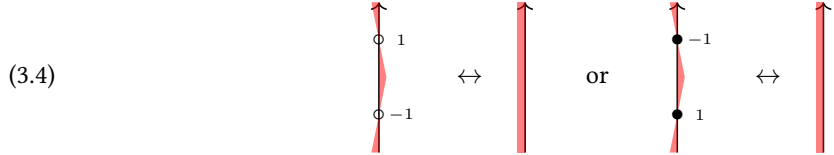
\begin{equation} \label{eq:FrChMorse2}
   \begin{tikzpicture}[anchorbase] 
    \draw [semithick, ->] (0,0) -- (0,2);
    \fill [opacity=.5, fill=red] (-.1,0) -- (.1,1) -- (-.1,2) -- (0,2) -- (0,0);
    \node at (0,1.5) {$\circ$};
    \node at (.3,1.5) {\tiny $1$};
    \node at (0,.5) {$\circ$};
    \node at (.3,.5) {\tiny $-1$};
  \end{tikzpicture}
  \quad \leftrightarrow \quad
   \begin{tikzpicture}[anchorbase] 
    \draw [semithick, ->] (0,0) -- (0,2);
    \fill [opacity=.5, fill=red] (-.1,0) -- (-.1,2) -- (0,2) -- (0,0);
  \end{tikzpicture}
  \qquad \text{or} \qquad
    \begin{tikzpicture}[anchorbase] 
    \draw [semithick, ->] (0,0) -- (0,2);
    \fill [opacity=.5, fill=red] (-.1,0) -- (.1,1) -- (-.1,2) -- (0,2) -- (0,0);
    \node at (0,1.5) {$\bullet$};
    \node at (.3,1.5) {\tiny $-1$};
    \node at (0,.5) {$\bullet$};
    \node at (.3,.5) {\tiny $1$};
  \end{tikzpicture}
  \quad \leftrightarrow \quad
   \begin{tikzpicture}[anchorbase] 
    \draw [semithick, ->] (0,0) -- (0,2);
    \fill [opacity=.5, fill=red] (-.1,0) -- (-.1,2) -- (0,2) -- (0,0);
  \end{tikzpicture}
\end{equation}

With $\beta'<0$ one gets:
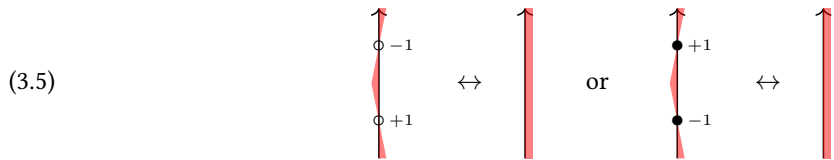
\begin{equation} \label{eq:FrChMorse3}
   \begin{tikzpicture}[anchorbase] 
    \draw [semithick, ->] (0,0) -- (0,2);
    \fill [opacity=.5, fill=red] (.1,0) -- (-.1,1) -- (.1,2) -- (0,2) -- (0,0);
    \node at (0,1.5) {$\circ$};
    \node at (.3,1.5) {\tiny $-1$};
    \node at (0,.5) {$\circ$};
    \node at (.3,.5) {\tiny $+1$};
  \end{tikzpicture}
  \quad \leftrightarrow \quad
   \begin{tikzpicture}[anchorbase] 
    \draw [semithick, ->] (0,0) -- (0,2);
    \fill [opacity=.5, fill=red] (.1,0) -- (.1,2) -- (0,2) -- (0,0);
  \end{tikzpicture}
  \qquad \text{or} \qquad
   \begin{tikzpicture}[anchorbase] 
    \draw [semithick, ->] (0,0) -- (0,2);
    \fill [opacity=.5, fill=red] (.1,0) -- (-.1,1) -- (.1,2) -- (0,2) -- (0,0);
    \node at (0,1.5) {$\bullet$};
    \node at (.3,1.5) {\tiny $+1$};
    \node at (0,.5) {$\bullet$};
    \node at (.3,.5) {\tiny $-1$};
  \end{tikzpicture}
  \quad \leftrightarrow \quad
   \begin{tikzpicture}[anchorbase] 
    \draw [semithick, ->] (0,0) -- (0,2);
    \fill [opacity=.5, fill=red] (.1,0) -- (.1,2) -- (0,2) -- (0,0);
  \end{tikzpicture}
\end{equation}

\subsubsection{Isolated points on $W_0$ with framing breach}

Assume that $m=(u_0,v_0,t_0)\in W_0\times \{t_0\}$ with:
\[
  (df)_m=\begin{pmatrix}
    a & \lambda a & \cdot \\ b & \lambda b & \cdot \\ c & \lambda c +\mu &\cdot 
  \end{pmatrix}
  \quad \mu \neq 0.
\]
We again reduce to:
\[(df)m=\begin{pmatrix}  a &0 &  \cdot \\ 0 & 0 & \cdot \\ c & \mu & \cdot \end{pmatrix},\;\mu \neq 0,\; a>0.\]

We consider the case of split vertex (a merge vertex would be treated similarly), and for convenience we number the two outgoing strands $1$ and $2$ so that the $v$-coordinate on strand $1$ is positive, while the $v$-coordinate on strand $2$ is negative.

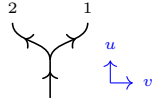
\[
  \begin{tikzpicture}[anchorbase,decoration={markings, mark=at position 0.5 with {\arrow{<}}; }]
    \draw [semithick, postaction={decorate}] (0,1) to [out=-90,in=90] (.5,.5);
    \draw [semithick, postaction={decorate}] (1,1) to [out=-90,in=90] (.5,.5);
    \draw [semithick, postaction={decorate}] (.5,.5) to [out=-90,in=90] (.5,0);
      \draw [blue,->] (1.3,.2) -- (1.6,.2);
    \node [blue] at (1.8,.2) {\tiny $v$};
    \draw [blue,->] (1.3,.2) -- (1.3,.5);
    \node [blue] at (1.3,.7) {\tiny $u$};
    \node at (1,1.2) {\tiny $1$};
    \node at (0,1.2) {\tiny $2$};
  \end{tikzpicture}
\]

Denote:
\[
  \frac{\partial^2 f}{\partial u \partial v}=\begin{pmatrix} d_1 \\ d_2 \\ d_3
  \end{pmatrix}\quad,\quad 
  \frac{\partial^2 f}{\partial t \partial v}=\begin{pmatrix} d_4 \\ d_5 \\ d_6
  \end{pmatrix}\quad,\quad
    \frac{\partial^2 f}{ \partial v^2}=\begin{pmatrix} d_7 \\ d_8 \\ d_9
  \end{pmatrix}
\]
One can assume that $d_2$ and $d_5$ are non-zero. Let us first look at the shape of the web through time. Since at first order $\frac{\partial f}{\partial t}$ only produces a drift of the web through time, we can ignore its contribution. Below it is treated as zero, as well as all purely-$t$ higher derivatives. Similarly the mixed derivative in $\partial u\partial t$ will not contribute to the change in shape of the web.
\[
  f(u_0+\epsilon_1,v_0+\epsilon_2,t_0+\delta)\sim f(u_0,v_0,t_0)+\begin{pmatrix} a \epsilon_1 +d_1\epsilon_1\epsilon_2+d_4\epsilon_2\delta+d_7\epsilon_2^2\\d_2\epsilon_1\epsilon_2 +d_5\epsilon_2\delta+d_8\epsilon_2^2\\ c \epsilon_1+\mu \epsilon_2 +d_3\epsilon_1\epsilon_2+d_6\epsilon_2\delta+d_0\epsilon_2^2 \end{pmatrix}
\]
Let's analyze the expression in the $y$ coordinate, and first consider $\delta=0$. $d_8\epsilon_2^2$ will not change the general shape of the image surface. $d_2\epsilon_1\epsilon_2$ determines whether the leg number $1$ is sent to positive or negative $y$ coordinates (resp., leg numbered $2$ being sent to negative or positive coordinates). Assume $d_2>0$. Thus at $\delta=0$ we simply observe:
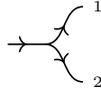
\[
  \begin{tikzpicture}[anchorbase,decoration={markings, mark=at position 0.5 with {\arrow{>}}; }]
    \draw [semithick, postaction={decorate}] (0,.5) -- (.5,.5);
    \draw [semithick, postaction={decorate}] (.5,.5) to [out=0,in=180] (1,1);
    \node at (1.2,0) {\tiny $2$};
    \draw [semithick, postaction={decorate}] (.5,.5) to [out=0,in=180] (1,0);
    \node at (1.2,1) {\tiny $1$};
  \end{tikzpicture}
\]
In the time direction we add a contribution of $d_5\epsilon_2\delta$. Assume $d_5>0$ (up to time reversion one can reduce to this case). Then the strand $1$ is pushed higher in the $y$ direction and the strand $2$ is pushed lower in the $y$ direction if $\delta>0$, while for $\delta<0$ the strand $1$ is pushed lower in the $y$ direction and the strand numbered $2$ higher in the $y$ direction. This phenomenon is predominant while $\epsilon_1<<\delta$, but as $\epsilon_1$ grows, $d_2\epsilon_1\epsilon_2$ controls the shape of the web. One thus reads (the first picture corresponds to $\mu>0$, the second one to $\mu <0$):
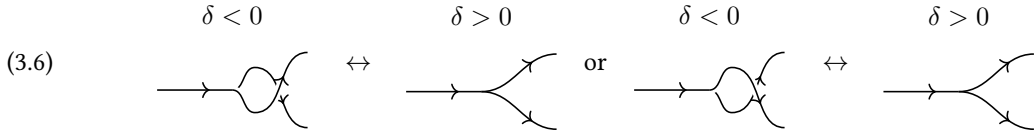
\begin{equation} \label{eq:twistFork}

\end{equation}
Indeed the crossing sign is determined by $\mu \epsilon_2$: the strand labeled $2$ corresponds to negative values of $\epsilon_2$, and thus the sign of $\mu<0$ determines which of the strands is pushed up in the $z$ direction.

We now look at the framing vector and read:
\[
(  \frac{\partial f}{\partial v})(u_0+\epsilon_2,v_0+\epsilon_2,t_0+\delta)\sim%
\]
Recall that we assumed $d_5>0$, so along time the framing travels positively in the $y$ direction. Furthermore, at $\delta=0$, it is controlled by $d_2\epsilon_1$, and we have already assumed that $d_2>0$. One thus has:

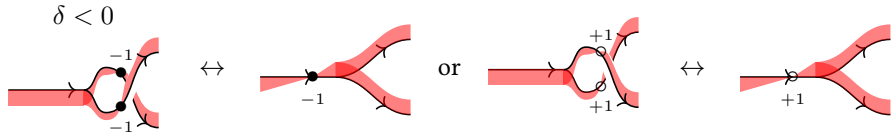
\[
    %
  \]

Similarly, one gets for $d_2<0$ the following moves:

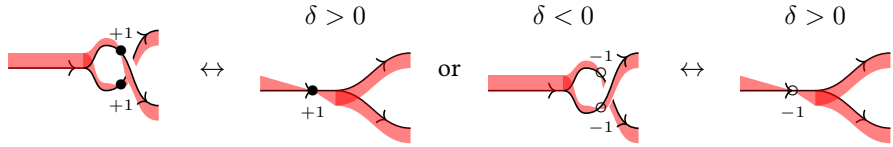
\[
    %
  \]

Note that the orientation of the strands play no role, thus one gets the moves involving a merge vertex by simply reversing the orientations.

\subsubsection{Isolated double points with framing breach}

The analysis is similar to the previous one, but simpler. One gets (we haven't shown orientation):

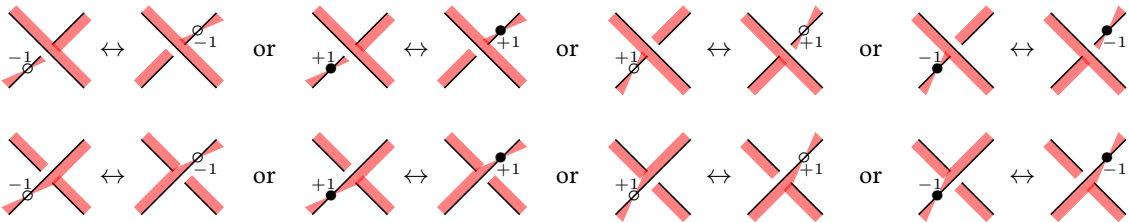
\[
  %
\]

\subsubsection{Statements}

Wrapping up the previous analysis, one gets:

\begin{theorem} \label{th:ReidHTwist}
  Oriented, framed web-tangles admit diagrams locally generated by the following pieces (orientations are to be added in any compatible way):
  \begin{gather*}
      %
\end{gather*}

Isotopies of oriented, framed web-tangles induce planar isotopies together a finite number of moves from the following list (in pictures not containing framing information, this can be freely chosen without framing changes):
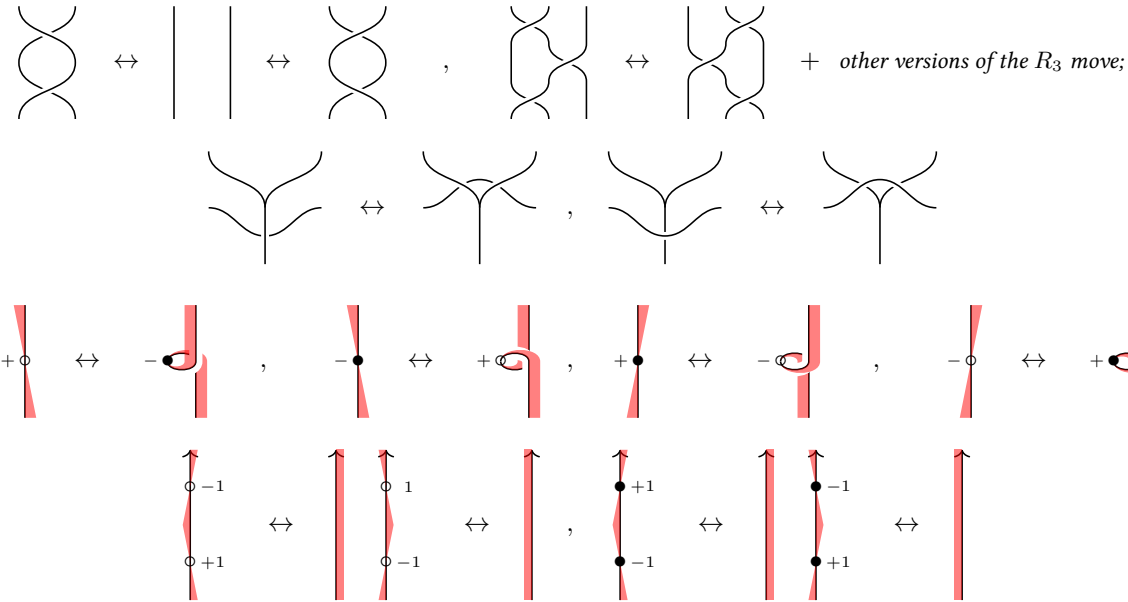
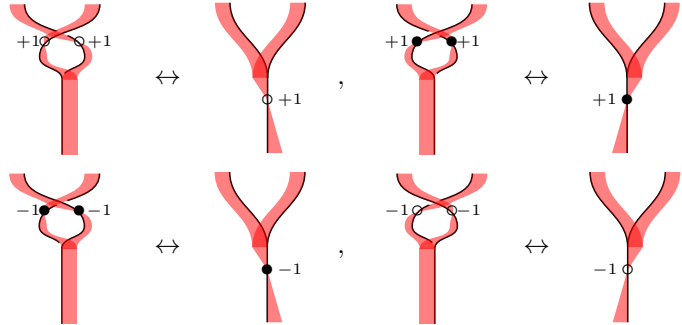
\begin{equation*}
  %
   \quad +\text{\small \; other versions of the $R_{3}$ move; }
\end{equation*}

\begin{equation*}
    %
\]

\end{theorem}
\begin{proof}
Going from the multi-jet bundle stratification to the statement of the theorem is a standard argument in 
geometric topology, but we will review the high-level structure of that argument here.

Start with an arbitrary framed web-tangle.
After a small isotopy, we may assume that the multi-jet map associated to the web-tangle is transverse to the 
stratification (collection of submanifolds) described above (because the set of transverse maps is dense).
As explained earlier in this section, 
near each transverse intersection point the corresponding web-tangle projection looks like one of the generators
listed in the first part of the theorem (strand, half-twist, crossing, and trivalent vertex).
This proves the first part of the theorem.

Now consider an arbitrary isotopy of web-tangles, with the beginning and end of the isotopy already transverse (i.e.\ the multijet maps associated to the beginning and end of the isotopy are transverse to the jet-bundle stratification described above).
Because the set of transverse maps is dense, we can make a small perturbation (second-order isotopy) to the isotopy such that the new isotopy is transverse.
We now examine the possible transverse intersections.
Each such intersection yields one of the framed Reidemeister moves listed in the second part of the theorem.
\end{proof}

Recall that away from the points where the framing points downward, we have a preferred isotopy that makes the framing flat and rightward-pointing with respect to the orientation of the web. Notice that in the case of a trivalent, this changes the shape of the web:
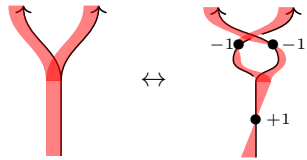
\[
 \begin{tikzpicture}[anchorbase,rotate=90]
    \draw [semithick,->] (0,0) to [out=0,in=180] (1,.5);
    \draw [semithick] (-1,0) -- (0,0);
    \draw [semithick, ->] (0,0) to [out=0,in=180] (1,-.5);
    \fill [red,opacity=.5] (-1,0) -- (-1,.2) -- (0,.2) -- (0,0);
    \fill [red,opacity=.5] (0,0) to [out=0,in=180] (1,.5) -- (1,.7) to [out=180,in=0] (0,.2);
    \fill [red,opacity=.5] (0,0) to [out=0,in=180] (1,-.5) -- (1,-.3) to [out=180,in=0] (0,.2);
  \end{tikzpicture}
  \quad \leftrightarrow \quad
\begin{tikzpicture}[anchorbase,rotate=90]
    \draw [semithick, ->] (0,0) to [out=0,in=180] (.3,.3) to [out=0,in=180] (1,-.5);
    \draw [white, line width=3] (0,0) to [out=0,in=180] (.3,-.3) to [out=0,in=180] (1,.5);
    \draw [semithick] (-1,0) -- (0,0);
    \draw [semithick, ->] (0,0) to [out=0,in=180] (.3,-.3) to [out=0,in=180] (1,.5);
    \fill [red,opacity=.5] (-1,0) -- (0,0) -- (0,-.2) -- (-1,.2);
    \fill [red,opacity=.5] (0,0) to [out=0,in=180] (.3,.3) to [out=0,in=180] (1,-.5) -- (1,-.3) to [out=180,in=-60] (.6,.2) to [out=120,in=0] (.3,.2) to [out=180,in=0] (0,-.2);
    \fill [red, opacity=.5] (0,0) to [out=0,in=180] (.3,-.3) to [out=0,in=180] (1,.5) -- (1,.7) to [out=180,in=60] (.6,.2) to [out=-120,in=0] (.3,-.4) to [out=180,in=0] (0,-.2);
    \node at (-.5,0) {$\bullet$};
    \node at (-.5,-.3) {\tiny $+1$};
   \node at (.5,-.23) {$\bullet$};
    \node at (.5,-.5) {\tiny $-1$};
    \node at (.5,.23) {$\bullet$};
    \node at (.5,.45) {\tiny $-1$};
      \end{tikzpicture}
  \]
One can thus translate the previous theorem into the following one (where we do not show the framing anymore, as it is on the right hand side of the strand at all times, except for the twists). For full twists, we use the symbol $\LEFTcircle$.

\begin{theorem} \label{th:ReidFT}
Oriented, framed web-tangles admit preferred diagrams locally generated  by the following pieces: 
  
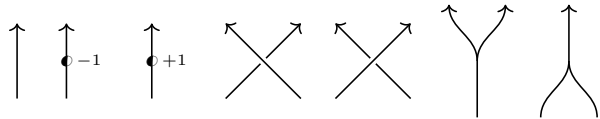
\begin{gather*}
      \begin{tikzpicture}[anchorbase] 
    \draw [semithick,->] (0,0) -- (0,1);
  \end{tikzpicture}
  \quad
  \begin{tikzpicture}[anchorbase] 
    \draw [semithick,->] (0,0) -- (0,1);
    \node at (0,.5) {\tiny $\LEFTcircle$};
    \node at (.3,.5) {\tiny $-1$};
  \end{tikzpicture}
  \quad
  \begin{tikzpicture}[anchorbase] 
    \draw [semithick,->] (0,0) -- (0,1);
    \node at (0,.5) {\tiny $\LEFTcircle$};
    \node at (.3,.5) {\tiny $+1$};
  \end{tikzpicture}
\quad
  \begin{tikzpicture}[anchorbase]
    \draw [semithick,->] (0,0) -- (1,1);
    \draw [white, line width=3] (1,0) -- (0,1);
    \draw [semithick,->] (1,0) -- (0,1);
  \end{tikzpicture}
\quad
  \begin{tikzpicture}[anchorbase,rotate=90]
    \draw [semithick,->] (0,0) -- (1,1);
    \draw [white, line width=3] (1,0) -- (0,1);
    \draw [semithick,<-] (1,0) -- (0,1);
  \end{tikzpicture}
\quad
    \begin{tikzpicture}[anchorbase,rotate=90,scale=.75]
    \draw [semithick,->] (0,0) to [out=0,in=180] (1,.5);
    \draw [semithick] (-1,0) -- (0,0);
    \draw [semithick,->] (0,0) to [out=0,in=180] (1,-.5);
  \end{tikzpicture}
  \quad
    \begin{tikzpicture}[anchorbase,rotate=-90,scale=.75]
    \draw [semithick] (0,0) to [out=0,in=180] (1,.5);
    \draw [semithick,<-] (-1,0) -- (0,0);
    \draw [semithick] (0,0) to [out=0,in=180] (1,-.5);
   \end{tikzpicture}
\end{gather*}

Isotopies of oriented, framed web-tangles induce planar isotopies together a finite number of moves from the following list (in pictures not containing framing information, this can be freely chosen without framing changes):
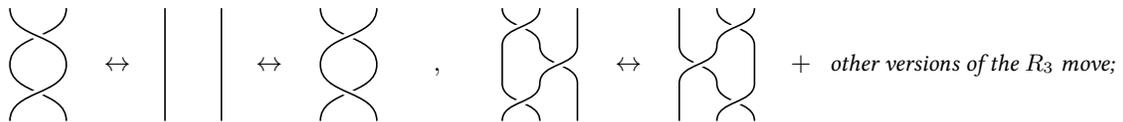
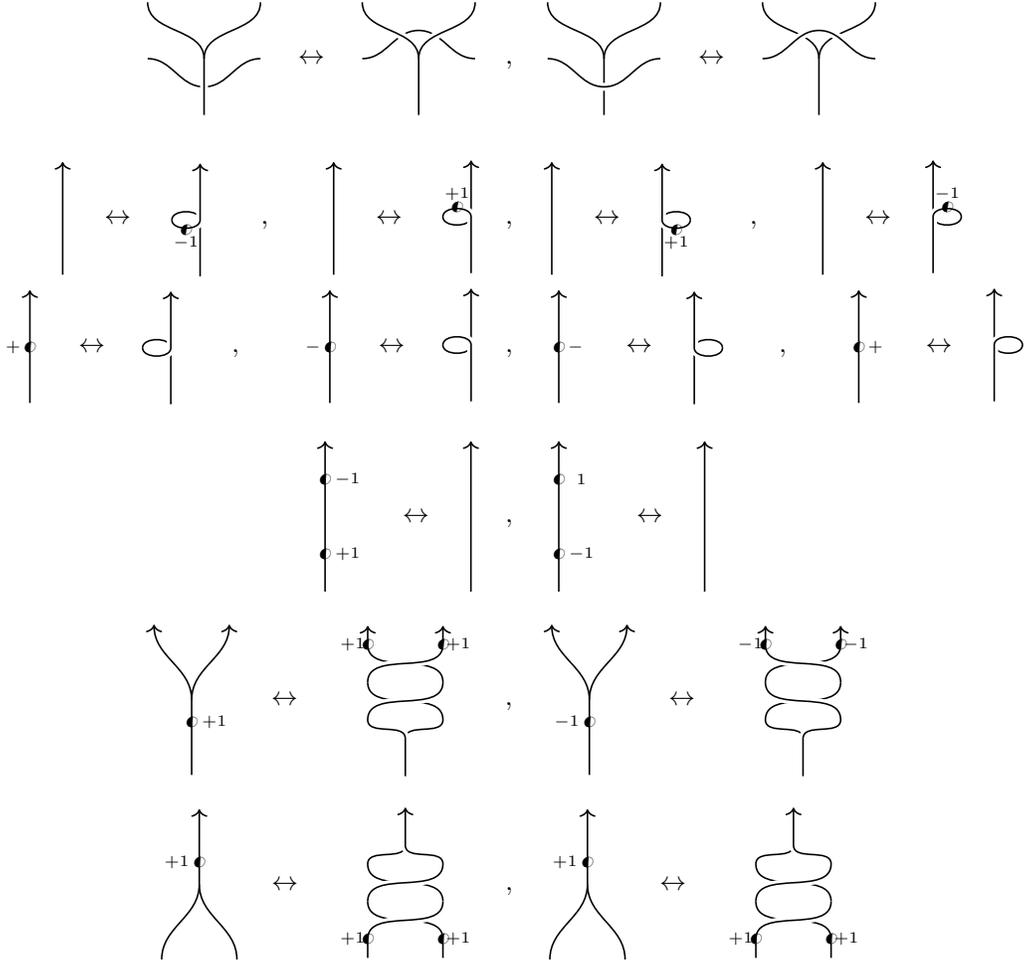
\begin{equation*}
  \begin{tikzpicture}[anchorbase,scale=.75]
    \draw [semithick] (1,0) to [out=90,in=-90] (0,1) to [out=90,in=-90] (1,2);
    \draw [white, line width=3] (0,0) to [out=90,in=-90] (1,1) to [out=90,in=-90] (0,2);
     \draw [semithick] (0,0) to [out=90,in=-90] (1,1) to [out=90,in=-90] (0,2);
   \end{tikzpicture}
   \quad \leftrightarrow \quad
   \begin{tikzpicture}[anchorbase,scale=.75]
    \draw [semithick] (0,0) -- (0,2);
    \draw [semithick] (1,0) -- (1,2);
  \end{tikzpicture}
  \quad \leftrightarrow \quad
  \begin{tikzpicture}[anchorbase,scale=.75]
    \draw [semithick] (0,0) to [out=90,in=-90] (1,1) to [out=90,in=-90] (0,2);
    \draw [line width=3, white] (1,0) to [out=90,in=-90] (0,1) to [out=90,in=-90] (1,2);
    \draw [semithick] (1,0) to [out=90,in=-90] (0,1) to [out=90,in=-90] (1,2);
  \end{tikzpicture}
   \qquad,\qquad
   \begin{tikzpicture}[anchorbase,scale=.5]
     \draw [semithick] (2,0) -- (2,1) to [out=90,in=-90] (1,2) to [out=90,in=-90] (0,3);
     \draw [white, line width=3] (1,0) to [out=90,in=-90] (0,1) -- (0,2) to [out=90,in=-90] (1,3);
     \draw [semithick] (1,0) to [out=90,in=-90] (0,1) -- (0,2) to [out=90,in=-90] (1,3);
     \draw [white, line width=3] (0,0) to [out=90,in=-90] (1,1) to [out=90,in=-90] (2,2) -- (2,3);
     \draw [semithick] (0,0) to [out=90,in=-90] (1,1) to [out=90,in=-90] (2,2) -- (2,3);
   \end{tikzpicture}   
   \quad \leftrightarrow \quad
   \begin{tikzpicture}[anchorbase,scale=.5]
     \draw [semithick] (2,0) to [out=90,in=-90] (1,1) to [out=90,in=-90] (0,2) -- (0,3);
     \draw [white, line width=3] (1,0) to [out=90,in=-90] (2,1) -- (2,2) to [out=90,in=-90] (1,3);
     \draw [semithick] (1,0) to [out=90,in=-90] (2,1) -- (2,2) to [out=90,in=-90] (1,3);
     \draw [white, line width=3] (0,0) -- (0,1)  to [out=90,in=-90] (1,2) to [out=90,in=-90] (2,3);
     \draw [semithick] (0,0) -- (0,1) to [out=90,in=-90] (1,2) to [out=90,in=-90] (2,3);
   \end{tikzpicture}
   \quad +\text{\small \; other versions of the $R_{3}$ move; }
\end{equation*}

\begin{equation*}
    \begin{tikzpicture}[anchorbase,scale=.75]
    % strand
    \draw [semithick] (-1,0) to [out=0,in=180] (0,-.5) to [out=0,in=180]  (1,0);
    % web
    \draw [line width=3, white] (0,-1) -- (0,0); 
    \draw [line width=3, white] (0,0) to [out=90,in=-90] (1,1);
    \draw [line width=3, white] (0,0) to [out=90,in=-90] (-1,1);
    \draw [semithick] (0,-1) -- (0,0);
    \draw [semithick] (0,0) to [out=90,in=-90] (1,1);
    \draw [semithick] (0,0) to [out=90,in=-90] (-1,1);
  \end{tikzpicture}
  \quad \leftrightarrow
  \quad
    \begin{tikzpicture}[anchorbase, scale=.75]
    % strand
    \draw [semithick] (-1,0) to [out=0,in=180] (0,.5) to [out=0,in=180]  (1,0);
    % web
    \draw [line width=3, white] (0,-1) -- (0,0); 
    \draw [line width=3, white] (0,0) to [out=90,in=-90] (1,1);
    \draw [line width=3, white] (0,0) to [out=90,in=-90] (-1,1);
    \draw [semithick] (0,-1) -- (0,0);
    \draw [semithick] (0,0) to [out=90,in=-90] (1,1);
    \draw [semithick] (0,0) to [out=90,in=-90] (-1,1);
  \end{tikzpicture}
\quad , \quad
  \begin{tikzpicture}[anchorbase, scale=.75]
    % web
    \draw [line width=3, white] (0,-1) -- (0,0); 
    \draw [line width=3, white] (0,0) to [out=90,in=-90] (1,1);
    \draw [line width=3, white] (0,0) to [out=90,in=-90] (-1,1);
    \draw [semithick] (0,-1) -- (0,0);
    \draw [semithick] (0,0) to [out=90,in=-90] (1,1);
    \draw [semithick] (0,0) to [out=90,in=-90] (-1,1);
    % strand
    \draw [line width=3, white] (-1,0) to [out=0,in=180] (0,-.5) to [out=0,in=180]  (1,0);
    \draw [semithick] (-1,0) to [out=0,in=180] (0,-.5) to [out=0,in=180]  (1,0);
  \end{tikzpicture}
  \quad \leftrightarrow \quad
  \begin{tikzpicture}[anchorbase, scale=.75]
    % web
    \draw [line width=3, white] (0,-1) -- (0,0); 
    \draw [line width=3, white] (0,0) to [out=90,in=-90] (1,1);
    \draw [line width=3, white] (0,0) to [out=90,in=-90] (-1,1);
    \draw [semithick] (0,-1) -- (0,0);
    \draw [semithick] (0,0) to [out=90,in=-90] (1,1);
    \draw [semithick] (0,0) to [out=90,in=-90] (-1,1);
    % strand
    \draw [line width=3, white] (-1,0) to [out=0,in=180] (0,.5) to [out=0,in=180]  (1,0);
    \draw [semithick] (-1,0) to [out=0,in=180] (0,.5) to [out=0,in=180]  (1,0);
  \end{tikzpicture}
\end{equation*}

\begin{gather*}
  %% 1
  \begin{tikzpicture}[anchorbase,scale=.75]
    \draw[semithick,->] (0,-1) -- (0,1);
  \end{tikzpicture}
  \quad \leftrightarrow
  \quad
    \begin{tikzpicture}[anchorbase,scale=.75]
      \draw[semithick] (.5,-1) to [out=90,in=-90] (.5,0) to [out=90,in=90] (0,0);
      \draw [line width=3, white] (0,0) to [out=-90,in=-90] (.5,0) to [out=90,in=-90] (.5,1);
      \draw [semithick,->] (0,0) to [out=-90,in=-90] (.5,0) to [out=90,in=-90] (.5,1);
      \node at (.25,-.18) {\tiny $\LEFTcircle$};
      \node at (.25,-.4) {\tiny $-1$};
  \end{tikzpicture}
  \qquad,\qquad
  %% 2
  \begin{tikzpicture}[anchorbase,scale=.75,yscale=-1]
    \draw[semithick,<-] (0,-1) -- (0,1);
  \end{tikzpicture}
  \quad \leftrightarrow
  \quad
    \begin{tikzpicture}[anchorbase,scale=.75,yscale=-1]
      \draw[semithick,<-] (.5,-1) to [out=90,in=-90] (.5,0) to [out=90,in=90] (0,0);
      \draw [line width=3, white] (0,0) to [out=-90,in=-90] (.5,0) to [out=90,in=-90] (.5,1);
      \draw [semithick,] (0,0) to [out=-90,in=-90] (.5,0) to [out=90,in=-90] (.5,1);
      \node at (.25,-.18) {\tiny $\LEFTcircle$};
      \node at (.25,-.4) {\tiny $+1$};
  \end{tikzpicture}
  \quad,\quad
  %% 3
  \begin{tikzpicture}[anchorbase,scale=.75,xscale=-1]
    \draw[semithick,->] (0,-1) -- (0,1);
     \end{tikzpicture}
  \quad \leftrightarrow
  \quad
    \begin{tikzpicture}[anchorbase,scale=.75,xscale=-1]
      \draw[semithick] (.5,-1) to [out=90,in=-90] (.5,0) to [out=90,in=90] (0,0);
      \draw [line width=3, white] (0,0) to [out=-90,in=-90] (.5,0) to [out=90,in=-90] (.5,1);
      \draw [semithick,->] (0,0) to [out=-90,in=-90] (.5,0) to [out=90,in=-90] (.5,1);
         \node at (.25,-.18) {\tiny $\LEFTcircle$};
      \node at (.25,-.4) {\tiny $+1$};
  \end{tikzpicture}
  \qquad,\qquad
  %%4
  \begin{tikzpicture}[anchorbase,scale=-.75]
    \draw[semithick,<-] (0,-1) -- (0,1);
  \end{tikzpicture}
  \quad \leftrightarrow
  \quad
    \begin{tikzpicture}[anchorbase,scale=-.75]
      \draw[semithick,<-] (.5,-1) to [out=90,in=-90] (.5,0) to [out=90,in=90] (0,0);
      \draw [line width=3, white] (0,0) to [out=-90,in=-90] (.5,0) to [out=90,in=-90] (.5,1);
      \draw [semithick] (0,0) to [out=-90,in=-90] (.5,0) to [out=90,in=-90] (.5,1);
      \node at (.25,-.18) {\tiny $\LEFTcircle$};
      \node at (.25,-.4) {\tiny $-1$};
  \end{tikzpicture}
  \\
  %% 1
  \begin{tikzpicture}[anchorbase,scale=.75]
    \draw[semithick,->] (0,-1) -- (0,1);
    \node at (0,0) {\tiny $\LEFTcircle$};
    \node at (-.3,0) {\tiny $+$};
  \end{tikzpicture}
  \quad \leftrightarrow
  \quad
    \begin{tikzpicture}[anchorbase,scale=.75]
      \draw[semithick] (.5,-1) to [out=90,in=-90] (.5,0) to [out=90,in=90] (0,0);
      \draw [line width=3, white] (0,0) to [out=-90,in=-90] (.5,0) to [out=90,in=-90] (.5,1);
      \draw [semithick,->] (0,0) to [out=-90,in=-90] (.5,0) to [out=90,in=-90] (.5,1);
  \end{tikzpicture}
  \qquad,\qquad
  %% 2
  \begin{tikzpicture}[anchorbase,scale=.75,yscale=-1]
    \draw[semithick,<-] (0,-1) -- (0,1);
    \node at (0,0) {\tiny $\LEFTcircle$};
    \node at (-.3,0) {\tiny $-$};
  \end{tikzpicture}
  \quad \leftrightarrow
  \quad
    \begin{tikzpicture}[anchorbase,scale=.75,yscale=-1]
      \draw[semithick,<-] (.5,-1) to [out=90,in=-90] (.5,0) to [out=90,in=90] (0,0);
      \draw [line width=3, white] (0,0) to [out=-90,in=-90] (.5,0) to [out=90,in=-90] (.5,1);
      \draw [semithick,] (0,0) to [out=-90,in=-90] (.5,0) to [out=90,in=-90] (.5,1);
  \end{tikzpicture}
  \quad,\quad
  %% 3
  \begin{tikzpicture}[anchorbase,scale=.75,xscale=-1]
    \draw[semithick,->] (0,-1) -- (0,1);
    \node at (0,0) {\tiny $\LEFTcircle$};
    \node at (-.3,0) {\tiny $-$};
     \end{tikzpicture}
  \quad \leftrightarrow
  \quad
    \begin{tikzpicture}[anchorbase,scale=.75,xscale=-1]
      \draw[semithick] (.5,-1) to [out=90,in=-90] (.5,0) to [out=90,in=90] (0,0);
      \draw [line width=3, white] (0,0) to [out=-90,in=-90] (.5,0) to [out=90,in=-90] (.5,1);
      \draw [semithick,->] (0,0) to [out=-90,in=-90] (.5,0) to [out=90,in=-90] (.5,1);
  \end{tikzpicture}
  \qquad,\qquad
  %%4
  \begin{tikzpicture}[anchorbase,scale=-.75]
    \draw[semithick,<-] (0,-1) -- (0,1);
    \node at (0,0) {\tiny $\LEFTcircle$};
    \node at (-.3,0) {\tiny $+$};
  \end{tikzpicture}
  \quad \leftrightarrow
  \quad
    \begin{tikzpicture}[anchorbase,scale=-.75]
      \draw[semithick,<-] (.5,-1) to [out=90,in=-90] (.5,0) to [out=90,in=90] (0,0);
      \draw [line width=3, white] (0,0) to [out=-90,in=-90] (.5,0) to [out=90,in=-90] (.5,1);
      \draw [semithick] (0,0) to [out=-90,in=-90] (.5,0) to [out=90,in=-90] (.5,1);
  \end{tikzpicture}
\end{gather*}

\[
   \begin{tikzpicture}[anchorbase] 
    \draw [semithick, ->] (0,0) -- (0,2);
    \node at (0,1.5) {\tiny $\LEFTcircle$};
    \node at (.3,1.5) {\tiny $-1$};
    \node at (0,.5) {\tiny $\LEFTcircle$};
    \node at (.3,.5) {\tiny $+1$};
  \end{tikzpicture}
  \quad \leftrightarrow \quad
   \begin{tikzpicture}[anchorbase] 
    \draw [semithick, ->] (0,0) -- (0,2);
  \end{tikzpicture}
\quad,\quad
   \begin{tikzpicture}[anchorbase] 
    \draw [semithick, ->] (0,0) -- (0,2);
    \node at (0,1.5) {\tiny $\LEFTcircle$};
    \node at (.3,1.5) {\tiny $1$};
    \node at (0,.5) {\tiny $\LEFTcircle$};
    \node at (.3,.5) {\tiny $-1$};
  \end{tikzpicture}
  \quad \leftrightarrow \quad
   \begin{tikzpicture}[anchorbase] 
    \draw [semithick, ->] (0,0) -- (0,2);
  \end{tikzpicture}
\]

\[
  \begin{tikzpicture}[anchorbase,rotate=90]
    \draw [semithick, ->] (0,0) to [out=0,in=180] (1,.5);
    \draw [semithick] (-1,0) -- (0,0);
    \draw [semithick, ->] (0,0) to [out=0,in=180] (1,-.5);
    \node at (-.3,0) {\tiny $\LEFTcircle$};
    \node at (-.3,-.3) {\tiny $+1$};
  \end{tikzpicture}
  \quad \leftrightarrow \quad
  \begin{tikzpicture}[anchorbase,rotate=90]
      \draw [semithick] (-.5,0) to [out=0,in=180] (-.25,-.5) to [out=0,in=180] (.25,.5);
      \draw [white, line width=3] (-.5,0) to [out=0,in=180] (-.25,.5) to [out=0,in=180] (.25,-.5);
      \draw [semithick] (-.5,0) to [out=0,in=180] (-.25,.5) to [out=0,in=180] (.25,-.5);
      \draw [semithick,->] (.25,-.5) to [out=0,in=180] (.75,.5) -- (1,.5);
      \draw [white, line width=3] (.25,.5) to [out=0,in=180] (.75,-.5) -- (1,-.5);
      \draw [semithick,->] (.25,.5) to [out=0,in=180] (.75,-.5) -- (1,-.5);
      \draw [semithick] (-1,0) -- (-.5,0);
  \node at (.75,-.5) {\tiny $\LEFTcircle$};
    \node at (.75,-.7) {\tiny $+1$};
    \node at (.75,.5) {\tiny $\LEFTcircle$};
    \node at (.75,.7) {\tiny $+1$};
      \end{tikzpicture}
  \quad,\quad
  \begin{tikzpicture}[anchorbase,rotate=-90, xscale=-1]
    \draw [semithick, ->] (0,0) to [out=0,in=180] (1,.5);
    \draw [semithick] (-1,0) -- (0,0);
    \draw [semithick, ->] (0,0) to [out=0,in=180] (1,-.5);
    \node at (-.3,0) {\tiny $\LEFTcircle$};
    \node at (-.3,-.3) {\tiny $-1$};
  \end{tikzpicture}
  \quad \leftrightarrow \quad
  \begin{tikzpicture}[anchorbase,rotate=-90,xscale=-1]
      \draw [semithick] (-.5,0) to [out=0,in=180] (-.25,-.5) to [out=0,in=180] (.25,.5);
      \draw [white, line width=3] (-.5,0) to [out=0,in=180] (-.25,.5) to [out=0,in=180] (.25,-.5);
      \draw [semithick] (-.5,0) to [out=0,in=180] (-.25,.5) to [out=0,in=180] (.25,-.5);
      \draw [semithick,->] (.25,-.5) to [out=0,in=180] (.75,.5) -- (1,.5);
      \draw [white, line width=3] (.25,.5) to [out=0,in=180] (.75,-.5) -- (1,-.5);
      \draw [semithick,->] (.25,.5) to [out=0,in=180] (.75,-.5) -- (1,-.5);
      \draw [semithick] (-1,0) -- (-.5,0);
  \node at (.75,-.5) {\tiny $\LEFTcircle$};
    \node at (.75,-.7) {\tiny $-1$};
    \node at (.75,.5) {\tiny $\LEFTcircle$};
    \node at (.75,.7) {\tiny $-1$};
      \end{tikzpicture}
    \]

    \[
  \begin{tikzpicture}[anchorbase,rotate=-90]
    \draw [semithick] (0,0) to [out=0,in=180] (1,.5);
    \draw [semithick,<-] (-1,0) -- (0,0);
    \draw [semithick] (0,0) to [out=0,in=180] (1,-.5);
    \node at (-.3,0) {\tiny $\LEFTcircle$};
    \node at (-.3,-.3) {\tiny $+1$};
  \end{tikzpicture}
  \quad \leftrightarrow \quad
  \begin{tikzpicture}[anchorbase,rotate=-90]
      \draw [semithick] (-.5,0) to [out=0,in=180] (-.25,-.5) to [out=0,in=180] (.25,.5);
      \draw [white, line width=3] (-.5,0) to [out=0,in=180] (-.25,.5) to [out=0,in=180] (.25,-.5);
      \draw [semithick] (-.5,0) to [out=0,in=180] (-.25,.5) to [out=0,in=180] (.25,-.5);
      \draw [semithick] (.25,-.5) to [out=0,in=180] (.75,.5) -- (1,.5);
      \draw [white, line width=3] (.25,.5) to [out=0,in=180] (.75,-.5) -- (1,-.5);
      \draw [semithick] (.25,.5) to [out=0,in=180] (.75,-.5) -- (1,-.5);
      \draw [semithick,<-] (-1,0) -- (-.5,0);
  \node at (.75,-.5) {\tiny $\LEFTcircle$};
    \node at (.75,-.7) {\tiny $+1$};
    \node at (.75,.5) {\tiny $\LEFTcircle$};
    \node at (.75,.7) {\tiny $+1$};
      \end{tikzpicture}
  \quad,\quad
  \begin{tikzpicture}[anchorbase,rotate=-90]
    \draw [semithick] (0,0) to [out=0,in=180] (1,.5);
    \draw [semithick,<-] (-1,0) -- (0,0);
    \draw [semithick] (0,0) to [out=0,in=180] (1,-.5);
    \node at (-.3,0) {\tiny $\LEFTcircle$};
    \node at (-.3,-.3) {\tiny $+1$};
  \end{tikzpicture}
  \quad \leftrightarrow \quad
  \begin{tikzpicture}[anchorbase,rotate=-90]
      \draw [semithick] (-.5,0) to [out=0,in=180] (-.25,-.5) to [out=0,in=180] (.25,.5);
      \draw [white, line width=3] (-.5,0) to [out=0,in=180] (-.25,.5) to [out=0,in=180] (.25,-.5);
      \draw [semithick] (-.5,0) to [out=0,in=180] (-.25,.5) to [out=0,in=180] (.25,-.5);
      \draw [semithick] (.25,-.5) to [out=0,in=180] (.75,.5) -- (1,.5);
      \draw [white, line width=3] (.25,.5) to [out=0,in=180] (.75,-.5) -- (1,-.5);
      \draw [semithick] (.25,.5) to [out=0,in=180] (.75,-.5) -- (1,-.5);
      \draw [semithick,<-] (-1,0) -- (-.5,0);
  \node at (.75,-.5) {\tiny $\LEFTcircle$};
    \node at (.75,-.7) {\tiny $+1$};
    \node at (.75,.5) {\tiny $\LEFTcircle$};
    \node at (.75,.7) {\tiny $+1$};
      \end{tikzpicture}
\]
\end{theorem}
%%%%%%%%%%%%%%%%%%%%%%%%%%

%%%%%%%%%%%%%%%%%%%%%%%%%%%%%%%%%%%%%%%%%%%%%%%%%%%%%%%%%%%%%%%%%%%%%%%%%%%%%%%%%%%%%%%%%%
\section{Movie moves for framed foams} \label{sec:foams}
%%%%%%%%%%%%%%%%%%%%%%%%%%%%%%%%%%%%%%%%%%%%%%%%%%%%%%%%%%%%%%%%%%%%%%%%%%%%%%%%%%%%%%%%%%

In this section, we will work on upgrading the Reidemeister theorem for framed, tangled webs one dimension higher. Natural cobordisms between webs are foams, and we will be looking for a presentation of the category of framed foams between tangled webs.

Classically, a foam $F$ is a special kind of 2-dimensional singular surface. Just as in the previous section, we will encode the framing by considering a thickened version $N(F)$ containing a preferred copy $F\subset N(F)$. There should be an atlas making it locally smoothly diffeomorphic to one of the following elementary pieces:
\begin{itemize}
\item for $x\in F^{2}$, a local model is given by $\D^2\times [0,1]$, with preferred copy $\D^2\times \{\frac{1}{2}\}$;
\item for $x\in F^{1}$, a local model for $F$ is given by $Y\times [0,1]$, with $Y$ a trivalent web: \[
Y=\quad    \begin{tikzpicture}[anchorbase,decoration={markings, mark=at position 0.5 with {\arrow{>}}; }]
    \draw [semithick, postaction={decorate}] (0,0) to [out=90,in=-90] (.5,.5);
    \draw [semithick, postaction={decorate}] (1,0) to [out=90,in=-90] (.5,.5);
    \draw [semithick, postaction={decorate}] (.5,.5) to [out=90,in=-90] (.5,1);
  \end{tikzpicture}
  \quad \text{or}\quad
  Y=\quad 
    \begin{tikzpicture}[anchorbase,decoration={markings, mark=at position 0.5 with {\arrow{<}}; }]
    \draw [semithick, postaction={decorate}] (0,1) to [out=-90,in=90] (.5,.5);
    \draw [semithick, postaction={decorate}] (1,1) to [out=-90,in=90] (.5,.5);
    \draw [semithick, postaction={decorate}] (.5,.5) to [out=-90,in=90] (.5,0);
  \end{tikzpicture}
\]
Recalling from the previous section that $Y$ can be turned into a $2$-manifold with corners $\tilde{Y}$, we define $N(F)=\tilde{Y}\times [0,1]$ with preferred copy $Y\times [0,1]$;
\item for $x\in F^0$, the local model is illustrated below:
  
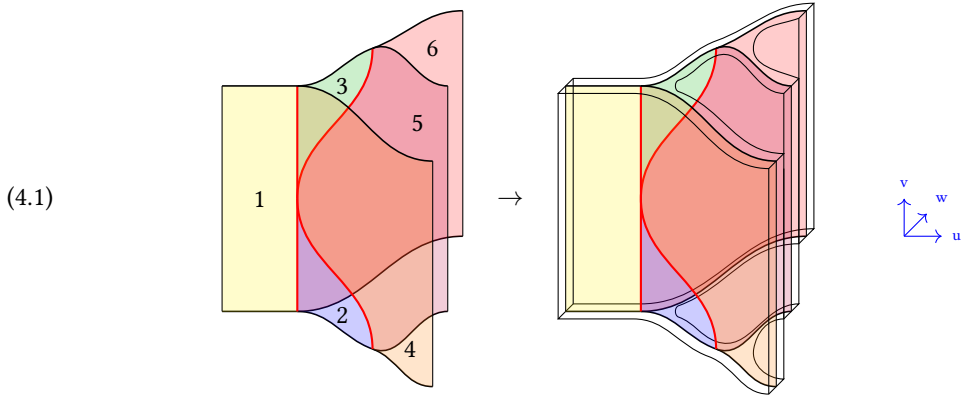
\begin{equation} \label{eq:infl6val}
    \begin{tikzpicture}[anchorbase]
      %% Bottom web
      \draw [semithick] (-1,0) -- (0,0);
      \draw [semithick] (0,0) to [out=0,in=160] (1,-.5);
      \draw [semithick] (1,-.5) to [out=-20,in=180] (1.8,-1);
      \draw [semithick] (1,-.5) to [out=-20,in=180] (2,0);
      \draw [semithick] (0,0) to [out=0,in=180] (2.2,1);
      %% Foam
      % LHS sheet
      \fill [fill=yellow, fill opacity=.2] (-1,0) -- (0,0) -- (0,3) -- (-1,3) -- (-1,0);
      \draw (-1,0) -- (-1,3);
      % Back sheet
      \fill [semithick, fill=red, fill opacity=.2] (0,0) to [out=0,in=180] (2.2,1) -- (2.2,4) to [out=180,in=20] (1,3.5) to [out=-90,in=90] (0,1.5) -- (0,0);
      \draw (2.2,1) -- (2.2,4);
      % Middle sheets
      \fill [fill=green, fill opacity=.2] (0,3) to [out=0,in=-160] (1,3.5) to [out=-90,in=90] (0,1.5) -- (0,3);
      \fill [fill=blue, fill opacity=.2] (0,0) to [out=0,in=160] (1,-.5) to [out=90,in=-90] (0,1.5) -- (0,0);
      \fill [fill=purple, fill opacity=.2] (1,-.5) to [out=-20,in=180] (2,0) -- (2,3) to [out=180,in=20] (1,3.5) to [out=-90,in=90] (0,1.5) to [out=-90,in=90] (1,-.5);
      \draw (2,0) -- (2,3);
      % Back seam
      \draw [red, thick] (0,1.5) to [out=90,in=-90] (1,3.5);
      % Front sheet
      \fill [fill=orange, fill opacity=.2] (1.8,-1) to [out=180,in=-20] (1, -.5) to [out=90,in=-90] (0,1.5) -- (0,3) to [out=0,in=180] (1.8,2) -- (1.8,-1);
      \draw (1.8,-1) -- (1.8,2);
      % Other seams
      \draw [red, thick] (0,0) -- (0,3);
      \draw [red, thick] (1,-.5) to [out=90,in=-90] (0,1.5);
      %% Top web
      \draw [semithick] (-1,3) -- (0,3);
      \draw [semithick] (0,3) to [out=0,in=-160] (1,3.5);
      \draw [semithick] (1,3.5) to [out=20,in=180] (2.2,4);
      \draw [semithick] (1,3.5) to [out=20,in=180] (2,3);
      \draw [semithick] (0,3) to [out=0,in=180] (1.8,2);
      % Numbering the facets
      \node [circle] at (-.5,1.5) {1};
      \node [circle] at (.6,-.05) {2};
      \node [circle] at (.6,3) {3};
      \node [circle] at (1.5,-.5) {4};
      \node [circle] at (1.6,2.5) {5};
      \node [circle] at (1.8,3.5) {6};
    \end{tikzpicture}
    \quad \rightarrow \quad
    \begin{tikzpicture}[anchorbase]
      %% Bottom web
      \draw [semithick] (-1,0) -- (0,0);
      \draw [semithick] (0,0) to [out=0,in=160] (1,-.5);
      \draw [semithick] (1,-.5) to [out=-20,in=180] (1.8,-1);
      \draw [semithick] (1,-.5) to [out=-20,in=180] (2,0);
      \draw [semithick] (0,0) to [out=0,in=180] (2.2,1);
      %% Inflated bottom web
      \draw (-.9,.1) -- (-1.1,-.1) -- (-.1,-.1) to [out=0,in=160] (.9,-.6) to [out=-20,in=180] (1.7,-1.1) -- (1.9,-.9) to [out=180,in=-100] (1.5,-.5) to [out=80,in=180] (1.9,-.1) -- (2.1,.1) to [out=180,in=-20] (1,-.4) to [out=160, in=-20] (.5,-.075) to [out=160,in=-160] (.5,.075) to [out=20,in=180] (2.1,.9) -- (2.3,1.1) to [out=180,in=0] (0,.1) --  (-.9,.1);
      %% Foam
      % LHS sheet
      \fill [fill=yellow, fill opacity=.2] (-1,0) -- (0,0) -- (0,3) -- (-1,3) -- (-1,0);
      \draw (-1,0) -- (-1,3);
      % Back inflated vertical lines
      \draw (-.9,.1) -- (-.9,3.1);
      \draw (2.3,1.1) -- (2.3,4.1);
      % Back sheet
      \fill [semithick, fill=red, fill opacity=.2] (0,0) to [out=0,in=180] (2.2,1) -- (2.2,4) to [out=180,in=20] (1,3.5) to [out=-90,in=90] (0,1.5) -- (0,0);
      \draw (2.2,1) -- (2.2,4);
      % Inflated vertical lines
      \draw (2.1,.9) -- (2.1,3.9);
      \draw (2.1,.1) -- (2.1,3.1);
      % Middle sheets
      \fill [fill=green, fill opacity=.2] (0,3) to [out=0,in=-160] (1,3.5) to [out=-90,in=90] (0,1.5) -- (0,3);
      \fill [fill=blue, fill opacity=.2] (0,0) to [out=0,in=160] (1,-.5) to [out=90,in=-90] (0,1.5) -- (0,0);
      \fill [fill=purple, fill opacity=.2] (1,-.5) to [out=-20,in=180] (2,0) -- (2,3) to [out=180,in=20] (1,3.5) to [out=-90,in=90] (0,1.5) to [out=-90,in=90] (1,-.5);
      \draw (2,0) -- (2,3);
      % Back seam
      \draw [red, thick] (0,1.5) to [out=90,in=-90] (1,3.5);
      % Inflated vertical lines
      \draw (1.9,-.1) -- (1.9,2.9);
      \draw (1.9,-.9) -- (1.9,1.9);
      % Front sheet
      \fill [fill=orange, fill opacity=.2] (1.8,-1) to [out=180,in=-20] (1, -.5) to [out=90,in=-90] (0,1.5) -- (0,3) to [out=0,in=180] (1.8,2) -- (1.8,-1);
      \draw (1.8,-1) -- (1.8,2);
      % Other seams
      \draw [red, thick] (0,0) -- (0,3);
      \draw [red, thick] (1,-.5) to [out=90,in=-90] (0,1.5);
      % Front inflated vertical lines
      \draw (-1.1,-.1) -- (-1.1,2.9);
      \draw (1.7,-1.1) -- (1.7,1.9);
      %% Top web
      \draw [semithick] (-1,3) -- (0,3);
      \draw [semithick] (0,3) to [out=0,in=-160] (1,3.5);
      \draw [semithick] (1,3.5) to [out=20,in=180] (2.2,4);
      \draw [semithick] (1,3.5) to [out=20,in=180] (2,3);
      \draw [semithick] (0,3) to [out=0,in=180] (1.8,2);
      % Inflated top web
      \draw (-.9,3.1) -- (-1.1,2.9) -- (-.1,2.9) to [out=0,in=180] (1.7,1.9) -- (1.9,2.1) to [out=180,in=-20] (.5,2.925) to [out=160,in=-160] (.5,3.075) to [out=20,in=-160] (1,3.4) to [out=20,in=180] (1.9,2.9) -- (2.1,3.1) to [out=160,in=-100] (1.5,3.55) to [out=80,in=-160] (2.1,3.9) -- (2.3,4.1) to [out=180,in=20] (1,3.6) to [out=-160,in=0] (0,3.1) to [out=180,in=0] (-.9,3.1);
      % Local coordinates
      \draw [blue,->] (3.5,1) -- (4,1);
      \node [blue] at (4.2,1) {\tiny u};
      \draw [blue,->] (3.5,1) -- (3.5,1.5);
      \node [blue] at (3.5,1.7) {\tiny v};
      \draw [blue,->] (3.5,1) -- (3.8,1.3);
      \node [blue] at (4,1.5) {\tiny w};
    \end{tikzpicture}
  \end{equation}
  Above we have indicated our preferred system of local coordinates. More precisely, we want to have a singular surface that is combinatorially as shown above, and so that the following pieces assemble into a smooth rectangle:
  \begin{itemize}
  \item facets $1$, $2$ and $4$;
  \item facets $1$, $2$, $3$ and $5$;
  \item facets $1$, $3$ and $6$.
  \end{itemize}
\end{itemize}

An explicit smooth realization of such a local model for $F$ at a 6-valent point can be built as follows (but other, non-diffeomorphic ones can exist):
  \begin{itemize}
  \item $\{x\leq 0,y=0,z\}$;
  \item $\{x>0,y=e^{-\frac{1}{x}},z\}$;
  \item $\{x>0,y=-e^{-\frac{1}{x}},z\}$;
  \item $\left\{
        x>e^{-\frac{1}{|z|}},y=\begin{cases}
      e^{-\frac{1}{x}}-e^{-\frac{1}{x-e^{-\frac{1}{z}}}}\; \text{if}\; z>0
      \\
      y=0 \; \text{if}\; z=0
      \\
      -e^{-\frac{1}{x}}+e^{-\frac{1}{x-e^{\frac{1}{z}}}}\; \text{if}\; z<0
    \end{cases}
    \right\}$
  \end{itemize}
Below we provide pictures of the slices at $z=1$, $z=0$ and $z=-1$.

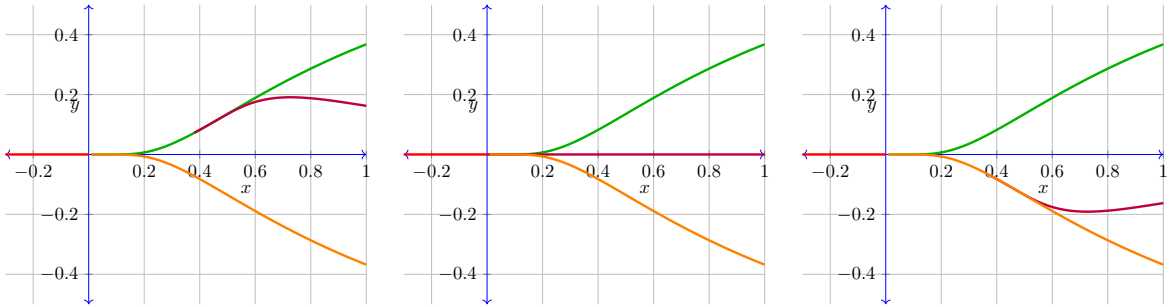
\[
\begin{tikzpicture}[scale=.7]
    \begin{axis}[
            xmin=-.3,xmax=1,
            ymin=-.5,ymax=.5,
            grid=both,
            ]
            \addplot [domain=-.3:0,samples=50,line width=1.3, color=red]({x},{0});
            \addplot [domain=.01:1,samples=50, line width=1.3, color=green]({x},{exp(-1/x)});
            \addplot [domain=exp(-1):1,samples=50,line width=1.3, color=purple]({x},{exp(-1/x)-exp(-1/(x-exp(-1)))});
            \addplot [domain=.01:1,samples=50, line width=1.3, color=orange]({x},{-exp(-1/x)});
    \end{axis}
\end{tikzpicture}
\quad
\begin{tikzpicture}[scale=.7]
    \begin{axis}[
            xmin=-.3,xmax=1,
            ymin=-.5,ymax=.5,
            grid=both,
            ]
            \addplot [domain=-.3:0,samples=50,line width=1.3, color=red]({x},{0});
            \addplot [domain=.01:1,samples=50, line width=1.3, color=green]({x},{exp(-1/x)});
            \addplot [domain=0:1,samples=50,line width=1.3, color=purple]({x},{0});
            \addplot [domain=.01:1,samples=50, line width=1.3, color=orange]({x},{-exp(-1/x)});
    \end{axis}
\end{tikzpicture}
\quad
\begin{tikzpicture}[scale=.7]
    \begin{axis}[
            xmin=-.3,xmax=1,
            ymin=-.5,ymax=.5,
            grid=both,
            ]
            \addplot [domain=-.3:0,samples=50,line width=1.3, color=red]({x},{0});
            \addplot [domain=.01:1,samples=50, line width=1.3, color=green]({x},{exp(-1/x)});
            \addplot [domain=exp(-1):1,samples=50, line width=1.3, color=purple]({x},{-exp(-1/x)+exp(-1/(x-exp(-1)))});
            \addplot [domain=.01:1,samples=50, line width=1.3, color=orange]({x},{-exp(-1/x)});
    \end{axis}
\end{tikzpicture}
\]
Then $N(F)$ is obtained by taking a neighborhood in the 3-ball.

We then consider smooth maps $f:N(F) \rightarrow \R^4$, and denoting $x,y,z,s$ the coordinates in $\R^4$, we consider the projection maps $\pi:\R^4\rightarrow \R^3$ that forgets the coordinate $z$. A framed foam in $\R^4$ is the image of $F$ under such a map $f$ that is injective on $F$, together with a non-tangent non-vanishing vector field, obtained by looking at the first derivative in the $w$ direction.

We will also look at isotopies of foams, described by families $f_t$ of such maps, $t\in [0,1]$.

We require the map $f$ (and all maps $f_t$) to respect the following conditions:
\begin{enumerate} \label{enum:condFoam}
\item \label{enum:condFoam1} $f$ is injective when restricted to $F$;
\item \label{enum:condFoam2} $\left(\frac{\partial f}{\partial u},\frac{\partial f}{\partial v},\frac{\partial f}{\partial w}\right)$ is full-rank on $F$.
\end{enumerate}

Consider the first condition. The union of submanifolds where this is not met is:
\[
  \{(M_1,N_1,M_2,N_1)\;|\;M_1,M_2\in F,\; N_1=N_2\in \R^4\}
\]
This is a union of codimension 6 submanifolds in the 2-fold jet bundle, while the graph of $f$ is 6-dimensional. One thus expects generically isolated intersections. We restrict our attention to maps that have no such intersections.

For the second condition, we have a set:
\[
  \{(M,N,D)\;|\;M\in F,\; N\in \R^4, D\;\text{not full rank}\}
\]
The codimension is: $1$ for the restriction to $F$, and since $D$ is a 4 by 3 matrix, requiring it not to be full rank amounts to freely choosing two columns, and asking the last one to be a linear combination of the first two ones. Since there are two parameters to choose for 4 entries, this is a codimension 2 condition. This all adds up to codimension $3$, while the graph is of dimension $3$ ($4$ if one adds the time parameter). As before, one expects generically isolated points where the condition is not fulfilled. We restrict out attention to maps that do not have such points.

%For the third condition, we read a codimension $4$ condition ($3$ for the . This is thus generic for foams, but could happen at isolated points through time. We thus restrict to maps $f_t$ that have $0$ such points, and again this set is stable under small perturbations. The reason to include this condition is as follows. Given $M\in F$, let $V(M)$ be  a neighborhood of $M$ in $F$. From the discussions to follow, we will see that generically, $V(M)\cap \R^4_{x,y,z,s=s(f(M))}\simeq \R^3$ is a web in $\R^3$. To obtain a framed web, we will want to 

%Let us say a bit more about this condition, which is related to the framing. Given $M\in F$, let $V(M)$ be  a neighborhood of $M$ in $F$. From the discussions to follow, we will see that generically, $V(M)\cap \R^4_{x,y,z,s=s(f(M))}\simeq \R^3$ is a web in $\R^3$. To obtain the framing, one can proceed as follows: 

Now, we will be looking for local models for generators of foams and foam isotopies, and relations between them.

As we go into this analysis, we will be using the following notations for derivatives: 
\[
\frac{\partial f}{\partial u}=\begin{pmatrix} a_1 \\ a_2 \\ a_3 \\ a_4 \end{pmatrix},\quad
\frac{\partial f}{\partial v}=\begin{pmatrix} b_1 \\ b_2 \\ b_3 \\ b_4 \end{pmatrix},\quad
\frac{\partial f}{\partial w}=\begin{pmatrix} c_1 \\ c_2 \\ c_3 \\ c_4 \end{pmatrix}.
\]

Similarly for second derivatives, we will use the following notations:
\[
\frac{\partial^2 f}{\partial u^2}=\begin{pmatrix} d_1 \\ d_2 \\ d_3 \\ d_4 \end{pmatrix},\quad
\frac{\partial^2 f}{\partial u\partial v}=\begin{pmatrix} e_1 \\ e_2 \\ e_3 \\ e_4 \end{pmatrix},\quad
\frac{\partial^2 f}{\partial u\partial w}=\begin{pmatrix} f_1 \\ f_2 \\ f_3 \\ f_4 \end{pmatrix},\quad
\frac{\partial^2 f}{\partial v^2}=\begin{pmatrix} g_1 \\ g_2 \\ g_3 \\ g_4 \end{pmatrix},\quad
\frac{\partial^2 f}{\partial v\partial w}=\begin{pmatrix} h_1 \\ h_2 \\ h_3 \\ h_4 \end{pmatrix},\quad
\frac{\partial^2 f}{\partial w^2}=\begin{pmatrix} i_1 \\ i_2 \\ i_3 \\ i_4 \end{pmatrix}.
\]
Higher derivatives will occasionally appear. We do not define global notations for them, and will choose local ones when required.

We will first list conditions that make the neighborhood of a given point trivial (meaning that it corresponds to an identity movie). Then we will investigate the failure of each of these conditions. We consider successively points on $F^2$, $F^1$ or $F^0$.

%%%%%%%%%%%%%%%%%%%%%%%%%%%%%%%%%%%%%%%%%%%%%%%%
\subsection{Conditions on $F^2$}
%%%%%%%%%%%%%%%%%%%%%%%%%%%%%%%%%%%%%%%%%%%%%%%

Given a point $M$ on $F^2$, we can find a trivial neighborhood for the image $f(M)$ if:
\begin{itemize}
\item $f(M)$ is not a multiple point under $\pi$;
\item $(\frac{\partial f}{\partial u},\frac{\partial f}{\partial v})$ is not contained in the $s=0$ sub-space of $\R^4$;
\item $(\frac{\partial f}{\partial u},\frac{\partial f}{\partial v})$ does not contain the direction $\vec{z}$;
\item $\frac{\partial f}{\partial w}$ is not contained in the space spanned by $\vec{z},\frac{\partial f}{\partial u},\frac{\partial f}{\partial v}$.
\end{itemize}

Indeed, consider $M\in F^2$ that enjoys the previous properties. Then:
\[
  (df)_m=\begin{pmatrix}  \cdot & \cdot & \cdot \\ \cdot & \cdot & \cdot \\\cdot & \cdot & \cdot \\\star & \star & \cdot \end{pmatrix}
\]
and we can assume that at least one of the entries marked with a star is non-zero (since the first two vectors are not contained in the $s=0$ subspace). Up to reparametrization in $u,v$, we can reduce to:
\[
  (df)_m=\begin{pmatrix}  \cdot & \cdot & \cdot \\ \cdot & \cdot & \cdot \\\cdot & \cdot & \cdot \\0 & \neq 0 & \cdot \end{pmatrix}
\]
Now, since  $(\frac{\partial f}{\partial u},\frac{\partial f}{\partial v})$ does not contain the direction $\vec{z}$, we can assume that:
\[
  (df)_m\neq\begin{pmatrix}  \textcolor{red}{0} & \cdot & \cdot \\ \textcolor{red}{0} & \cdot & \cdot \\\cdot & \cdot & \cdot \\0 & \neq 0 & \cdot \\\end{pmatrix}
\]
Up to rotation in the $x,y$ plane in the target space, we can assume that the differential is:
\[
  (df)_m=\begin{pmatrix}  >0 & \cdot & \cdot \\ 0 & \cdot & \cdot \\\cdot & \cdot & \cdot \\0 & \neq 0 & \cdot \end{pmatrix}
\]
and up to adding to the $v$ coordinate a multiple of $u$, and to the $w$ coordinates multiples of $u$ and $v$, we reduce to:
\[
  (df)_m=\begin{pmatrix}  >0 & 0 & 0 \\ 0 & \cdot & \cdot \\\cdot & \cdot & \cdot \\0 & \neq 0 & \cdot \end{pmatrix}=
  \begin{pmatrix}  a_1 & 0 & 0 \\ 0 & b_2 & c_2 \\ a_3 & b_3 & c_3 \\ 0 & b_4 & 0 \end{pmatrix}\quad a_1>0,\; b_4\neq 0
\]
The change of coordinate in $w$ modifies the framing by a multiple of a vector that belongs to the tangent plane to $F$. This is of no effect from our perspective. 

Furthermore, $c_2$ can be assumed non-zero thanks to the condition on $\frac{\partial f}{\partial w}$.

Then one can run a Taylor expansion:
\begin{align*}
  f(m_1+\epsilon_1,m_2+\epsilon_2,m3)&=\begin{pmatrix}  \epsilon_1 a_1 \\ \epsilon_2 b_2 \\ \epsilon_1 a_3 + \epsilon_2 b_3 \\ b_4 \epsilon_2 \end{pmatrix} \\
  \left(\frac{\partial f}{\partial w}\right)_{m_1+\epsilon_1,m_2+\epsilon_2,m_3}&=\left(\frac{\partial f}{\partial w}\right)_{m_1,m_2,m_3}+o(\epsilon_1,\epsilon_2)
\end{align*}
In terms of movies, one reads the following trivial movie:
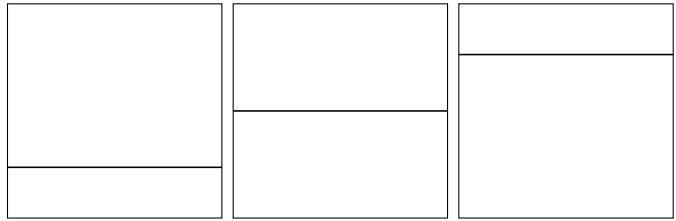
\[
  \begin{tikzpicture}[anchorbase,scale=.75]
    % Frame
    \draw (.1,.1) rectangle (3.9,3.9);
    \draw[semithick] (.1,1) -- (3.9,1);
    % Frame
    \draw (4.1,.1) rectangle (7.9,3.9);
    \draw[semithick] (4.1,2) -- (7.9,2);
    % Frame
    \draw (8.1,.1) rectangle (11.9,3.9);
    \draw[semithick] (8.1,3) -- (11.9,3);
  \end{tikzpicture}
\]
We haven't indicated the framing, which stays on the same side at all time and never vanishes.

%%%%%%%%%%%%%%%%%%%%%%%%%%%%%%%%%%%%%%%%%%%
\subsection{Conditions on $F^1$}
%%%%%%%%%%%%%%%%%%%%%%%%%%%%%%%%%%%%%%%%%%%%

Let us now look for a local model on $F^1$ in the easiest case. We will use the following conditions, for $M\in F^{1}$, with the local parametrization in the source space so that the direction $v$ agrees with the direction of the seam:
\begin{itemize}
\item $f(M)$ is not a multiple point under $\pi$;
\item $\frac{\partial f}{\partial v}\neq \begin{pmatrix} \cdot \\ \cdot \\ \cdot \\ 0\end{pmatrix}$;
\item $(\frac{\partial f}{\partial u},\frac{\partial f}{\partial v})$ does not contain the direction $\vec{z}$;
\item $\frac{\partial f}{\partial v}$ is not contained in the space spanned by $\vec{z},\frac{\partial f}{\partial u},\frac{\partial f}{\partial v}$.
\end{itemize}

We use the condition on $\frac{\partial f}{\partial v}$ to add a multiple of $v$ to $u$ so that:
\[
  (df)_m=\begin{pmatrix}  \cdot & \cdot & \cdot \\ \cdot & \cdot & \cdot \\\cdot & \cdot & \cdot \\ 0 & \neq 0 & \cdot \end{pmatrix}
\]
This has the effect to not leave the seam strictly vertical, but makes it drift in the $u$ direction.

Then the condition on the $\vec{z}$ direction implies that up to rotation in the $x,y$ plane, we can restrict to:
\[
  (df)_m=\begin{pmatrix}  >0 & \cdot & \cdot \\ 0 & \cdot & \cdot \\\cdot & \cdot & \cdot \\ 0 & \neq 0 & \cdot \end{pmatrix}
\]

Then we can use the condition on the framing and add to $w$ multiples of $u$ and $v$ so that:
\[
    (df)_m=\begin{pmatrix}  >0 & \cdot & 0 \\ 0 & \cdot & \neq 0 \\\cdot & \cdot & \cdot \\ 0 & \neq 0 & 0 \end{pmatrix}
  \]

  As before, we read from a Taylor expansion a trivial movie with non-vanishing framing in the projection:
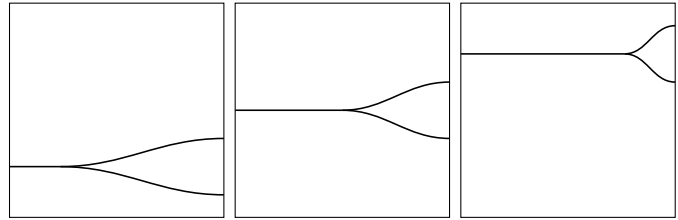
\[
  \begin{tikzpicture}[anchorbase,scale=.75]
    % Frame
    \draw (.1,.1) rectangle (3.9,3.9);
    \draw[semithick] (.1,1) -- (1,1);
    \draw[semithick] (1,1) to [out=0,in=180] (3.9,1.5);
    \draw[semithick] (1,1) to [out=0,in=180] (3.9,.5);
    % Frame
    \draw (4.1,.1) rectangle (7.9,3.9);
    \draw[semithick] (4.1,2) -- (6,2);
    \draw[semithick] (6,2) to [out=0,in=180] (7.9,2.5);
    \draw[semithick] (6,2) to [out=0,in=180] (7.9,1.5);
    % Frame
    \draw (8.1,.1) rectangle (11.9,3.9);
    \draw[semithick] (8.1,3) -- (11,3);
    \draw[semithick] (11,3) to [out=0,in=180] (11.9,3.5);
    \draw[semithick] (11,3) to [out=0,in=180] (11.9,2.5);
  \end{tikzpicture}
\]

%%%%%%%%%%%%%%%%%%%%%%%%%%%%%%%%%%%%%%%%%%%%%%%
\subsection{Conditions on $F^0$}
%%%%%%%%%%%%%%%%%%%%%%%%%%%%%%%%%%%%%%%%%%%%%%%

Finally, we focus on $M\in F^0$, with choice of coordinates as in Equation~\eqref{eq:infl6val}.
We assume the following:
\begin{itemize}
\item $f(M)$ is not a multiple point under $\pi$;
\item $\frac{\partial f}{\partial v}\neq \begin{pmatrix} \cdot \\ \cdot \\ \cdot \\ 0\end{pmatrix}$;
\item $(\frac{\partial f}{\partial u},\frac{\partial f}{\partial v})$ does not contain the direction $\vec{z}$;
\item $\frac{\partial f}{\partial w}$ is not contained in the space spanned by $\vec{z},\frac{\partial f}{\partial u},\frac{\partial f}{\partial v}$.
\end{itemize}

We run the very same analysis as in the $F^1$ case, and obtain the following generator for a foam:
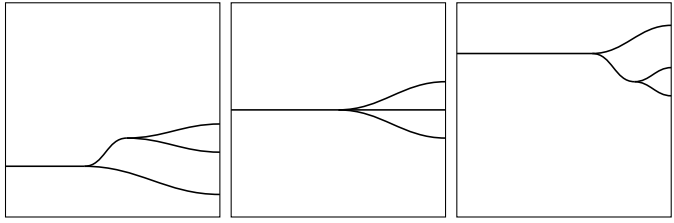
\[
  \begin{tikzpicture}[anchorbase,scale=.75]
    % Frame
    \draw (.1,.1) rectangle (3.9,3.9);
    \draw[semithick] (.1,1) -- (1.5,1);
    \draw[semithick] (1.5,1) to [out=0,in=180] (2.25,1.5);
    \draw [semithick] (2.25,1.5) to [out=0,in=180] (3.9,1.75);
    \draw [semithick] (2.25,1.5) to [out=0,in=180] (3.9,1.25);
    \draw[semithick] (1.5,1) to [out=0,in=180] (3.9,.5);
    % Frame
    \draw (4.1,.1) rectangle (7.9,3.9);
    \draw[semithick] (4.1,2) -- (6,2);
    \draw[semithick] (6,2) to [out=0,in=180] (7.9,2.5);
    \draw [semithick] (6,2) -- (7.9,2);
    \draw[semithick] (6,2) to [out=0,in=180] (7.9,1.5);
    % Frame
    \draw (8.1,.1) rectangle (11.9,3.9);
    \draw[semithick] (8.1,3) -- (10.5,3);
    \draw[semithick] (10.5,3) to [out=0,in=180] (11.9,3.5);
    \draw[semithick] (10.5,3) to [out=0,in=180] (11.25,2.5);
    \draw [semithick] (11.25,2.5) to [out=0,in=180] (11.9,2.75);
    \draw [semithick] (11.25,2.5) to [out=0,in=180] (11.9,2.25);    
  \end{tikzpicture}
\]

In general the middle picture is omitted and we represent this generator as:
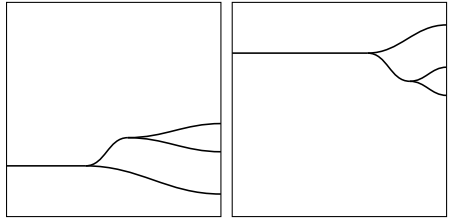
\[
  \begin{tikzpicture}[anchorbase,scale=.75]
    % Frame
    \draw (.1,.1) rectangle (3.9,3.9);
    \draw[semithick] (.1,1) -- (1.5,1);
    \draw[semithick] (1.5,1) to [out=0,in=180] (2.25,1.5);
    \draw [semithick] (2.25,1.5) to [out=0,in=180] (3.9,1.75);
    \draw [semithick] (2.25,1.5) to [out=0,in=180] (3.9,1.25);
    \draw[semithick] (1.5,1) to [out=0,in=180] (3.9,.5);
    % Frame
    \draw (4.1,.1) rectangle (7.9,3.9);
    \draw[semithick] (4.1,3) -- (6.5,3);
    \draw[semithick] (6.5,3) to [out=0,in=180] (7.9,3.5);
    \draw[semithick] (6.5,3) to [out=0,in=180] (7.25,2.5);
    \draw [semithick] (7.25,2.5) to [out=0,in=180] (7.9,2.75);
    \draw [semithick] (7.25,2.5) to [out=0,in=180] (7.9,2.25);    
  \end{tikzpicture}
\]

Again, nothing interesting happens to the framing.

\vspace{1cm}

We will now analyze all situations we have excluded in the above discussion, and look for other foam generators (movies) as well as generators for isotopies (movie moves).

%%%%%%%%%%%%%%%%%%%%%%%%%%%%%%%%%%%%%%%%%%%%
\subsection{Multiple points in projection}
%%%%%%%%%%%%%%%%%%%%%%%%%%%%%%%%%%%%%%%%%%%

In $J_0^{k}$, we consider:
\[
  \{(M_1,N_1,\dots,M_k,N_k)\;|\; M_1\in F,\dots,M_k\in F,\; \pi(N_1)=\pi(N_2)=\dots=\pi(N_k)\}
\]
The codimension is: $k$ times $1$ for the restriction to $F$, and $(k-1)$ times $3$ for the identification of three coordinates in $N_i$, $i>1$, with those of $N_1$. In total we read $4k-3$. The dimension of the graph on the other hand is $3k$, or $3k+1$ if one considers a 1-parameter family of functions.

If $k>3$ then $3k<4k-3$, and if $k>4$, $3k+1<4k-3$, so one only has to consider double or triple points for foams, and quadruple points for isotopies.

{\bf Double points in projection}

We have a codimension $5$ condition for a $6$-dimensional graph, so we expect $1$-dimensional generic intersections (2-dimensional through time). If instead of looking at multiple points from $F\times F$ one restricts to $F^1\times F$ or $F\times F^1$, then one has isolated points ($1$-dimensional sets through time). For $F^1\times F^1$, $F^2\times F^0$ or $F^0\times F^2$, the intersection is generically empty for a foam, and isolated for a time family. Finally, there are no intersections of the kind $F^0\times F^0$, even through time.

Let us start with double points $\pi(f(M_1))=\pi(f(M_2))$ with $M_1$ and $M_2\in F^2$. Asking the intersection to be transverse is equivalent to requiring that the following vectors generate a 3-dimensional space under $\pi$:
\[
\left(\frac{\partial f}{\partial u}\right)_{M_1},\;\left(\frac{\partial f}{\partial v}\right)_{M_1},\;\left(\frac{\partial f}{\partial u}\right)_{M_2},\;\left(\frac{\partial f}{\partial u}\right)_{M_2}
\]
Failure of this requirement is a codimension 2 condition (up to permutation, one fixes the first two vectors, then the third and fourth ones are linear combinations of the first two ones: this takes two parameters for each of them, out of three coordinates). Since we already have restricted our attention to a $1$-dimensional situation, this can be assumed for free for a foam. For a $1$-parameter family there will be isolated points where this condition is not met (see page~\pageref{page:nontransverse2pt} for details).

Asking both $\left(\frac{\partial f}{\partial u}\right)_{M_1}$ and $\left(\frac{\partial f}{\partial v}\right)_{M_1}$ to have zero $s$-coordinate is a codimension 2 condition, so for a foam this can be assumed not to happen at a double point (there will be isolated points through time, yielding $\rm{MM}_{14}$ and $\rm{MM}_{15}$ as we shall see later). So we can reparametrize at $M_1$ and $M_2$ so that $\frac{\partial f}{\partial u}$ has zero $s$-coordinate. It might happen (this is a codimension $1$ condition) that $\pi(\left(\frac{\partial f}{\partial u}\right)_{M_1})$ and $\pi(\left(\frac{\partial f}{\partial u}\right)_{M_2})$ are colinear. We will analyze this situation later (this will give $\rm{MM}_3$, $\rm{MM}_4$ and $\rm{MM}_9$) and assume for now that this is not the case. We will also assume that both $u$ derivatives are not purely vertical. As this is a codimension $2$ condition, this could happen for only one of the two vectors at isolated points through time. This will yield a version of the classical movie move $\rm{MM}_8$, as we will see later. Up to reparametrization in the $x,y$ variables, we can assume that $\pi(\left(\frac{\partial f}{\partial u}\right)_{M_1})$ is parallel to $\vec{x}$ and $\pi(\left(\frac{\partial f}{\partial u}\right)_{M_2})$ is parallel to $\vec{y}$.

Then one reads an identity movie over a crossing:
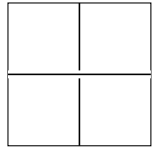
\[
  \begin{tikzpicture}[anchorbase,scale=.5]
    % Frame
    \draw (.1,.1) rectangle (3.9,3.9);
    \draw[semithick] (2,.1) -- (2,3.9);
    \draw [white, line width=3] (.1,2) -- (3.9,2);
    \draw [semithick] (.1,2) -- (3.9,2);
  \end{tikzpicture}
\]

Let us now look at the framing. If $\pi(\left(\frac{\partial f}{\partial w})_{M_1}\right)$ and $\pi(\left(\frac{\partial f}{\partial w})_{M_2}\right)$ lie outside of the space spanned by the other two tangent vectors and the vertical direction, then nothing happens. Requiring that the framing is vertical is a codimension $1$ condition, so there are isolated points where this happens on one of the two strands. Through time, there also are isolated points where this happens on both strands at the same time.

    Consider first the foam case, and assume that the framing vector projects onto the tangent vector to the strand parallel to the $x$ direction (corresponding to $M_1$). Then the shape of the framing on this strand will be determined by $\left(\frac{\partial^2 f}{\partial u\partial w}\right)$ that can be assumed to have a non-zero $y$ component (through time, this could be false). In the $s$ direction, the framing is simply transported parallel if $\left(\frac{\partial^2 f}{\partial u\partial w}\right)_y\neq 0$, while in the other case the control is ensured by $\left(\frac{\partial^2 f}{\partial v\partial w}\right)$, assumed to have a non-zero $y$ entry for codimension reason (see page~\pageref{page:framingcross} for the moves). Depending on signs, we are in one of the following situations (each generator can be read from left to right or right to left):

    \begin{gather*}

\end{gather*}

Now, assume that $\left(\frac{\partial f}{\partial u}\right)_{M_1}$ and $\left(\frac{\partial f}{\partial u}\right)_{M_2}$ are colinear. Again, this only happens at isolated points. Up to reparametrization, we assume that both are supported in the $x$ direction. We have:
\[
  \left(\frac{\partial f}{\partial u}\right)_{M_1}=%
\]
There is a similar expansion at $M_2$.

Focusing first on the $s=0$ case, since $b_4\neq 0$ one reads: $\epsilon_2=-\frac{d_4}{b_4+e_4\epsilon_1}\simeq -\frac{d_4}{b_4}\epsilon_1^2$. This justifies that the $s=0$ slice is described by:
\[%
\]

Then, for codimension reason, we can assume that $d_2-\frac{d_4}{b_4}$ and its $M_2$ analog are different (through time this will happen at isolated points, yielding MM9 as we will see later). This determines a local model of the following kind (depending on the relative sign, and where we have colored the strands for better visibility):

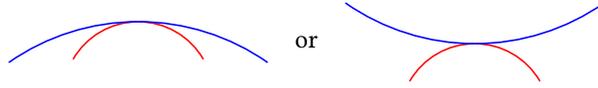
\[
  %
\]

The behavior in the $s$ direction is controlled by the $y$-coordinates $b_2$ of $\frac{\partial f}{\partial v}$ at $M_1$ and $M_2$. One can assume that the difference of the two entries is non-zero (but this will be zero at isolated points through time, yielding $\rm{MM}_3$ and $\rm{MM}_4$). In both cases, we arrive at the Reidemeister II move (to be read in one or the other direction):

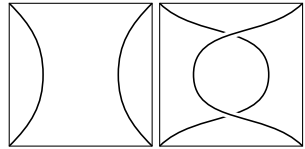
\[
%
  
\]

Let us now consider the case where the entries $b_2$ at $M_1$ and $M_2$ balance. Then (the prime stands for the entry corresponding to $M_2$):
\[\left(f(M_1+\epsilon_2)-f(M_2+\epsilon_2)\right)_y\sim (g_2-g'_2)\epsilon_2^2\]

This yields the classical movie moves $\text{MM}_3$, $\text{MM}_4$.

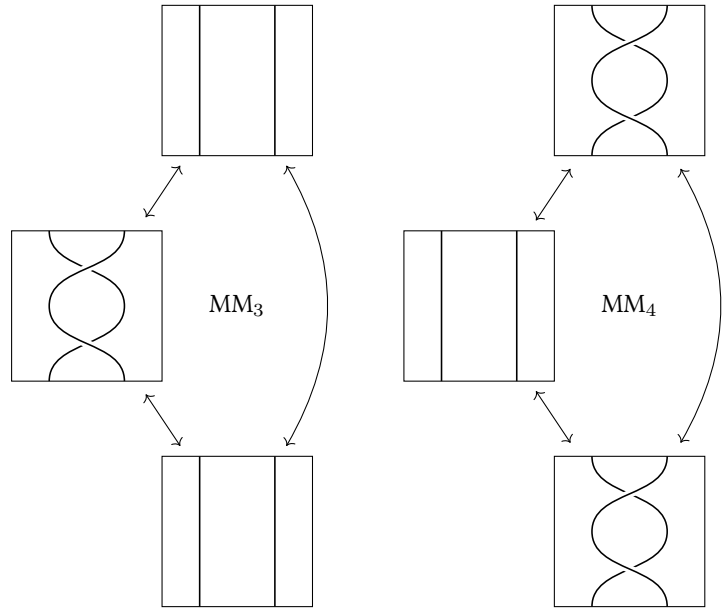
\[
  %
\]

Now, if $d_2-\frac{d_4}{b_4}$ and $d'_2-\frac{d'_4}{b'_4}$ are equal, then the difference in $y$-coordinate is controlled by a term proportional to $\epsilon_1^3$. This yields the classical movie move $\text{MM}_9$. 

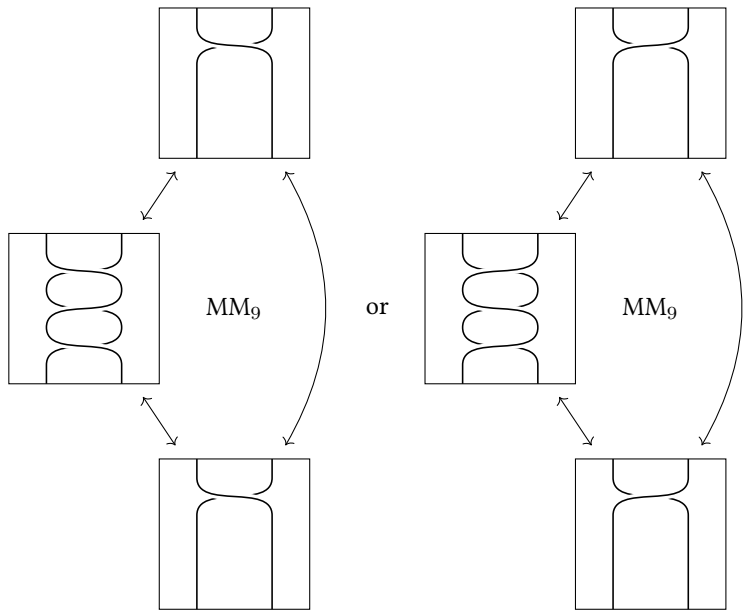
\[
  %
\]

By codimension count, we see that the framing can be assumed to be non-vertical at all time.

Going back to the Reidemeister II foam generator, one could see a framing change on one of the two strands. This yields the following movie moves:
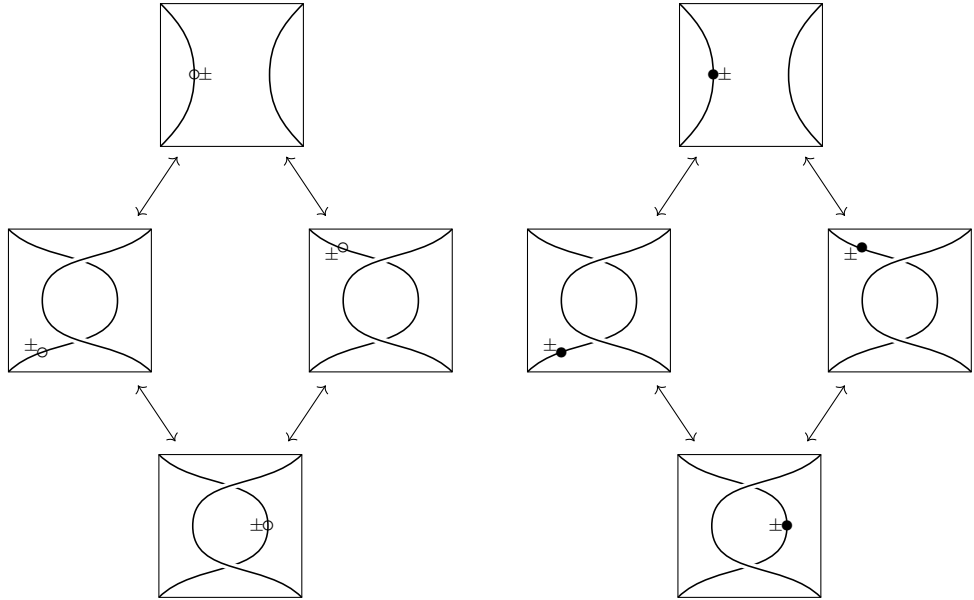
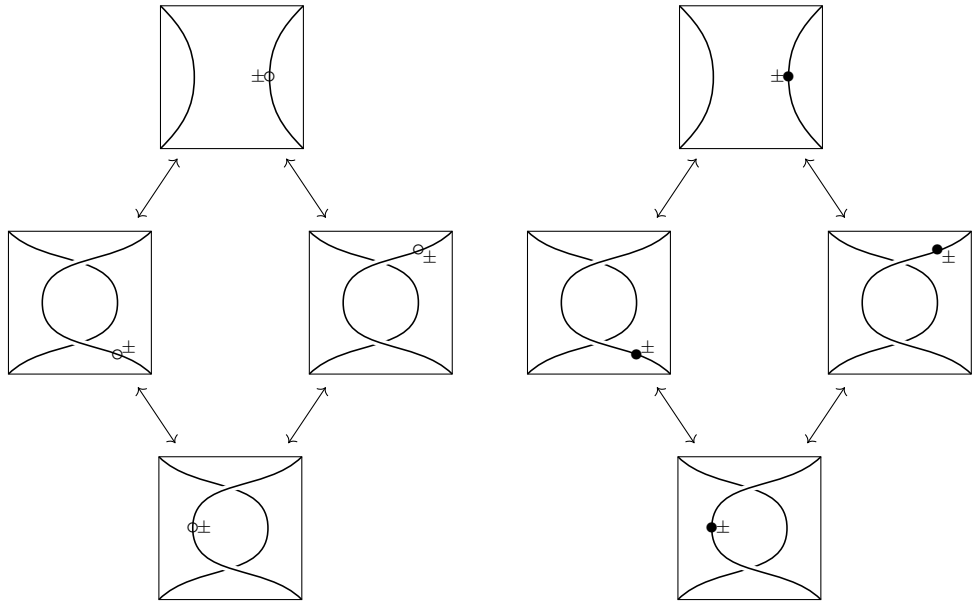
\[
  %
\]
as well as:
\[
  %
\]

Going back to the case of a framing change going through a crossing, three things can happen through time: either only one change goes through but the vector $\frac{\partial^2 f}{\partial u \partial w}$ has zero $y$ coordinate or the vector $\frac{\partial^2 f}{\partial v \partial w}$ has zero $y$ coordinate, or a crossing change happens on each strand at the same time. In the first case we have the following movie move: 
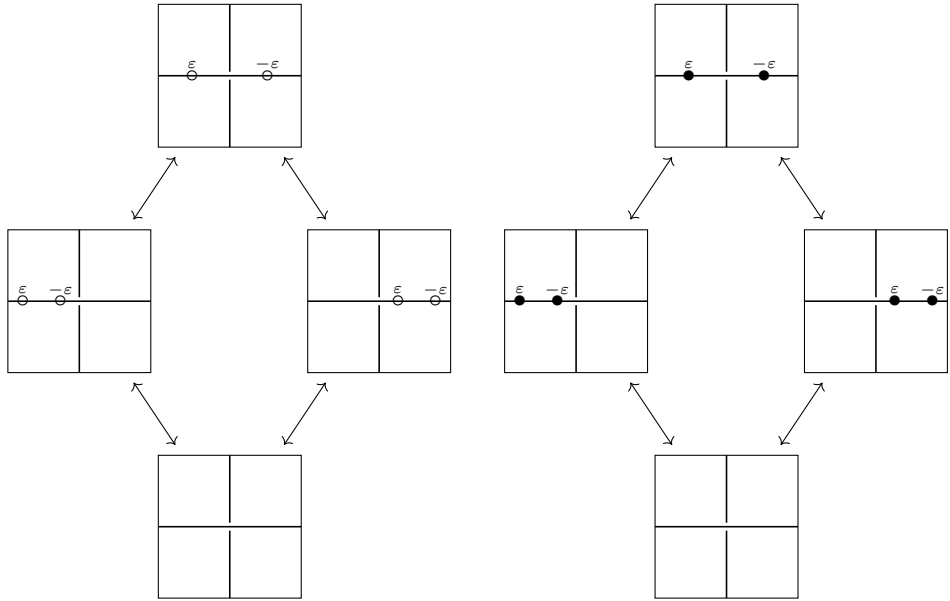
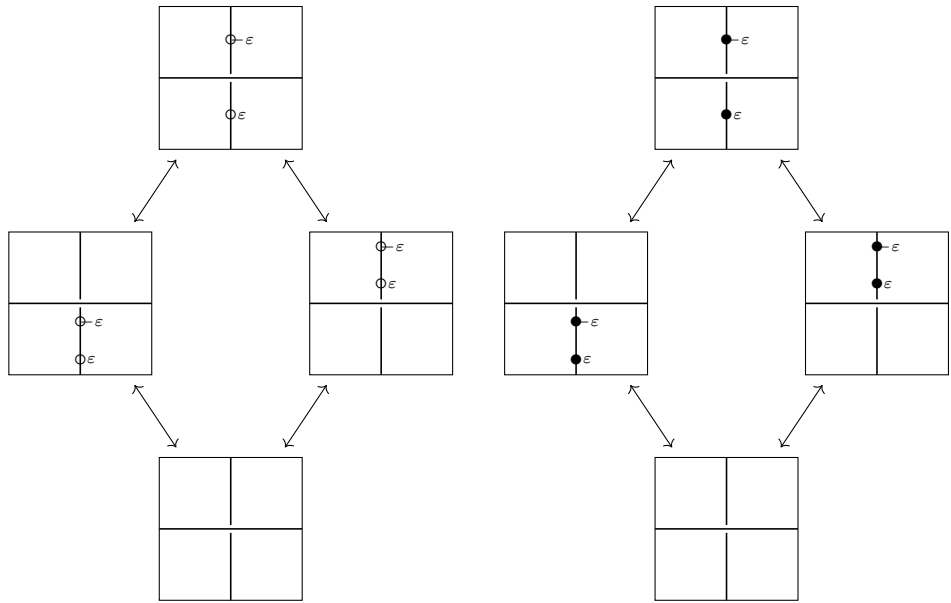
\label{page:framingcross}

\[
  %
\]

In the second case we get:
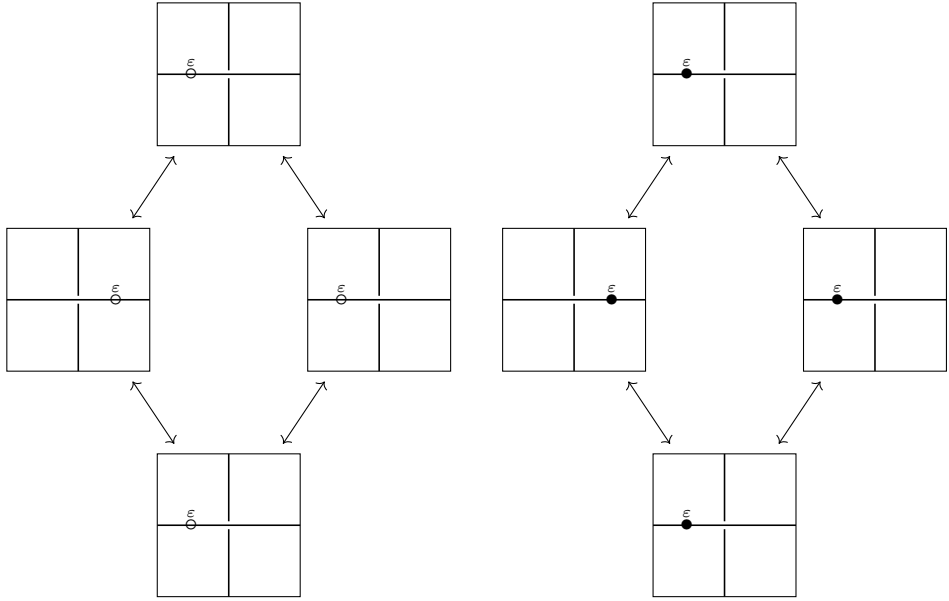
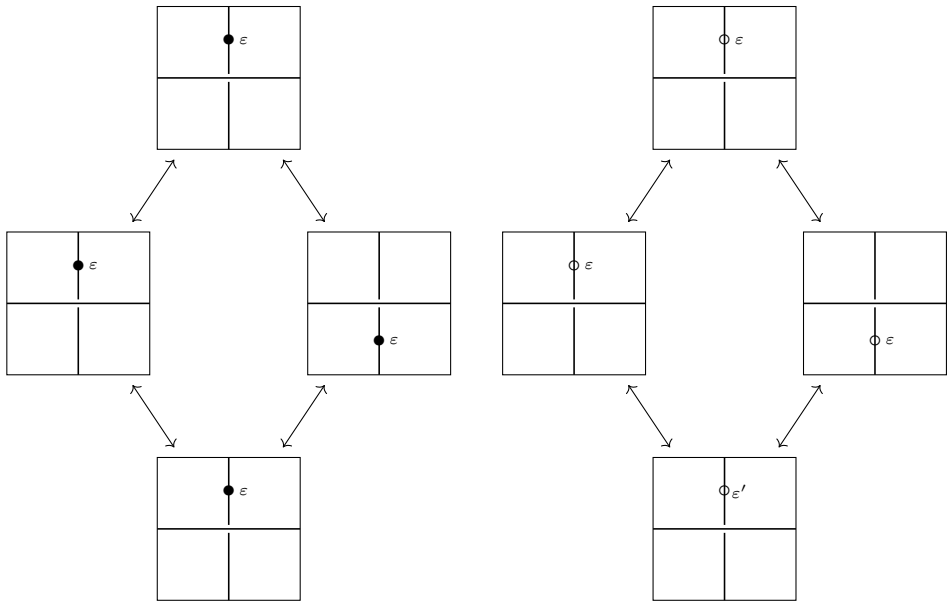
\[
%
\]

In the case where both strands carry a crossing change, then one obtains the following movie move:
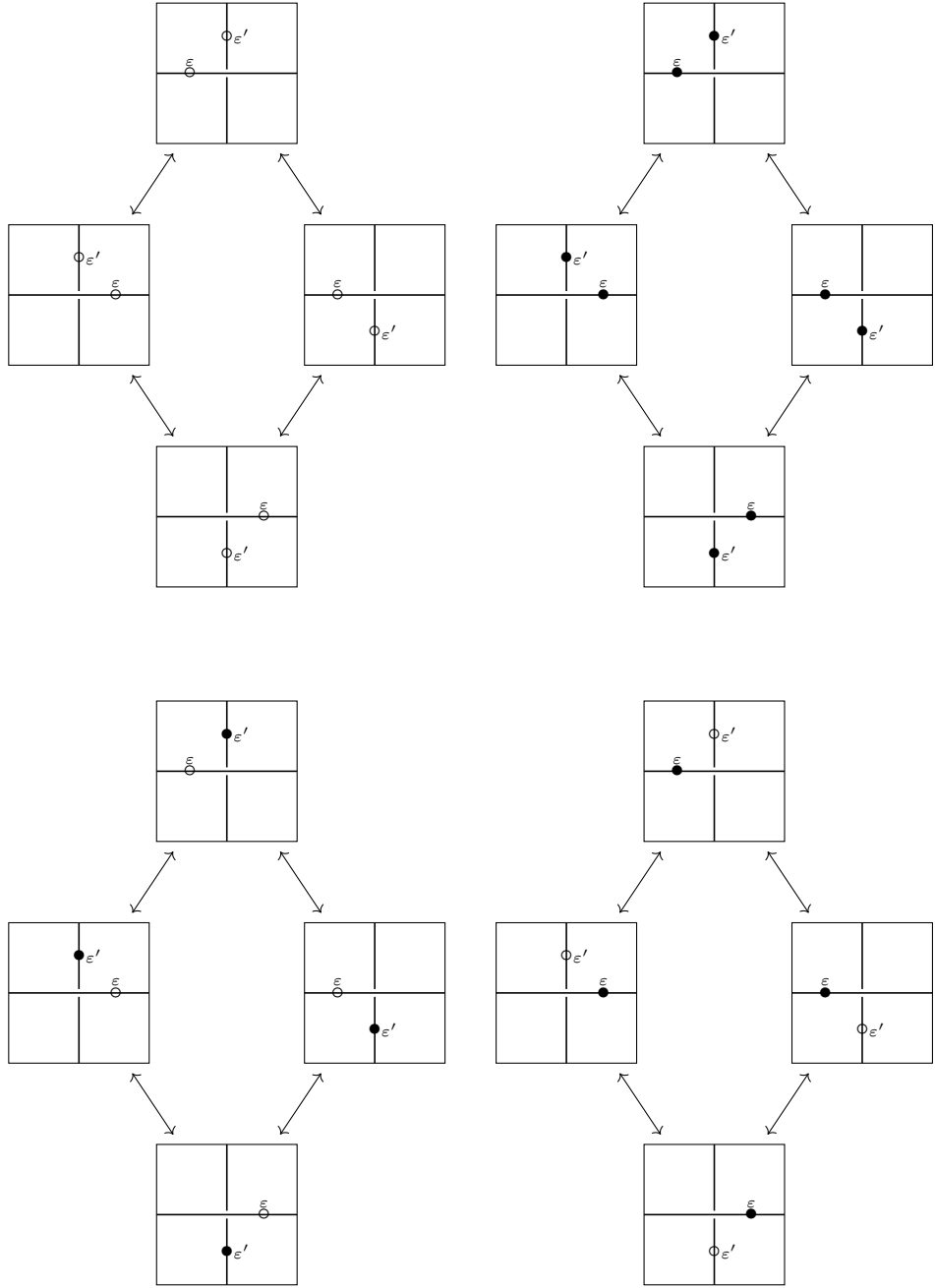
\[
%
\]

Now, assuming that both $\left(\left(\frac{\partial f}{\partial u}\right)_{M_1}\right)$ and $\left(\left(\frac{\partial f}{\partial v}\right)_{M_1}\right)$ have zero $s$-coordinate amounts to saying that the neighborhood of $M_1$ is sent to (to be read in one or the other direction) (see also the beginning of Section~\ref{subsec:nF2} for more details):
    \[
      %
    \]

        Superposed with an extra strand, one gets the following classical movie move:
    \[
      %
    \]
    Because of the codimension, one can assume that nothing occurs to the framing during this move.
    In the saddle case one gets the following familiar movie move (again the framing remains constant):
    
 \[
      %
        };
        \node at (0,-3) {$\text{MM}_{15}$};
        \draw [<->] (A) -- (B1);
        \draw [<->] (A) -- (B2);
        \draw [<->] (B1) -- (C);
        \draw [<->] (B2) -- (C);
      \end{tikzpicture}
\]

Let us now argue about the case where the intersection is not transverse. Since $\left(\frac{\partial f}{\partial u}\right)_{M_1}$ and $\left(\frac{\partial f}{\partial v}\right)_{M_1}$ always form a 2-dimensional vector space (and the same holds for $M_2$), two things can happen for the intersection not to be transverse:
\begin{itemize} \label{page:nontransverse2pt}
\item both projections remain of dimension $2$, and become equal. This is a codimension 2 condition, and it will recover the movie moves $\rm{MM}_3$ and $\rm{MM}_4$ already considered;
\item one of the two projections is only of dimension $1$, and included in the projection of the other tangent space. But this is a codimension $3$ condition and thus cannot happen generically.
\end{itemize}

Finally, we consider the case where one of the $u$ derivatives is purely vertical, say at $M_1$. One sees a generator from~\eqref{eq:frR1} happening on one of the two strands, superposed with an extra strand that intersects transversely (for codimension reasons). One then obtains the following movie move (and its analogs with other choices for the crossings and twist sign):

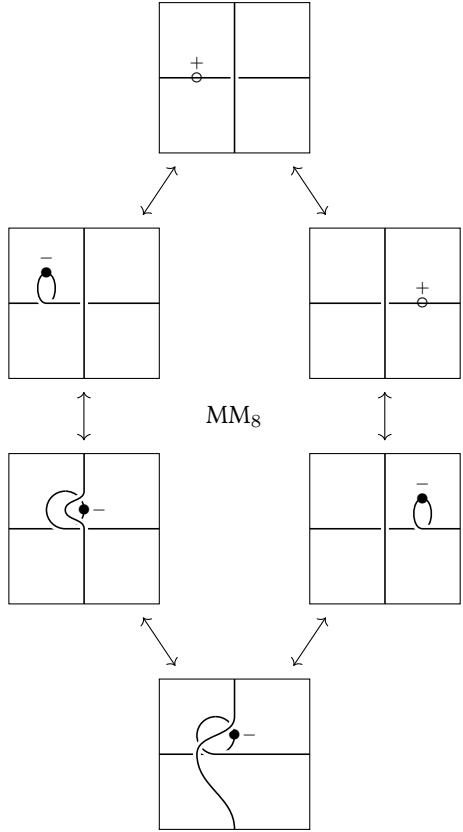
\[
  \begin{tikzpicture}[anchorbase]
    \node (A) at (0,0) {
\begin{tikzpicture}[anchorbase, scale=.5]
  % R1
  \draw [line width=3, white] (0,2) -- (4,2);
  \draw [semithick] (0,2) -- (4,2);
  % Strand
  \draw [line width=3,white] (2,0) -- (2,4);
  \draw [semithick] (2,0) -- (2,4);
  % twists
  \node at (1,2) {$\circ$};
  \node at (1,2.4) {\tiny $+$};
  % Rectangle
  \draw (0,0) rectangle (4,4);
  \end{tikzpicture}};
    \node (B1) at (-2,-3) {
\begin{tikzpicture}[anchorbase, scale=.5]
  % R1
  \draw [line width=3, white] (0,2) to [out=0,in=180] (1,2) to [out=0,in=0] (1,2.8);
   \draw [semithick] (0,2) to [out=0,in=180] (1,2) to [out=0,in=0] (1,2.8);
  \draw [line width=3, white] (1,2.8) to [out=180,in=180] (1,2) -- (4,2);
  \draw [semithick] (1,2.8) to [out=180,in=180] (1,2) -- (4,2);
  \node at (1,2.8) {$\bullet$};
  \node at (1,3.2) {\tiny $-$};
  % Strand
  \draw [line width=3,white] (2,0) -- (2,4);
  \draw [semithick] (2,0) -- (2,4);
  % Rectangle
  \draw (0,0) rectangle (4,4);
  \end{tikzpicture}};
    \node (C1) at (-2,-6) {
\begin{tikzpicture}[anchorbase, scale=.5]
  % R1
  \draw [line width=3, white] (0,2) to [out=0,in=180] (1.5,2) to [out=0,in=-90] (2,2.5) to [out=90,in=0] (1.5,3);
   \draw [semithick] (0,2) to [out=0,in=180] (1.5,2) to [out=0,in=-90] (2,2.5) to [out=90,in=0] (1.5,3);
  \draw [line width=3, white] (1.5,3) to [out=180,in=90] (1,2.5) to [out=-90,in=180] (1.5,2) -- (4,2);
  \draw [semithick] (1.5,3) to [out=180,in=90] (1,2.5) to [out=-90,in=180] (1.5,2) -- (4,2);
  \node at (2,2.5) {$\bullet$};
  \node at (2.4,2.5) {\tiny $-$};
  % Strand
  \draw [line width=3,white] (2,0) -- (2,2) to [out=90,in=-90] (1.5,2.5) to [out=90,in=-90] (2,3) -- (2,4);
  \draw [semithick] (2,0) -- (2,2) to [out=90,in=-90] (1.5,2.5) to [out=90,in=-90] (2,3) -- (2,4);
  % Rectangle
  \draw (0,0) rectangle (4,4);
  \end{tikzpicture}};
    \node (D1) at (0,-9) {
\begin{tikzpicture}[anchorbase, scale=.5]
  % R1
  \draw [line width=3, white] (0,2) to [out=0,in=180] (1.5,2) to [out=0,in=-90] (2,2.5) to [out=90,in=0] (1.5,3);
   \draw [semithick] (0,2) to [out=0,in=180] (1.5,2) to [out=0,in=-90] (2,2.5) to [out=90,in=0] (1.5,3);
  \draw [line width=3, white] (1.5,3) to [out=180,in=90] (1,2.5) to [out=-90,in=180] (1.5,2) -- (4,2);
  \draw [semithick] (1.5,3) to [out=180,in=90] (1,2.5) to [out=-90,in=180] (1.5,2) -- (4,2);
  \node at (2,2.5) {$\bullet$};
  \node at (2.4,2.5) {\tiny $-$};
  % Strand
  \draw [line width=3,white] (2,0) to [out=90,in=-90]  (1,2) to [out=90,in=-90] (2,3) -- (2,4);
  \draw [semithick]  (2,0) to [out=90,in=-90]  (1,2) to [out=90,in=-90] (2,3) -- (2,4);
  % Rectangle
  \draw (0,0) rectangle (4,4);
  \end{tikzpicture}};
    \node (E1) at (2,-6) {
\begin{tikzpicture}[anchorbase, scale=.5]
  % R1
  \draw [line width=3, white] (0,2) to [out=0,in=180] (3,2) to [out=0,in=0] (3,2.8);
  \draw [semithick] (0,2) to [out=0,in=180] (3,2) to [out=0,in=0] (3,2.8);
  \draw [line width=3, white] (3,2.8) to [out=180,in=180] (3,2) -- (4,2);
  \draw [semithick] (3,2.8) to [out=180,in=180] (3,2) -- (4,2);
  \node at (3,2.8) {$\bullet$};
  \node at (3,3.2) {\tiny $-$};
  % Strand
  \draw [line width=3,white] (2,0) -- (2,4);
  \draw [semithick] (2,0) -- (2,4);
  % Rectangle
  \draw (0,0) rectangle (4,4);
  \end{tikzpicture}};
    \node (F) at (2,-3) {
\begin{tikzpicture}[anchorbase, scale=.5]
  % R1
  \draw [line width=3, white] (0,2) -- (4,2);
  \draw [semithick] (0,2) -- (4,2);
  % Strand
  \draw [line width=3,white] (2,0) -- (2,4);
  \draw [semithick] (2,0) -- (2,4);
  % twists
  \node at (3,2) {$\circ$};
  \node at (3,2.4) {\tiny $+$};
  % Rectangle
  \draw (0,0) rectangle (4,4);
  \end{tikzpicture}};
\draw [<->] (A) -- (B1);
\draw [<->] (B1) -- (C1);
\draw [<->] (C1) -- (D1);
\draw [<->] (D1) -- (E1);
\draw [<->] (E1) -- (F);
\draw [<->] (A) -- (F);
\node at (0,-4.5) {$\text{MM}_8$};
\end{tikzpicture}
\]

{\bf Double points from $F^{1}\times F^2$}

Now we consider double points from $F^1\times F^2$. We have isolated points in a foam, and 1-dimensional sets in isotopies. Let us first establish a local model. We consider $M_1\in F^1$ and $M_2\in F^2$, and we suppose that the parametrization has be chosen so that the seam is supported in the $v$ direction. By a codimension argument, one may assume that $\left(\frac{\partial f}{\partial v}\right)_{M_1}$ has non-zero $s$ coordinate (this will happen at isolated points through time). Then upon reparametrizing in the $u$ direction, one can assume that $\left(\frac{\partial f}{\partial u}\right)_{M_1}$ has no $s$ coordinate. Again because of the codimension, we can assume that $\pi(\left(\frac{\partial f}{\partial u}\right)_{M_1})\neq 0$ (this is a codimension $2$ condition, so this can be safely assumed even for isotopies), and up to rotation in the $x,y$ plane that it is parallel to the $x$ direction with positive coordinate.

At this point, we have:
\[
    (df)_{M_1}=\begin{pmatrix}  >0 & \cdot & \cdot \\ 0 & \cdot & \cdot \\\cdot & \cdot & \cdot \\ 0 & \neq 0 & \cdot \end{pmatrix}
  \]

  Let us now look at $M_2$. As in the $F^2\times F^2$ case, we reparametrize at $M_2$ so that $\left(\frac{\partial f}{\partial u}\right)_{M_2}$ has zero $s$-coordinate and $\left(\frac{\partial f}{\partial v}\right)_{M_2}$ has non-zero $s$ coordinate. Then we can assume (and this will fail through time at isolated points, see  the move from Equation~\eqref{StrandAroundVertex}) that $\pi\left(\left(\frac{\partial f}{\partial u}\right)_{M_1}\right)$ and $\pi\left(\left(\frac{\partial f}{\partial u}\right)_{M_2}\right)$ are not colinear. We thus have:

\[
  (df)_{M_1}=
\]

The behavior as $v$ varies is controlled by the difference of the $y$-coordinates in $\left(\frac{\partial f}{\partial v}\right)_{M_1}$ and $\left(\frac{\partial f}{\partial v}\right)_{M_2}$. Let us denote them as follows: 
\[
  (df)_{M_1}=%
\quad 
\]

For a foam, one may assume that $b_2-b'_2\neq 0$ (but for $1$-parameter families this won't hold at isolated points, see Equation~\eqref{forkslide_inv}). Then one reads the following movie generator (reading from left to right corresponds to the case where $b'_2>b_2$, the other case is obtained by reading in reverse direction):
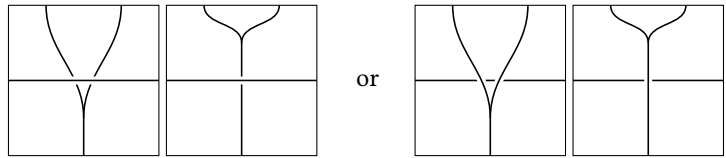
\[
  %
\]
Again, the framing can be assumed not to be subject to any change (through time, this will yield moves from Equations~\eqref{eq:frthrucrossF1} and~\eqref{eq:frthrucrossF1_2}).

Now we add the time parameter and consider the failure of the conditions considered in the previous paragraph.

If $(\frac{\partial f}{\partial v})_{M_1}$ has zero $s$-coordinate, since this is a codimension $1$ condition we are down to isolated points, and we can assume that the second derivative has non-zero $s$-coordinate. This implies that around $M_1$ we have:
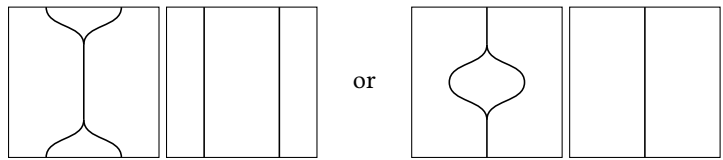
\[
  %
\]
At the singularity we have the following situation (we only draw this in the first case, as the corresponding picture in the second case is not illuminating):
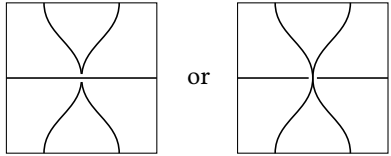
\[
%
\]

The evolution through time is controlled by the relative value of the $y$-coordinates of $\left(\frac{\partial f}{\partial t}\right)_{M_1}$ and of  $\left(\frac{\partial f}{\partial t}\right)_{M_2}$. For codimension reason, the difference can be assumed to be non-zero, and we thus read the following moves:
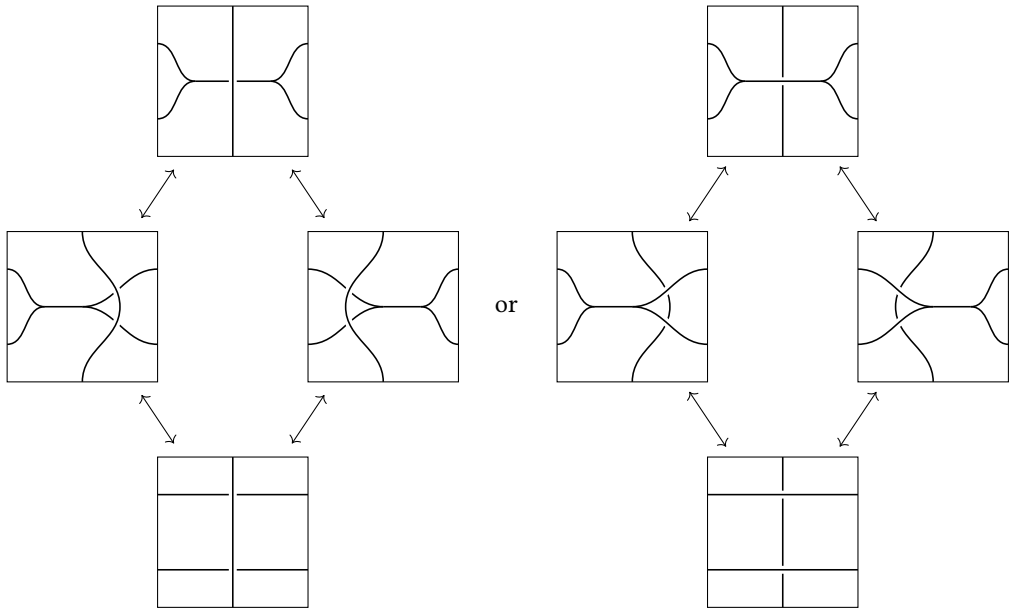
\[
  %
\]
In the other case we get the following move:
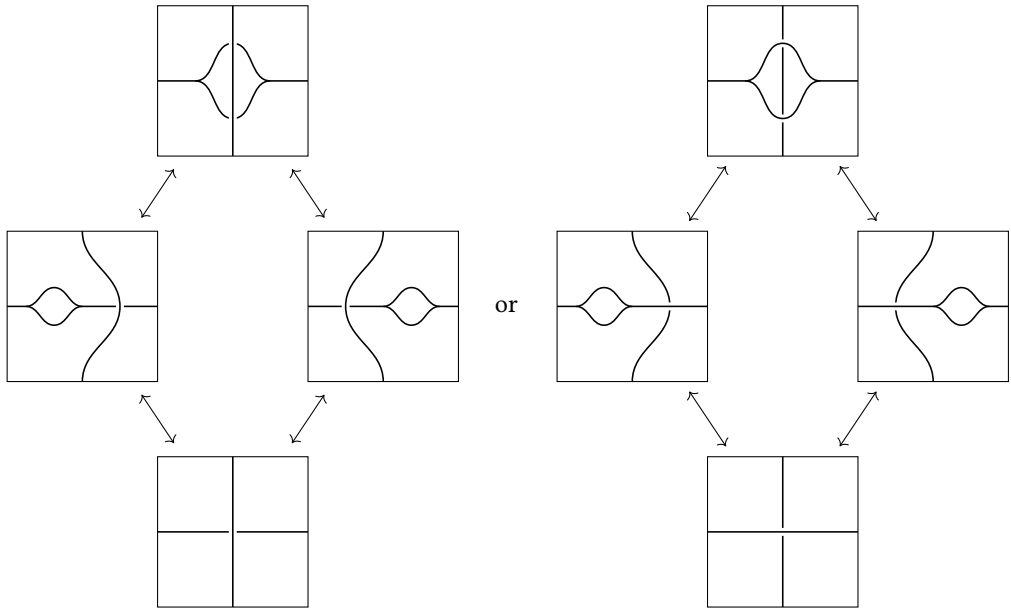
\[
  %
\]

Again nothing happens to the framing.

Now, we consider the case of non-transverse intersections. The case where $\pi\left(\left(\frac{\partial f}{\partial u}\right)_{M_1}\right)$ and $\pi\left(\left(\frac{\partial f}{\partial u}\right)_{M_2}\right)$ are colinear corresponds to a singular situation as follows (we have colored the web for better visibility). 
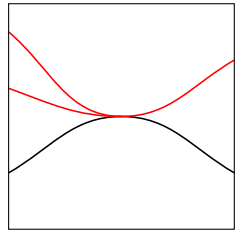
\[
  %
\]
Indeed, one can compare the lines drawn by $f(M_1+(\epsilon_1,0,0))$ and $f(M_2+(\epsilon_1,0,0)$ as was done for the movie moves $\rm{MM}_3$ and $\rm{MM}_4$. Then the strand goes away from the other line at pace proportional to $\epsilon_1^2$, while the two legs of the web glue smoothly and thus diverge slower than $\epsilon_1^2$. This excludes the following situation: \
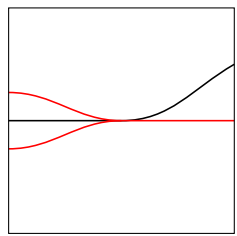
\[
  %
\]

The shape before and after in the $t$ direction will be controlled by the relative $x$ coordinates of the $t$ derivatives. The difference can be assumed to be non-zero, and we read the following move:
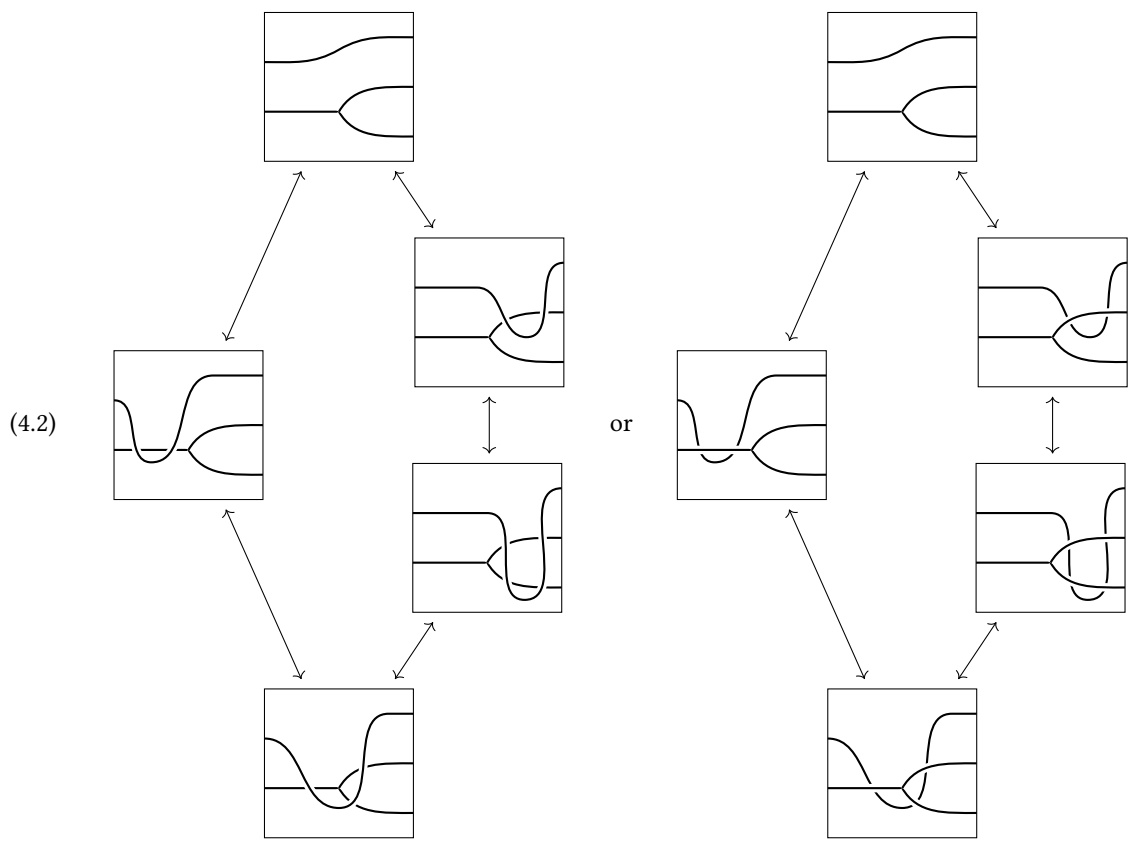
\begin{equation}  \label{StrandAroundVertex}
  %
\end{equation}

Under mirror image one reads:
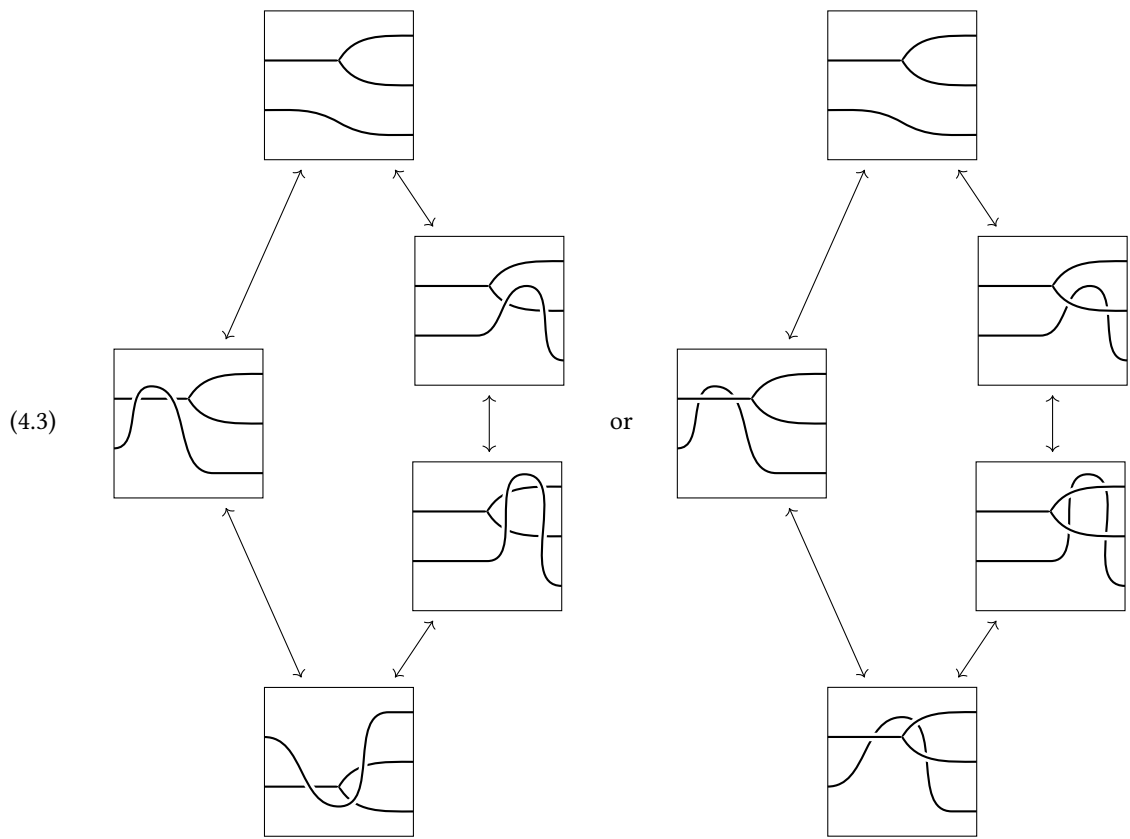
\begin{equation}
  \label{StrandAroundVertex2}
  %
\end{equation}

The case where the $y$-coordinates in the $\partial v$ derivatives balance induces the following movie move:
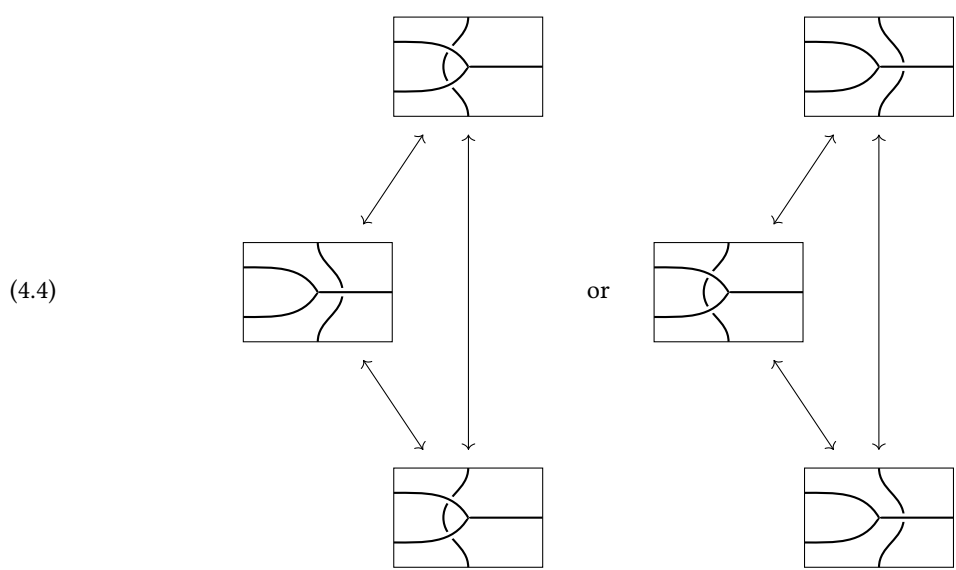
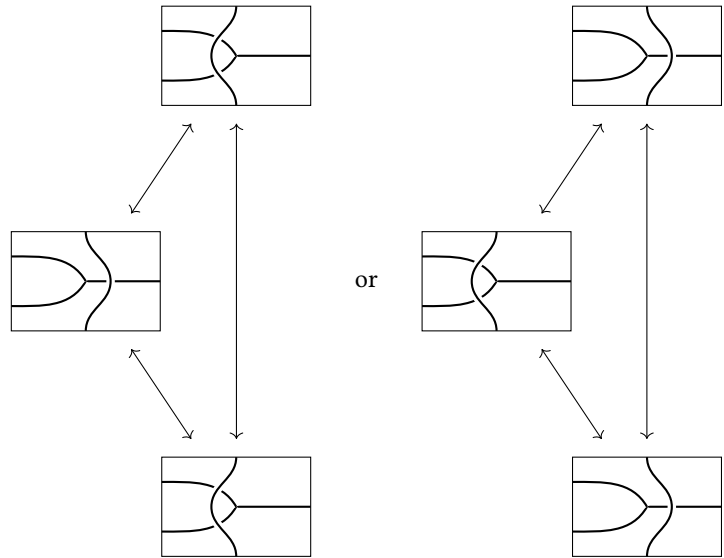
\begin{gather}
  \label{forkslide_inv}
  %
\end{gather}

It finally remains to analyze changes in framing. This can be supported on the strand or on the trivalent web, and we read the following moves (they are deduced from isotopies of webs by superposing an extra strand). Below we have drawn moves with framing change going under. The same ones should be considered for the other framing change.
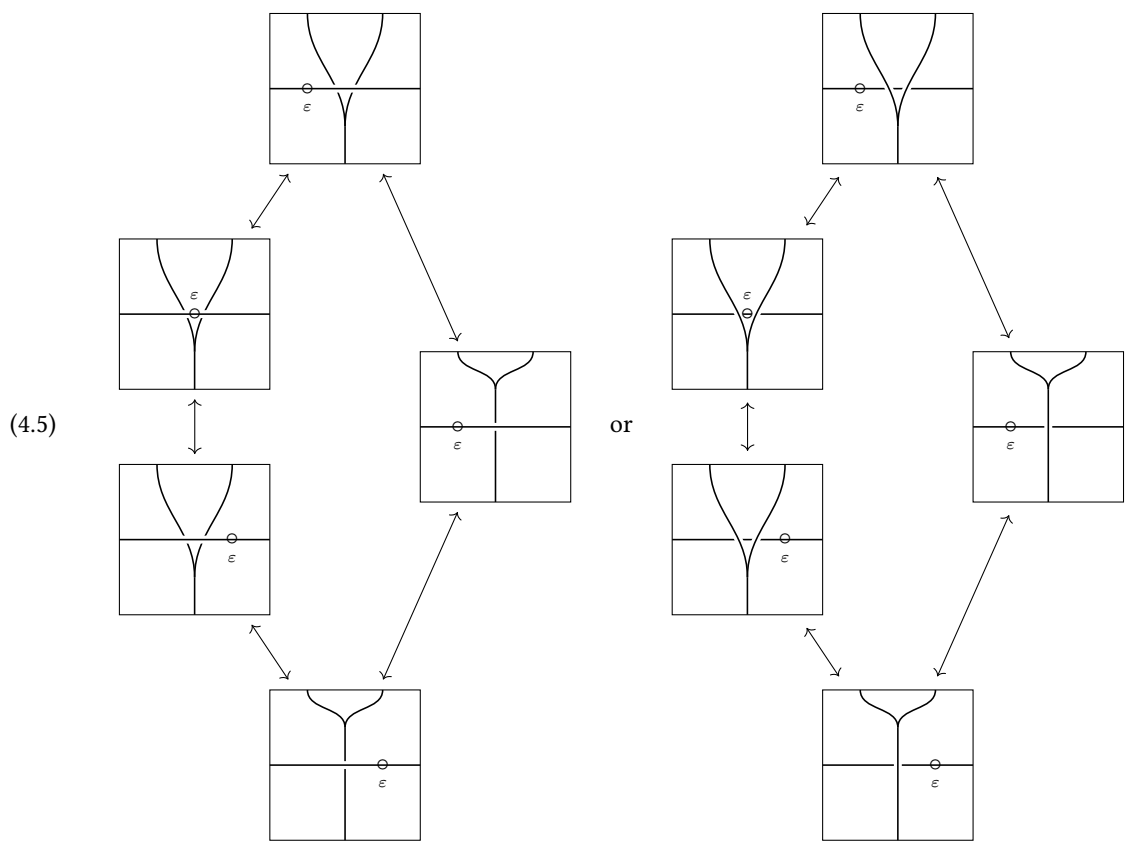
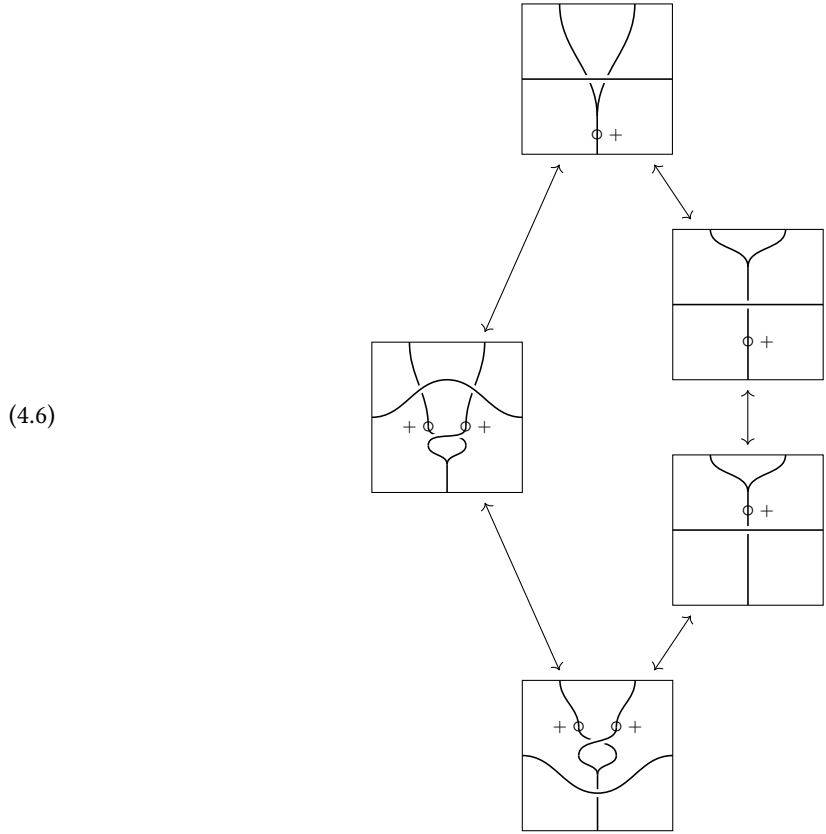
\begin{equation}\label{eq:frthrucrossF1}
  %
};
\draw [<->] (A) -- (B1);
\draw [<->] (A) -- (B2);
\draw [<->] (B1) -- (D);
\draw [<->] (B2) -- (C2);
\draw [<->] (C2) -- (D);
\end{tikzpicture}
\end{equation}

Again, similar moves should be considered with other framing changes, other crossing signs and also up to mirror image.

{\bf Double points in $F^1\times F^1$}

We are looking at a situation of codimension $7$ (twice $1$ for restriction to $F$, then twice $1$ for restriction to $F^1$, and $3$ for the equality of the $x,y,s$ coordinates). Such a situation is thus not generic for a foam, but does happen at isolated points during a foam isotopy.

One may assume that one of the $x$ or $y$ entries of the $\frac{\partial f}{\partial u}$ derivatives at $M_1$ and $M_2$, the two points of interest, is non-zero. Up to rotation, the one at $M_1$ can be assumed to be parallel to the $x$ direction. Then the other one can be assumed to have a non-zero $y$ coordinate: the intersection is transverse. One thus has a singular situation of the following kind:

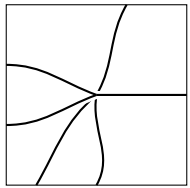
\[

\]

Depending on the relative value of $b_1$ and $b_2$ (all of which can be assumed non-zero) at $M_1$ and $M_2$, the webs obtained for positive and negative values of $v$ will be of the following kind: 
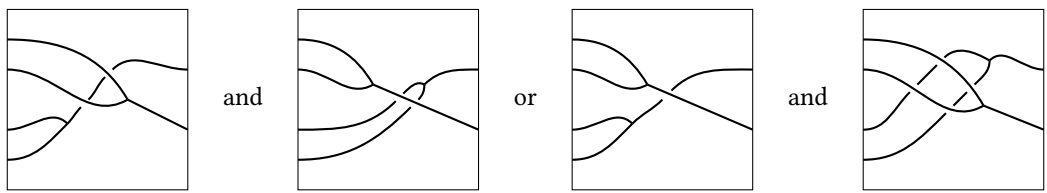
\[
 %
\]

Letting now $t$ vary, depending on the relative value of the $x$ and $y$ entries of $\frac{\partial f}{\partial t}$ at $M_1$ and $M_2$ (and for codimension reason these values can be assumed to be different), one resolves the singularity by pushing one of the trivalent vertices in one direction for positive $t$ and in the other direction for negative $t$, yielding the following movie move.

  \begin{equation}
\label{C1}
%
\end{equation}

For codimension reasons, the framing can be assumed not to change.

{\bf Double points in $F^0\times F^2$}

The analysis is similar to the one sketched previously and yields the following movie moves:

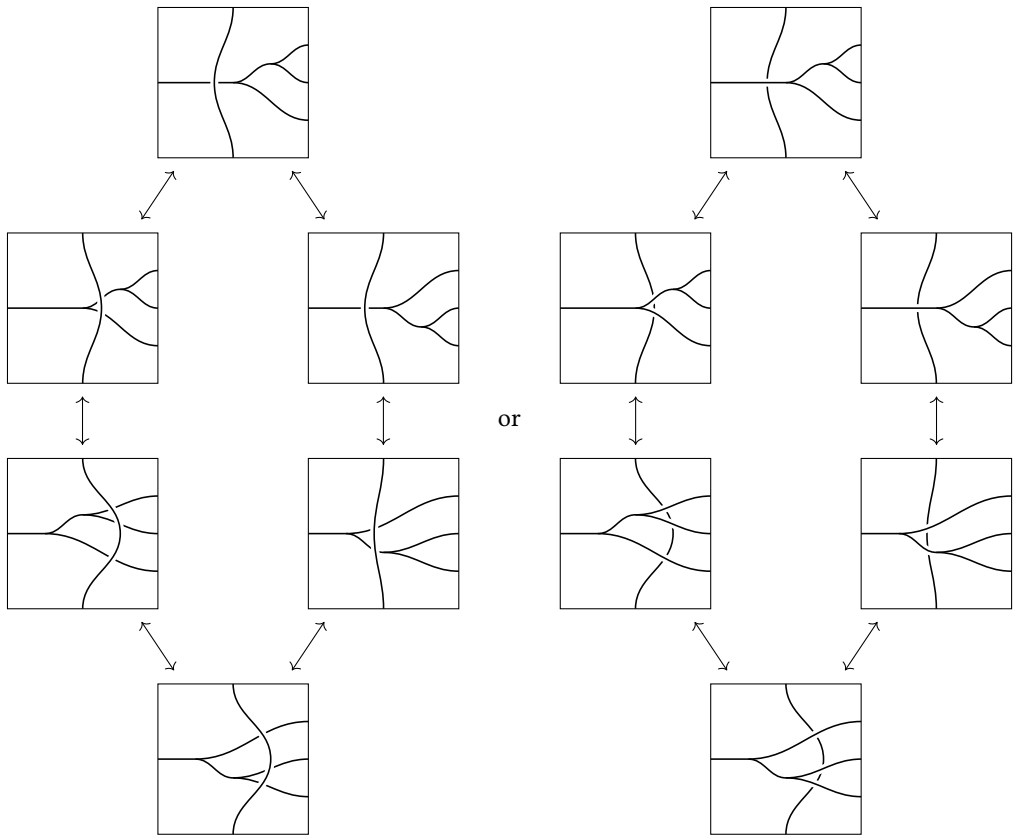
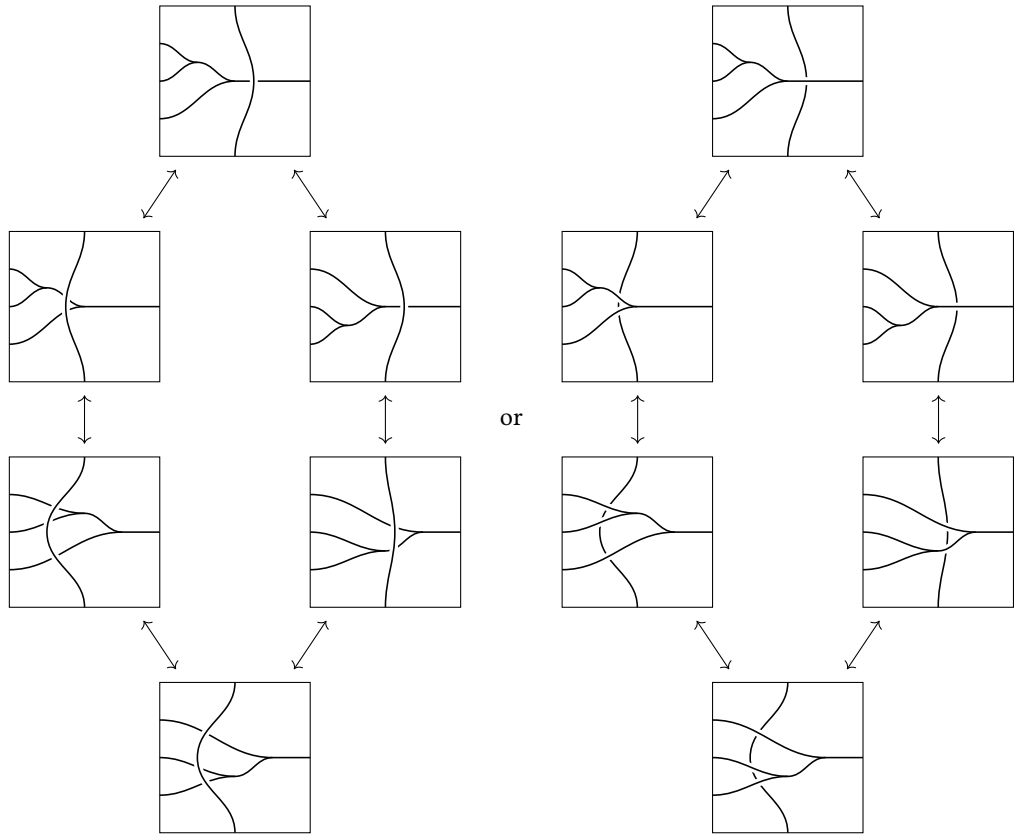
\[
  %
\]

This exhausts the analysis for double points.

{\bf Triple points}

In the case of foams, we have isolated singularities, so no additional restriction can be imposed. Generically, triple points involve three points on $F^2$. The analysis thus parallels the classical situation, and we find as generators Reidemeister III moves.

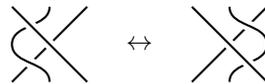
\[
  %
\]
    
These moves can occur with all possible orientations and compatible choices for crossings.

Adding up the time direction, one recovers the following classical movie moves: MM5, controlling the invertibility of the third Reidemeister move, and MM6, corresponding to two of the strands being non-transverse. In the pictures below, one should consider all versions for crossings.

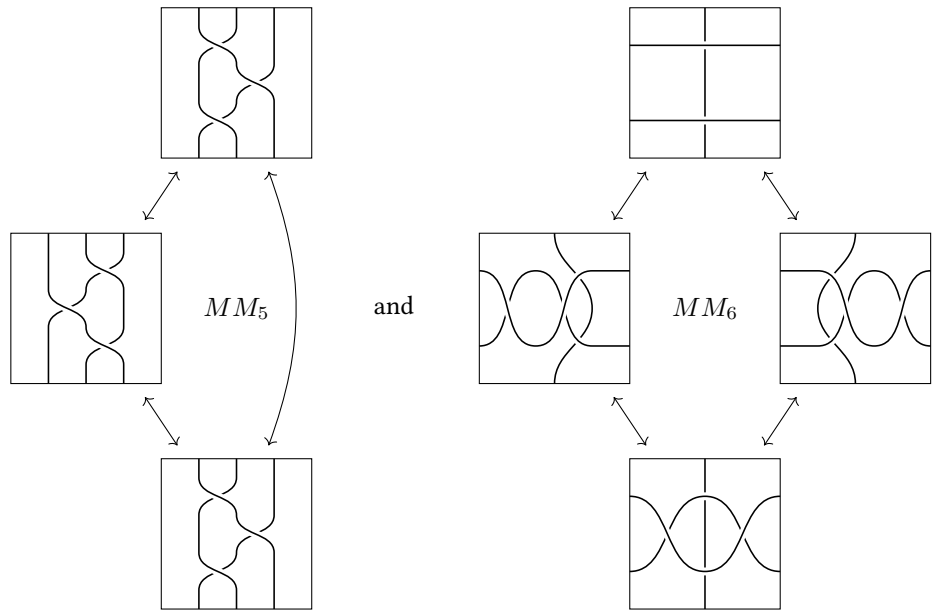
\[
  %
\]

It might also happen that the underlying surface is just a Reidemeister 3 generator moving along time without special feature, but crossing a frame change on one of the three strands. This yields the following move (and analogs of it with other crossing conventions and with the half-twist traveling on other strands):
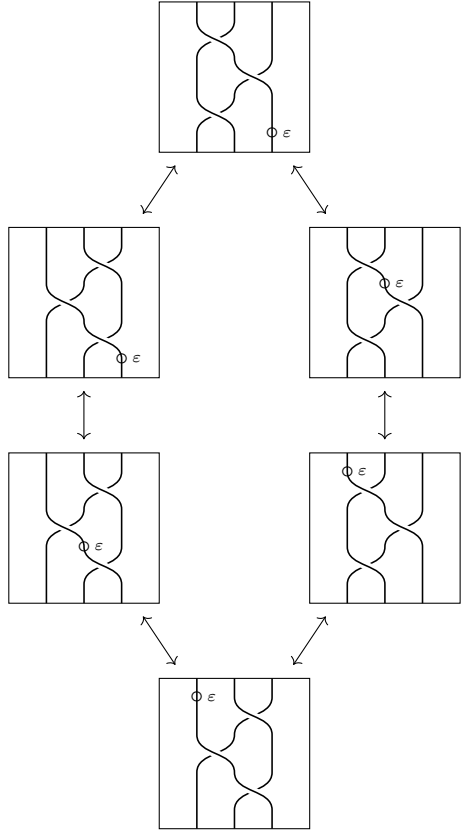
\[
  %
\]

Then one can look at triple points involving one point on $F^{1}$ and two on $F^{2}$. This happens at isolated points, and no other degeneracy generically happens. One gets the following moves, as in~\cite{Carter_foams}:
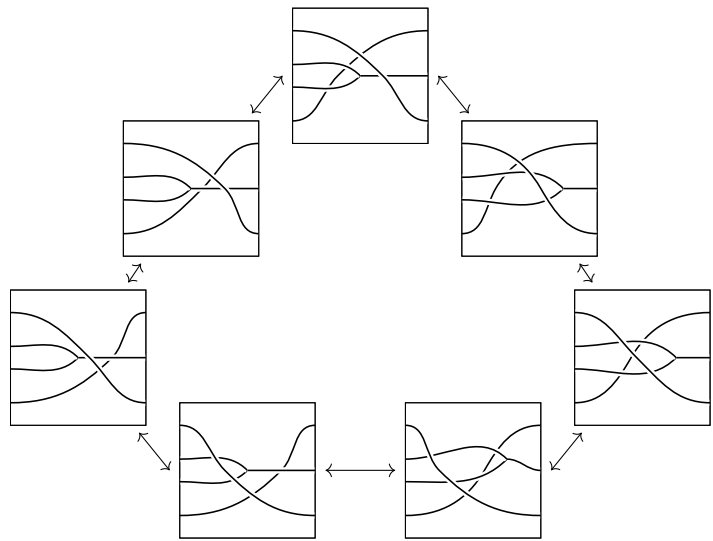
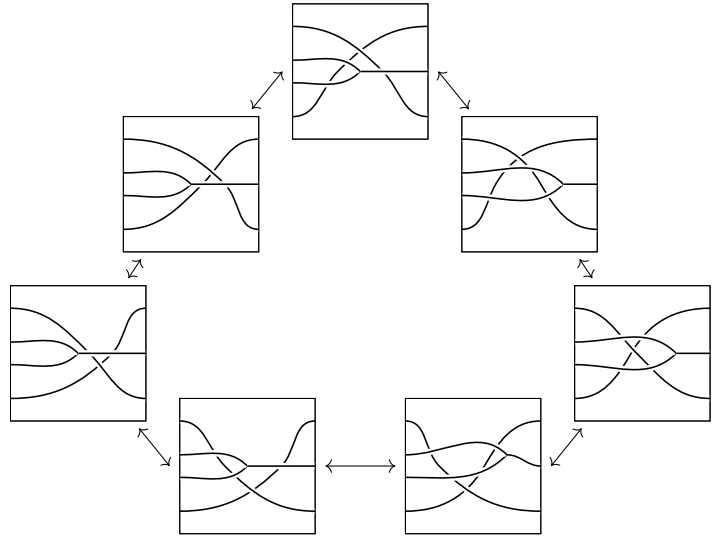
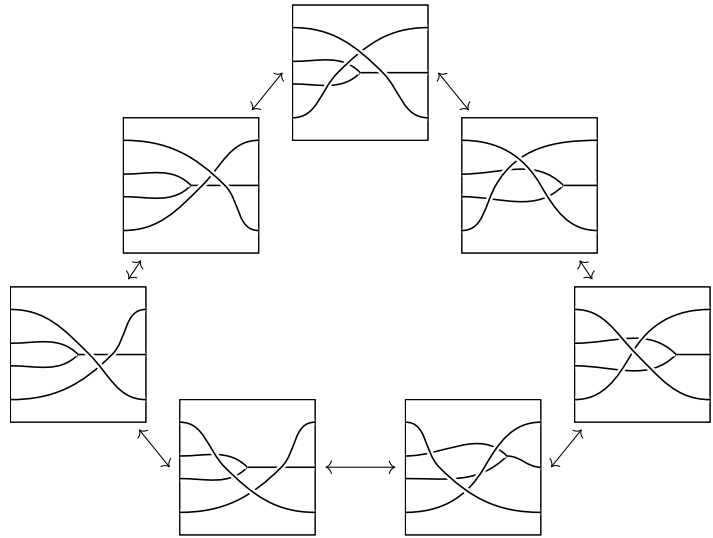
  \begin{equation}
    \label{C2}
%
\end{equation}

{\bf Quadruple points}

Quadruple points only occur at isolated places, involving four points from $F^2$: this is covered by the classical, link situation. One gets the classical movie move $\text{MM}_{10}$. Again all versions of it should be considered.

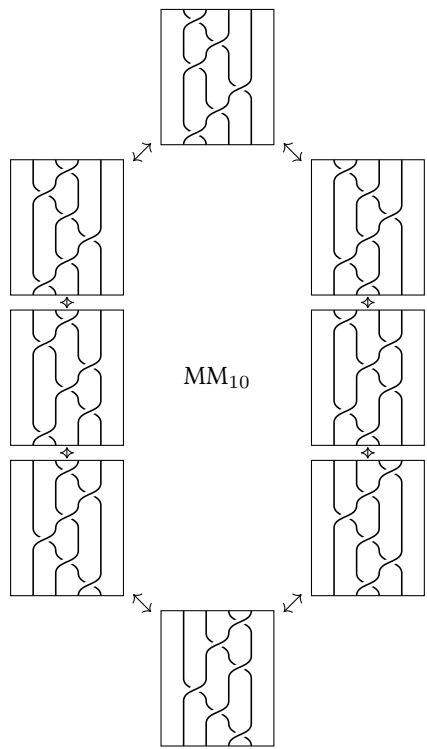
\[
  %
};
\draw [<->] (A) -- (B1);
\draw [<->] (A) -- (B2);
\draw [<->] (B1) -- (C1);
\draw [<->] (B2) -- (C2);
\draw [<->] (C1) -- (D1);
\draw [<->] (C2) -- (D2);
\draw [<->] (D1) -- (E);
\draw [<->] (D2) -- (E);
\node at (0,-4) {$\text{MM}_{10}$};
\end{tikzpicture}
\]

This exhausts the classification of multiple points. We now consider the neighborhood of a single point.

%%%%%%%%%%%%%%%%%%%%%%%%%%%%%
\subsection{Neighborhood of a point on $F^{2}$} \label{subsec:nF2}
%%%%%%%%%%%%%%%%%%%%%%%%%%%%%%%%%

When framing is ignored, the analysis to be run here is classical and equivalent to Morse theory techniques. As foam generators, one gets birth cobordisms, death cobordisms, and saddles.

    
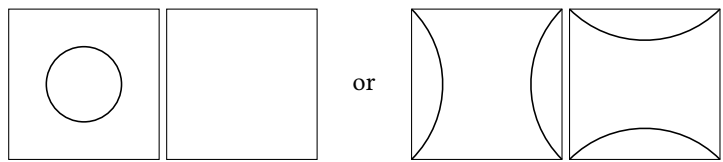
\[
      \begin{tikzpicture}[anchorbase]
        \node at (0,0)
        {
          \begin{tikzpicture}[anchorbase,scale=.5]
            \draw [semithick] (2,2) circle (1);
            \draw (0,0) rectangle (4,4);
          \end{tikzpicture}
        };
        \node at (2.1,0)
        {
          \begin{tikzpicture}[anchorbase,scale=.5]
                     \draw (0,0) rectangle (4,4);
          \end{tikzpicture}
        };
      \end{tikzpicture}
      \quad \text{or}\quad
      \begin{tikzpicture}[anchorbase]
        \node at (0,0)
        {
          \begin{tikzpicture}[anchorbase,scale=.5]
            \draw [semithick] (0,0) to [out=45,in=-45] (0,4);
            \draw [semithick] (4,0) to [out=135,in=-135] (4,4);
            \draw (0,0) rectangle (4,4);
          \end{tikzpicture}
        };
        \node at (2.1,0)
        {
          \begin{tikzpicture}[anchorbase,scale=.5]
            \draw [semithick] (0,0) to [out=45,in=135] (4,0);
            \draw [semithick] (0,4) to [out=-45,in=-135] (4,4);
            \draw (0,0) rectangle (4,4);
          \end{tikzpicture}
        };
      \end{tikzpicture}
    \]

When adding the time dimension, one gets classical crossingless movie moves, which reduce to MM11, as well as MM12 and MM13 which will need additional attention to incorporate the framing data.

Recall from the previous paragraphs that Morse points occur when we assume that both $\frac{\partial f}{\partial u}$ and $\frac{\partial f}{\partial v}$ have vanishing $s$ coordinates. This corresponds to a codimension $3$ condition, while the graph of one function is of dimension $3$. This thus occurs at isolated points. Denote:
\[
\frac{\partial f}{\partial u}= \begin{pmatrix} a_1 \\ a_2 \\ a_3 \\0 \end{pmatrix},\quad \frac{\partial f}{\partial v}=\begin{pmatrix}b_1 \\ b_2 \\ b_3 \\0 \end{pmatrix}.
\]
By $\rm{SL}_2$ action in the $(x,y)$ plane one reduces as usual to (unless the two vectors are colinear, which will yield moves $\rm{MM}_{12}$ and $\rm{MM}_{13}$):
\[
  df_M=\begin{pmatrix} 1 & 0 & c_1 \\ 0 & \pm 1 & c_2 \\ a_3 & b_3 & c_3 \\ 0 & 0 & c_4 \end{pmatrix}
\]
Up to symmetry we will assume the entry $b_2$ to be $1$. One can then write:
\[
  f(u_0+\epsilon_1,v_0+\epsilon_2,w_0)=\begin{pmatrix} \epsilon_1 \\ \epsilon_2 \\ \cdot \\ d_4\epsilon_1^2+e_4\epsilon_1\epsilon_2+g_4\epsilon_2^2 \end{pmatrix}
\]
By a codimension argument one can assume that $\begin{vmatrix}d_4 & e_4 \\ e_4 & d_4\end{vmatrix}\neq 0$ (this will fail through time at isolated points, see the movie move~$\rm{MM}_{11}$), yielding a cap, cup, or saddle generator.

Now for the framing, since we have:
\[
  df_M=\begin{pmatrix} 1 & 0 & c_1 \\ 0 & \pm 1 & c_2 \\ a_3 & b_3 & c_3 \\ 0 & 0 & c_4 \end{pmatrix}
\]
one can replace $w$ by a linear combination of $u$, $v$ and $w$ to get:
\[
  df_M=\begin{pmatrix} 1 & 0 & 0 \\ 0 & \pm 1 & 0 \\ a_3 & b_3 & c_3 \\ 0 & 0 & c_4 \end{pmatrix}
\]
Through time it might happen that a framing change occurs at the Morse point ($c_4=0$, see the analysis on page~\pageref{page:b4=0}), but for a generic foam $c_4\neq 0$ and thus:
\begin{equation} \label{eq:framingforcup}
  \frac{\partial f}{\partial w}\notin \Span_{\R}\left(\frac{\partial f}{\partial u},\frac{\partial f}{\partial v}, \vec{z}\right)
\end{equation}
Since this is an open condition (as the three vectors span a space of dimension $3$), it will remain true around $M$ and thus the framing does not become singular on the generators.

Let us now go back to the case where $\begin{pmatrix}d_4 & e_4 \\ e_4 & d_4\end{pmatrix}$ is of rank $1$. One can reduce to the following situation:
\[
  f(u_0+\epsilon_1,v_0+\epsilon_2,w_0)=f(u_0,v_0,w_0)+\begin{pmatrix} \epsilon_1 \\ \epsilon_2 \\ \cdot \\ d_4 \epsilon_1^2+g_4 \epsilon_2^3\end{pmatrix},\;d_4\neq 0,\;g_4\neq 0
\]
This yields the classical movie move MM11, with generically no framing change.

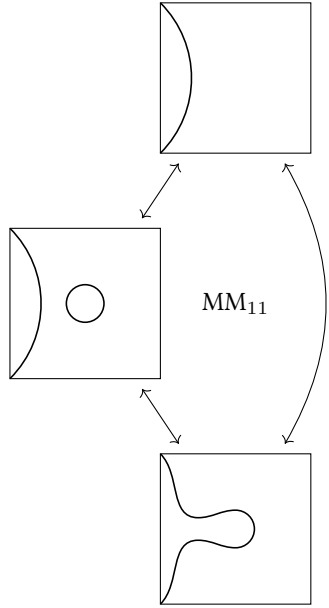
\[
  \begin{tikzpicture}[anchorbase]
    \node (A) at (0,0) {
\begin{tikzpicture}[anchorbase, scale=.5]
  % Strands
    \draw [semithick] (0,0) to [out=45,in=-45] (0,4);
  % Rectangle
  \draw (0,0) rectangle (4,4);
  \end{tikzpicture}};
    \node (B1) at (-2,-3) {
\begin{tikzpicture}[anchorbase, scale=.5]
  % Strands
  \draw [semithick] (0,0) to [out=45,in=-45] (0,4);
  \draw [semithick] (2,2) circle (.5);
  % Rectangle
  \draw (0,0) rectangle (4,4);
  \end{tikzpicture}};
    \node (C) at (0,-6) {
\begin{tikzpicture}[anchorbase, scale=.5]
  % Strands
  \draw [semithick] (0,0) to [out=45,in=180] (1,1.7) to [out=0,in=180] (2,1.5) to [out=0,in=-90] (2.5,2) to [out=90,in=0] (2,2.5) to [out=180,in=0] (1,2.3) to [out=180,in=-45] (0,4);
  % Rectangle
  \draw (0,0) rectangle (4,4);
  \end{tikzpicture}};
\draw [<->] (A) -- (B1);
\draw [<->] (B1) -- (C);
\draw [<->] (A) to [out=-60,in=60] (C);
\node at (0,-3) {$\text{MM}_{11}$};
\end{tikzpicture}
\]

Let us now consider the case where the $\partial u$ and $\partial v$ derivatives have colinear $x,y$ projections. Up to rotation in the $x,y$ plane one can assume that:
\[
  \frac{\partial f}{\partial u}= \begin{pmatrix} a_1 \\ 0 \\ a_3 \\0 \end{pmatrix},\quad \frac{\partial f}{\partial v}=\begin{pmatrix}b_1 \\ 0 \\ b_3 \\0 \end{pmatrix},b_1>0.
\]

% \[
%   \frac{\partial f}{\partial u}= \begin{pmatrix} a_1 \\ a_2 \\ a_3 \\0 \end{pmatrix},\quad \frac{\partial f}{\partial v}=\begin{pmatrix}b_1 \\ b_2 \\ b_3 \\0 \end{pmatrix}.
% \]
% and:

One can assume that $\left|\begin{pmatrix} d_4 & e_4 \\ e_4 & g_4\end{pmatrix}\right|\neq 0$. This implies that there exists a change in coordinates $u,v$ such that:
\[
  \frac{\partial^2 f}{\partial u^2}= \begin{pmatrix} d_1 \\ d_2 \\ d_3 \\ \pm 1 \end{pmatrix},\quad \frac{\partial^2 f}{\partial u\partial v}=\begin{pmatrix}e_1 \\ e_2 \\ e_3 \\0 \end{pmatrix},  \frac{\partial^2 f}{\partial v^2}= \begin{pmatrix} g_1 \\ g_2 \\ g_3 \\ \pm 1 \end{pmatrix}.
\]

We first consider the case where the last line is $\epsilon_1^2+\epsilon_2^2$. The case with two minus signs will follow by mirror. The mixed case will be considered later. This gives as Taylor expansion:
\[
  f(u_0+\epsilon_1,v_0+\epsilon_2,w_0)=\begin{pmatrix} a_1 \epsilon_1+b_1\epsilon_2 \\ d_2 \epsilon_1^2+e_2\epsilon_1\epsilon_2+ g_2\epsilon_2^2 \\ a_3\epsilon_1+b_3\epsilon_2 \\ \epsilon_1^2+\epsilon_2^2 \end{pmatrix}
\]

Let us focus on the coordinates $x,y$ and $s$, and fix a given value $s=r$. Since the three equations are homogeneous, the surface drawn is topologically just the cone on one of these curves, say for $r=1$. At $r=1$, we consider the set of points such that $\epsilon_1^2+\epsilon_2^2=1$, which we can reparametrize as $\epsilon_1=\cos(\theta)$ and $\epsilon_2=\sin(\theta)$ for $\theta\in [-\pi,\pi]$. So we care about the following parametric curve:
\[
  \{(a_1\cos(\theta)+b_1\sin(\theta),d_2\cos(\theta)^2+e_2\cos(\theta)\sin(\theta)+g_2\sin(\theta)^2),\quad \theta\in [-\pi,\pi]\}
\]
Below is an example with parameters $a_1=b_1=1$ and $d_2=1$, $e_2=2$, $g_2=3$.
  
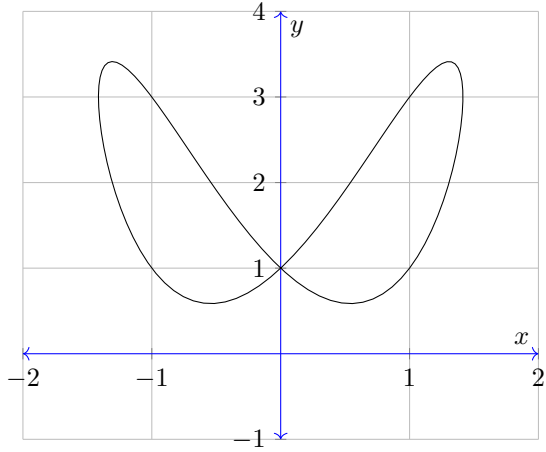
\[
\begin{tikzpicture}
    \begin{axis}[
            xmin=-2,xmax=2,
            ymin=-1,ymax=4,
            grid=both,
            ]
            \addplot [domain=0:360,samples=100]({cos(x)+sin(x)},{cos(x)^2+2*cos(x)*sin(x)+3*sin(x)^2}); 
    \end{axis}
\end{tikzpicture}
\]

We want to claim that the above curve is representative of the generic situation. Let us go back to the curve:
\[
  \left\{
  \begin{pmatrix} a_1 \epsilon_1+b_1\epsilon_2 \\ d_2 \epsilon_1^2+e_2\epsilon_1\epsilon_2+ g_2\epsilon_2^2 \end{pmatrix},\; \epsilon_1^2+\epsilon_2^2=1
  \right\}
\]

Up to changing coordinates: $\epsilon_1=\frac{a_1\epsilon'_1+b_1\epsilon'_2}{\sqrt{a_1^2+b_1^2}}$ and $\epsilon_2=\frac{b_1\epsilon'_1-a_1\epsilon'_2}{\sqrt{a_1^2+b_1^2}}$, one reduces to:
\[
  \left\{
  \begin{pmatrix} \sqrt{a_1^2+b_1^2} \epsilon'_1 \\ d'_2 {\epsilon'_1}^2+e'_2\epsilon'_1\epsilon'_2+ f'_2{\epsilon'_2}^2 \end{pmatrix},\; {\epsilon'_1}^2+{\epsilon'_2}^2=1
  \right\}
\]

Let us now look at multiple points: assume that $(\xi_1,\xi_2)$ and $(\xi'_1,\xi'_2)$ yield the same point. From the first coordinate, one reads that $\xi_1=\xi'_1$, and thus from the equation ${\epsilon'_1}^2+{\epsilon'_2}^2=1$ one gets that $\xi_2=\pm \xi'_2$. Now, $(\xi_1,\xi_2)$ and $(\xi_1,-\xi_2)$ produce the same point if $2e'_2\xi_1\xi_2=0$. Generically, $e'_2\neq 0$ and one gets $\xi_1=0$ and thus as only double points, $(0,1)$ and $(0,-1)$. Thus the curve is a circle with a single double point, the only topological solution being the one depicted above.

Adding up the time parameter, we will get a framed version of the classical move $\rm{MM}_{12}$. Let us focus on the framing data, and compare to the situation in Equation~\eqref{eq:framingforcup}. Now
\[
  \Span_R\left(\frac{\partial f}{\partial u},\frac{\partial f}{\partial v},\vec{z}\right)
\]
is only two-dimensional, but will become three-dimensional when perturbed. Although at the singular point the framing does not belong to this spanned space, this is not an open condition anymore. To see this, notice that one can only reduce to:
\[
  \frac{\partial f}{\partial w}=\begin{pmatrix}0 \\ c_2 \\ c_3 \\ c_4 \end{pmatrix}
\]
The non-vanishing of the $y$-coordinate will dictate the shape of the framing, and justifies the following move:

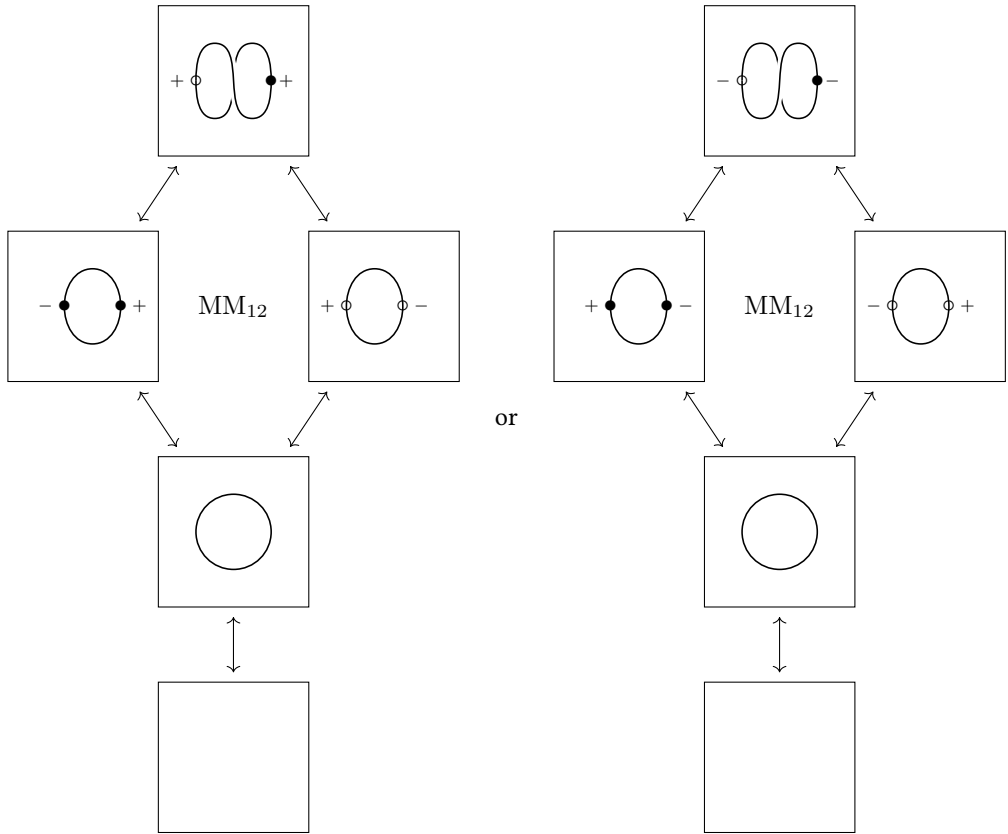
\[
  \begin{tikzpicture}[anchorbase]
    \node (A) at (0,0) {
\begin{tikzpicture}[anchorbase, scale=.5]
  % Strands
  \draw [line width=3, white] (1,2) to [out=-90,in=-180] (1.5,1) to [out=0,in=180] (2.5,3) to [out=0,in=90] (3,2);
  \draw [semithick] (1,2) to [out=-90,in=-180] (1.5,1) to [out=0,in=180] (2.5,3) to [out=0,in=90] (3,2);
  \draw [line width=3, white] (1,2) to [out=90,in=-180] (1.5,3) to [out=0,in=180] (2.5,1) to [out=0,in=-90] (3,2);
  \draw [semithick] (1,2) to [out=90,in=-180] (1.5,3) to [out=0,in=180] (2.5,1) to [out=0,in=-90] (3,2);
  \node at (3,2) {$\bullet$};
  \node at (3.4,2) {\tiny $+$};
  \node at (1,2) {$\circ$};
  \node at (.5,2) {\tiny $+$};
  % Rectangle
  \draw (0,0) rectangle (4,4);
  \end{tikzpicture}};
    \node (B1) at (-2,-3) {
\begin{tikzpicture}[anchorbase, scale=.5]
  % Strands
  \draw [semithick] (1.5,2) to [out=-90,in=-180] (2.25,1) to [out=0,in=-90] (3,2) to [out=90,in=0] (2.25,3) to [out=180,in=90] (1.5,2);
  \node at (1.5,2) {$\bullet$};
  \node at (1,2) {\tiny $-$};
  \node at (3,2) {$\bullet$};
  \node at (3.5,2) {\tiny $+$};
   % Rectangle
  \draw (0,0) rectangle (4,4);
  \end{tikzpicture}};
    \node (B2) at (2,-3) {
\begin{tikzpicture}[anchorbase, scale=.5]
  % Strands
  \draw [semithick] (1,2) to [out=-90,in=-180] (1.75,1) to [out=0,in=-90] (2.5,2) to [out=90,in=0] (1.75,3) to [out=180,in=90] (1,2);
  \node at (1,2) {$\circ$};
  \node at (.5,2) {\tiny $+$};
  \node at (2.5,2) {$\circ$};
  \node at (3,2) {\tiny $-$};
  % Rectangle
  \draw (0,0) rectangle (4,4);
  \end{tikzpicture}};
    \node (C) at (0,-6) {
\begin{tikzpicture}[anchorbase, scale=.5]
  % Strands
  \draw [semithick] (1,2) to [out=-90,in=-180] (2,1) to [out=0,in=-90] (3,2) to [out=90,in=0] (2,3) to [out=180,in=90] (1,2);
  % Rectangle
  \draw (0,0) rectangle (4,4);
  \end{tikzpicture}};
    \node (D) at (0,-9) {
\begin{tikzpicture}[anchorbase, scale=.5]
  % Strands
  % Rectangle
  \draw (0,0) rectangle (4,4);
  \end{tikzpicture}};
\draw [<->] (A) -- (B1);
\draw [<->] (A) -- (B2);
\draw [<->] (B1) -- (C);
\draw [<->] (B2) -- (C);
\draw [<->] (C) -- (D);
\node at (0,-3) {$\rm{MM}_{12}$};
\end{tikzpicture}
\quad \text{or}\quad
  \begin{tikzpicture}[anchorbase]
    \node (A) at (0,0) {
\begin{tikzpicture}[anchorbase, scale=.5]
  % Strands
  \draw [line width=3, white] (1,2) to [out=90,in=-180] (1.5,3) to [out=0,in=180] (2.5,1) to [out=0,in=-90] (3,2);
  \draw [semithick] (1,2) to [out=90,in=-180] (1.5,3) to [out=0,in=180] (2.5,1) to [out=0,in=-90] (3,2);
  \draw [line width=3, white] (1,2) to [out=-90,in=-180] (1.5,1) to [out=0,in=180] (2.5,3) to [out=0,in=90] (3,2);
  \draw [semithick] (1,2) to [out=-90,in=-180] (1.5,1) to [out=0,in=180] (2.5,3) to [out=0,in=90] (3,2);
  \node at (3,2) {$\bullet$};
  \node at (3.4,2) {\tiny $-$};
  \node at (1,2) {$\circ$};
  \node at (.5,2) {\tiny $-$};
  % Rectangle
  \draw (0,0) rectangle (4,4);
  \end{tikzpicture}};
    \node (B1) at (-2,-3) {
\begin{tikzpicture}[anchorbase, scale=.5]
  % Strands
  \draw [semithick] (1.5,2) to [out=-90,in=-180] (2.25,1) to [out=0,in=-90] (3,2) to [out=90,in=0] (2.25,3) to [out=180,in=90] (1.5,2);
  \node at (1.5,2) {$\bullet$};
  \node at (1,2) {\tiny $+$};
  \node at (3,2) {$\bullet$};
  \node at (3.5,2) {\tiny $-$};
   % Rectangle
  \draw (0,0) rectangle (4,4);
  \end{tikzpicture}};
    \node (B2) at (2,-3) {
\begin{tikzpicture}[anchorbase, scale=.5]
  % Strands
  \draw [semithick] (1,2) to [out=-90,in=-180] (1.75,1) to [out=0,in=-90] (2.5,2) to [out=90,in=0] (1.75,3) to [out=180,in=90] (1,2);
  \node at (1,2) {$\circ$};
  \node at (.5,2) {\tiny $-$};
  \node at (2.5,2) {$\circ$};
  \node at (3,2) {\tiny $+$};
  % Rectangle
  \draw (0,0) rectangle (4,4);
  \end{tikzpicture}};
    \node (C) at (0,-6) {
\begin{tikzpicture}[anchorbase, scale=.5]
  % Strands
  \draw [semithick] (1,2) to [out=-90,in=-180] (2,1) to [out=0,in=-90] (3,2) to [out=90,in=0] (2,3) to [out=180,in=90] (1,2);
  % Rectangle
  \draw (0,0) rectangle (4,4);
  \end{tikzpicture}};
    \node (D) at (0,-9) {
\begin{tikzpicture}[anchorbase, scale=.5]
  % Strands
  % Rectangle
  \draw (0,0) rectangle (4,4);
  \end{tikzpicture}};
\draw [<->] (A) -- (B1);
\draw [<->] (A) -- (B2);
\draw [<->] (B1) -- (C);
\draw [<->] (B2) -- (C);
\draw [<->] (C) -- (D);
\node at (0,-3) {$\rm{MM}_{12}$};
\end{tikzpicture}
\]

Note that at the bottom, the side of the circle used to pair the half-twists together is irrelevant, thanks to Equation~\ref{eq:twistpaircup}.

We now go back to the case where the last line in the Taylor expansion is $\epsilon_1^2-\epsilon_2^2$:
\[
  f(u_0+\epsilon_1,v_0+\epsilon_2,w_0)=\begin{pmatrix} a_1 \epsilon_1+b_1\epsilon_2 \\ d_2 \epsilon_1^2+e_2\epsilon_1\epsilon_2+ g_2\epsilon_2^2 \\ a_3\epsilon_1+b_3\epsilon_2 \\ \epsilon_1^2-\epsilon_2^2 \end{pmatrix}
\]

Looking at the last line in the above matrix, one sees that $s=0$ when $\epsilon_1=\pm \epsilon_2$. The $s$-coordinate is then positive or negative as illustrated below:
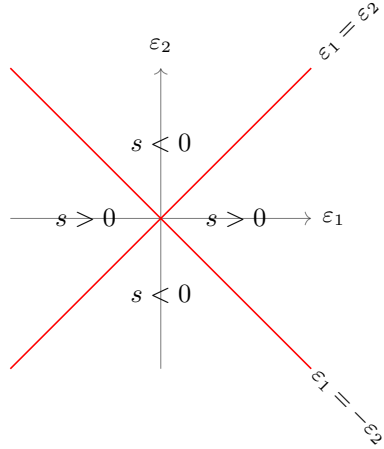
\[
  \begin{tikzpicture}[anchorbase]
    \draw [opacity=.5,->] (-2,0) -- (2,0);
    \node at (2.3,0) {\small $\epsilon_1$};
    \draw [opacity=.5,->] (0,-2) -- (0,2);
    \node at (0,2.3) {\small $\epsilon_2$};
    \draw [semithick, red] (-2,-2) -- (2,2);
    \node [rotate=45] at (2.5,2.5) {\small $\epsilon_1=\epsilon_2$};
    \draw [semithick, red] (-2,2) -- (2,-2);
    \node [rotate=-45] at (2.5,-2.5) {\small $\epsilon_1=-\epsilon_2$};
    \node at (1,0) {$s>0$};
    \node at (-1,0) {$s>0$};
    \node at (0,1) {$s<0$};
    \node at (0,-1) {$s<0$};
  \end{tikzpicture}
\]

At $s=0$, then one has:
\[
  f(u_0+\epsilon_1,v_0\pm\epsilon_1,w_0)=\begin{pmatrix} (a_1\pm b_1) \epsilon_1 \\ (d_2\pm e_2 +g_2) \epsilon_1^2 \\ \cdot \\ 0 \end{pmatrix}
\]

One may assume that generically $a_1\pm b_1\neq 0$ and $d_2\pm e_2+g_2\neq 0$. One then sees two parabolas, tangent at one point.

Fixing $s=\mu^2$ with $\mu>0$, one has that $\epsilon_1^2=\epsilon_2^2+\mu^2\geq \mu^2$, thus one expects to see two connected components, depending that $\epsilon_1>\mu$ or $\epsilon_1<-\mu$. Let us first focus on $\epsilon_1>\mu$. The curve one will see is close to the two half-parabolas formed by the $\epsilon_1\geq 0$ part at $s=0$. Similarly, the curve corresponding to $\epsilon_1<-\mu$ is close to the two half-lines formed by the $\epsilon_1\leq 0$ part at $s=0$. Depending on the relative signs of $a_1\pm b_1$, the parts at $s=0$ assemble either into two smooth curves with one intersection point, or two singular curves meeting at their singular point (that then get smoothed at $\mu\neq 0$). Replacing $s=\mu^2$ by $s=-\mu^2$ passes from one to the other situation.

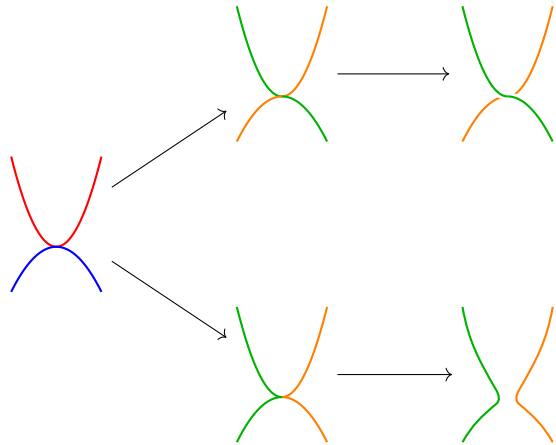
\[
  \begin{tikzpicture}[anchorbase]
    \node (A) at (0,0) {
      \begin{tikzpicture}[anchorbase,scale=.3]
        \draw [red, thick,  domain=-2:2, samples=40] 
        plot ({\x}, {\x*\x});
        \draw [blue, thick,  domain=-2:2, samples=40] 
        plot ({\x}, {-.5*\x*\x});
      \end{tikzpicture}
    };
    \node (B) at (3,2){
      \begin{tikzpicture}[anchorbase,scale=.3]
        \draw [green, thick,  domain=-2:0, samples=40] 
        plot ({\x}, {\x*\x});
        \draw [orange, thick,  domain=0:2, samples=40] 
        plot ({\x}, {\x*\x});
        \draw [orange, thick,  domain=-2:0, samples=40] 
        plot ({\x}, {-.5*\x*\x});
        \draw [green, thick,  domain=0:2, samples=40] 
        plot ({\x}, {-.5*\x*\x});
      \end{tikzpicture}
    };
    \node (C) at (6,2){
      \begin{tikzpicture}[anchorbase,scale=.3]
        \draw [orange, thick,  domain=0:2, samples=40] 
        plot ({\x}, {\x*\x});
        \draw [orange, thick,  domain=-2:0, samples=40] 
        plot ({\x}, {-.5*\x*\x});
        \draw [white, line width=3,  domain=-2:0, samples=40] 
        plot ({\x}, {\x*\x});
        \draw [white, line width=3,  domain=0:2, samples=40] 
        plot ({\x}, {-.5*\x*\x});
        \draw [green, thick,  domain=-2:0, samples=40] 
        plot ({\x}, {\x*\x});
        \draw [green, thick,  domain=0:2, samples=40] 
        plot ({\x}, {-.5*\x*\x});
      \end{tikzpicture}
    };
    \node (D) at (3,-2){
      \begin{tikzpicture}[anchorbase,scale=.3]
        \draw [green, thick,  domain=-2:0, samples=40] 
        plot ({\x}, {\x*\x});
        \draw [orange, thick,  domain=0:2, samples=40] 
        plot ({\x}, {\x*\x});
        \draw [green, thick,  domain=-2:0, samples=40] 
        plot ({\x}, {-.5*\x*\x});
        \draw [orange, thick,  domain=0:2, samples=40] 
        plot ({\x}, {-.5*\x*\x});
      \end{tikzpicture}
    };
    \node (E) at (6,-2){
      \begin{tikzpicture}[anchorbase,scale=.3]
        \draw [thick, green] (-2,4) to [out=-80,in=120] (-.5,.3) to [out=-60,in=40] (-.5,-.3) to [out=-140,in=60] (-2,-2);
        \draw [thick, orange] (2,4) to [out=-100,in=60] (.5,.3) to [out=-120,in=140] (.5,-.3) to [out=-40,in=120] (2,-2);
      \end{tikzpicture}
    };
    \draw [->] (A) -- (B);
    \draw [->] (B) -- (C);
    \draw [->] (A) -- (D);
    \draw [->] (D) -- (E);
  \end{tikzpicture}
\]

Notice that if the $x$-coordinate takes the value $0$, then one has: $\epsilon_2=-\frac{a_1}{b_1}\epsilon_1$. Replacing $\epsilon_2$ by this value in the last line implies that if $a_1>b_1$, then this can only occur for negative value of the $s$-coordinate, while if $a_1<b_1$, this will only occur for positive values of $s$. Since both positive and negative values can be taken in the $x$-coordinate, this implies the following: if $a_1>b_1$, then the two connected components living at fixed $s$-coordinate of positive value live each on one side of the $x=0$ line (and in particular do not meet); if $a_1<b_1$, then the two connected components living at fixed $s$-coordinate of negative value do not meet. This justifies that the green and orange lines at the bottom right corner of the above picture are disjoint.

For codimension reason, the framing can be assumed to be generic at the singular point, and then transported around, giving (notice the framing changes on the crossingless picture):
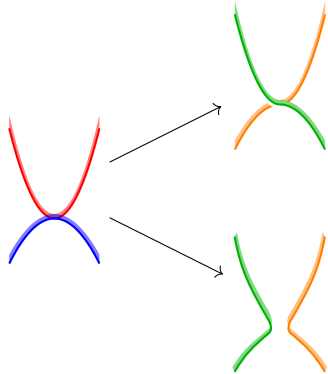
\[
  \begin{tikzpicture}[anchorbase]
    \node (A) at (0,0) {
      \begin{tikzpicture}[anchorbase,scale=.3]
        \draw [red, thick,  domain=-2:2, samples=40] 
        plot ({\x}, {\x*\x});
        \fill [opacity=.5,red,  domain=-2:2, samples=40] 
        plot ({\x}, {\x*\x}) --   plot ({-\x}, {1.1*\x*\x+.2});
        \draw [blue, thick,  domain=-2:2, samples=40] 
        plot ({\x}, {-.5*\x*\x});
        \fill [opacity=.5,blue,  domain=-2:2, samples=40] 
        plot ({\x}, {-.5*\x*\x}) -- plot ({-\x}, {-.45*\x*\x+.2});
      \end{tikzpicture}
    };
    \node (B) at (3,1.5){
      \begin{tikzpicture}[anchorbase,scale=.3]
        \draw [orange, thick,  domain=0:2, samples=40] 
        plot ({\x}, {\x*\x});
        \fill [opacity=.5,orange,  domain=0:2, samples=40] 
        plot ({\x}, {\x*\x}) --        plot ({2-\x}, {1.1*(2-\x)*(2-\x)+.2});
        \draw [orange, thick,  domain=-2:0, samples=40] 
        plot ({\x}, {-.5*\x*\x});
        \fill [opacity=.5,orange,  domain=-2:0, samples=40] 
        plot ({\x}, {-.5*\x*\x}) -- plot ({-2-\x}, {-.45*(-2-\x)*(-2-\x)+.2});
        \draw [white, line width=3,  domain=-2:0, samples=40] 
        plot ({\x}, {\x*\x});
        \draw [white, line width=3,  domain=0:2, samples=40] 
        plot ({\x}, {-.5*\x*\x});
        \draw [green, thick,  domain=-2:0, samples=40] 
        plot ({\x}, {\x*\x});
        \fill [opacity=.5,green, thick,  domain=-2:0, samples=40] 
        plot ({\x}, {\x*\x}) -- plot ({-2-\x}, {1.1*(-2-\x)*(-2-\x)+.2});
        \draw [green, thick,  domain=0:2, samples=40] 
        plot ({\x}, {-.5*\x*\x});
        \fill [opacity=.5,green, domain=0:2, samples=40] 
        plot ({\x}, {-.5*\x*\x}) --        plot ({2-\x}, {-.45*(2-\x)*(2-\x)+.2});
      \end{tikzpicture}
    };
    \node (D) at (3,-1.5){
      \begin{tikzpicture}[anchorbase,scale=.3]
        \draw [thick, green] (-2,4) to [out=-80,in=120] (-.5,.3) to [out=-60,in=40] (-.5,-.3) to [out=-140,in=60] (-2,-2);
        \fill [opacity=.5, green] (-2,4) to [out=-80,in=120] (-.5,.3) to [out=-60,in=40] (-.5,-.3) to [out=-140,in=60] (-2,-2) -- (-2.1,-1.8) to [out=60,in=-140] (-.5,-.1) to [out=40,in=-60] (-.5,.7) to [out=120,in=-80] (-1.9,4.2);
        \draw [thick, orange] (2,4) to [out=-100,in=60] (.5,.3) to [out=-120,in=140] (.5,-.3) to [out=-40,in=120] (2,-2);
        \fill [opacity=.5, orange] (2,4) to [out=-100,in=60] (.5,.3) to [out=-120,in=140] (.5,-.3) to [out=-40,in=120] (2,-2) -- (2.1,-1.8) to [out=120,in=-40] (.5,-.1) to [out=140,in=-120] (.4,.5) to [out=60,in=-100] (1.9,4.2);
      \end{tikzpicture}
    };
    \draw [->] (A) -- (B);
    \draw [->] (A) -- (D);
  \end{tikzpicture}
\]

When one adds the time parameter, the place where the $\partial u$ and $\partial v$ derivatives align will be taken away from the Morse singularity, in one or the other side of the saddle. One gets the following move (and all variants of it):

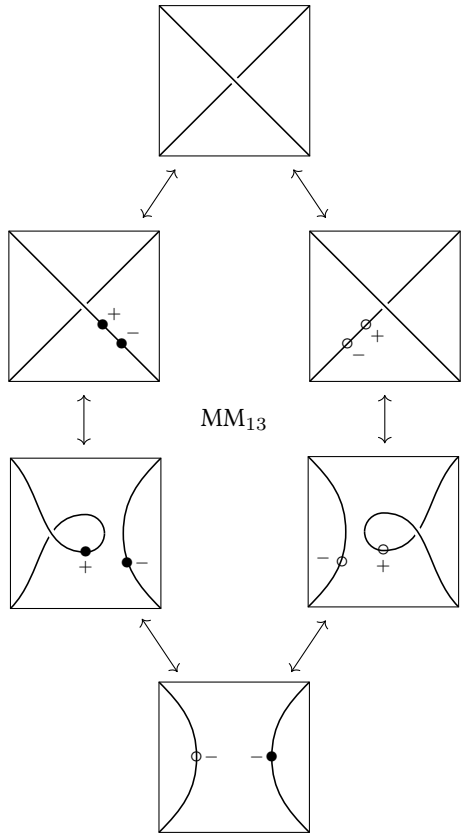
\[
  \begin{tikzpicture}[anchorbase]
    \node (A) at (0,0) {
\begin{tikzpicture}[anchorbase, scale=.5]
  % Strands
  \draw [line width=3, white] (0,0) -- (4,4);
  \draw [semithick] (0,0) -- (4,4);
  \draw [line width=3, white] (0,4) -- (4,0);
  \draw [semithick] (0,4) -- (4,0);
  % Twists
  % Rectangle
  \draw (0,0) rectangle (4,4);
  \end{tikzpicture}};
    \node (B1) at (-2,-3) {
\begin{tikzpicture}[anchorbase, scale=.5]
  % Strands
  \draw [line width=3, white] (0,0) -- (4,4);
  \draw [semithick] (0,0) -- (4,4);
  \draw [line width=3, white] (0,4) -- (4,0);
  \draw [semithick] (0,4) -- (4,0);
  % Twists
  \node at (2.5,1.5) {$\bullet$};
  \node at (2.8,1.8) {\tiny $+$};
  \node at (3,1) {$\bullet$};
  \node at (3.3,1.3) {\tiny $-$};
  % Rectangle
  \draw (0,0) rectangle (4,4);
  \end{tikzpicture}};
    \node (B2) at (2,-3) {
\begin{tikzpicture}[anchorbase, scale=.5]
  % Strands
  \draw [line width=3, white] (0,0) -- (4,4);
  \draw [semithick] (0,0) -- (4,4);
  \draw [line width=3, white] (0,4) -- (4,0);
  \draw [semithick] (0,4) -- (4,0);
  % Twists
  \node at (1.5,1.5) {$\circ$};
  \node at (1.8,1.2) {\tiny $+$};
  \node at (1,1) {$\circ$};
  \node at (1.3,.7) {\tiny $-$};
  % Rectangle
  \draw (0,0) rectangle (4,4);
  \end{tikzpicture}};
    \node (C2) at (2,-6) {
\begin{tikzpicture}[anchorbase, scale=.5]
  % Strands
  \draw [draw=white, line width=3] (4,4) to [out=-135,in=0] (2,1.5) to [out=180,in=-90] (1.5,2);
  \draw [semithick] (4,4) to [out=-135,in=0] (2,1.5) to [out=180,in=-90] (1.5,2);
  \draw [draw=white, line width=3] (1.5,2) to [out=90,in=180] (2,2.5) to [out=0,in=135] (4,0);
  \draw [semithick] (1.5,2) to [out=90,in=180] (2,2.5) to [out=0,in=135] (4,0);
  \draw [semithick] (0,0) to [out=45,in=-90] (1,2) to [out=90,in=-45] (0,4);
  % Twists
  \node at (.9,1.2) {$\circ$};
  \node at (.4,1.3) {\tiny $-$};
  \node at (2,1.5) {$\circ$};
  \node at (2,1.1) {\tiny $+$};
    % Rectangle
  \draw (0,0) rectangle (4,4);
  \end{tikzpicture}};
    \node (C1) at (-2,-6) {
\begin{tikzpicture}[anchorbase, scale=.5]
  % Strands
  \draw [semithick] (2.5,2) to [out=90,in=0] (2,2.5) to [out=180,in=45] (0,0);
  \draw [draw=white, line width=3] (0,4) to [out=-45,in=180] (2,1.5) to [out=0,in=-90] (2.5,2);
  \draw [semithick] (0,4) to [out=-45,in=180] (2,1.5) to [out=0,in=-90] (2.5,2);
  \draw [semithick] (4,0) to [out=135,in=-90] (3,2) to [out=90,in=-135] (4,4);
  % Twists
  \node at (2,1.5) {$\bullet$};
  \node at (2,1.1) {\tiny $+$};
  \node at (3.1,1.2) {$\bullet$};
  \node at (3.5,1.2) {\tiny $-$};
  % Rectangle
  \draw (0,0) rectangle (4,4);
  \end{tikzpicture}};
    \node (D) at (0,-9) {
\begin{tikzpicture}[anchorbase, scale=.5]
  % Strands
  \draw [semithick] (0,0) to [out=45,in=-90] (1,2) to [out=90,in=-45] (0,4);
  \draw [semithick] (4,0) to [out=135,in=-90] (3,2) to [out=90,in=-135] (4,4);
  \node at (1,2) {$\circ$};
  \node at (1.4,2) {\tiny $-$};
  \node at (3,2) {$\bullet$};
  \node at (2.6,2) {\tiny $-$};
  % Rectangle
  \draw (0,0) rectangle (4,4);
  \end{tikzpicture}};
\draw [<->] (A) -- (B1);
\draw [<->] (A) -- (B2);
\draw [<->] (B1) -- (C1);
\draw [<->] (B2) -- (C2);
\draw [<->] (C1) -- (D);
\draw [<->] (C2) -- (D);
\node at (0,-4.5) {$\text{MM}_{13}$};
\end{tikzpicture}
\]

\vspace{2cm}

Let us now go back to framing considerations we had eluded at the beginning of our analysis. The graph of $f$ on $N(F)\times [0,1]$ is 4-dimensional. Requiring that the framing lies in the subspace generated by the tangent plane to $F$ and the vertical direction is a codimension 1 condition. Furthermore, the restriction to $F$ is also a codimension $1$ condition. Thus the set of points where the framing is vertical is of dimension $1$ in a generic foam, and of dimension $2$ in a time-family of foams.

Assume that $(df)_M=\begin{pmatrix} \fbox{$\begin{matrix} \\ \\ A \\ ~ \end{matrix}$} & \fbox{$\begin{matrix} \\ \\ B \\~  \end{matrix}$} & \fbox{$A+B+\begin{matrix} 0\\ 0\\ 1\\ 0 \end{matrix}$}   \end{pmatrix}$. Up to change of coordinates in $u$ and $v$ we can assume that $\left(\frac{\partial f}{\partial u}\right)_M$ has zero $s$ coordinate. We first assume that $\left(\frac{\partial f}{\partial u}\right)$ has non-zero coordinates in $x$ or $y$, which allows us to reduce to the case where its $y$-coordinate is zero and the $x$-coordinate is strictly positive. Then one can change coordinates in $u,v,w$ so that:
\[
  \left(df\right)_M=\begin{pmatrix}
    a_1 & 0 & \lambda a_1  \\
    0 & b_2 & \mu b_2 \\
    a_3 & b_3 & \mu b_3 + \eta \\
    0 & b_4 & \mu b_4 
  \end{pmatrix}\quad ,\quad a_1>0, \eta \neq 0
\]

Let us first assume that $b_4\neq 0$. In that case, the local shape for the foam is a trivial movie over a single segment and we are basically brought to the case of an isotopy of a web. We will be interested in $\frac{\partial^2 f}{\partial u \partial w}$. Assume that this vector has a non-zero $y$ coordinate $f_2$. This ensures that the framing change is isolated (the framing has non-vanishing projection through $\pi$ around $M$) and we have a trivial movie move.

Now, assume that $f_2=0$. This is a codimensions $1$ condition, so we expect through time a $1$-dimensional set of such points. Consider $\beta'=(\frac{\partial^3f}{\partial u^2\partial w})_M$. If $\beta'\neq 0$, then we locally are in a situation similar to the one from equation~\eqref{eq:FrChMorse}. If the $y$-coordinate $\nu$ of $\frac{\partial^2 f}{\partial v \partial w}$ is non-zero, then at $t=t_0$ we see the foam from equations~\ref{eq:FrChMorse2} or ~\ref{eq:FrChMorse3}. Along time, we see no change provided $\frac{\partial^2f}{\partial t\partial w}$ has non-zero $y$ coordinates. Otherwise, one gets the following movie move, which is only a twist version of the framed $\mathrm{R}_I$ move:
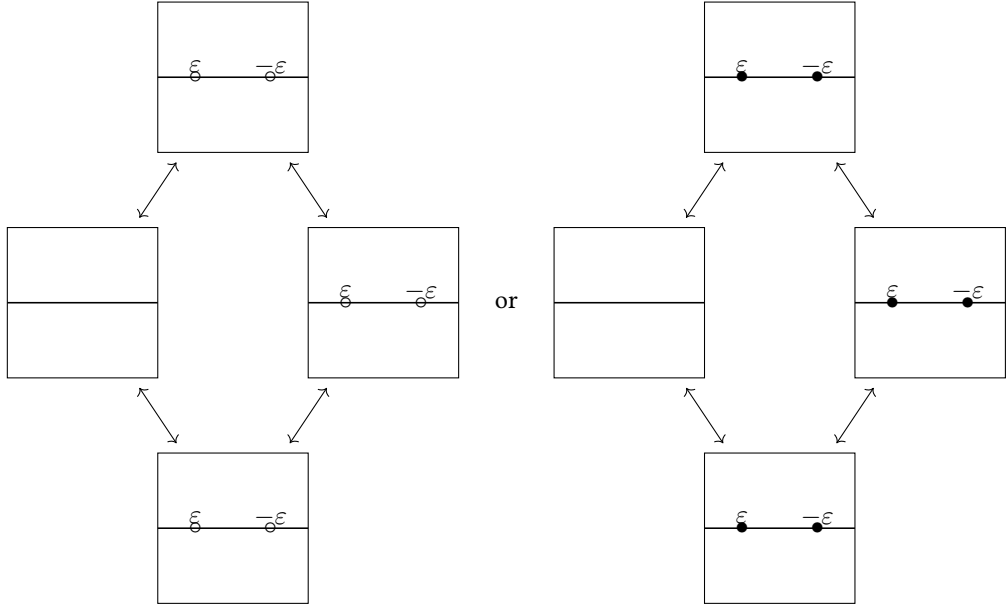
\[
  \begin{tikzpicture}[anchorbase]
    \node (A) at (0,0) {
\begin{tikzpicture}[anchorbase, scale=.5]
  % Strand
  \draw [semithick] (0,2) -- (4,2);
  % Framing changes
  \node at (1,2) {$\circ$};
  \node at (1,2.3) {$\epsilon$};
  \node at (3,2) {$\circ$};
  \node at (3,2.3) {$-\epsilon$};
  % Rectangle
  \draw (0,0) rectangle (4,4);
  \end{tikzpicture}};
    \node (B1) at (-2,-3) {
\begin{tikzpicture}[anchorbase, scale=.5]
  % Strand
  \draw [semithick] (0,2) -- (4,2);
  % Rectangle
  \draw (0,0) rectangle (4,4);
  \end{tikzpicture}};
    \node (B2) at (2,-3) {
\begin{tikzpicture}[anchorbase, scale=.5]
  % Strand
  \draw [semithick] (0,2) -- (4,2);
  % Framing changes
  \node at (1,2) {$\circ$};
  \node at (1,2.3) {$\epsilon$};
  \node at (3,2) {$\circ$};
  \node at (3,2.3) {$-\epsilon$};
  % Rectangle
  \draw (0,0) rectangle (4,4);
  \end{tikzpicture}};
    \node (C) at (0,-6) {
\begin{tikzpicture}[anchorbase, scale=.5]
  % Strand
  \draw [semithick] (0,2) -- (4,2);
  % Framing changes
  \node at (1,2) {$\circ$};
  \node at (1,2.3) {$\epsilon$};
  \node at (3,2) {$\circ$};
  \node at (3,2.3) {$-\epsilon$};
  % Rectangle
  \draw (0,0) rectangle (4,4);
  \end{tikzpicture}};
\draw [<->] (A) -- (B1);
\draw [<->] (A) -- (B2);
\draw [<->] (B1) -- (C);
\draw [<->] (B2) -- (C);
\end{tikzpicture}
\quad \text{or}\quad
  \begin{tikzpicture}[anchorbase]
    \node (A) at (0,0) {
\begin{tikzpicture}[anchorbase, scale=.5]
  % Strand
  \draw [semithick] (0,2) -- (4,2);
  % Framing changes
  \node at (1,2) {$\bullet$};
  \node at (1,2.3) {$\epsilon$};
  \node at (3,2) {$\bullet$};
  \node at (3,2.3) {$-\epsilon$};
  % Rectangle
  \draw (0,0) rectangle (4,4);
  \end{tikzpicture}};
    \node (B1) at (-2,-3) {
\begin{tikzpicture}[anchorbase, scale=.5]
  % Strand
  \draw [semithick] (0,2) -- (4,2);
  % Rectangle
  \draw (0,0) rectangle (4,4);
  \end{tikzpicture}};
    \node (B2) at (2,-3) {
\begin{tikzpicture}[anchorbase, scale=.5]
  % Strand
  \draw [semithick] (0,2) -- (4,2);
  % Framing changes
  \node at (1,2) {$\bullet$};
  \node at (1,2.3) {$\epsilon$};
  \node at (3,2) {$\bullet$};
  \node at (3,2.3) {$-\epsilon$};
  % Rectangle
  \draw (0,0) rectangle (4,4);
  \end{tikzpicture}};
    \node (C) at (0,-6) {
\begin{tikzpicture}[anchorbase, scale=.5]
  % Strand
  \draw [semithick] (0,2) -- (4,2);
  % Framing changes
  \node at (1,2) {$\bullet$};
  \node at (1,2.3) {$\epsilon$};
  \node at (3,2) {$\bullet$};
  \node at (3,2.3) {$-\epsilon$};
  % Rectangle
  \draw (0,0) rectangle (4,4);
  \end{tikzpicture}};
\draw [<->] (A) -- (B1);
\draw [<->] (A) -- (B2);
\draw [<->] (B1) -- (C);
\draw [<->] (B2) -- (C);
\end{tikzpicture}
\]

If $\beta'=0$, then for codimension reason this only happens at isolated points, and we can assume that the $y$-coordinate $\beta"$ in $\frac{\partial^4f}{\partial^3u \partial w}$ is non-zero as well as $\nu$. The equation for the curve of annihilation of the points is then: $\beta"\epsilon_1^3+\nu \epsilon_2=0$. As we add the time parameter, we can assume that $\kappa$ the $y$-coordinate in $\frac{\partial^2 f}{\partial w \partial t}$ is non-zero as well $\xi$ the $y$-coordinate in $\frac{\partial^2 f}{\partial u \partial t}$, and we get an equation: $\beta"\epsilon_1^3+\nu\epsilon_2+\kappa \delta+\xi \epsilon_1\delta=0$.Below we illustrate the equations: $v=u^3+u+1$ and $v=u^3-u-1$ (corresponding to $\beta"=1$, $\nu=1$, $\kappa=1$, $\xi=1$, and $\delta=\pm 1$). The point is that in all cases, we pass between $0$ and $2$ inflection points.

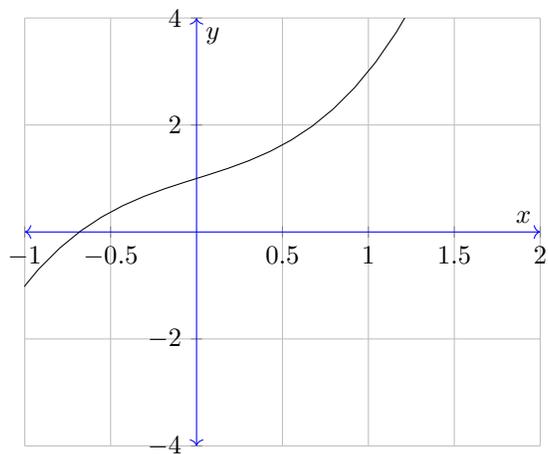
\begin{equation} \label{eq:pic3change1}
\begin{tikzpicture}
    \begin{axis}[
            xmin=-1,xmax=2,
            ymin=-4,ymax=4,
            grid=both,
            ]
            \addplot [domain=-3:3,samples=50]({x},{x^3+x+1}); 
    \end{axis}
\end{tikzpicture}
\end{equation}

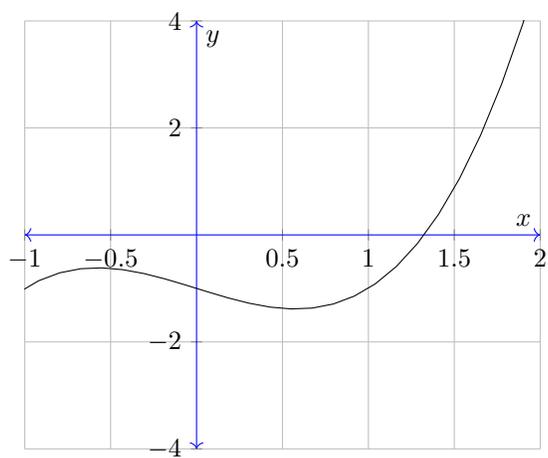
\begin{equation} \label{eq:pic3change2}
\begin{tikzpicture}
    \begin{axis}[
            xmin=-1,xmax=2,
            ymin=-4,ymax=4,
            grid=both,
            ]
            \addplot [domain=-3:3,samples=50]({x},{x^3-x-1}); 
    \end{axis}
\end{tikzpicture}
\end{equation}

This corresponds to the following movie-move:
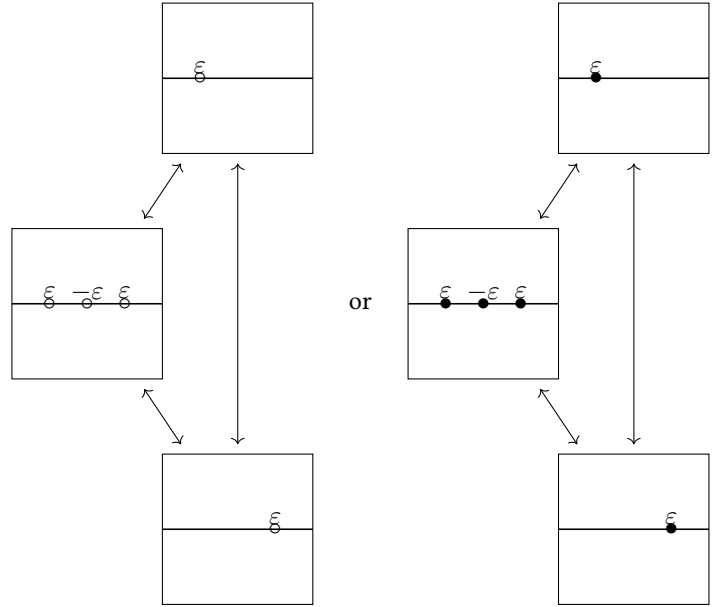
\[
  \begin{tikzpicture}[anchorbase]
    \node (A) at (0,0) {
\begin{tikzpicture}[anchorbase, scale=.5]
  % Strand
  \draw [semithick] (0,2) -- (4,2);
  % Framing changes
  \node at (1,2) {$\circ$};
  \node at (1,2.3) {$\epsilon$};
  % Rectangle
  \draw (0,0) rectangle (4,4);
  \end{tikzpicture}};
    \node (B1) at (-2,-3) {
\begin{tikzpicture}[anchorbase, scale=.5]
  % Strand
  \draw [semithick] (0,2) -- (4,2);
  % Framing changes
  \node at (1,2) {$\circ$};
  \node at (1,2.3) {$\epsilon$};
  \node at (2,2) {$\circ$};
  \node at (2,2.3) {$-\epsilon$};
  \node at (3,2) {$\circ$};
  \node at (3,2.3) {$\epsilon$};
  % Rectangle
  \draw (0,0) rectangle (4,4);
  \end{tikzpicture}};
    \node (C) at (0,-6) {
\begin{tikzpicture}[anchorbase, scale=.5]
  % Strand
  \draw [semithick] (0,2) -- (4,2);
  % Framing changes
  \node at (3,2) {$\circ$};
  \node at (3,2.3) {$\epsilon$};
  % Rectangle
  \draw (0,0) rectangle (4,4);
  \end{tikzpicture}};
\draw [<->] (A) -- (B1);
\draw [<->] (B1) -- (C);
\draw [<->] (A) -- (C);
\end{tikzpicture}
\quad \text{or}\quad
  \begin{tikzpicture}[anchorbase]
    \node (A) at (0,0) {
\begin{tikzpicture}[anchorbase, scale=.5]
  % Strand
  \draw [semithick] (0,2) -- (4,2);
  % Framing changes
  \node at (1,2) {$\bullet$};
  \node at (1,2.3) {$\epsilon$};
  % Rectangle
  \draw (0,0) rectangle (4,4);
  \end{tikzpicture}};
    \node (B1) at (-2,-3) {
\begin{tikzpicture}[anchorbase, scale=.5]
  % Strand
  \draw [semithick] (0,2) -- (4,2);
  % Framing changes
  \node at (1,2) {$\bullet$};
  \node at (1,2.3) {$\epsilon$};
  \node at (2,2) {$\bullet$};
  \node at (2,2.3) {$-\epsilon$};
  \node at (3,2) {$\bullet$};
  \node at (3,2.3) {$\epsilon$};
  % Rectangle
  \draw (0,0) rectangle (4,4);
  \end{tikzpicture}};
    \node (C) at (0,-6) {
\begin{tikzpicture}[anchorbase, scale=.5]
  % Strand
  \draw [semithick] (0,2) -- (4,2);
  % Framing changes
  \node at (3,2) {$\bullet$};
  \node at (3,2.3) {$\epsilon$};
  % Rectangle
  \draw (0,0) rectangle (4,4);
  \end{tikzpicture}};
\draw [<->] (A) -- (B1);
\draw [<->] (B1) -- (C);
\draw [<->] (A) -- (C);
\end{tikzpicture}
\]

\label{page:b4=0} We now go back to the case where $b_4=0$, that is:
\[
  \left(df\right)_M=\begin{pmatrix}
    a_1 & 0 & \lambda a_1  \\
    0 & b_2 & \mu b_2  \\
    a_3 & b_3 &  \lambda a_3+\mu b_3 + \eta \\
    0 & 0 & 0 
  \end{pmatrix}\quad ,\quad a_1>0, \eta \neq 0
\]
This is a co-dimension $1$ condition, so we expect such points to happen at isolated places (through time). Up to adding multiples of $u$ and $v$ to $w$, we can simplify the matrix to:
\[
  \left(df\right)_M=\begin{pmatrix}
    a_1 & 0 & 0  \\
    0 & b_2 & 0  \\
    a_3 & b_3 & \eta \\
    0 & 0 & 0 
  \end{pmatrix}\quad ,\quad a_1>0, \eta \neq 0
\]
One may assume that the 3 by 3 matrix formed by the projections under $\pi$ of $\left(\frac{\partial^2 f}{\partial^2 u}\right)_M$, $\left(\frac{\partial^2 f}{\partial u \partial v}\right)_M$ and $\left(\frac{\partial^2 f}{\partial^2 v}\right)_M$ is of rank $3$. This ensures that the system formed by the first-order expansions of the equation:
\[
  A\left(\frac{\partial f}{\partial u}\right)_{(u_0+\epsilon_1,v_0+\epsilon_2,w_0,t_0)}+  B\left(\frac{\partial f}{\partial v}\right)_{(u_0+\epsilon_1,v_0+\epsilon_2,w_0,t_0)}=  \left(\frac{\partial f}{\partial w}\right)_{(u_0+\epsilon_1,v_0+\epsilon_2,w_0,t_0)}
\]
has a unique solution in $A$, $B$ for each value of $\epsilon_1$ and $\epsilon_2$. Furthermore, we can assume that the RHS of the above equation is non-zero, which implies that the solution is non-zero linear combination in $\epsilon_1$ and $\epsilon_2$. In other terms, at first order, the locus of framing change draws a line that passes at $M$.

The same equation will transport through time, except that the line will drift in the direction given by $- \frac{\partial^2 f}{\partial w \partial t}$. This vector can be assumed to have non-zero projection in the $x,y$ plane. One thus gets one of the following moves.

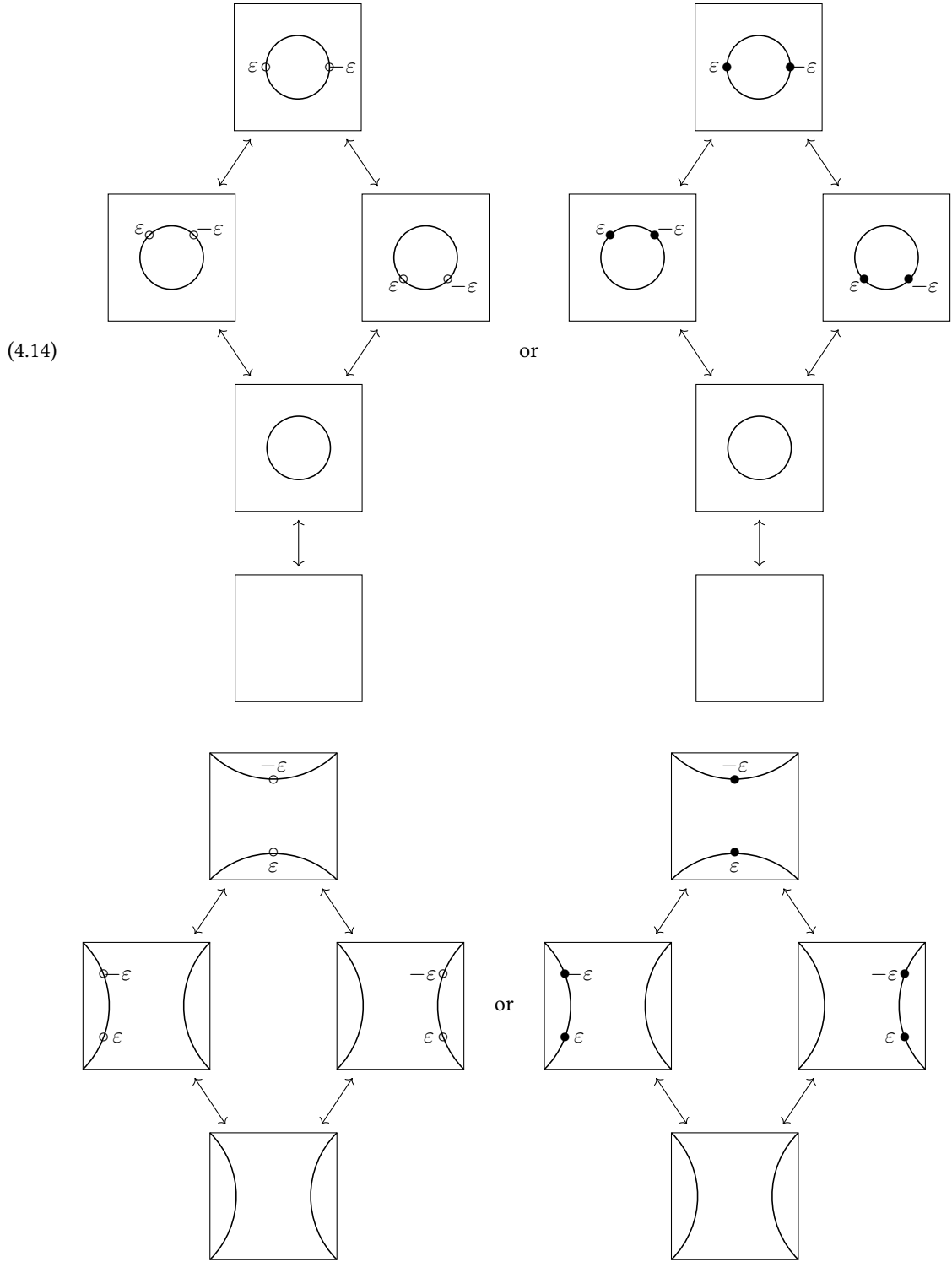
\begin{equation}\label{eq:twistpaircup}
  \begin{tikzpicture}[anchorbase]
    \node (A) at (0,0) {
\begin{tikzpicture}[anchorbase, scale=.5]
  % Circle
  \draw [semithick] (2,2) circle (1);
  % Framing changes
  \node at (1,2) {$\circ$};
  \node at (.6,2) {$\epsilon$};
  \node at (3,2) {$\circ$};
  \node at (3.4,2) {$-\epsilon$};
  % Rectangle
  \draw (0,0) rectangle (4,4);
  \end{tikzpicture}};
    \node (B1) at (-2,-3) {
\begin{tikzpicture}[anchorbase, scale=.5]
  % Circle
  \draw [semithick] (2,2) circle (1);
  % Framing changes
  \node at (1.3,2.7) {$\circ$};
  \node at (1,2.9) {$\epsilon$};
  \node at (2.7,2.7) {$\circ$};
  \node at (3.2,2.9) {$-\epsilon$};
  % Rectangle
  \draw (0,0) rectangle (4,4);
  \end{tikzpicture}};
    \node (B2) at (2,-3) {
\begin{tikzpicture}[anchorbase, scale=.5]
  % Circle
  \draw [semithick] (2,2) circle (1);
  % Framing changes
  \node at (1.3,1.3) {$\circ$};
  \node at (1,1.1) {$\epsilon$};
  \node at (2.7,1.3) {$\circ$};
  \node at (3.2,1.1) {$-\epsilon$};
  % Rectangle
  \draw (0,0) rectangle (4,4);
  \end{tikzpicture}};
    \node (C) at (0,-6) {
\begin{tikzpicture}[anchorbase, scale=.5]
  % Circle
  \draw [semithick] (2,2) circle (1);
  % Rectangle
  \draw (0,0) rectangle (4,4);
  \end{tikzpicture}};
    \node (D) at (0,-9) {
\begin{tikzpicture}[anchorbase, scale=.5]
  % Rectangle
  \draw (0,0) rectangle (4,4);
  \end{tikzpicture}};
\draw [<->] (A) -- (B1);
\draw [<->] (A) -- (B2);
\draw [<->] (B1) -- (C);
\draw [<->] (B2) -- (C);
\draw [<->] (C) -- (D);
\end{tikzpicture}
\quad\text{or}\quad
  \begin{tikzpicture}[anchorbase]
    \node (A) at (0,0) {
\begin{tikzpicture}[anchorbase, scale=.5]
  % Circle
  \draw [semithick] (2,2) circle (1);
  % Framing changes
  \node at (1,2) {$\bullet$};
  \node at (.6,2) {$\epsilon$};
  \node at (3,2) {$\bullet$};
  \node at (3.4,2) {$-\epsilon$};
  % Rectangle
  \draw (0,0) rectangle (4,4);
  \end{tikzpicture}};
    \node (B1) at (-2,-3) {
\begin{tikzpicture}[anchorbase, scale=.5]
  % Circle
  \draw [semithick] (2,2) circle (1);
  % Framing changes
  \node at (1.3,2.7) {$\bullet$};
  \node at (1,2.9) {$\epsilon$};
  \node at (2.7,2.7) {$\bullet$};
  \node at (3.2,2.9) {$-\epsilon$};
  % Rectangle
  \draw (0,0) rectangle (4,4);
  \end{tikzpicture}};
    \node (B2) at (2,-3) {
\begin{tikzpicture}[anchorbase, scale=.5]
  % Circle
  \draw [semithick] (2,2) circle (1);
  % Framing changes
  \node at (1.3,1.3) {$\bullet$};
  \node at (1,1.1) {$\epsilon$};
  \node at (2.7,1.3) {$\bullet$};
  \node at (3.2,1.1) {$-\epsilon$};
  % Rectangle
  \draw (0,0) rectangle (4,4);
  \end{tikzpicture}};
    \node (C) at (0,-6) {
\begin{tikzpicture}[anchorbase, scale=.5]
  % Circle
  \draw [semithick] (2,2) circle (1);
  % Rectangle
  \draw (0,0) rectangle (4,4);
  \end{tikzpicture}};
    \node (D) at (0,-9) {
\begin{tikzpicture}[anchorbase, scale=.5]
  % Rectangle
  \draw (0,0) rectangle (4,4);
  \end{tikzpicture}};
\draw [<->] (A) -- (B1);
\draw [<->] (A) -- (B2);
\draw [<->] (B1) -- (C);
\draw [<->] (B2) -- (C);
\draw [<->] (C) -- (D);
\end{tikzpicture}
\end{equation}

\[
  \begin{tikzpicture}[anchorbase]
    \node (A) at (0,0) {
\begin{tikzpicture}[anchorbase, scale=.5]
  % Strands
  \draw [semithick] (0,0) to [out=45,in=135] (4,0);
  \draw [semithick] (0,4) to [out=-45,in=-135] (4,4);
  % Framing changes
  \node at (2,.85) {$\circ$};
  \node at (2,.4) {$\epsilon$};
  \node at (2,3.15) {$\circ$};
  \node at (2,3.6) {$-\epsilon$};
  % Rectangle
  \draw (0,0) rectangle (4,4);
  \end{tikzpicture}};
    \node (B1) at (-2,-3) {
\begin{tikzpicture}[anchorbase, scale=.5]
  % Strands
  \draw [semithick] (0,0) to [out=45,in=-45] (0,4);
  \draw [semithick] (4,0) to [out=135,in=-135] (4,4);
  % Framing changes
  \node at (.65,1) {$\circ$};
  \node at (1.1,1) {$\epsilon$};
  \node at (.65,3) {$\circ$};
  \node at (1.1,3) {$-\epsilon$};
  % Rectangle
  \draw (0,0) rectangle (4,4);
  \end{tikzpicture}};
    \node (B2) at (2,-3) {
\begin{tikzpicture}[anchorbase, scale=.5]
  % Strands
  \draw [semithick] (0,0) to [out=45,in=-45] (0,4);
  \draw [semithick] (4,0) to [out=135,in=-135] (4,4);
  % Framing changes
  \node at (3.35,1) {$\circ$};
  \node at (2.9,1) {$\epsilon$};
  \node at (3.35,3) {$\circ$};
  \node at (2.7,3) {$-\epsilon$};
  % Rectangle
  \draw (0,0) rectangle (4,4);
\end{tikzpicture}};
    \node (C) at (0,-6) {
\begin{tikzpicture}[anchorbase, scale=.5]
  % Strands
  \draw [semithick] (0,0) to [out=45,in=-45] (0,4);
  \draw [semithick] (4,0) to [out=135,in=-135] (4,4);
  % Framing changes
  % Rectangle
  \draw (0,0) rectangle (4,4);
  \end{tikzpicture}};
\draw [<->] (A) -- (B1);
\draw [<->] (A) -- (B2);
\draw [<->] (B1) -- (C);
\draw [<->] (B2) -- (C);
\end{tikzpicture}
\quad\text{or}\quad
  \begin{tikzpicture}[anchorbase]
    \node (A) at (0,0) {
\begin{tikzpicture}[anchorbase, scale=.5]
  % Strands
  \draw [semithick] (0,0) to [out=45,in=135] (4,0);
  \draw [semithick] (0,4) to [out=-45,in=-135] (4,4);
  % Framing changes
  \node at (2,.85) {$\bullet$};
  \node at (2,.4) {$\epsilon$};
  \node at (2,3.15) {$\bullet$};
  \node at (2,3.6) {$-\epsilon$};
  % Rectangle
  \draw (0,0) rectangle (4,4);
  \end{tikzpicture}};
    \node (B1) at (-2,-3) {
\begin{tikzpicture}[anchorbase, scale=.5]
  % Strands
  \draw [semithick] (0,0) to [out=45,in=-45] (0,4);
  \draw [semithick] (4,0) to [out=135,in=-135] (4,4);
  % Framing changes
  \node at (.65,1) {$\bullet$};
  \node at (1.1,1) {$\epsilon$};
  \node at (.65,3) {$\bullet$};
  \node at (1.1,3) {$-\epsilon$};
  % Rectangle
  \draw (0,0) rectangle (4,4);
  \end{tikzpicture}};
    \node (B2) at (2,-3) {
\begin{tikzpicture}[anchorbase, scale=.5]
  % Strands
  \draw [semithick] (0,0) to [out=45,in=-45] (0,4);
  \draw [semithick] (4,0) to [out=135,in=-135] (4,4);
  % Framing changes
  \node at (3.35,1) {$\bullet$};
  \node at (2.9,1) {$\epsilon$};
  \node at (3.35,3) {$\bullet$};
\node at (2.7,3) {$-\epsilon$};
  % Rectangle
  \draw (0,0) rectangle (4,4);
\end{tikzpicture}};
    \node (C) at (0,-6) {
\begin{tikzpicture}[anchorbase, scale=.5]
  % Strands
  \draw [semithick] (0,0) to [out=45,in=-45] (0,4);
  \draw [semithick] (4,0) to [out=135,in=-135] (4,4);
  % Framing changes
  % Rectangle
  \draw (0,0) rectangle (4,4);
  \end{tikzpicture}};
\draw [<->] (A) -- (B1);
\draw [<->] (A) -- (B2);
\draw [<->] (B1) -- (C);
\draw [<->] (B2) -- (C);
\end{tikzpicture}
\]

We now look at points where $\vec{z}\in \Span_{\R}\left(\frac{\partial f}{\partial u},\frac{\partial f}{\partial v}\right)$. If we recap on codimension, we have: a codimension $1$ condition when restricting to $F$, and a codimension $2$ condition coming from the verticality assumption. We thus have isolated points on a foam (this was already analyzed when looking at isotopies of a web) yielding the movies from Equation~\eqref{eq:frR1}, and $1$-dimensional sets through time.
% Two things could happen (and they are generically mutually exclusive):
% \begin{itemize}
% \item this happens at a Morse extremum ($a_4=b_4=0$);
% \item a framing change occurs at the same time.
% \end{itemize}

It can happen that $a_4=b_4=0$: this is the case of Morse extrema which has already been analyzed, yielding movie moves $\text{MM}_{12}$ and $\text{MM}_{13}$.

Notice that there cannot be a framing change happening at the same time. Indeed:
\[
  \Span_\R\left(\frac{\partial f}{\partial u},\frac{\partial f}{\partial v},\vec{z}\right)=  \Span_\R\left(\frac{\partial f}{\partial u},\frac{\partial f}{\partial v}\right)
\]
and $\frac{\partial f}{\partial w}$ cannot belong to the above vector space without violating condition~\ref{enum:condFoam2} on page~\pageref{enum:condFoam}.

Let us run a process similar to the one that yielded Equation~\eqref{eq:frR1}. Up to change of $(u,v)$ basis, one can reduce to:
\[
  \frac{\partial f}{\partial u}=\begin{pmatrix} 0 \\ 0 \\ a_3 \\ 0\end{pmatrix},   \quad \frac{\partial f}{\partial v}=\begin{pmatrix} b_1 \\ b_2 \\ 0 \\ b_4 \end{pmatrix}
\]

Up to rotation in the $x,y$ plane, one further reduces to:
\[
\frac{\partial^2 f}{\partial u^2}=\begin{pmatrix} d_1 \\ 0 \\ d_3 \\ d_4 \end{pmatrix},\quad \frac{\partial^2 f}{\partial u\partial v}=\begin{pmatrix} e_1 \\ e_2 \\ a_3 \\ e_4\end{pmatrix},\quad \frac{\partial f^3}{\partial u^3}=\begin{pmatrix} l_1 \\ l_2 \\ l_3 \\ l_4\end{pmatrix}
\]
If $d_1\neq 0$, $l_2\neq 0$ and $e_2\neq 0$, 
then we simply have a framed Reidemeister 1 move traveling through time. This will be a trivial movie move provided the line of double points has a tangent at origin with non-zero $s$-coordinate. If this is not the case, then one gets the following move (which is just the classical $\text{MM}_7$ if one reads it from bottom to top and back to bottom through the other side).

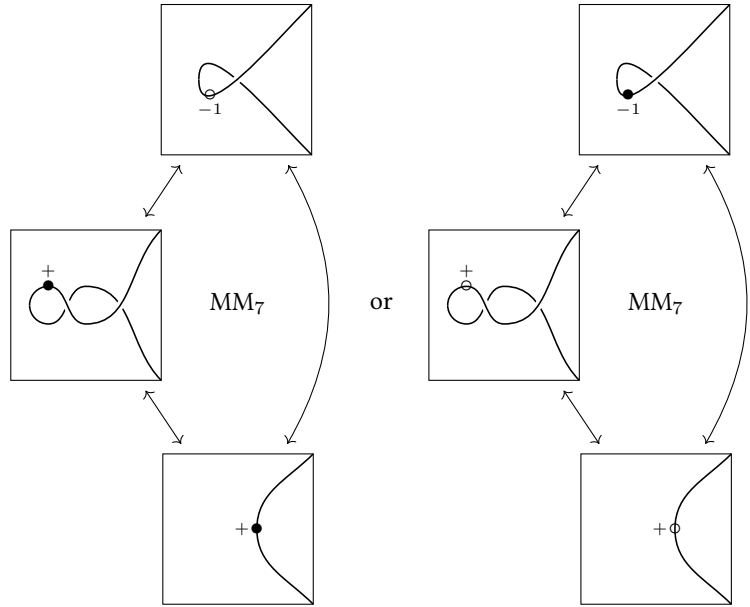
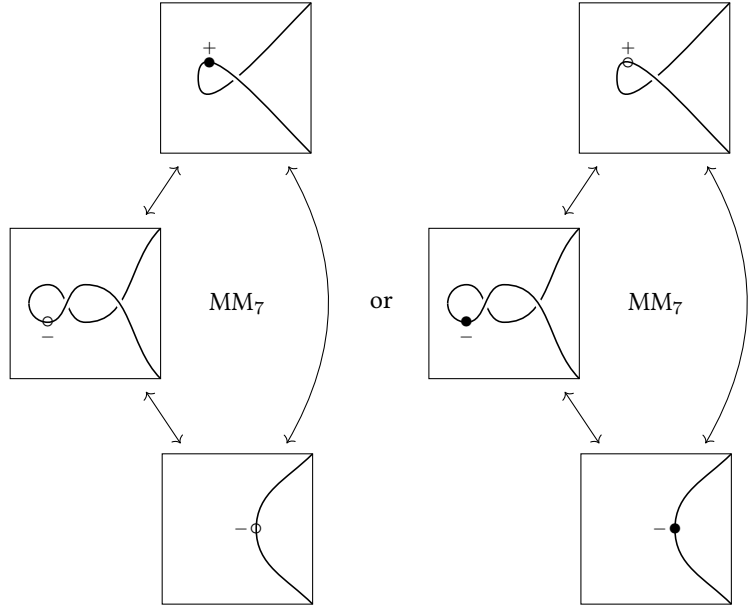
\[
  \begin{tikzpicture}[anchorbase]
    \node (A) at (0,0) {
\begin{tikzpicture}[anchorbase, scale=.5]
  % Strands
  \draw [semithick] (1,2) to [out=90,in=135] (4,0);
  \draw [white, line width=3] (4,4) to [out=-135,in=-90] (1,2);
  \draw [semithick] (4,4) to [out=-135,in=-90] (1,2);
  % Framing changes
  \node at (1.3,1.6) {$\circ$};
  \node at (1.3,1.2) {\tiny $-1$};
  % Rectangle
  \draw (0,0) rectangle (4,4);
  \end{tikzpicture}};
    \node (B1) at (-2,-3) {
\begin{tikzpicture}[anchorbase, scale=.5]
  % Strands
  \draw [semithick] (4,0) to [out=135,in=0] (2,2.5) to [out=180,in=0] (1,1.5) to [out=180,in=-90] (.5,2);
  \draw [white, line width=3] (4,4) to [out=-135,in=0] (2,1.5) to [out=180,in=0] (1,2.5) to [out=180,in=90] (.5,2);
  \draw [semithick] (4,4) to [out=-135,in=0] (2,1.5) to [out=180,in=0] (1,2.5) to [out=180,in=90] (.5,2);
  % Framing changes
  \node at (1,2.5) {$\bullet$};
  \node at (1,2.9) {\tiny $+$};
  % Rectangle
  \draw (0,0) rectangle (4,4);
  \end{tikzpicture}};
    \node (C) at (0,-6) {
\begin{tikzpicture}[anchorbase, scale=.5]
  % Strands
  \draw [semithick] (4,0) to [out=135,in=-90] (2.5,2) to [out=90,in=-135] (4,4);
  % Framing changes
  \node at (2.5,2) {$\bullet$};
  \node at (2.1,2) {\tiny $+$};
  % Rectangle
  \draw (0,0) rectangle (4,4);
  \end{tikzpicture}};
\draw [<->] (A) -- (B1);
\draw [<->] (B1) -- (C);
\draw [<->] (A) to [out=-60,in=60] (C);
\node at (0,-3) {$\text{MM}_{7}$};
\end{tikzpicture}
\quad \text{or} \quad
  \begin{tikzpicture}[anchorbase]
    \node (A) at (0,0) {
\begin{tikzpicture}[anchorbase, scale=.5]
  % Strands
  \draw [semithick] (1,2) to [out=90,in=135] (4,0);
  \draw [white, line width=3] (4,4) to [out=-135,in=-90] (1,2);
  \draw [semithick] (4,4) to [out=-135,in=-90] (1,2);
  % Framing changes
  \node at (1.3,1.6) {$\bullet$};
  \node at (1.3,1.2) {\tiny $-1$};
   % Rectangle
  \draw (0,0) rectangle (4,4);
  \end{tikzpicture}};
    \node (B1) at (-2,-3) {
\begin{tikzpicture}[anchorbase, scale=.5]
  % Strands
  \draw [semithick] (4,0) to [out=135,in=0] (2,2.5) to [out=180,in=0] (1,1.5) to [out=180,in=-90] (.5,2);
  \draw [white, line width=3] (4,4) to [out=-135,in=0] (2,1.5) to [out=180,in=0] (1,2.5) to [out=180,in=90] (.5,2);
  \draw [semithick] (4,4) to [out=-135,in=0] (2,1.5) to [out=180,in=0] (1,2.5) to [out=180,in=90] (.5,2);
  % Framing changes
  \node at (1,2.5) {$\circ$};
  \node at (1,2.9) {\tiny $+$};
  % Rectangle
  \draw (0,0) rectangle (4,4);
  \end{tikzpicture}};
    \node (C) at (0,-6) {
\begin{tikzpicture}[anchorbase, scale=.5]
  % Strands
  \draw [semithick] (4,0) to [out=135,in=-90] (2.5,2) to [out=90,in=-135] (4,4);
  % Framing changes
  \node at (2.5,2) {$\circ$};
  \node at (2.1,2) {\tiny $+$};
  % Rectangle
  \draw (0,0) rectangle (4,4);
  \end{tikzpicture}};
\draw [<->] (A) -- (B1);
\draw [<->] (B1) -- (C);
\draw [<->] (A) to [out=-60,in=60] (C);
\node at (0,-3) {$\text{MM}_{7}$};
\end{tikzpicture}
\]

\[
  \begin{tikzpicture}[anchorbase]
    \node (A) at (0,0) {
\begin{tikzpicture}[anchorbase, scale=.5]
  % Strands
  \draw [white, line width=3] (4,4) to [out=-135,in=-90] (1,2);
  \draw [semithick] (4,4) to [out=-135,in=-90] (1,2);
  \draw [white, line width=3] (1,2) to [out=90,in=135] (4,0);
  \draw [semithick] (1,2) to [out=90,in=135] (4,0);
  % Framing changes
  \node at (1.3,2.4) {$\bullet$};
  \node at (1.3,2.8) {\tiny $+$};
  % Rectangle
  \draw (0,0) rectangle (4,4);
  \end{tikzpicture}};
    \node (B1) at (-2,-3) {
\begin{tikzpicture}[anchorbase, scale=.5]
  % Strands
  \draw [white, line width=3] (4,4) to [out=-135,in=0] (2,1.5) to [out=180,in=0] (1,2.5) to [out=180,in=90] (.5,2);
  \draw [semithick] (4,4) to [out=-135,in=0] (2,1.5) to [out=180,in=0] (1,2.5) to [out=180,in=90] (.5,2);
  \draw [white, line width=3] (4,0) to [out=135,in=0] (2,2.5) to [out=180,in=0] (1,1.5) to [out=180,in=-90] (.5,2);
  \draw [semithick] (4,0) to [out=135,in=0] (2,2.5) to [out=180,in=0] (1,1.5) to [out=180,in=-90] (.5,2);
  % Framing changes
  \node at (1,1.5) {$\circ$};
  \node at (1,1.1) {\tiny $-$};
  % Rectangle
  \draw (0,0) rectangle (4,4);
  \end{tikzpicture}};
    \node (C) at (0,-6) {
\begin{tikzpicture}[anchorbase, scale=.5]
  % Strands
  \draw [semithick] (4,0) to [out=135,in=-90] (2.5,2) to [out=90,in=-135] (4,4);
  % Framing changes
  \node at (2.5,2) {$\circ$};
  \node at (2.1,2) {\tiny $-$};
  % Rectangle
  \draw (0,0) rectangle (4,4);
  \end{tikzpicture}};
\draw [<->] (A) -- (B1);
\draw [<->] (B1) -- (C);
\draw [<->] (A) to [out=-60,in=60] (C);
\node at (0,-3) {$\text{MM}_{7}$};
\end{tikzpicture}
\quad \text{or} \quad
  \begin{tikzpicture}[anchorbase]
    \node (A) at (0,0) {
\begin{tikzpicture}[anchorbase, scale=.5]
  % Strands
  \draw [white, line width=3] (4,4) to [out=-135,in=-90] (1,2);
  \draw [semithick] (4,4) to [out=-135,in=-90] (1,2);
  \draw [white, line width=3] (1,2) to [out=90,in=135] (4,0);
  \draw [semithick] (1,2) to [out=90,in=135] (4,0);
  % Framing changes
  \node at (1.3,2.4) {$\circ$};
  \node at (1.3,2.8) {\tiny $+$};
  % Rectangle
  \draw (0,0) rectangle (4,4);
  \end{tikzpicture}};
    \node (B1) at (-2,-3) {
\begin{tikzpicture}[anchorbase, scale=.5]
  % Strands
  \draw [white, line width=3] (4,4) to [out=-135,in=0] (2,1.5) to [out=180,in=0] (1,2.5) to [out=180,in=90] (.5,2);
  \draw [semithick] (4,4) to [out=-135,in=0] (2,1.5) to [out=180,in=0] (1,2.5) to [out=180,in=90] (.5,2);
  \draw [white, line width=3] (4,0) to [out=135,in=0] (2,2.5) to [out=180,in=0] (1,1.5) to [out=180,in=-90] (.5,2);
  \draw [semithick] (4,0) to [out=135,in=0] (2,2.5) to [out=180,in=0] (1,1.5) to [out=180,in=-90] (.5,2);
  % Framing changes
  \node at (1,1.5) {$\bullet$};
  \node at (1,1.1) {\tiny $-$};
  % Rectangle
  \draw (0,0) rectangle (4,4);
  \end{tikzpicture}};
    \node (C) at (0,-6) {
\begin{tikzpicture}[anchorbase, scale=.5]
  % Strands
  \draw [semithick] (4,0) to [out=135,in=-90] (2.5,2) to [out=90,in=-135] (4,4);
  % Framing changes
  \node at (2.5,2) {$\bullet$};
  \node at (2.1,2) {\tiny $-$};
  % Rectangle
  \draw (0,0) rectangle (4,4);
  \end{tikzpicture}};
\draw [<->] (A) -- (B1);
\draw [<->] (B1) -- (C);
\draw [<->] (A) to [out=-60,in=60] (C);
\node at (0,-3) {$\text{MM}_{7}$};
\end{tikzpicture}
\]

To illustrate this case, consider the following particular situation:
\[
  f_{t_0}(u_0+\epsilon_1,v_0+\epsilon_2)=
  \begin{pmatrix} b_1\epsilon_2 +d_1\epsilon_1^2 \\ b_2\epsilon_2+e_2\epsilon_1\epsilon_2 +l_2\epsilon_1^3 \\ b_4 \epsilon_2+d_4\epsilon_1^2+e_4\epsilon_1\epsilon_2
  \end{pmatrix} 
\]

Then the line of double points is governed by the equation $\epsilon_1^2=-\frac{e_2}{l_2}\epsilon_2$. The tangent vector at $\epsilon_2=0$ is:
\[
  \begin{pmatrix}b_1 +\frac{d_1e_2}{l_2}\\ b_2 \\ b_4+\frac{d_4e_2}{l_2}\end{pmatrix}
\]
It might indeed happen that the last line vanishes.

Asking that $d_1=0$ corresponds to the simultaneous vanishing of both $x$ and $y$ coordinates in $\frac{\partial^2 f}{\partial u^2}$, so this is of too big codimension and generically does not happen.

Assume that $e_2=0$. Then we will use $m_2$ the $y$-coordinate of $\frac{\partial f^3}{\partial u^2 \partial v}$ and $n_2$ the $y$-coordinate of $\frac{\partial f^3}{\partial u \partial v^2}$ to write, at fixed $\epsilon_2$ and $\epsilon_1$ small enough:
\[
  f_{t_0}(u_0+\epsilon_1,v_0+\epsilon_2,w_0)\sim \begin{pmatrix}b_1\epsilon_2+d_1\epsilon_1^2+e_1\epsilon_1\epsilon_2 \\ b_2\epsilon_2+l_2\epsilon_1^3+m_2\epsilon_1^2\epsilon_2+n_2\epsilon_1\epsilon_2^2 \\ \cdot \\ b_4\epsilon_2+d_4\epsilon_1^2+e_4\epsilon_1\epsilon_2 \end{pmatrix}
\]

Starting from the cusp seen at $\epsilon_2=0$, one can study the effect of $e_1\epsilon_1\epsilon_2$, $m_2\epsilon_1^2\epsilon_2$ and $n_2\epsilon_1\epsilon_2^2$. One finds that only the latter term plays a role, and this time this role is the same both for $\epsilon_2$ positive or negative.  So depending on the signs of $l_2$ and $n_2$, either a crossing gets created both for $\epsilon_2>0$ and $\epsilon_2<0$, or for none of them.

The first case yields the classical movie move $\text{MM}_1$, while the second one yields $\text{MM}_2$.

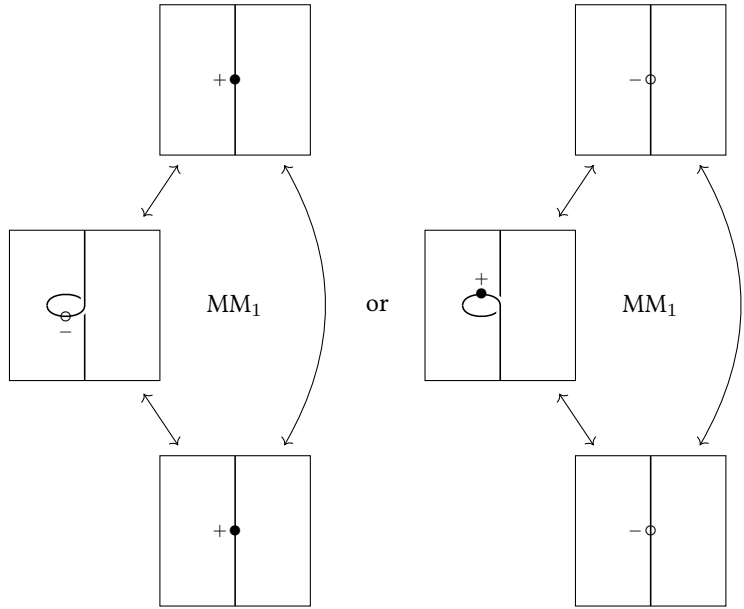
\[
  \begin{tikzpicture}[anchorbase]
    \node (A) at (0,0) {
\begin{tikzpicture}[anchorbase, scale=.5]
  % Strands
  \draw [semithick] (2,0) -- (2,4);
  % Twists
  \node at (2,2) {$\bullet$};
  \node at (1.6,2) {\tiny $+$};
  % Rectangle
    \draw (0,0) rectangle (4,4);
  \end{tikzpicture}};
    \node (B1) at (-2,-3) {
\begin{tikzpicture}[anchorbase, scale=.5]
  % Strands
  \draw [white, line width=3] (2,0) -- (2,2) to [out=90,in=90] (1,2);
  \draw [semithick] (2,0) -- (2,2) to [out=90,in=90] (1,2);
  \draw [white, line width=3] (1,2) to [out=-90,in=-90] (2,2) -- (2,4);
  \draw [semithick] (1,2) to [out=-90,in=-90] (2,2) -- (2,4);
  % Halft twists
  \node at (1.5,1.7) {$\circ$};
  \node at (1.5,1.3) {\tiny $-$};
  % Rectangle
  \draw (0,0) rectangle (4,4);
  \end{tikzpicture}};
    \node (C) at (0,-6) {
\begin{tikzpicture}[anchorbase, scale=.5]
  % Strands
    \draw [semithick] (2,0) -- (2,4);
  % Twists
  \node at (2,2) {$\bullet$};
  \node at (1.6,2) {\tiny $+$};
  % Rectangle
  \draw (0,0) rectangle (4,4);
  \end{tikzpicture}};
\draw [<->] (A) -- (B1);
\draw [<->] (B1) -- (C);
\draw [<->] (A) to [out=-60,in=60] (C);
\node at (0,-3) {$\text{MM}_{1}$};
\end{tikzpicture}
\quad \text{or} \quad
  \begin{tikzpicture}[anchorbase]
    \node (A) at (0,0) {
\begin{tikzpicture}[anchorbase, scale=.5]
  % Strands
    \draw [semithick] (2,0) -- (2,4);
  % Twists
  \node at (2,2) {$\circ$};
  \node at (1.6,2) {\tiny $-$};
  % Rectangle
  \draw (0,0) rectangle (4,4);
  \end{tikzpicture}};
    \node (B1) at (-2,-3) {
\begin{tikzpicture}[anchorbase, scale=.5]
  % Strands
  \draw [white, line width=3] (1,2) to [out=-90,in=-90] (2,2) -- (2,4);
  \draw [semithick] (1,2) to [out=-90,in=-90] (2,2) -- (2,4);
  \draw [white, line width=3] (2,0) -- (2,2) to [out=90,in=90] (1,2);
  \draw [semithick] (2,0) -- (2,2) to [out=90,in=90] (1,2);
  % Halft twists
  \node at (1.5,2.3) {$\bullet$};
  \node at (1.5,2.7) {\tiny $+$};
  % Rectangle
  \draw (0,0) rectangle (4,4);
  \end{tikzpicture}};
    \node (C) at (0,-6) {
\begin{tikzpicture}[anchorbase, scale=.5]
  % Strands
    \draw [semithick] (2,0) -- (2,4);
  % Twists
  \node at (2,2) {$\circ$};
  \node at (1.6,2) {\tiny $-$};
  % Rectangle
  \draw (0,0) rectangle (4,4);
  \end{tikzpicture}};
\draw [<->] (A) -- (B1);
\draw [<->] (B1) -- (C);
\draw [<->] (A) to [out=-60,in=60] (C);
\node at (0,-3) {$\text{MM}_{1}$};
\end{tikzpicture}
\]

Above and below one should also consider the case where the $\circ$ and $\bullet$ signs are exchanged.

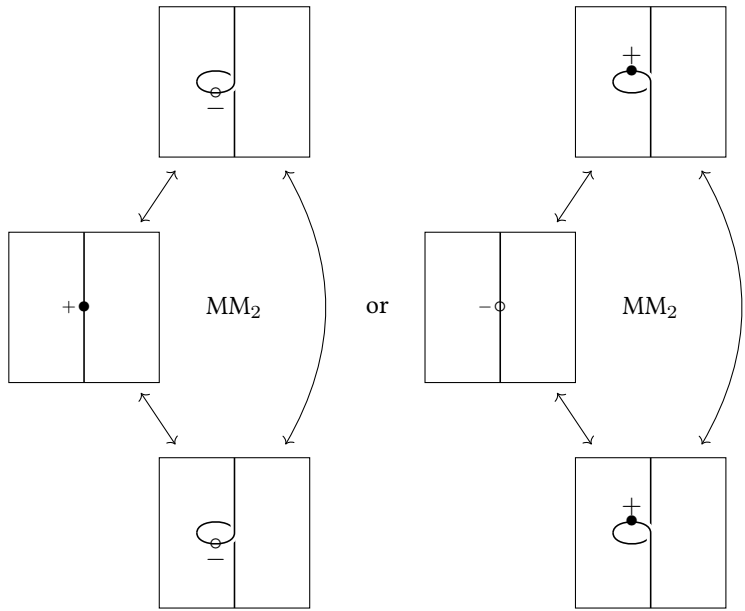
\[
  \begin{tikzpicture}[anchorbase]
    \node (A) at (0,0) {
\begin{tikzpicture}[anchorbase, scale=.5]
  % Strands
  \draw [white, line width=3] (2,0) -- (2,2) to [out=90,in=90] (1,2);
  \draw [semithick] (2,0) -- (2,2) to [out=90,in=90] (1,2);
  \draw [white, line width=3] (1,2) to [out=-90,in=-90] (2,2) -- (2,4);
  \draw [semithick] (1,2) to [out=-90,in=-90] (2,2) -- (2,4);
  % Halft twists
  \node at (1.5,1.7) {$\circ$};
  \node at (1.5,1.3) {$-$};
  % Rectangle
  \draw (0,0) rectangle (4,4);
  \end{tikzpicture}};
    \node (B1) at (-2,-3) {
\begin{tikzpicture}[anchorbase, scale=.5]
  % Strands
    \draw [semithick] (2,0) -- (2,4);
    % Twists
    \node at (2,2) {$\bullet$};
    \node at (1.6,2) {\tiny $+$};
    % Rectangle
  \draw (0,0) rectangle (4,4);
  \end{tikzpicture}};
    \node (C) at (0,-6) {
\begin{tikzpicture}[anchorbase, scale=.5]
  % Strands
  \draw [white, line width=3] (2,0) -- (2,2) to [out=90,in=90] (1,2);
  \draw [semithick] (2,0) -- (2,2) to [out=90,in=90] (1,2);
  \draw [white, line width=3] (1,2) to [out=-90,in=-90] (2,2) -- (2,4);
  \draw [semithick] (1,2) to [out=-90,in=-90] (2,2) -- (2,4);
  % Halft twists
  \node at (1.5,1.7) {$\circ$};
  \node at (1.5,1.3) {$-$};
  % Rectangle
  \draw (0,0) rectangle (4,4);
  \end{tikzpicture}};
\draw [<->] (A) -- (B1);
\draw [<->] (B1) -- (C);
\draw [<->] (A) to [out=-60,in=60] (C);
\node at (0,-3) {$\text{MM}_{2}$};
\end{tikzpicture}
\quad \text{or} \quad
  \begin{tikzpicture}[anchorbase]
    \node (A) at (0,0) {
\begin{tikzpicture}[anchorbase, scale=.5]
  % Strands
  \draw [white, line width=3] (1,2) to [out=-90,in=-90] (2,2) -- (2,4);
  \draw [semithick] (1,2) to [out=-90,in=-90] (2,2) -- (2,4);
  \draw [white, line width=3] (2,0) -- (2,2) to [out=90,in=90] (1,2);
  \draw [semithick] (2,0) -- (2,2) to [out=90,in=90] (1,2);
  % Halft twists
  \node at (1.5,2.3) {$\bullet$};
  \node at (1.5,2.7) {$+$};
  % Rectangle
  \draw (0,0) rectangle (4,4);
  \end{tikzpicture}};
    \node (B1) at (-2,-3) {
\begin{tikzpicture}[anchorbase, scale=.5]
  % Strands
    \draw [semithick] (2,0) -- (2,4);
    % Twists
    \node at (2,2) {$\circ$};
    \node at (1.6,2) {\tiny $-$};
  % Rectangle
  \draw (0,0) rectangle (4,4);
  \end{tikzpicture}};
    \node (C) at (0,-6) {
\begin{tikzpicture}[anchorbase, scale=.5]
  % Strands
  \draw [white, line width=3] (1,2) to [out=-90,in=-90] (2,2) -- (2,4);
  \draw [semithick] (1,2) to [out=-90,in=-90] (2,2) -- (2,4);
  \draw [white, line width=3] (2,0) -- (2,2) to [out=90,in=90] (1,2);
  \draw [semithick] (2,0) -- (2,2) to [out=90,in=90] (1,2);
  % Halft twists
  \node at (1.5,2.3) {$\bullet$};
  \node at (1.5,2.7) {$+$};
  % Rectangle
  \draw (0,0) rectangle (4,4);
  \end{tikzpicture}};
\draw [<->] (A) -- (B1);
\draw [<->] (B1) -- (C);
\draw [<->] (A) to [out=-60,in=60] (C);
\node at (0,-3) {$\text{MM}_{2}$};
\end{tikzpicture}
\]

The case where $l_2=0$ goes as follows: one replaces $l_2\epsilon_1^3$ by $p_2\epsilon_1^4$. One can push the Taylor expansion by adding a term $q_2\epsilon_1^5$, but then one has to account for terms $r_2\epsilon_1^3\epsilon_2$ too (we only keep track of those with odd exponent in $\epsilon_1$). We find multiple points by looking at the roots of the polynomial $q_2\epsilon_1^5+r_2\epsilon_1^3\epsilon_2 + e_2\epsilon_1\epsilon_2=0$. $\epsilon_1=0$ gives one root, and then depending on the values of the coefficients, one can get $0$, $2$ or $4$ real roots. Finding these roots goes by considering the degree two polynomial in the variable $X^2$ given by $q_2X^4+(r_2\epsilon_2)X^2+(e_2\epsilon_2)$. Changing $\epsilon_2$ to $-\epsilon_2$ does change the sign of the determinant. So on one of the two sides in the $v$ direction, one sees no multiple points. On the other side, the determinant is equivalent to $-4q_2e_2\epsilon_2$, and one finds that one of the two roots of the degree two polynomial has to be negative. This only creates two real roots, yielding a single double point, as above.

%%%%%%%%%%%%%%%%%%%%%%%%%%%%%%%%%%%%%%%%%%%%%%%%
\subsection{Neighborhood of a point on $F^{1}$}
%%%%%%%%%%%%%%%%%%%%%%%%%%%%%%%%%%%%%%%%%%%%%%%

Recall from the beginning of the section that we have to consider the following degeneracies:
\begin{itemize} \label{itemize:F1}
\item $\frac{\partial f}{\partial v} = \begin{pmatrix} \cdot \\ \cdot \\ \cdot \\ 0\end{pmatrix}$;
\item $(\frac{\partial f}{\partial u},\frac{\partial f}{\partial v})$ contains the direction $\vec{z}$;
\item $\frac{\partial f}{\partial w}$ is contained in the space spanned by $\vec{z},\frac{\partial f}{\partial u},\frac{\partial f}{\partial v}$.
\end{itemize}

Recall that the graph of $f$ is 3-dimensional ($+1$ through time). Regarding codimension, restricting to a point on $F^1$ is a codimension 2 condition. Asking that  $\frac{\partial f}{\partial v} = \begin{pmatrix} \cdot \\ \cdot \\ \cdot \\ 0\end{pmatrix}$ is a codimension $1$ condition, so this happens at isolated points with no other degeneracy in a foam, and such points form 1-dimensional sets through time.

To obtain a local foam model, consider a point $M=(u_0,v_0,w_0)\in F^{1}$, and assume that $\frac{\partial f}{\partial v} = \begin{pmatrix} \cdot \\ \cdot \\ \cdot \\ 0\end{pmatrix}$. Recall we use the local chart that makes the $v$ direction parallel to the seam. Up to rotation in the $x,y$ plane, one may assume that:
\[
  \frac{\partial f}{\partial v} = \begin{pmatrix} b_1 \\ 0 \\ b_3 \\ 0\end{pmatrix}, b_1>0.
\]
The assumption that $b_1>0$ is generic, even through time, as its failure expresses the codimension $2$ condition that both $x$ and $y$ are zero. Consider:
\[
 \frac{\partial^2 f}{\partial v^2} = \begin{pmatrix} g_1 \\ g_2 \\ g_3 \\ g_4\end{pmatrix}
\]
One may assume that $g_4\neq 0$. This is a codimension $1$ condition, so through time we will have to consider failure of it.

One may write, following the seam:
\[
  f(u_0,v_0+\epsilon_2,w_0)=f(u_0,v_0,w_0)+\begin{pmatrix} b_1 \epsilon_2 \\ g_2\epsilon_2^2 \\ b_3\epsilon_2 \\ g_4\epsilon_2^2\end{pmatrix}
\]
Depending on the sign of $g_4$, this corresponds to a Morse minimum ($g_4>0$) or maximum ($g_4<0$) on the seam. We run our analysis in the case with $g_4>0$, the other case being symmetric. Then, denoting:
\[
  \frac{\partial f}{\partial u} = \begin{pmatrix} a_1 \\ a_2 \\ a_3 \\ a_4\end{pmatrix}
\]
one may assume that $a_4\neq 0$ (this will also have to be considered through time). Up to adding to $u$ a scalar multiple of $v$, one can assume that:
\[
  \frac{\partial f}{\partial u} = \begin{pmatrix} 0 \\ a_2 \\ a_3 \\ a_4\end{pmatrix}
\]
Furthermore, up to adding to $w$ scalar multiples of $u$ and $v$, one can assume that:
\[
  \frac{\partial f}{\partial w} = \begin{pmatrix} 0 \\ c_2 \\ c_3 \\ 0\end{pmatrix}
\]

Thus one may write:
\begin{equation} \label{eq:Taylorzip}
 f(u_0+\epsilon_1,v_0+\epsilon_2,w_0)= \begin{pmatrix} b_1\epsilon_2 + d_1\epsilon_1^2 \\ a_2\epsilon_1+g_2\epsilon_2^2+e_2\epsilon_1\epsilon_2 \\ \ast \\  g_4\epsilon_2^2+a_4\epsilon_1  \end{pmatrix}
\end{equation}

At fixed level $s=r$, one thus has:
\[
  \epsilon_1=\frac{r-g_4\epsilon_2^2}{a_4}
\]
yielding a parametric curve equivalent to:
\[
  (d_1\frac{r^2}{a_4^2}+b_1\epsilon_2,\frac{a_2r}{a_4}+e_2\frac{r}{a_4}\epsilon_2
\]
provided $e_2\neq 0$, or:
\[
  (d_1\frac{r^2}{a_4^2}+b_1\epsilon_2,\frac{a_2r}{a_4}+(g_2-\frac{a_2g_4}{a_4})\epsilon_2^2
\]
if $e_2=0$ (provided $g_2-\frac{a_2g_4}{a_4}\neq 0$, which can be assumed generically if $e_2$ has already been set to zero).

In the first case, one just sees a line, while in the second case one gets a parabola.

Let us now identify the pieces that correspond to $\epsilon_1\geq 0$ (this is the part that will be replaced by the two legs of the web). $\epsilon_1\geq 0$ corresponds to:
\[
  \epsilon_2^2\leq \frac{r}{g_4}\; \text{if}\; a_4>0,\quad \epsilon_2^2\geq \frac{r}{g_4}\; \text{if}\; a_4<0.
\]

Now, to pass from the line to the web, one will replace the portions corresponding to $\epsilon_1\geq 0$ by a two curves, obtained by pushing positively or negatively in the framing direction. Notice that at $(u_0,v_0,w_0)$, the tangent to the line has slope $\frac{e_2r}{a_4b_1}\neq \infty$ (since $b_1$ and $a_4$ both are non-zero) if $e_2\neq 0$, and $0$ if $e_2=0$. Since the framing at $(u_0,v_0,w_0)$ has $(x,y)$ coordinates $(0,c_2)$, in both cases the process of creating the legs creates no crossing nor special framing features (provided $c_2\neq 0)$.

Thus depending whether $a_4>0$ or $a_4<0$, one gets one of the following foam generators:
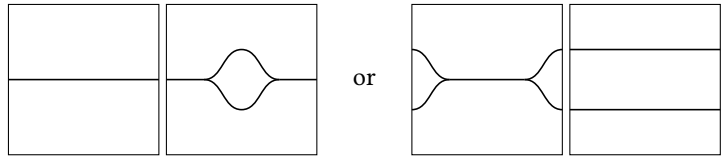
\[
  \begin{tikzpicture}[anchorbase]
    \node at (0,0) {
\begin{tikzpicture}[anchorbase, scale=.5]
  % Web
  \draw [semithick] (0,2) -- (4,2);
  % Square
  \draw (0,0) rectangle (4,4);
  \end{tikzpicture}};
\node at (2.1,0) {\begin{tikzpicture}[anchorbase, scale=.5]
  % Web
    \draw [semithick] (0,2) -- (1,2);
    \draw [semithick] (1,2) to [out=0,in=180] (2,2.8) to [out=0,in=180] (3,2);
    \draw [semithick] (1,2) to [out=0,in=180] (2,1.2) to [out=0,in=180] (3,2);
    \draw [semithick] (3,2) -- (4,2);
  % Square
  \draw (0,0) rectangle (4,4);
  \end{tikzpicture}};
\end{tikzpicture}
\quad \text{or}\quad
  \begin{tikzpicture}[anchorbase]
    \node at (0,0) {
\begin{tikzpicture}[anchorbase, scale=.5]
  % Web
  \draw [semithick] (0,1.2) to [out=0,in=180] (1,2);
  \draw [semithick] (0,2.8) to [out=0,in=180] (1,2);
  \draw [semithick] (1,2) -- (3,2);
  \draw [semithick] (3,2) to [out=0,in=180] (4,1.2);
  \draw [semithick] (3,2) to [out=0,in=180] (4,2.8);
  % Square
  \draw (0,0) rectangle (4,4);
  \end{tikzpicture}};
\node at (2.1,0) {\begin{tikzpicture}[anchorbase, scale=.5]
  % Web
  \draw [semithick] (0,1.2) -- (4,1.2);
  \draw [semithick] (0,2.8) -- (4,2.8);
  % Square
  \draw (0,0) rectangle (4,4);
  \end{tikzpicture}};
\end{tikzpicture}
\]

Through time, we have to consider the following special cases (which are mutually exclusive):
\begin{enumerate}
\item $g_4=0$;
\item $a_4=0$;
\item $c_2=0$ (which also corresponds to a framing change);
\end{enumerate}

In the case where $g_4=0$, one has to go one order further and appeal to:
\[
  \frac{\partial^3 f}{\partial v^3} = \begin{pmatrix} j_1 \\ j_2 \\ j_3 \\ j_4\end{pmatrix}
\]
One may assume that $j_4\neq 0$ and thus obtain:
\[
  f(u_0,v_0+\epsilon_2,w_0)=f(u_0,v_0,w_0)+\begin{pmatrix} b_1 \epsilon_2 \\ d_2\epsilon_2^2 \\ b_3\epsilon_2 \\ j_4\epsilon_2^3\end{pmatrix}
\]
We care about the last line, that draws a polynomial with a triple root at $0$.

Let us consider also:
\[
    \frac{\partial f}{ \partial t} = \begin{pmatrix} k_1 \\ k_2 \\ k_3 \\ k_4\end{pmatrix},\quad
  \frac{\partial^2 f}{\partial v \partial t} = \begin{pmatrix} l_1 \\ l_2 \\ l_3 \\ l_4\end{pmatrix}
\]
$l_4$ can be assumed to be non-zero. The Taylor expansion thus becomes (we still only focus on the seam):
\[
  f_{t_0+\delta}(u_0,v_0+\epsilon_2,w_0)=f_{t_0}(u_0,v_0,w_0)+\begin{pmatrix} b_1 \epsilon_2 +k_1\delta + l_1\epsilon_2\delta\\ d_2\epsilon_2^2+k_2\delta +l_2\epsilon_2\delta \\ k_3\epsilon_2 +l_3\delta+g_3\epsilon_2\delta\\ j_4\epsilon_2^3+k_4\delta + l_4\epsilon_2\delta \end{pmatrix}
\]
Through time and focusing on the $x$ and $s$ coordinates, one is brought back to a situation similar to the one investigated in the figures drawn at~\eqref{eq:pic3change1} and~\eqref{eq:pic3change2}. Both signs for $a_4$ yield mirror versions of the following move:

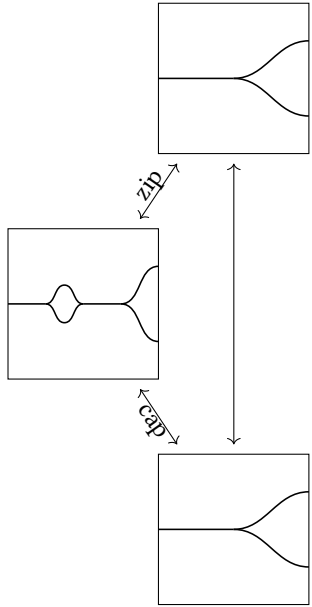
\[
  \begin{tikzpicture}[anchorbase]
    \node (A) at (0,0) {
\begin{tikzpicture}[anchorbase, scale=.5]
  % Web
  \draw [semithick] (0,2) -- (2,2);
  \draw [semithick] (2,2) to [out=0,in=180] (4,3);
  \draw [semithick] (2,2) to [out=0,in=180] (4,1);
  % Rectangle
  \draw (0,0) rectangle (4,4);
  \end{tikzpicture}};
    \node (B1) at (-2,-3) {
\begin{tikzpicture}[anchorbase, scale=.5]
  % Web
  \draw [semithick] (0,2) -- (1,2);
  \draw [semithick] (1,2) to [out=0,in=180] (1.5,2.5) to [out=0,in=180] (2,2);
  \draw [semithick] (1,2) to [out=0,in=180] (1.5,1.5) to [out=0,in=180] (2,2);
  \draw [semithick] (2,2) -- (3,2);
  \draw [semithick] (3,2) to [out=0,in=180] (4,3);
  \draw [semithick] (3,2) to [out=0,in=180] (4,1);
  % Rectangle
  \draw (0,0) rectangle (4,4);
  \end{tikzpicture}};
    \node (C) at (0,-6) {
\begin{tikzpicture}[anchorbase, scale=.5]
  % Web
  \draw [semithick] (0,2) -- (2,2);
  \draw [semithick] (2,2) to [out=0,in=180] (4,3);
  \draw [semithick] (2,2) to [out=0,in=180] (4,1);
  % Rectangle
  \draw (0,0) rectangle (4,4);
  \end{tikzpicture}};
\draw [<->] (A) -- (B1)  node [rotate=55, yshift=4, midway] {zip} ;
\draw [<->] (A) -- (C);
\draw [<->] (B1) -- (C)  node [rotate=-55, yshift=-4, midway] {cap} ;
\end{tikzpicture}
\]

Let us now suppose that $a_4=0$. One should consider the second derivative in the $u$ direction, the $s$-coordinate of which we denote $d_4$. One may assume that $d_4\neq 0$. This means that on both sides of the seam, the image of the $w=0$ sheet goes similarly up or down. Since following both split facets corresponds to $\epsilon_3<<\epsilon_1$, we have the same behavior for the two thinnest facets.
Up to mirror, one can assume that $g_4>0$, and we are left with two situations: either $d_4>0$ or $d_4<0$. In the case where $d_4>0$ one reads:
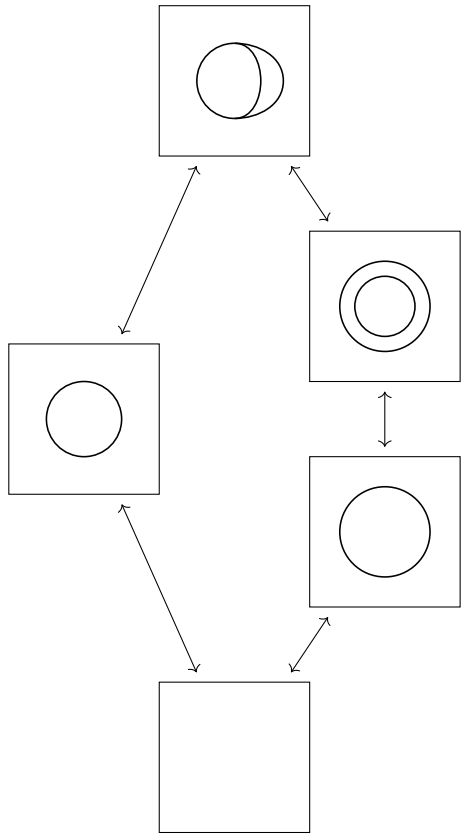
\begin{equation}
  \begin{tikzpicture}[scale=1]
    % % Surrounding square
    % \draw [dashed] (-4,5) rectangle (4,-5);
    \node (A) at (0,0) {
      \begin{tikzpicture}[anchorbase,scale=.5]
        % Web
        \draw [semithick] (2,1) to [out=0,in=-90] (3.3,2) to [out=90,in=0] (2,3);
        \draw [semithick] (2,1) to [out=0,in=-90] (2.7,2) to [out=90,in=0] (2,3);
        \draw [semithick] (2,3) to [out=180,in=90] (1,2) to [out=-90,in=180] (2,1);
        % Square
        \draw (0,0) rectangle (4,4);
      \end{tikzpicture}
    };
    \node (B1) at (-2,-4.5) {
      \begin{tikzpicture}[anchorbase,scale=.5]
        % Web
        \draw [semithick] (2,2) circle (1);
        % Square
        \draw (0,0) rectangle (4,4);
      \end{tikzpicture}
    };
    \node (B2) at (2,-3) {
      \begin{tikzpicture}[anchorbase,scale=.5]
        % Web
        \draw [semithick] (2,2) circle (1.2);
        \draw [semithick] (2,2) circle (.8);
        % Square
        \draw (0,0) rectangle (4,4);
      \end{tikzpicture}
    };
        \node (C2) at (2,-6) {
      \begin{tikzpicture}[anchorbase,scale=.5]
        % Web
        \draw [semithick] (2,2) circle (1.2);
        % Square
        \draw (0,0) rectangle (4,4);
      \end{tikzpicture}
    };
        \node (D) at (0,-9) {
      \begin{tikzpicture}[anchorbase,scale=.5]
        % Web
        % Square
        \draw (0,0) rectangle (4,4);
      \end{tikzpicture}
    };
    \draw [<->] (A) -- (B1);
    \draw [<->] (A) -- (B2);
    \draw [<->] (B2) -- (C2);
    \draw [<->] (B1) -- (D);
    \draw [<->] (C2) -- (D);
  \end{tikzpicture}
\end{equation}

When $d_4<0$, one gets:

\begin{equation}
  \begin{tikzpicture}[scale=1]
    % % Surrounding square
    % \draw [dashed] (-4,5) rectangle (4,-5);
    \node (A) at (0,0) {
      \begin{tikzpicture}[anchorbase,scale=.5]
        % Web
        \draw [semithick] (0,1) -- (4,1);
        \draw [semithick] (0,2) -- (4,2);
        \draw [semithick] (0,3) -- (4,3);
        % Square
        \draw (0,0) rectangle (4,4);
      \end{tikzpicture}
    };
    \node (B1) at (-2,-4.5) {
      \begin{tikzpicture}[anchorbase,scale=.5]
        % Web
        \draw [semithick] (0,1) -- (4,1);
        \draw [semithick] (0,2) to [out=0,in=180] (1.5,2.5);
        \draw [semithick] (0,3) to [out=0,in=180] (1.5,2.5);
        \draw [semithick] (1.5,2.5) -- (2.5,2.5);
        \draw [semithick] (2.5,2.5) to [out=0,in=180] (4,2);
        \draw [semithick] (2.5,2.5) to [out=0,in=180] (4,3);
        % Square
        \draw (0,0) rectangle (4,4);
      \end{tikzpicture}
    };
    \node (B2) at (2,-3) {
      \begin{tikzpicture}[anchorbase,scale=.5]
        % Web
        \draw [semithick] (0,1) -- (1.5,1);
        \draw [semithick] (1.5,1) to [out=0,in=180] (2,1.5) to [out=0,in=180] (2.5,1);
        \draw [semithick] (1.5,1) to [out=0,in=180] (2,.5) to [out=0,in=180] (2.5,1);
        \draw [semithick] (2.5,1) -- (4,1);
        \draw [semithick] (0,2) -- (4,2);
        \draw [semithick] (0,3) -- (4,3);
        % Square
        \draw (0,0) rectangle (4,4);
      \end{tikzpicture}
    };
        \node (C2) at (2,-6) {
      \begin{tikzpicture}[anchorbase,scale=.5]
        % Web
        \draw [semithick] (0,1) -- (1,1);
        \draw [semithick] (1,1) to [out=0,in=0] (1,2) -- (0,2);
        \draw [semithick] (1,1) to [out=0,in=180] (2,.5) to [out=0,in=180] (3,1);
        \draw [semithick] (3,1) to [out=180,in=180]  (3,2) -- (4,2);
        \draw [semithick] (3,1) -- (4,1);
        \draw [semithick] (0,3) -- (4,3);
        % Square
        \draw (0,0) rectangle (4,4);
      \end{tikzpicture}
    };
        \node (D) at (0,-9) {
      \begin{tikzpicture}[anchorbase,scale=.5]
        % Web
        \draw [semithick] (0,2) to [out=0,in=180] (1,2.5);
        \draw [semithick] (0,3) to [out=0,in=180] (1,2.5);
        \draw [semithick] (1,2.5) to [out=0,in=0] (1,1) -- (0,1);
        \draw [semithick] (3,2.5) to [out=180,in=180] (3,1) -- (4,1);
        \draw [semithick] (3,2.5) to [out=0,in=180] (4,2);
        \draw [semithick] (3,2.5) to [out=0,in=180] (4,3);
        % Square
        \draw (0,0) rectangle (4,4);
      \end{tikzpicture}
    };
    \draw [<->] (A) -- (B1);
    \draw [<->] (A) -- (B2);
    \draw [<->] (B2) -- (C2);
    \draw [<->] (B1) -- (D);
    \draw [<->] (C2) -- (D);
  \end{tikzpicture}
\end{equation}

Let us now consider the case where $a_4$ and $g_4$ are both non-zero, but $c_2=0$. One can start the same analysis as in Equation~\eqref{eq:Taylorzip}, but splitting the piece with $\epsilon_1\geq 0$ will be more involved (since the framing crosses the line). It is best seen by pushing the Taylor expansion also in the $w$ direction, writing:
\begin{gather*}
  (f(u,v_0+\epsilon_2,w))_x\sim b_1\epsilon_2+e_1\epsilon_1\epsilon_2 +f_1\epsilon_1\epsilon_3+h_1\epsilon_2\epsilon_3
  \\
  (f(u,v_0+\epsilon_2,w))_y\sim a_2\epsilon_1 +e_2\epsilon_1\epsilon_2 +g_2\epsilon_2^2+f_1\epsilon_1\epsilon_3+h_2\epsilon_2\epsilon_3+i_2\epsilon_3^2.
  \end{gather*}
Again one sees a line at first order in $\epsilon_2$, but the change when $\epsilon_3$ becomes $-\epsilon_3$ is now controlled by $h_2\epsilon_2\epsilon_3$. This creates a crossing, and at $t=t_0$ we get a singular version of the zip/unzip or cap/cup generators (the crossings come of course with a sign):
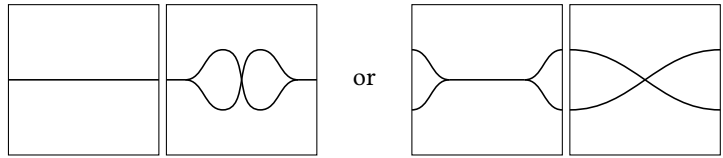
\[

\]
Furthermore, the line of multiple points ends at $(u_0,v_0,w_0)\in F^{1}$, our Morse extremum. Bringing in the time parameter leaves the end of this line to the same point, while the Morse minimum moves in a direction determined by the $s$ coordinate of $\frac{\partial^2f}{\partial v \partial t}$. One thus reads one of the following moves (we haven't yet studied the framing data):

\begin{equation}
  %
  \end{equation}

  Recall that we have:
  \[
    \frac{\partial f}{\partial u}=%
\]
Thus the framing is purely in the $\vec{z}$ direction, which produces a framing change. Following the seam, one reads:
\[
  \frac{\partial f}{\partial w}_{u_0,v_0+\epsilon_2,w_0}=%
\]
$h_2$ may be assumed to be non-zero. Around the singular point we get something that looks like the following:
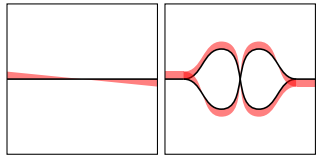
\[
  %
\]

Depending on the signs for $h_2$ and $h_4$ one gets one of the following framed moves:
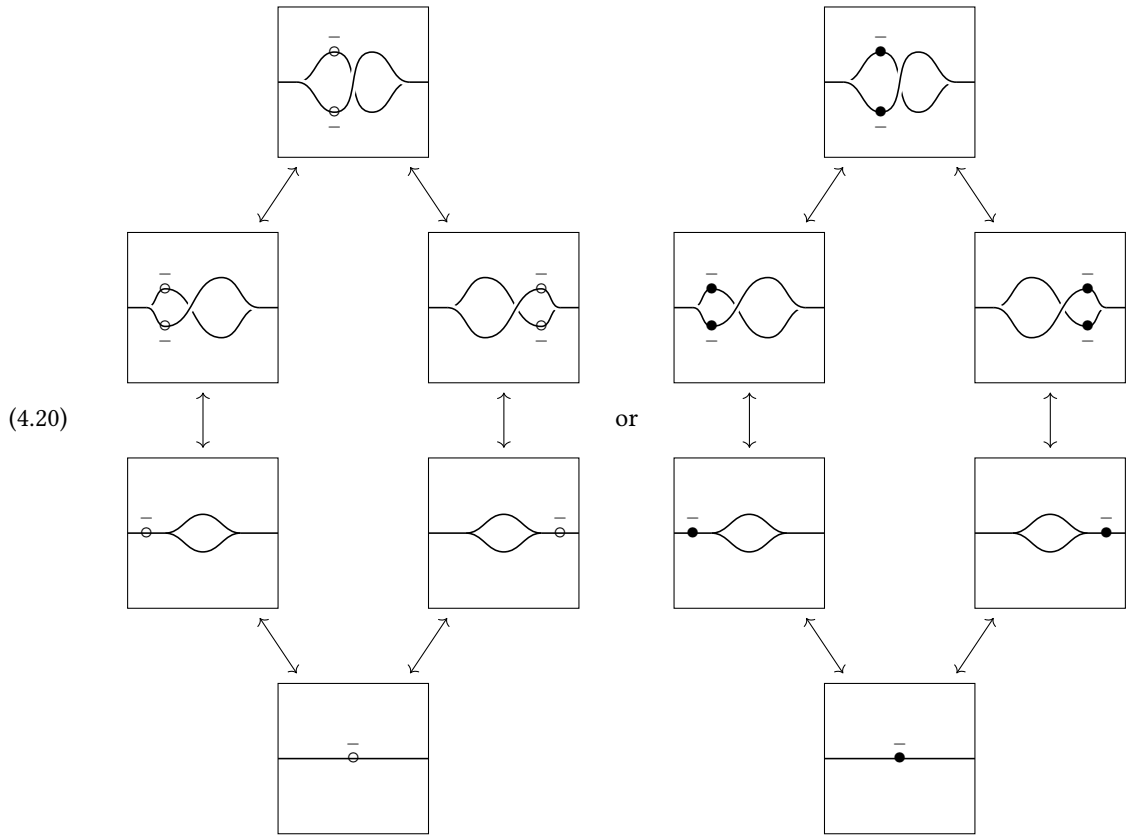
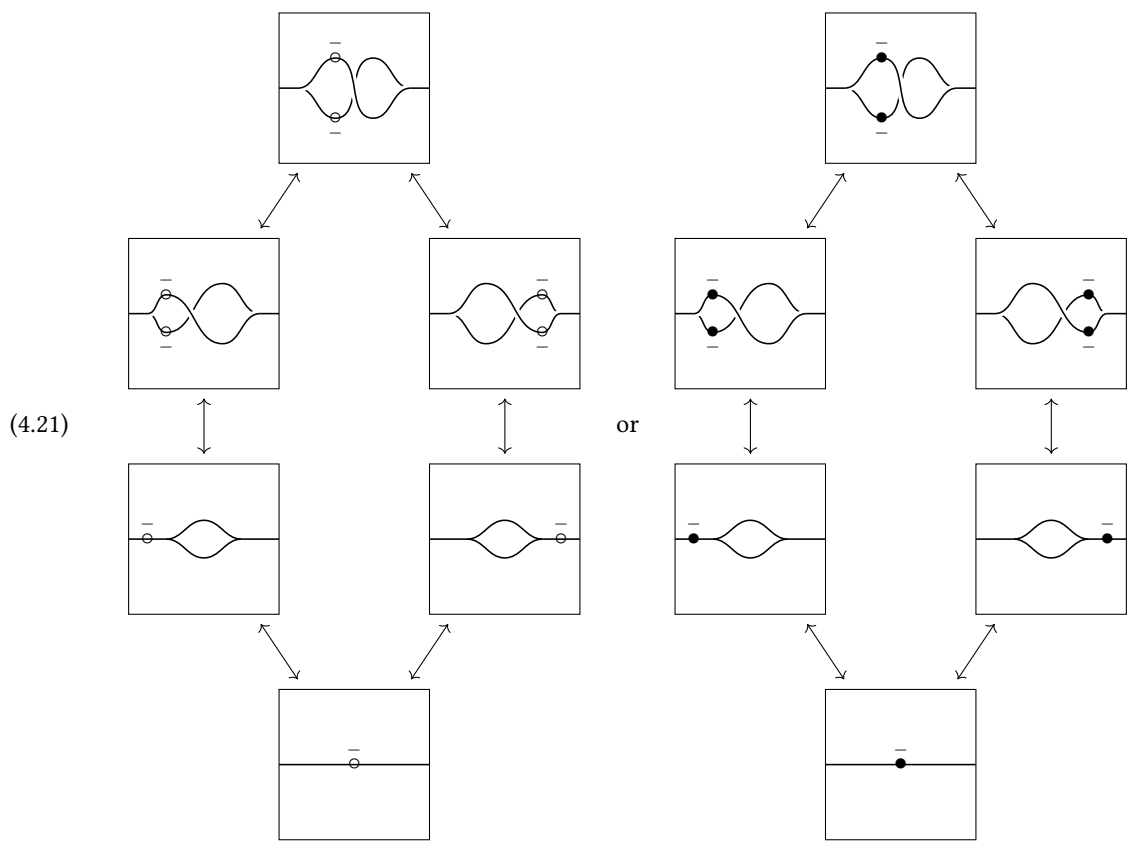
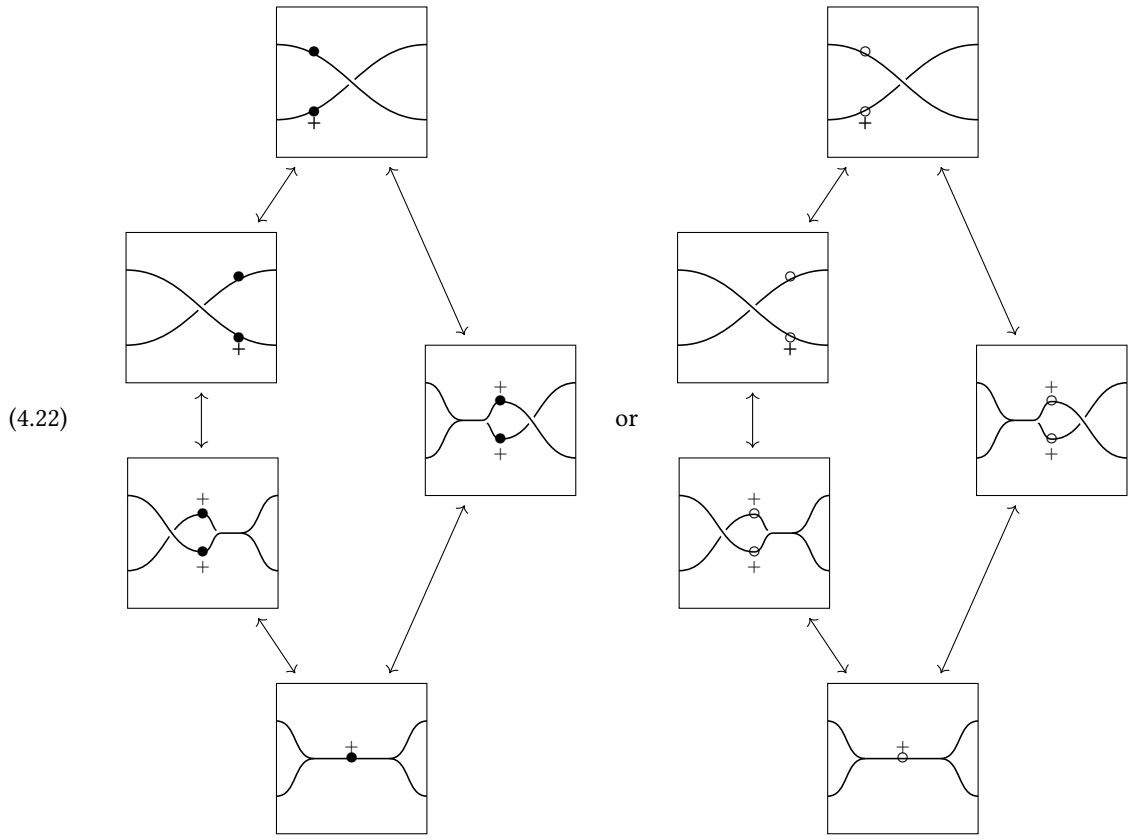
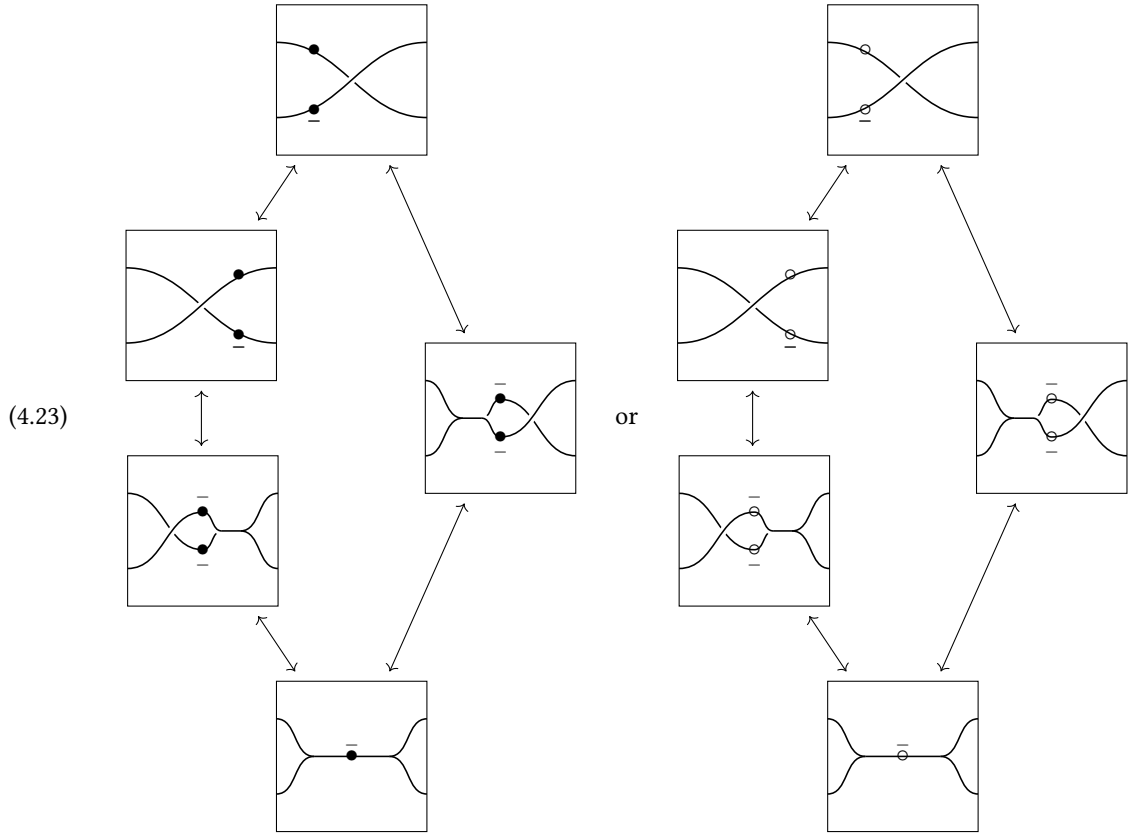
\begin{equation} \label{eq:crzcup}
  %
    };
    \draw [<->] (A) -- (B1);
    \draw [<->] (A) -- (B2);
    \draw [<->] (B1) -- (C1);
    \draw [<->] (C1) -- (D);
    \draw [<->] (B2) -- (D);
  \end{tikzpicture}
\end{equation}

\vspace{1cm}

We now consider the second bullet point from the list written at the beginning of this section:
\[
  \vec{z}\in  \Span_\R(\frac{\partial f}{\partial u},\frac{\partial f}{\partial v})
\]
Recall that we have set the $v$ direction to be the one of the seam. Asking that this is parallel to the $\vec{z}$ direction is a codimension $3$ condition, so this is never generic. On the other hand, asking that $\frac{\partial f}{\partial u}$ is a linear combination of the $\vec{z}$ direction and of $\frac{\partial f}{\partial v}$ is a codimension $2$ condition. Through time, we expect isolated points. At such a point we can reparametrize $u$ by a scalar multiple of $v$ so that $\frac{\partial f}{\partial u}$ is genuinely vertical.

Interestingly, the resulting move can be read entirely from the ones in Equation~\eqref{eq:frR1}. Indeed, forgetting one of the two legs of the web does bring us back to the isotopy investigated in Section~\ref{sec:webs}. Now, when the time moves the vertical vector on the side of the web that only has one strand, one gets the same isotopy as in Equation~\eqref{eq:frR1}. When it moves on the other side, focusing on each leg one at a time, we again read the same isotopy. We pretend that the way the two legs entangle can also be read from it: indeed, one leg is the image of the other one by a push along the framing: one just has to take the boundary of the red ribbon in~\eqref{eq:frR1}. The resulting move is equivalent to the following one (or to similar ones with other configurations for crossings and twists):

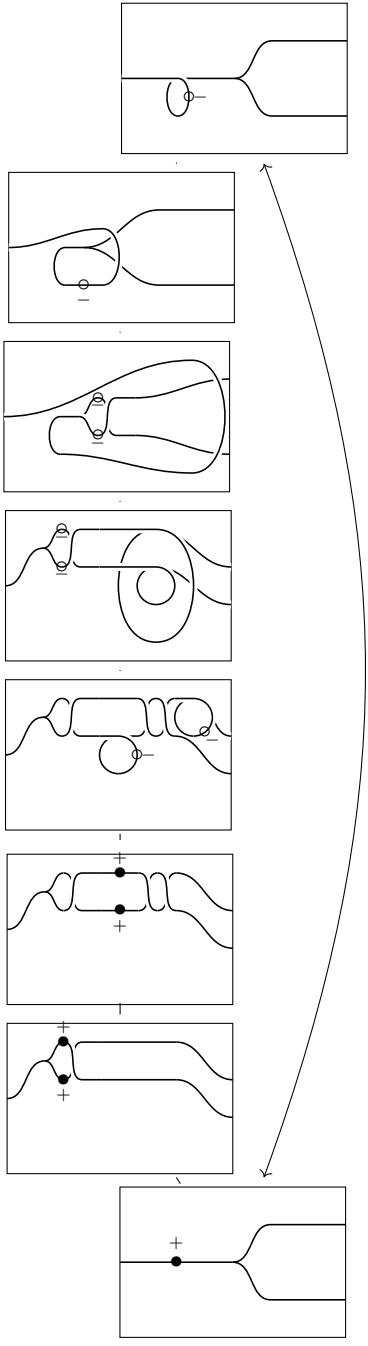
\begin{equation}
\begin{tikzpicture}[anchorbase,scale=.75]
  \node at (0,0) {
\begin{tikzpicture}[decoration={markings, mark=at position 0.6 with {\arrow{>}};  },fill opacity=.2,scale=.15]
%\draw[dashed] (25,32) rectangle (-25,-32);
%\node[opacity=1] at (0,11) {\bf{Sauron}};
%\node[opacity=1] at (0,-11) {};
\end{tikzpicture}
  };
  \node (A) at (0,0) {
    \begin{tikzpicture}[anchorbase,scale=.5]
      % Square
      \draw (0,0) rectangle (6,4);
      % Strands
      \draw [white, line width=3] (1.5,1) to [out=180,in=180] (1.5,2) -- (3,2);
      \draw [semithick] (1.5,1) to [out=180,in=180] (1.5,2) -- (3,2);
      \draw [white, line width=3] (0,2) to [out=0,in=180] (1.5,2) to [out=0,in=0] (1.5,1);
      \draw [semithick] (0,2) to [out=0,in=180] (1.5,2) to [out=0,in=0] (1.5,1);
      \draw [semithick] (3,2) to [out=0,in=180] (4,3) -- (6,3);
      \draw [semithick] (3,2) to [out=0,in=180] (4,1) -- (6,1);
      \node at (1.8,1.5) {$\circ$};
      \node at (2.1,1.5) {\tiny $-$};
    \end{tikzpicture}
  };
  \node (B1) at (-2,-3) {
    \begin{tikzpicture}[anchorbase,scale=.5]
      % Square
      \draw (0,0) rectangle (6,4);
      % Strands
      \draw [white, line width=3] (1.5,1) to [out=180,in=180] (1.5,2) -- (2,2);
      \draw [semithick] (1.5,1) to [out=180,in=180] (1.5,2) -- (2,2);
      \draw [semithick] (2,2) to [out=0,in=180] (4,3) -- (6,3);
      \draw [semithick] (2,2) to [out=0,in=180] (4,1) -- (6,1);
      \draw [white, line width=3] (0,2) to [out=0,in=180] (2.5,2.5) to [out=0,in=0] (2.5,1) -- (1.5,1);
      \draw [semithick]  (0,2) to [out=0,in=180] (2.5,2.5) to [out=0,in=0] (2.5,1) -- (1.5,1);
      \node at (2,1) {$\circ$};
      \node at (2,.6) {\tiny $-$};
     \end{tikzpicture}
  };
  \node (C1) at (-2,-6) {
    \begin{tikzpicture}[anchorbase,scale=.5]
      % Square
      \draw (0,0) rectangle (6,4);
      % Strands
      \draw [white, line width=3] (1.5,1) to [out=180,in=180] (1.5,2) -- (2,2);
      \draw [semithick] (1.5,1) to [out=180,in=180] (1.5,2) -- (2,2);
      \draw [white, line width=3] (2,2) to [out=0,in=180] (2.5,2.5) to [out=0,in=180] (3,1.5);
      \draw [semithick] (2,2) to [out=0,in=180] (2.5,2.5) to [out=0,in=180] (3,1.5) -- (3.5,1.5);
      \draw [white, line width=3] (2,2) to [out=0,in=180] (2.5,1.5) to [out=0,in=180] (3,2.5);
      \draw [semithick] (2,2) to [out=0,in=180] (2.5,1.5) to [out=0,in=180] (3,2.5) -- (3.5,2.5);
      % \draw [white, line width=3] (3,2.5) to [out=0,in=180] (3.5,1.5);
      % \draw [semithick] (3,2.5) to [out=0,in=180] (3.5,1.5);
      % \draw [white, line width=3] (3,1.5) to [out=0,in=180] (3.5,2.5);
      % \draw [semithick] (3,1.5) to [out=0,in=180] (3.5,2.5);
      \draw [semithick] (3.5,2.5) to [out=0,in=180] (6,3);
      \draw [semithick] (3.5,1.5) to [out=0,in=180]  (6,1);
      \draw [white, line width=3] (0,2) to [out=0,in=180] (5,3.5) to [out=0,in=0] (5,.5) to [out=180,in=0] (1.5,1);
      \draw [semithick]  (0,2) to [out=0,in=180] (5,3.5) to [out=0,in=0] (5,.5) to [out=180,in=0] (1.5,1);
      \node at (2.5,2.5) {$\circ$};
      \node at (2.5,2.3) {\tiny $-$};
      \node at (2.5,1.5) {$\circ$};
      \node at (2.5,1.3) {\tiny $-$};
      % \node at (3,2.5) {$\bullet$};
      % \node at (3,2.7) {\tiny $-$};
      % \node at (3,1.5) {$\bullet$};
      % \node at (3,1.3) {\tiny $-$};
    \end{tikzpicture}
  };
  \node (D1) at (-2,-9) {
    \begin{tikzpicture}[anchorbase,scale=.5]
      % Square
      \draw (0,0) rectangle (6,4);
      % Strands
      % Incoming strand
      \draw [semithick] (0,2) to [out=0,in=180] (1,3);
      % double crossings between outgoing edges
      \draw [semithick] (1,3) to [out=0,in=180] (1.5,3.5);
      \draw [semithick] (1,3) to [out=0,in=180] (1.5,2.5);
      \draw [white, line width=3] (1.5,3.5) to [out=0,in=180] (2,2.5);
      \draw [semithick] (1.5,3.5) to [out=0,in=180] (2,2.5) -- (2.5,2.5);
      \draw [white, line width=3] (1.5,2.5) to [out=0,in=180] (2,3.5);
      \draw [semithick] (1.5,2.5) to [out=0,in=180] (2,3.5) -- (2.5,3.5);
      % \draw [white, line width=3] (2,3.5) to [out=0,in=180] (2.5,2.5);
      % \draw [semithick] (2,3.5) to [out=0,in=180] (2.5,2.5);
      % \draw [white, line width=3] (2,2.5) to [out=0,in=180] (2.5,3.5);
      % \draw [semithick] (2,2.5) to [out=0,in=180] (2.5,3.5);
      % big double curl
      \draw [white, line width=3] (4,.5) to [out=180,in=-90] (3,2) to [out=90,in=180] (4,3.5) to [out=0,in=180] (6,2.5);
      \draw [semithick] (4,.5) to [out=180,in=-90] (3,2) to [out=90,in=180] (4,3.5) to [out=0,in=180] (6,2.5);
      \draw [white, line width=3] (4,1.5) to [out=180,in=-90] (3.5,2) to [out=90,in=180] (4,2.5) to [out=0,in=180] (6,1.5);
      \draw [semithick] (4,1.5) to [out=180,in=-90] (3.5,2) to [out=90,in=180] (4,2.5) to [out=0,in=180] (6,1.5);
      \draw [white, line width=3] (2.5,2.5) -- (4,2.5) to [out=0,in=90] (4.5,2) to [out=-90,in=0] (4,1.5);
      \draw [semithick] (2.5,2.5) -- (4,2.5) to [out=0,in=90] (4.5,2) to [out=-90,in=0] (4,1.5);
      \draw [white, line width=3] (2.5,3.5) -- (4,3.5) to [out=0,in=90] (5,2) to [out=-90,in=0] (4,.5);
      \draw [semithick] (2.5,3.5) -- (4,3.5) to [out=0,in=90] (5,2) to [out=-90,in=0] (4,.5);
      % Twists
      \node at (1.5,3.5) {$\circ$};
      \node at (1.5,3.3) {\tiny $-$};
      \node at (1.5,2.5) {$\circ$};
      \node at (1.5,2.3) {\tiny $-$};
      % \node at (2,3.5) {$\bullet$};
      % \node at (2,3.7) {\tiny $-$};
      % \node at (2,2.5) {$\bullet$};
      % \node at (2,2.3) {\tiny $-$};
    \end{tikzpicture}
  };
  \node (E1) at (-2,-12) {
    \begin{tikzpicture}[anchorbase,scale=.5]
      % Square
      \draw (0,0) rectangle (6,4);
      % Strands
      % Incoming strand
      \draw [semithick] (0,2) to [out=0,in=180] (1,3);
      % double crossings between outgoing edges
      \draw [semithick] (1,3) to [out=0,in=180] (1.5,3.5);
      \draw [semithick] (1,3) to [out=0,in=180] (1.5,2.5);
      \draw [white, line width=3] (1.5,3.5) to [out=0,in=180] (2,2.5);
      \draw [semithick] (1.5,3.5) to [out=0,in=180] (2,2.5) -- (2.5,2.5);
      \draw [white, line width=3] (1.5,2.5) to [out=0,in=180] (2,3.5);
      \draw [semithick] (1.5,2.5) to [out=0,in=180] (2,3.5) -- (2.5,3.5);
      % \draw [white, line width=3] (2,3.5) to [out=0,in=180] (2.5,2.5);
      % \draw [semithick] (2,3.5) to [out=0,in=180] (2.5,2.5);
      % \draw [white, line width=3] (2,2.5) to [out=0,in=180] (2.5,3.5);
      % \draw [semithick] (2,2.5) to [out=0,in=180] (2.5,3.5);
      % curl on bottom strand
      \draw [white, line width=3] (3,1.5) to [out=180,in=-90] (2.5,2) to [out=90,in=180] (3,2.5) -- (3.5,2.5);
      \draw [semithick]  (3,1.5) to [out=180,in=-90] (2.5,2) to [out=90,in=180] (3,2.5) -- (3.5,2.5);
      \draw [white, line width=3] (2.5,2.5) -- (3,2.5) to [out=0,in=90] (3.5,2) to [out=-90,in=0] (3,1.5);
      \draw [semithick] (2.5,2.5) -- (3,2.5) to [out=0,in=90] (3.5,2) to [out=-90,in=0] (3,1.5);
      % top strand
      \draw [semithick] (2.5,3.5) -- (3.5,3.5);
      % double crossing again
      \draw [white, line width=3] (3.5,2.5) to [out=0,in=180] (4,3.5);
      \draw [semithick] (3.5,2.5) to [out=0,in=180] (4,3.5);
      \draw [white, line width=3] (3.5,3.5) to [out=0,in=180] (4,2.5);
      \draw [semithick] (3.5,3.5) to [out=0,in=180] (4,2.5);
      \draw [white, line width=3] (4,2.5) to [out=0,in=180] (4.5,3.5);
      \draw [semithick] (4,2.5) to [out=0,in=180] (4.5,3.5);
      \draw [white, line width=3] (4,3.5) to [out=0,in=180] (4.5,2.5);
      \draw [semithick] (4,3.5) to [out=0,in=180] (4.5,2.5);
      % curl on top strand
      \draw [white, line width=3] (5,2.5) to [out=180,in=-90] (4.5,3) to [out=90,in=180] (5,3.5) to [out=0,in=180] (6,2.5);
      \draw [semithick] (5,2.5) to [out=180,in=-90] (4.5,3) to [out=90,in=180] (5,3.5) to [out=0,in=180] (6,2.5);
      \draw [white, line width=3] (4.5,3.5) -- (5,3.5) to [out=0,in=90] (5.5,3) to [out=-90,in=0] (5,2.5);
      \draw [semithick]  (4.5,3.5) -- (5,3.5) to [out=0,in=90] (5.5,3) to [out=-90,in=0] (5,2.5);
      % bottom strand
      \draw [semithick] (4.5,2.5) to [out=0,in=180] (6,1.5);
      % Twists
      \node at (3.5,2) {$\circ$};
      \node at (3.8,2) {\tiny $-$};
      % \node at (2.5,2) {$\bullet$};
      % \node at (2.2,2) {\tiny $-$};
      \node at (5.3,2.6) {$\circ$};
      \node at (5.5,2.4) {\tiny $-$};
      % \node at (4.5,2.9) {$\bullet$};
      % \node at (4.7,3.1) {\tiny $-$};
    \end{tikzpicture}
  };
  \node (F1) at (-2,-15) {
    \begin{tikzpicture}[anchorbase,scale=.5]
      % Square
      \draw (0,0) rectangle (6,4);
      % Strands
      % Incoming strand
      \draw [semithick] (0,2) to [out=0,in=180] (1,3);
      % double crossings between outgoing edges
      \draw [semithick] (1,3) to [out=0,in=180] (1.5,3.5);
      \draw [semithick] (1,3) to [out=0,in=180] (1.5,2.5);
      \draw [white, line width=3] (1.5,3.5) to [out=0,in=180] (2,2.5);
      \draw [semithick] (1.5,3.5) to [out=0,in=180] (2,2.5) -- (2.5,2.5);
      \draw [white, line width=3] (1.5,2.5) to [out=0,in=180] (2,3.5);
      \draw [semithick] (1.5,2.5) to [out=0,in=180] (2,3.5) -- (2.5,3.5);
      % \draw [white, line width=3] (2,3.5) to [out=0,in=180] (2.5,2.5);
      % \draw [semithick] (2,3.5) to [out=0,in=180] (2.5,2.5);
      % \draw [white, line width=3] (2,2.5) to [out=0,in=180] (2.5,3.5);
      % \draw [semithick] (2,2.5) to [out=0,in=180] (2.5,3.5);
      % curl on bottom strand
      \draw [semithick] (2.5,2.5) -- (3.5,2.5);
      % top strand
      \draw [semithick] (2.5,3.5) -- (3.5,3.5);
      % double crossing again
      \draw [white, line width=3] (3.5,2.5) to [out=0,in=180] (4,3.5);
      \draw [semithick] (3.5,2.5) to [out=0,in=180] (4,3.5);
      \draw [white, line width=3] (3.5,3.5) to [out=0,in=180] (4,2.5);
      \draw [semithick] (3.5,3.5) to [out=0,in=180] (4,2.5);
      \draw [white, line width=3] (4,2.5) to [out=0,in=180] (4.5,3.5);
      \draw [semithick] (4,2.5) to [out=0,in=180] (4.5,3.5);
      \draw [white, line width=3] (4,3.5) to [out=0,in=180] (4.5,2.5);
      \draw [semithick] (4,3.5) to [out=0,in=180] (4.5,2.5);
      % curl on top strand
      \draw [semithick]  (4.5,3.5) to [out=0,in=180] (6,2.5);
      % bottom strand
      \draw [semithick] (4.5,2.5) to [out=0,in=180] (6,1.5);
      % Twists
      \node at (3,2.5) {$\bullet$};
      \node at (3,2.1) {\tiny $+$};
      \node at (3,3.5) {$\bullet$};
      \node at (3,3.9) {\tiny $+$};
    \end{tikzpicture}
  };
    \node (G) at (-2,-18) {
    \begin{tikzpicture}[anchorbase,scale=.5]
      % Square
      \draw (0,0) rectangle (6,4);
      % Strands
      % Incoming strand
      \draw [semithick] (0,2) to [out=0,in=180] (1,3);
      %  crossings between outgoing edges
      \draw [semithick] (1,3) to [out=0,in=180] (1.5,3.5);
      \draw [semithick] (1,3) to [out=0,in=180] (1.5,2.5);
      \draw [white, line width=3] (1.5,2.5) to [out=0,in=180] (2,3.5);
      \draw [semithick] (1.5,2.5) to [out=0,in=180] (2,3.5) -- (2.5,3.5);
      \draw [white, line width=3] (1.5,3.5) to [out=0,in=180] (2,2.5);
      \draw [semithick] (1.5,3.5) to [out=0,in=180] (2,2.5) -- (2.5,2.5);
      % curl on bottom strand
      \draw [semithick] (2.5,2.5) -- (3.5,2.5) -- (4.5,2.5);
      % top strand
      \draw [semithick] (2.5,3.5) -- (3.5,3.5) -- (4.5,3.5);
      % curl on top strand
      \draw [semithick]  (4.5,3.5) to [out=0,in=180] (6,2.5);
      % bottom strand
      \draw [semithick] (4.5,2.5) to [out=0,in=180] (6,1.5);
      % Twists
      \node at (1.5,2.5) {$\bullet$};
      \node at (1.5,2.1) {\tiny $+$};
      \node at (1.5,3.5) {$\bullet$};
      \node at (1.5,3.9) {\tiny $+$};
    \end{tikzpicture}
  };
  \node (H) at (0,-21) {
    \begin{tikzpicture}[anchorbase,scale=.5]
      % Square
      \draw (0,0) rectangle (6,4);
      % Strands
      \draw [semithick] (0,2) -- (3,2);
      \draw [semithick] (3,2) to [out=0,in=180] (4,3) -- (6,3);
      \draw [semithick] (3,2) to [out=0,in=180] (4,1) -- (6,1);
      % Twists
      \node at (1.5,2) {$\bullet$};
      \node at (1.5,2.5) {\tiny $+$};
    \end{tikzpicture}
  };
  \draw [-] (A) -- (B1);
  \draw [-] (B1) -- (C1);
  \draw [-] (C1) -- (D1);
  \draw [-] (D1) -- (E1);
  \draw [-] (E1) -- (F1);
  \draw [-] (F1) -- (G);
  \draw [-] (G) -- (H);
  \draw [<->] (A) to [out=-70,in=70] (H);
\end{tikzpicture}
\end{equation}

For the third bullet point on page~\pageref{itemize:F1}, requiring that $\frac{\partial f}{\partial w}\in \Span_\R\left(\frac{\partial f}{\partial u},\frac{\partial f}{\partial v},\vec{z}\right)$ at $M\in F^{1}$ is a codimension $3$ condition ($2$ from restricting to $F^1$, $1$ from the framing assumption). This creates foam generators that we have already investigated in Theorem~\ref{th:ReidHTwist}, namely:
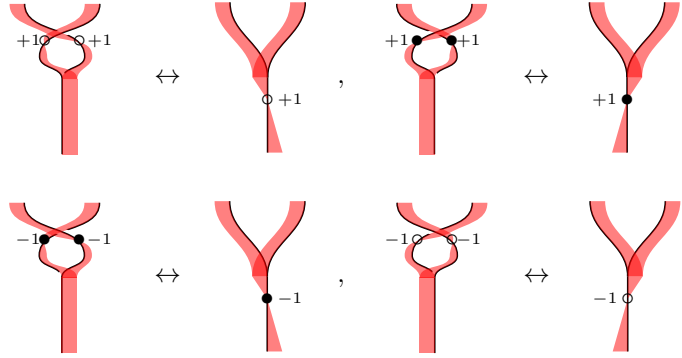
\[
  \begin{tikzpicture}[anchorbase,rotate=90]
    \draw [semithick, postaction={decorate}] (0,0) to [out=0,in=180] (.3,-.3) to [out=0,in=180] (1,.5);
    \draw [white, line width=3] (0,0) to [out=0,in=180] (.3,.3) to [out=0,in=180] (1,-.5);
    \draw [semithick, postaction={decorate}] (-1,0) -- (0,0);
    \draw [semithick, postaction={decorate}] (0,0) to [out=0,in=180] (.3,.3) to [out=0,in=180] (1,-.5);
    \fill [red,opacity=.5] (-1,0) -- (0,0) -- (0,-.2) -- (-1,-.2);
    \fill [red,opacity=.5] (0,0) to [out=0,in=180] (.3,.3) to [out=0,in=180] (1,-.5) -- (1,-.3) to [out=180,in=-60] (.6,.2) to [out=120,in=0] (.3,.2) to [out=180,in=0] (0,-.2);
    \fill [red, opacity=.5] (0,0) to [out=0,in=180] (.3,-.3) to [out=0,in=180] (1,.5) -- (1,.7) to [out=180,in=60] (.6,.2) to [out=-120,in=0] (.3,-.4) to [out=180,in=0] (0,-.2);
   \node at (.5,-.23) {$\circ$};
    \node at (.5,-.5) {\tiny $+1$};
    \node at (.5,.23) {$\circ$};
    \node at (.5,.45) {\tiny $+1$};
      \end{tikzpicture}
  \quad \leftrightarrow \quad
  \begin{tikzpicture}[anchorbase,rotate=90]
    \draw [semithick, postaction={decorate}] (0,0) to [out=0,in=180] (1,.5);
    \draw [semithick, postaction={decorate}] (-1,0) -- (0,0);
    \draw [semithick, postaction={decorate}] (0,0) to [out=0,in=180] (1,-.5);
    \fill [red,opacity=.5] (-1,0) -- (-.3,0) -- (-1,-.2);
    \fill [red,opacity=.5]  (-.3,0) -- (0,0) -- (0,.2);
    \fill [red,opacity=.5] (0,0) to [out=0,in=180] (1,.5) -- (1,.7) to [out=180,in=0] (0,.2);
    \fill [red,opacity=.5] (0,0) to [out=0,in=180] (1,-.5) -- (1,-.3) to [out=180,in=0] (0,.2);
    \node at (-.3,0) {$\circ$};
    \node at (-.3,-.3) {\tiny $+1$};
  \end{tikzpicture}
  \quad,\quad
  \begin{tikzpicture}[anchorbase,rotate=-90,xscale=-1]
    \draw [semithick, postaction={decorate}] (0,0) to [out=0,in=180] (.3,.3) to [out=0,in=180] (1,-.5);
    \draw [white, line width=3] (0,0) to [out=0,in=180] (.3,-.3) to [out=0,in=180] (1,.5);
    \draw [semithick, postaction={decorate}] (-1,0) -- (0,0);
    \draw [semithick, postaction={decorate}] (0,0) to [out=0,in=180] (.3,-.3) to [out=0,in=180] (1,.5);
    \fill [red,opacity=.5] (-1,0) -- (0,0) -- (0,-.2) -- (-1,-.2);
    \fill [red,opacity=.5] (0,0) to [out=0,in=180] (.3,.3) to [out=0,in=180] (1,-.5) -- (1,-.3) to [out=180,in=-60] (.6,.2) to [out=120,in=0] (.3,.2) to [out=180,in=0] (0,-.2);
    \fill [red, opacity=.5] (0,0) to [out=0,in=180] (.3,-.3) to [out=0,in=180] (1,.5) -- (1,.7) to [out=180,in=60] (.6,.2) to [out=-120,in=0] (.3,-.4) to [out=180,in=0] (0,-.2);
   \node at (.5,-.23) {$\bullet$};
    \node at (.5,-.5) {\tiny $+1$};
    \node at (.5,.23) {$\bullet$};
    \node at (.5,.45) {\tiny $+1$};
      \end{tikzpicture}
  \quad \leftrightarrow \quad
  \begin{tikzpicture}[anchorbase,rotate=-90,xscale=-1]
    \draw [semithick, postaction={decorate}] (0,0) to [out=0,in=180] (1,.5);
    \draw [semithick, postaction={decorate}] (-1,0) -- (0,0);
    \draw [semithick, postaction={decorate}] (0,0) to [out=0,in=180] (1,-.5);
    \fill [red,opacity=.5] (-1,0) -- (-.3,0) -- (-1,-.2);
    \fill [red,opacity=.5]  (-.3,0) -- (0,0) -- (0,.2);
    \fill [red,opacity=.5] (0,0) to [out=0,in=180] (1,.5) -- (1,.7) to [out=180,in=0] (0,.2);
    \fill [red,opacity=.5] (0,0) to [out=0,in=180] (1,-.5) -- (1,-.3) to [out=180,in=0] (0,.2);
    \node at (-.3,0) {$\bullet$};
    \node at (-.3,-.3) {\tiny $+1$};
  \end{tikzpicture}
\]

\[
  \begin{tikzpicture}[anchorbase,rotate=90]
    \draw [semithick, postaction={decorate}] (0,0) to [out=0,in=180] (.3,.3) to [out=0,in=180] (1,-.5);
    \draw [white, line width=3] (0,0) to [out=0,in=180] (.3,-.3) to [out=0,in=180] (1,.5);
    \draw [semithick, postaction={decorate}] (-1,0) -- (0,0);
    \draw [semithick, postaction={decorate}] (0,0) to [out=0,in=180] (.3,-.3) to [out=0,in=180] (1,.5);
    \fill [red,opacity=.5] (-1,0) -- (0,0) -- (0,-.2) -- (-1,-.2);
    \fill [red,opacity=.5] (0,0) to [out=0,in=180] (.3,.3) to [out=0,in=180] (1,-.5) -- (1,-.3) to [out=180,in=-60] (.6,.2) to [out=120,in=0] (.3,.2) to [out=180,in=0] (0,-.2);
    \fill [red, opacity=.5] (0,0) to [out=0,in=180] (.3,-.3) to [out=0,in=180] (1,.5) -- (1,.7) to [out=180,in=60] (.6,.2) to [out=-120,in=0] (.3,-.4) to [out=180,in=0] (0,-.2);
   \node at (.5,-.23) {$\bullet$};
    \node at (.5,-.5) {\tiny $-1$};
    \node at (.5,.23) {$\bullet$};
    \node at (.5,.45) {\tiny $-1$};
      \end{tikzpicture}
  \quad \leftrightarrow \quad
  \begin{tikzpicture}[anchorbase,rotate=90]
    \draw [semithick, postaction={decorate}] (0,0) to [out=0,in=180] (1,.5);
    \draw [semithick, postaction={decorate}] (-1,0) -- (0,0);
    \draw [semithick, postaction={decorate}] (0,0) to [out=0,in=180] (1,-.5);
    \fill [red,opacity=.5] (-1,0) -- (-.3,0) -- (-1,-.2);
    \fill [red,opacity=.5]  (-.3,0) -- (0,0) -- (0,.2);
    \fill [red,opacity=.5] (0,0) to [out=0,in=180] (1,.5) -- (1,.7) to [out=180,in=0] (0,.2);
    \fill [red,opacity=.5] (0,0) to [out=0,in=180] (1,-.5) -- (1,-.3) to [out=180,in=0] (0,.2);
    \node at (-.3,0) {$\bullet$};
    \node at (-.3,-.3) {\tiny $-1$};
  \end{tikzpicture}
  \quad,\quad
  \begin{tikzpicture}[anchorbase,rotate=-90,xscale=-1]
    \draw [semithick, postaction={decorate}] (0,0) to [out=0,in=180] (.3,-.3) to [out=0,in=180] (1,.5);
    \draw [white, line width=3] (0,0) to [out=0,in=180] (.3,.3) to [out=0,in=180] (1,-.5);
    \draw [semithick, postaction={decorate}] (-1,0) -- (0,0);
    \draw [semithick, postaction={decorate}] (0,0) to [out=0,in=180] (.3,.3) to [out=0,in=180] (1,-.5);
    \fill [red,opacity=.5] (-1,0) -- (0,0) -- (0,-.2) -- (-1,-.2);
    \fill [red,opacity=.5] (0,0) to [out=0,in=180] (.3,.3) to [out=0,in=180] (1,-.5) -- (1,-.3) to [out=180,in=-60] (.6,.2) to [out=120,in=0] (.3,.2) to [out=180,in=0] (0,-.2);
    \fill [red, opacity=.5] (0,0) to [out=0,in=180] (.3,-.3) to [out=0,in=180] (1,.5) -- (1,.7) to [out=180,in=60] (.6,.2) to [out=-120,in=0] (.3,-.4) to [out=180,in=0] (0,-.2);
   \node at (.5,-.23) {$\circ$};
    \node at (.5,-.5) {\tiny $-1$};
    \node at (.5,.23) {$\circ$};
    \node at (.5,.45) {\tiny $-1$};
      \end{tikzpicture}
  \quad \leftrightarrow \quad
  \begin{tikzpicture}[anchorbase,rotate=-90,xscale=-1]
    \draw [semithick, postaction={decorate}] (0,0) to [out=0,in=180] (1,.5);
    \draw [semithick, postaction={decorate}] (-1,0) -- (0,0);
    \draw [semithick, postaction={decorate}] (0,0) to [out=0,in=180] (1,-.5);
    \fill [red,opacity=.5] (-1,0) -- (-.3,0) -- (-1,-.2);
    \fill [red,opacity=.5]  (-.3,0) -- (0,0) -- (0,.2);
    \fill [red,opacity=.5] (0,0) to [out=0,in=180] (1,.5) -- (1,.7) to [out=180,in=0] (0,.2);
    \fill [red,opacity=.5] (0,0) to [out=0,in=180] (1,-.5) -- (1,-.3) to [out=180,in=0] (0,.2);
    \node at (-.3,0) {$\circ$};
    \node at (-.3,-.3) {\tiny $-1$};
  \end{tikzpicture}
\]

To get the above generators, one has to assume that we are away from the first two bullet points from the list on page~\pageref{itemize:F1}, and also that $\frac{\partial^2f}{\partial u \partial w}$ and $\frac{\partial^2 f}{\partial v \partial w}$ have non-zero $y$ coordinates. To give a bit more detail, one can reduce to the preferred situation where:
\[
  \frac{\partial f}{\partial u}=\begin{pmatrix} a_1 \\ 0 \\ a_3 \\ 0 \end{pmatrix}, \quad a_1>0,\quad  \frac{\partial f}{\partial v}=\begin{pmatrix}0 \\ b_2 \\ b_3 \\ b_4 \end{pmatrix}, \quad b_4\neq 0,\quad
  \frac{\partial f}{\partial w}=\begin{pmatrix}  0 \\ 0 \\ c_3 \\ 0 \end{pmatrix}
\]
Consider our usual notations that coordinates in $\partial u^2$ are denoted $d_i$, in $\partial u\partial v$ $e_i$, in $\partial u\partial w$ $f_i$, in $\partial v^2 $ $g_i$, in $\partial v \partial w$ $h_i$. Then on a typical slice, the framing grows parallel to a line generated by $\begin{pmatrix}f_1 \\ f_2 \\ f_3 \\ f_4\end{pmatrix}$. The framing vector will be on one or the other side of the tangent plane, the precise side being determined by the following determinant (at $(u_0+\epsilon_1,v_0,w_0)$):
\[
  \begin{vmatrix}
    f_1\epsilon_1 & a_1 & e_1\epsilon_1 \\
    f_2\epsilon_1 & d_2 \epsilon_1 & b_2 \\
    f_4\epsilon_1 & d_4 \epsilon_1 & b_4
    \end{vmatrix}
  \]
  One gets a polynomial in $\epsilon_1$ of the form $\epsilon_1(f_4a_1b_2-f_2a_1b_4+B\epsilon_1+C\epsilon_1^2)$. Provided $f_4a_1b_2-f_2a_1b_4\neq 0$, the sign changes when $\epsilon_1$ passes zero.

  Similarly, following the seam, one gets at $(u_0,v_0+\epsilon_2,w_0)$:
\[
  \begin{vmatrix}
    h_1\epsilon_2 & a_1 & g_1\epsilon_2 \\
    h_2\epsilon_2 & e_2 \epsilon_2 & b_2 \\
    h_4\epsilon_2 & e_4 \epsilon_2 & b_4
    \end{vmatrix}
  \]
We get that the sign will change if $h_4a_1b_2-h_2a_1b_4\neq 0$.
This ensures that the line of framing change is generic, in the sense that it goes from the bottom on one side of the seam to the top on the other side.

Through time, there could be more degeneration from the list on page~\pageref{itemize:F1}, or the two conditions listed above could fail. In each case we get codimensions $1$ conditions, so we only have to consider one case at a time. Notice that the second bullet point is not compatible with the third one, as we have assumed that the $3$ derivatives form a rank $3$ matrix.

In case one has $f_4a_1b_2-f_2a_1b_4=0$, then the framing line stays on one side of the seam. We get as movie move the invertibility of a half-twist passing through a vertex:

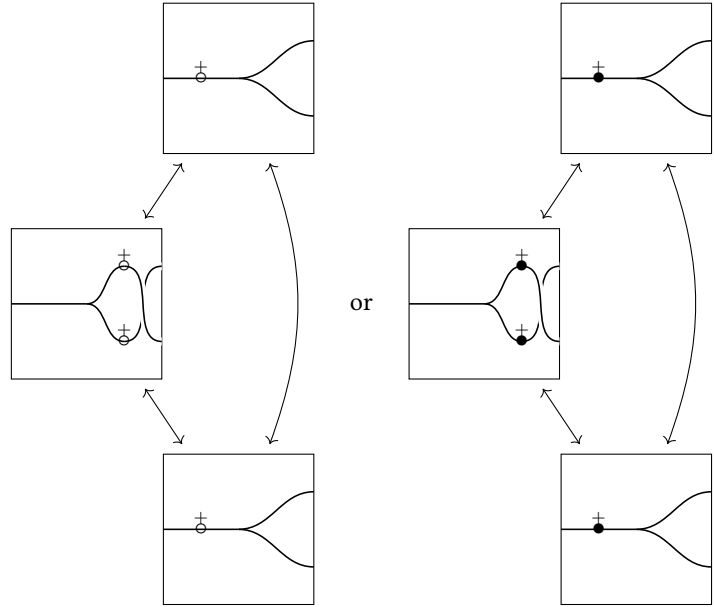
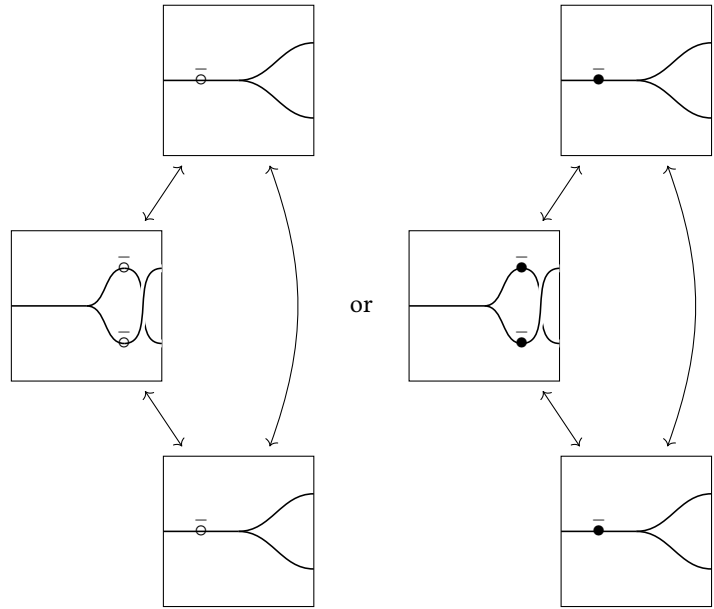
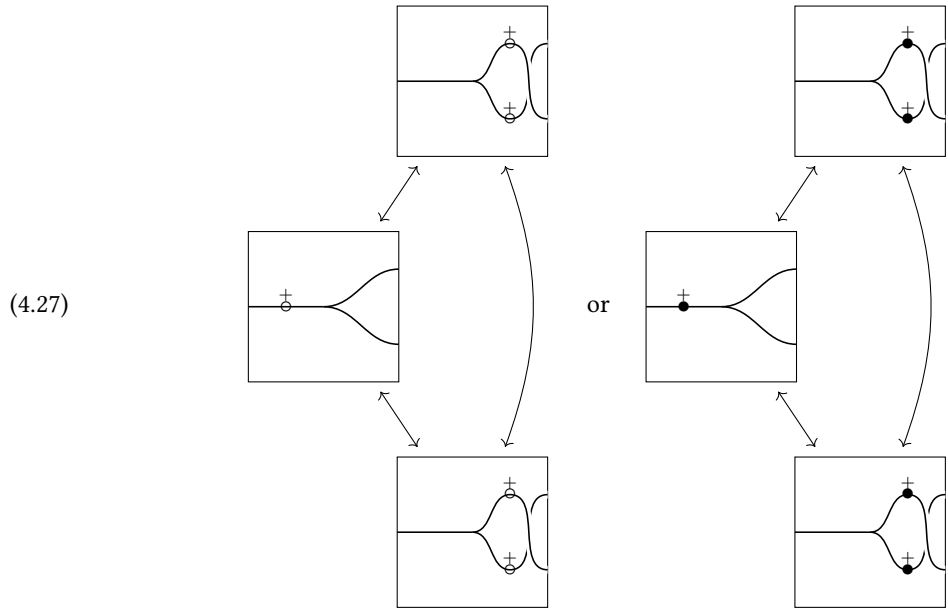
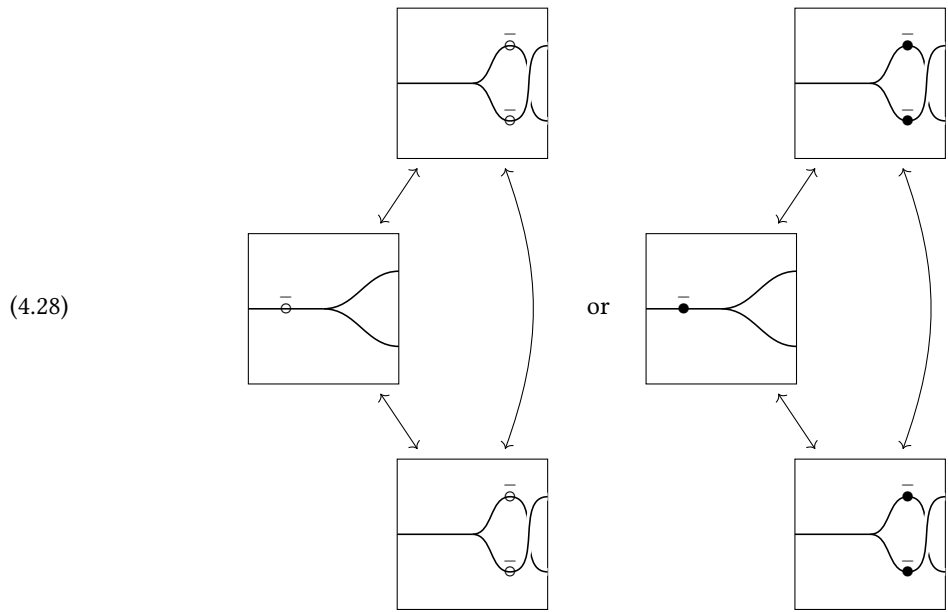
\begin{equation}

\end{equation}

In case of a framing change happening at a Morse extremum on the $1$-skeleton, corresponding to  $h_4a_1b_2-h_2a_1b_4=0$, one gets the following move:

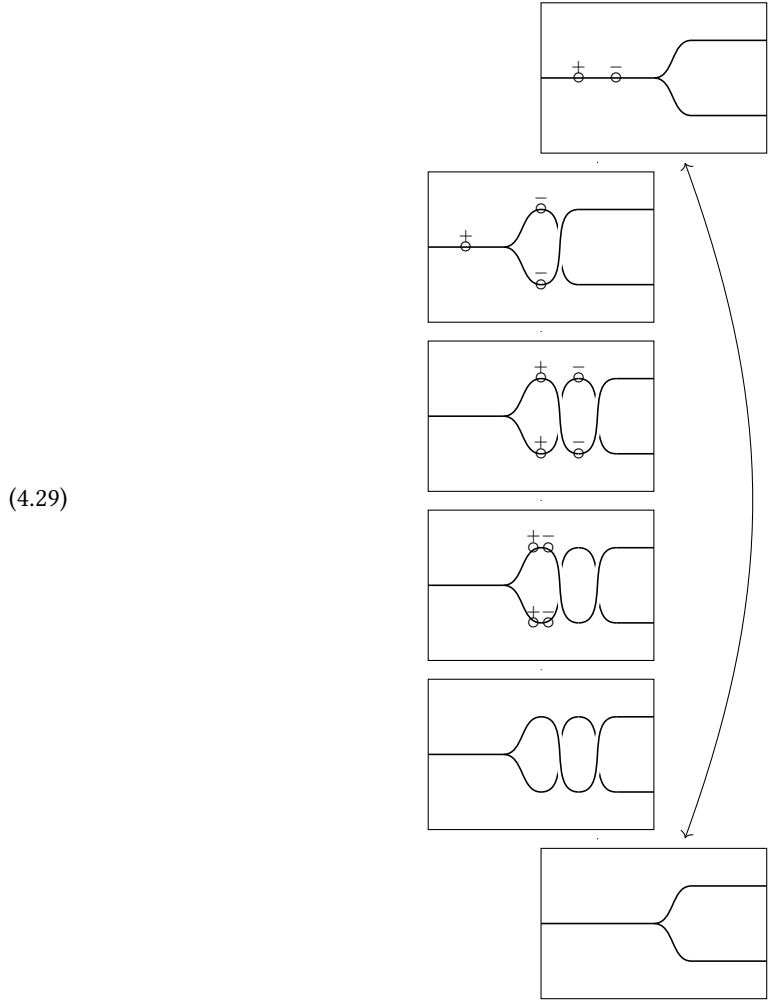
\begin{equation}
%
  };
  \draw [-] (A) -- (B1);
  \draw [-] (B1) -- (C1);
  \draw [-] (C1) -- (D1);
  \draw [-] (D1) -- (E1);
  \draw [-] (E1) -- (F);
  \draw [<->] (A) to [out=-70,in=70] (F);
\end{tikzpicture}
\end{equation}

Finally, if $\frac{\partial f}{\partial v}$ has zero $s$-coordinate, one gets back the moves~\eqref{eq:crzcup} to~\eqref{eq:crzip}.

%%%%%%%%%%%%%%%%%%%%%%%%%%%%%%%%
\subsection{Neighborhood of a point on $F^0$}
%%%%%%%%%%%%%%%%%%%%%%%%%%%%%%%%%

We now consider the neighborhood of a point $M=(u_0,v_0,w_0)\in F^0$. Restricting to $F^0$ is already a codimension $3$ condition, so at the level of foams no other degeneration will generically occur. Assume as in picture~\eqref{eq:infl6val} that the local chart in $u,v,w$ is such that the tangent vector along both seams at $M$ is parallel to the $v$ direction.

Locally on the vertical seam, one has:
\[
  f(u_0,v_0+\epsilon_2,w_0)=f(u_0,v_0,w_0)+\begin{pmatrix}b_1 \epsilon_2 \\ b_2 \epsilon_2 \\ b_3\epsilon_2 \\ b_4\epsilon_2\end{pmatrix}+o(\epsilon_2)
\]

Assuming $b_4\neq 0$, the image of the seam will have a non-trivial component in the $s$-direction. In order to have the expected generator, one wishes that $\begin{pmatrix} 0 \\ 0 \\ 1 \\ 0\end{pmatrix}\notin \Span_\R\left(\frac{\partial f}{\partial u},\frac{\partial f}{\partial v}\right)$. Assuming the converse is a codimension $2$ condition, so this will never happen, even through time.

Up to change of coordinates $u,v,w$ and rotation in the $x,y$ plane, one can thus assume that:
\[
  \frac{\partial f}{\partial u}=\begin{pmatrix} a_1 \\ 0 \\ a_3 \\ 0 \end{pmatrix}, \;a_1>0,\quad
  \frac{\partial f}{\partial v}=\begin{pmatrix} 0 \\ b_2 \\ b_3 \\ b_4 \end{pmatrix},
  \frac{\partial f}{\partial w}=\begin{pmatrix} 0 \\ c_2 \\ c_3 \\0 \end{pmatrix}
\]

Finally, one can assume that $c_2\neq 0$ (which also ensures that the framing is locally constant and non-vertical). This gives the classical foam generator:

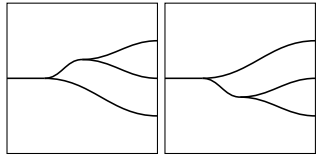
\[
  \begin{tikzpicture}[anchorbase]
    \node at (0,0)
    {
      \begin{tikzpicture}[anchorbase,scale=.5]
        % rectangle
        \draw (0,0) rectangle (4,4);
        % Web
        \draw [semithick] (0,2) -- (1,2);
        \draw [semithick] (1,2) to [out=0,in=180] (2,2.5);
        \draw [semithick] (2,2.5) to [out=0,in=180] (4,3);
        \draw [semithick] (2,2.5) to [out=0,in=180] (4,2);
        \draw [semithick] (1,2) to [out=0,in=180] (4,1);
      \end{tikzpicture}
    };
    \node at (2.1,0)
    {
      \begin{tikzpicture}[anchorbase,scale=.5]
        % rectangle
        \draw (0,0) rectangle (4,4);
        % Web
        \draw [semithick] (0,2) -- (1,2);
        \draw [semithick] (1,2) to [out=0,in=180] (2,1.5);
        \draw [semithick] (2,1.5) to [out=0,in=180] (4,1);
        \draw [semithick] (2,1.5) to [out=0,in=180] (4,2);
        \draw [semithick] (1,2) to [out=0,in=180] (4,3);
      \end{tikzpicture}
    };
  \end{tikzpicture}
\]

Through time, we have to analyze the codimension $1$ conditions that $b_4=0$ or $c_2=0$. For codimension reasons, they generically won't happen together.

Let us start with $b_4=0$. Then following the vertical seam, one reads:
\[
  f(u_0,v_0+\epsilon_2,w_0)=f(u_0,v_0,w_0)+ \begin{pmatrix} \epsilon_2b_1 \\ \epsilon_2 b_2 \\ \epsilon_2 b_3 \\ \epsilon_2^2 g_4 \end{pmatrix}
\]
For codimension reason, one may assume that $g_4\neq 0$. One thus sees a bending seam. At $t_0+\delta$ though, one picks up a contribution of the $s$-coordinate of $\frac{\partial^2 f}{\partial v \partial t}$, that can be assumed to be non-zero. Depending on the sign of $\delta$, the Morse extremum and the 6-valent point will split in one or the other direction. One reads the following movie-moves:

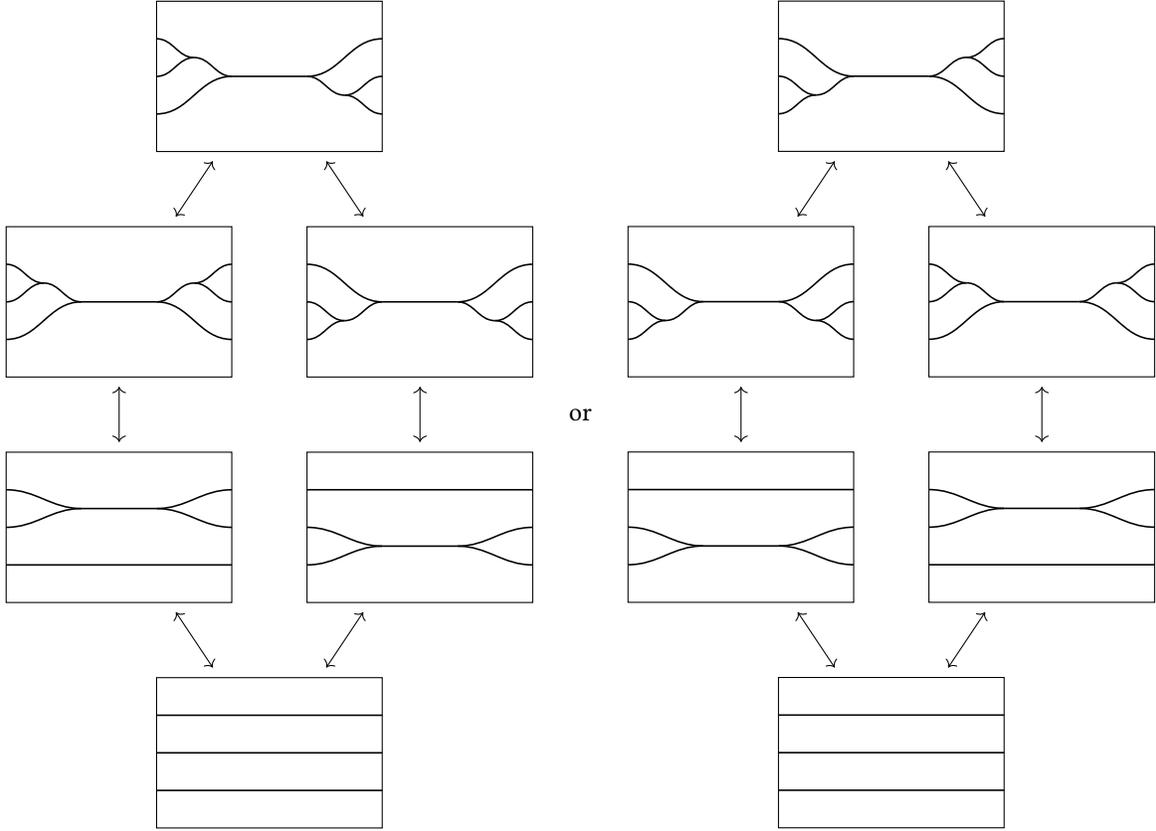
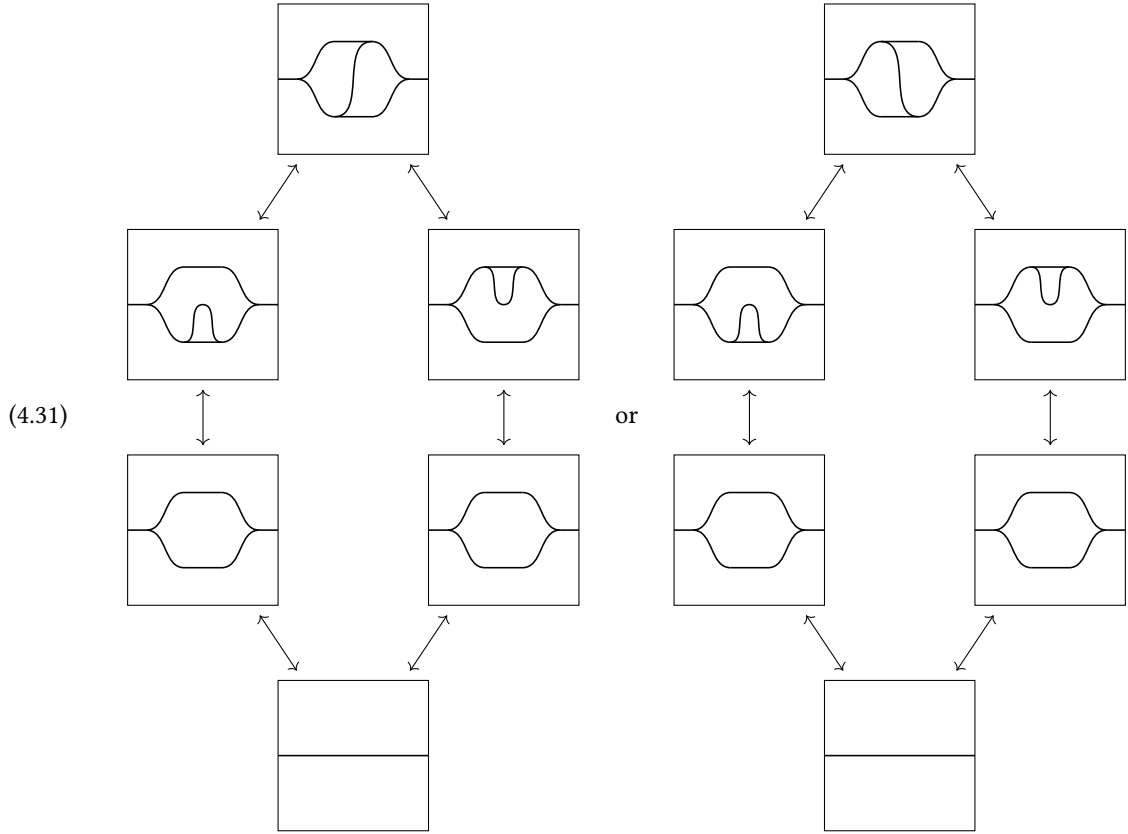
\begin{equation}
\begin{tikzpicture}[anchorbase]
  \node at (0,0) {
\begin{tikzpicture}[decoration={markings, mark=at position 0.6 with {\arrow{>}};  },fill opacity=.2,scale=.15]
\end{tikzpicture}
  };
  \node (A) at (0,0) {
    \begin{tikzpicture}[anchorbase,scale=.5]
      % Square
      \draw (0,0) rectangle (6,4);
      % Strands
      \draw [semithick] (0,3) to [out=0,in=180] (1,2.5);
      \draw [semithick] (0,2) to [out=0,in=180] (1,2.5);
      \draw [semithick] (1,2.5) to [out=0,in=180] (2,2);
      \draw [semithick] (0,1) to [out=0,in=180] (2,2);
      \draw [semithick] (2,2) -- (4,2);
      \draw [semithick] (4,2) to  [out=0,in=180] (5,1.5);
      \draw [semithick] (5,1.5) to [out=0,in=180] (6,1);
      \draw [semithick] (5,1.5) to [out=0,in=180] (6,2);
      \draw [semithick] (4,2) to [out=0,in=180] (6,3);
    \end{tikzpicture}
  };
  \node (B1) at (-2,-3) {
    \begin{tikzpicture}[anchorbase,scale=.5]
      % Square
      \draw (0,0) rectangle (6,4);
      % Strands
      \draw [semithick] (0,3) to [out=0,in=180] (1,2.5);
      \draw [semithick] (0,2) to [out=0,in=180] (1,2.5);
      \draw [semithick] (1,2.5) to [out=0,in=180] (2,2);
      \draw [semithick] (0,1) to [out=0,in=180] (2,2);
      \draw [semithick] (2,2) -- (4,2);
      \draw [semithick] (4,2) to  [out=0,in=180] (5,2.5);
      \draw [semithick] (5,2.5) to [out=0,in=180] (6,3);
      \draw [semithick] (5,2.5) to [out=0,in=180] (6,2);
      \draw [semithick] (4,2) to [out=0,in=180] (6,1);
     \end{tikzpicture}
  };
  \node (B2) at (2,-3) {
    \begin{tikzpicture}[anchorbase,scale=.5]
      % Square
      \draw (0,0) rectangle (6,4);
      % Strands
      \draw [semithick] (0,1) to [out=0,in=180] (1,1.5);
      \draw [semithick] (0,2) to [out=0,in=180] (1,1.5);
      \draw [semithick] (1,1.5) to [out=0,in=180] (2,2);
      \draw [semithick] (0,3) to [out=0,in=180] (2,2);
      \draw [semithick] (2,2) -- (4,2);
      \draw [semithick] (4,2) to  [out=0,in=180] (5,1.5);
      \draw [semithick] (5,1.5) to [out=0,in=180] (6,1);
      \draw [semithick] (5,1.5) to [out=0,in=180] (6,2);
      \draw [semithick] (4,2) to [out=0,in=180] (6,3);
    \end{tikzpicture}
  };
  \node (C1) at (-2,-6) {
    \begin{tikzpicture}[anchorbase,scale=.5]
      % Square
      \draw (0,0) rectangle (6,4);
      % Strands
      \draw [semithick] (0,3) to [out=0,in=180] (2,2.5);
      \draw [semithick] (0,2) to [out=0,in=180] (2,2.5);
      \draw [semithick] (2,2.5) -- (4,2.5);
      \draw [semithick] (0,1) -- (6,1);
      \draw [semithick] (4,2.5) to  [out=0,in=180] (6,3);
      \draw [semithick] (4,2.5) to [out=0,in=180] (6,2);
     \end{tikzpicture}
  };
  \node (C2) at (2,-6) {
    \begin{tikzpicture}[anchorbase,scale=.5]
      % Square
      \draw (0,0) rectangle (6,4);
      % Strands
      \draw [semithick] (0,1) to [out=0,in=180] (2,1.5);
      \draw [semithick] (0,2) to [out=0,in=180] (2,1.5);
      \draw [semithick] (2,1.5) -- (4,1.5);
      \draw [semithick] (0,3) -- (6,3);
      \draw [semithick] (4,1.5) to  [out=0,in=180] (6,1);
      \draw [semithick] (4,1.5) to [out=0,in=180] (6,2);
     \end{tikzpicture}
  };
  \node (D) at (0,-9) {
    \begin{tikzpicture}[anchorbase,scale=.5]
      % Square
      \draw (0,0) rectangle (6,4);
      % Strands
      \draw [semithick] (0,1) -- (6,1);
      \draw [semithick] (0,2) -- (6,2);
      \draw [semithick] (0,3) -- (6,3);
    \end{tikzpicture}
  };
  \draw [<->] (A) -- (B1);
  \draw [<->] (A) -- (B2);
  \draw [<->] (B1) -- (C1);
  \draw [<->] (B2) -- (C2);
  \draw [<->] (C1) -- (D);
  \draw [<->] (C2) -- (D);
\end{tikzpicture}
\quad \text{or}\quad
\begin{tikzpicture}[anchorbase]
  \node at (0,0) {
\begin{tikzpicture}[decoration={markings, mark=at position 0.6 with {\arrow{>}};  },fill opacity=.2,scale=.15]
\end{tikzpicture}
  };
  \node (A) at (0,0) {
    \begin{tikzpicture}[anchorbase,scale=.5,yscale=-1]
      % Square
      \draw (0,0) rectangle (6,4);
      % Strands
      \draw [semithick] (0,3) to [out=0,in=180] (1,2.5);
      \draw [semithick] (0,2) to [out=0,in=180] (1,2.5);
      \draw [semithick] (1,2.5) to [out=0,in=180] (2,2);
      \draw [semithick] (0,1) to [out=0,in=180] (2,2);
      \draw [semithick] (2,2) -- (4,2);
      \draw [semithick] (4,2) to  [out=0,in=180] (5,1.5);
      \draw [semithick] (5,1.5) to [out=0,in=180] (6,1);
      \draw [semithick] (5,1.5) to [out=0,in=180] (6,2);
      \draw [semithick] (4,2) to [out=0,in=180] (6,3);
    \end{tikzpicture}
  };
  \node (B1) at (-2,-3) {
    \begin{tikzpicture}[anchorbase,scale=.5,yscale=-1]
      % Square
      \draw (0,0) rectangle (6,4);
      % Strands
      \draw [semithick] (0,3) to [out=0,in=180] (1,2.5);
      \draw [semithick] (0,2) to [out=0,in=180] (1,2.5);
      \draw [semithick] (1,2.5) to [out=0,in=180] (2,2);
      \draw [semithick] (0,1) to [out=0,in=180] (2,2);
      \draw [semithick] (2,2) -- (4,2);
      \draw [semithick] (4,2) to  [out=0,in=180] (5,2.5);
      \draw [semithick] (5,2.5) to [out=0,in=180] (6,3);
      \draw [semithick] (5,2.5) to [out=0,in=180] (6,2);
      \draw [semithick] (4,2) to [out=0,in=180] (6,1);
     \end{tikzpicture}
  };
  \node (B2) at (2,-3) {
    \begin{tikzpicture}[anchorbase,scale=.5,yscale=-1]
      % Square
      \draw (0,0) rectangle (6,4);
      % Strands
      \draw [semithick] (0,1) to [out=0,in=180] (1,1.5);
      \draw [semithick] (0,2) to [out=0,in=180] (1,1.5);
      \draw [semithick] (1,1.5) to [out=0,in=180] (2,2);
      \draw [semithick] (0,3) to [out=0,in=180] (2,2);
      \draw [semithick] (2,2) -- (4,2);
      \draw [semithick] (4,2) to  [out=0,in=180] (5,1.5);
      \draw [semithick] (5,1.5) to [out=0,in=180] (6,1);
      \draw [semithick] (5,1.5) to [out=0,in=180] (6,2);
      \draw [semithick] (4,2) to [out=0,in=180] (6,3);
    \end{tikzpicture}
  };
  \node (C1) at (-2,-6) {
    \begin{tikzpicture}[anchorbase,scale=.5,yscale=-1]
      % Square
      \draw (0,0) rectangle (6,4);
      % Strands
      \draw [semithick] (0,3) to [out=0,in=180] (2,2.5);
      \draw [semithick] (0,2) to [out=0,in=180] (2,2.5);
      \draw [semithick] (2,2.5) -- (4,2.5);
      \draw [semithick] (0,1) -- (6,1);
      \draw [semithick] (4,2.5) to  [out=0,in=180] (6,3);
      \draw [semithick] (4,2.5) to [out=0,in=180] (6,2);
     \end{tikzpicture}
  };
  \node (C2) at (2,-6) {
    \begin{tikzpicture}[anchorbase,scale=.5,yscale=-1]
      % Square
      \draw (0,0) rectangle (6,4);
      % Strands
      \draw [semithick] (0,1) to [out=0,in=180] (2,1.5);
      \draw [semithick] (0,2) to [out=0,in=180] (2,1.5);
      \draw [semithick] (2,1.5) -- (4,1.5);
      \draw [semithick] (0,3) -- (6,3);
      \draw [semithick] (4,1.5) to  [out=0,in=180] (6,1);
      \draw [semithick] (4,1.5) to [out=0,in=180] (6,2);
     \end{tikzpicture}
  };
  \node (D) at (0,-9) {
    \begin{tikzpicture}[anchorbase,scale=.5,yscale=-1]
      % Square
      \draw (0,0) rectangle (6,4);
      % Strands
      \draw [semithick] (0,1) -- (6,1);
      \draw [semithick] (0,2) -- (6,2);
      \draw [semithick] (0,3) -- (6,3);
    \end{tikzpicture}
  };
  \draw [<->] (A) -- (B1);
  \draw [<->] (A) -- (B2);
  \draw [<->] (B1) -- (C1);
  \draw [<->] (B2) -- (C2);
  \draw [<->] (C1) -- (D);
  \draw [<->] (C2) -- (D);
\end{tikzpicture}
\end{equation}

\begin{equation}
  \begin{tikzpicture}[anchorbase]
  \node at (0,0) {
\begin{tikzpicture}[decoration={markings, mark=at position 0.6 with {\arrow{>}};  },fill opacity=.2,scale=.15]
\end{tikzpicture}
  };
  \node (A) at (0,0) {
    \begin{tikzpicture}[anchorbase,scale=.5]
      % Square
      \draw (0,0) rectangle (4,4);
      % Strands
      \draw [semithick] (0,2) -- (.5,2);
      \draw [semithick] (.5,2) to [out=0,in=180] (1.5,1);
      \draw [semithick] (.5,2) to [out=0,in=180] (1.5,3) -- (2.5,3);
      \draw [semithick] (1.5,1) to [out=0,in=180] (2.5,3);
      \draw [semithick] (1.5,1) -- (2.5,1) to [out=0,in=180] (3.5,2);
      \draw [semithick] (2.5,3) to [out=0,in=180] (3.5,2);
      \draw [semithick] (3.5,2) -- (4,2);
    \end{tikzpicture}
  };
  \node (B1) at (-2,-3) {
    \begin{tikzpicture}[anchorbase,scale=.5]
      % Square
      \draw (0,0) rectangle (4,4);
      % Strands
      \draw [semithick] (0,2) -- (.5,2);
      \draw [semithick] (.5,2) to [out=0,in=180] (1.5,1);
      \draw [semithick] (.5,2) to [out=0,in=180] (1.5,3) -- (2.5,3);
      \draw [semithick] (1.5,1) to [out=0,in=180] (2,2) to [out=0,in=180] (2.5,1);
      \draw [semithick] (1.5,1) -- (2.5,1) to [out=0,in=180] (3.5,2);
      \draw [semithick] (2.5,3) to [out=0,in=180] (3.5,2);
      \draw [semithick] (3.5,2) -- (4,2);
     \end{tikzpicture}
  };
  \node (B2) at (2,-3) {
    \begin{tikzpicture}[anchorbase,scale=.5]
      % Square
      \draw (0,0) rectangle (4,4);
      % Strands
      \draw [semithick] (0,2) -- (.5,2);
      \draw [semithick] (.5,2) to [out=0,in=180] (1.5,1);
      \draw [semithick] (.5,2) to [out=0,in=180] (1.5,3) -- (2.5,3);
      \draw [semithick] (1.5,3) to [out=0,in=180] (2,2) to [out=0,in=180] (2.5,3);
      \draw [semithick] (1.5,1) -- (2.5,1) to [out=0,in=180] (3.5,2);
      \draw [semithick] (2.5,3) to [out=0,in=180] (3.5,2);
      \draw [semithick] (3.5,2) -- (4,2);
    \end{tikzpicture}
  };
  \node (C1) at (-2,-6) {
    \begin{tikzpicture}[anchorbase,scale=.5]
      % Square
      \draw (0,0) rectangle (4,4);
      % Strands
      \draw [semithick] (0,2) -- (.5,2);
      \draw [semithick] (.5,2) to [out=0,in=180] (1.5,1);
      \draw [semithick] (.5,2) to [out=0,in=180] (1.5,3) -- (2.5,3);
      \draw [semithick] (1.5,1) -- (2.5,1) to [out=0,in=180] (3.5,2);
      \draw [semithick] (2.5,3) to [out=0,in=180] (3.5,2);
      \draw [semithick] (3.5,2) -- (4,2);
     \end{tikzpicture}
  };
  \node (C2) at (2,-6) {
    \begin{tikzpicture}[anchorbase,scale=.5]
      % Square
      \draw (0,0) rectangle (4,4);
      % Strands
      \draw [semithick] (0,2) -- (.5,2);
      \draw [semithick] (.5,2) to [out=0,in=180] (1.5,1);
      \draw [semithick] (.5,2) to [out=0,in=180] (1.5,3) -- (2.5,3);
      \draw [semithick] (1.5,1) -- (2.5,1) to [out=0,in=180] (3.5,2);
      \draw [semithick] (2.5,3) to [out=0,in=180] (3.5,2);
      \draw [semithick] (3.5,2) -- (4,2);
     \end{tikzpicture}
  };
  \node (D) at (0,-9) {
    \begin{tikzpicture}[anchorbase,scale=.5]
      % Square
      \draw (0,0) rectangle (4,4);
      % Strands
      \draw [semithick] (0,2) -- (4,2);
    \end{tikzpicture}
  };
  \draw [<->] (A) -- (B1);
  \draw [<->] (A) -- (B2);
  \draw [<->] (B1) -- (C1);
  \draw [<->] (B2) -- (C2);
  \draw [<->] (C1) -- (D);
  \draw [<->] (C2) -- (D);
\end{tikzpicture}
\quad \text{or} \quad
  \begin{tikzpicture}[anchorbase]
  \node at (0,0) {
\begin{tikzpicture}[decoration={markings, mark=at position 0.6 with {\arrow{>}};  },fill opacity=.2,scale=.15]
\end{tikzpicture}
  };
  \node (A) at (0,0) {
    \begin{tikzpicture}[anchorbase,scale=.5]
      % Square
      \draw (0,0) rectangle (4,4);
      % Strands
      \draw [semithick] (0,2) -- (.5,2);
      \draw [semithick] (.5,2) to [out=0,in=180] (1.5,1);
      \draw [semithick] (.5,2) to [out=0,in=180] (1.5,3) -- (2.5,3);
      \draw [semithick] (1.5,3) to [out=0,in=180] (2.5,1);
      \draw [semithick] (1.5,1) -- (2.5,1) to [out=0,in=180] (3.5,2);
      \draw [semithick] (2.5,3) to [out=0,in=180] (3.5,2);
      \draw [semithick] (3.5,2) -- (4,2);
    \end{tikzpicture}
  };
  \node (B1) at (-2,-3) {
    \begin{tikzpicture}[anchorbase,scale=.5]
      % Square
      \draw (0,0) rectangle (4,4);
      % Strands
      \draw [semithick] (0,2) -- (.5,2);
      \draw [semithick] (.5,2) to [out=0,in=180] (1.5,1);
      \draw [semithick] (.5,2) to [out=0,in=180] (1.5,3) -- (2.5,3);
      \draw [semithick] (1.5,1) to [out=0,in=180] (2,2) to [out=0,in=180] (2.5,1);
      \draw [semithick] (1.5,1) -- (2.5,1) to [out=0,in=180] (3.5,2);
      \draw [semithick] (2.5,3) to [out=0,in=180] (3.5,2);
      \draw [semithick] (3.5,2) -- (4,2);
     \end{tikzpicture}
  };
  \node (B2) at (2,-3) {
    \begin{tikzpicture}[anchorbase,scale=.5]
      % Square
      \draw (0,0) rectangle (4,4);
      % Strands
      \draw [semithick] (0,2) -- (.5,2);
      \draw [semithick] (.5,2) to [out=0,in=180] (1.5,1);
      \draw [semithick] (.5,2) to [out=0,in=180] (1.5,3) -- (2.5,3);
      \draw [semithick] (1.5,3) to [out=0,in=180] (2,2) to [out=0,in=180] (2.5,3);
      \draw [semithick] (1.5,1) -- (2.5,1) to [out=0,in=180] (3.5,2);
      \draw [semithick] (2.5,3) to [out=0,in=180] (3.5,2);
      \draw [semithick] (3.5,2) -- (4,2);
    \end{tikzpicture}
  };
  \node (C1) at (-2,-6) {
    \begin{tikzpicture}[anchorbase,scale=.5]
      % Square
      \draw (0,0) rectangle (4,4);
      % Strands
      \draw [semithick] (0,2) -- (.5,2);
      \draw [semithick] (.5,2) to [out=0,in=180] (1.5,1);
      \draw [semithick] (.5,2) to [out=0,in=180] (1.5,3) -- (2.5,3);
      \draw [semithick] (1.5,1) -- (2.5,1) to [out=0,in=180] (3.5,2);
      \draw [semithick] (2.5,3) to [out=0,in=180] (3.5,2);
      \draw [semithick] (3.5,2) -- (4,2);
     \end{tikzpicture}
  };
  \node (C2) at (2,-6) {
    \begin{tikzpicture}[anchorbase,scale=.5]
      % Square
      \draw (0,0) rectangle (4,4);
      % Strands
      \draw [semithick] (0,2) -- (.5,2);
      \draw [semithick] (.5,2) to [out=0,in=180] (1.5,1);
      \draw [semithick] (.5,2) to [out=0,in=180] (1.5,3) -- (2.5,3);
      \draw [semithick] (1.5,1) -- (2.5,1) to [out=0,in=180] (3.5,2);
      \draw [semithick] (2.5,3) to [out=0,in=180] (3.5,2);
      \draw [semithick] (3.5,2) -- (4,2);
     \end{tikzpicture}
  };
  \node (D) at (0,-9) {
    \begin{tikzpicture}[anchorbase,scale=.5]
      % Square
      \draw (0,0) rectangle (4,4);
      % Strands
      \draw [semithick] (0,2) -- (4,2);
    \end{tikzpicture}
  };
  \draw [<->] (A) -- (B1);
  \draw [<->] (A) -- (B2);
  \draw [<->] (B1) -- (C1);
  \draw [<->] (B2) -- (C2);
  \draw [<->] (C1) -- (D);
  \draw [<->] (C2) -- (D);
\end{tikzpicture}
\end{equation}

Let us now consider the case where $b_4\neq 0$ but $c_2=0$. Around $(u_0,v_0,w_0)$, the $y$-coordinate is controlled by:
\[
  b_2\epsilon_2+d_2\epsilon_1^2+f_2\epsilon_1\epsilon_3+h_2\epsilon_2\epsilon_3+i_2\epsilon_3^2
\]
The situation is similar to the one that lead to~\eqref{eq:twistFork}: we have either contradictory or parallel effects of $h_2\epsilon_2\epsilon_3$ and $f_2\epsilon_1\epsilon_3$ depending on the sign of $\epsilon_2$. One thus gets the following movie move (and analogs of it with other twist/crossing versions):

\begin{equation}
\begin{tikzpicture}[anchorbase,scale=.75]
  \node at (0,0) {
\begin{tikzpicture}[decoration={markings, mark=at position 0.6 with {\arrow{>}};  },fill opacity=.2,scale=.15]
%\draw[dashed] (25,32) rectangle (-25,-32);
%\node[opacity=1] at (0,11) {\bf{Sauron}};
%\node[opacity=1] at (0,-11) {};
\end{tikzpicture}
  };
  \node (A) at (0,0) {
    \begin{tikzpicture}[anchorbase,scale=.5]
      % Square
      \draw (0,0) rectangle (6,4);
      % Strands
      \draw [semithick] (0,2) -- (2,2);
      \draw [semithick] (2,2) to [out=0,in=180] (3,2.5);
      \draw [semithick] (2,2) to [out=0,in=180] (6,1);
      \draw [semithick] (3,2.5) to [out=0,in=180] (6,3);
      \draw [semithick] (3,2.5) to [out=0,in=180] (6,2);
      % Twists
      \node at (1,2) {$\circ$};
      \node at (1,2.3) {\tiny $+$};
    \end{tikzpicture}
  };
  \node (B1) at (-3,-4.5) {
    \begin{tikzpicture}[anchorbase,scale=.5]
      % Square
      \draw (0,0) rectangle (6,4);
      % Strands
      \draw [semithick] (0,2) -- (1.5,2);
      \draw [semithick] (1.5,2) to [out=0,in=180] (2,2.5);
      \draw [semithick] (1.5,2) to [out=0,in=180] (2,1.5);
      % Crossing
      \draw [white, line width=3] (2,1.5) to [out=0,in=180] (3,2.5);
      \draw [semithick] (2,1.5) to [out=0,in=180] (3,2.5);
      \draw [white, line width=3] (2,2.5) to [out=0,in=180] (3,1.5);
      \draw [semithick] (2,2.5) to [out=0,in=180] (3,1.5);
      % Free strand
      \draw [semithick] (3,1.5) to [out=0,in=180] (6,1);
      % Web
      \draw [semithick] (3,2.5) to [out=0,in=180] (6,3);
      \draw [semithick] (3,2.5) to [out=0,in=180] (6,2);
      % Twists
      \node at (2,2.5) {$\circ$};
      \node at (2,2.8) {\tiny $+$};
      \node at (2,1.5) {$\circ$};
      \node at (2,1.2) {\tiny $+$};
    \end{tikzpicture}
  };
  \node (B2) at (3,-3) {
    \begin{tikzpicture}[anchorbase,scale=.5]
      % Square
      \draw (0,0) rectangle (6,4);
      % Strands
      \draw [semithick] (0,2) -- (2,2);
      \draw [semithick] (2,2) to [out=0,in=180] (3,1.5);
      \draw [semithick] (2,2) to [out=0,in=180] (6,3);
      \draw [semithick] (3,1.5) to [out=0,in=180] (6,1);
      \draw [semithick] (3,1.5) to [out=0,in=180] (6,2);
      % Twists
      \node at (1,2) {$\circ$};
      \node at (1,2.3) {\tiny $+$};
    \end{tikzpicture}
  };
  \node (C1) at (-3,-7.5) {
    \begin{tikzpicture}[anchorbase,scale=.5]
      % Square
      \draw (0,0) rectangle (6,4);
      % Strands
      \draw [semithick] (0,2) -- (1.5,2);
      \draw [semithick] (1.5,2) to [out=0,in=180] (2,2.5) -- (3,2.5);
      \draw [semithick] (1.5,2) to [out=0,in=180] (2,1.5) -- (2.5,1.5);
      \draw [semithick] (2.5,1.5) to [out=0,in=180]  (3,2);
      \draw [semithick] (2.5,1.5) to [out=0,in=180] (3,1);
      % Crossing
      \draw [white, line width=3] (3,1) to [out=0,in=180] (4,2);
      \draw [semithick] (3,2) to [out=0,in=180] (4,3);
      \draw [white, line width=3] (3,1) to [out=0,in=180] (4,2);
      \draw [semithick] (3,1) to [out=0,in=180] (4,2);
      \draw [white, line width=3] (3,2.5) to [out=0,in=180] (4,1);
      \draw [semithick] (3,2.5) to [out=0,in=180] (4,1);
      % Free strand
      \draw [semithick] (4,1) to [out=0,in=180] (6,1);
      % Web strands
      \draw [semithick] (4,3) to [out=0,in=180] (6,3);
      \draw [semithick] (4,2) to [out=0,in=180] (6,2);
      % Twists
      \node at (2,2.5) {$\circ$};
      \node at (2,2.8) {\tiny $+$};
      \node at (2,1.5) {$\circ$};
      \node at (2,1.2) {\tiny $+$};
    \end{tikzpicture}
  };
  \node (C2) at (3,-6) {
    \begin{tikzpicture}[anchorbase,scale=.5]
      % Square
      \draw (0,0) rectangle (6,4);
      % Strands
      \draw [semithick] (0,2) -- (1.5,2);
      \draw [semithick] (1.5,2) to [out=0,in=180] (2,1.5);
      \draw [semithick] (1.5,2) to [out=0,in=180] (2,2.5);
      % Crossing
      \draw [white, line width=3] (2,1.5) to [out=0,in=180] (3,2.5);
      \draw [semithick] (2,1.5) to [out=0,in=180] (3,2.5);
      \draw [white, line width=3] (2,2.5) to [out=0,in=180] (3,1.5);
      \draw [semithick] (2,2.5) to [out=0,in=180] (3,1.5);
      % Free strand
      \draw [semithick] (3,2.5) to [out=0,in=180] (6,3);
      % Web
      \draw [semithick] (3,1.5) to [out=0,in=180] (6,1);
      \draw [semithick] (3,1.5) to [out=0,in=180] (6,2);
      % Twists
      \node at (2,1.5) {$\circ$};
      \node at (2,1.2) {\tiny $+$};
      \node at (2,2.5) {$\circ$};
      \node at (2,2.8) {\tiny $+$};
    \end{tikzpicture}
  };
  \node (D1) at (-3,-10.5) {
    \begin{tikzpicture}[anchorbase,scale=.5]
      % Square
      \draw (0,0) rectangle (6,4);
      % Strands
      \draw [semithick] (0,2) -- (1.5,2);
      \draw [semithick] (1.5,2) to [out=0,in=180] (2,3) -- (4,3);
      \draw [semithick] (1.5,2) to [out=0,in=180] (2.5,1.5);
      \draw [semithick] (2.5,1.5) to [out=0,in=180]  (3,2);
      \draw [semithick] (2.5,1.5) to [out=0,in=180] (3,1);
      % Crossings
      \draw [white, line width=3] (3,1) to [out=0,in=180] (4,2);
      \draw [semithick] (3,1) to [out=0,in=180] (4,2);
      \draw [white, line width=3] (3,2) to [out=0,in=180] (4,1);
      \draw [semithick] (3,2) to [out=0,in=180] (4,1);      
      \draw [white, line width=3] (4,1) to [out=0,in=180] (5,2);
      \draw [semithick] (4,1) to [out=0,in=180] (5,2);
      \draw [white, line width=3] (4,2) to [out=0,in=180] (5,3);
      \draw [semithick] (4,2) to [out=0,in=180] (5,3);
      \draw [white, line width=3] (4,3) to [out=0,in=180] (5,1);
      \draw [semithick] (4,3) to [out=0,in=180] (5,1);
      % Free strand
      \draw [semithick] (5,1) to [out=0,in=180] (6,1);
      % Web strands
      \draw [semithick] (5,3) to [out=0,in=180] (6,3);
      \draw [semithick] (5,2) to [out=0,in=180] (6,2);
      % Twists
      \node at (2,3) {$\circ$};
      \node at (2,3.3) {\tiny $+$};
      \node at (3,1) {$\circ$};
      \node at (3,.7) {\tiny $+$};
      \node at (3,2) {$\circ$};
      \node at (3,2.3) {\tiny $+$};
    \end{tikzpicture}
  };
  \node (D2) at (3,-9) {
    \begin{tikzpicture}[anchorbase,scale=.5]
      % Square
      \draw (0,0) rectangle (6,4);
      % Strands
      \draw [semithick] (0,2) -- (1.5,2);
      \draw [semithick] (1.5,2) to [out=0,in=180] (2,1.5) -- (3,1.5);
      \draw [semithick] (1.5,2) to [out=0,in=180] (2,2.5) -- (2.5,2.5);
      \draw [semithick] (2.5,2.5) to [out=0,in=180]  (3,2);
      \draw [semithick] (2.5,2.5) to [out=0,in=180] (3,3);
      % Crossing
      \draw [white, line width=3] (3,1.5) to [out=0,in=180] (4,3);
      \draw [semithick] (3,1.5) to [out=0,in=180] (4,3);
      \draw [white, line width=3] (3,2) to [out=0,in=180] (4,1);
      \draw [semithick] (3,2) to [out=0,in=180] (4,1);
      \draw [white, line width=3] (3,3) to [out=0,in=180] (4,2);
      \draw [semithick] (3,3) to [out=0,in=180] (4,2);
      % Free strand
      \draw [semithick] (4,3) to [out=0,in=180] (6,3);
      % Web strands
      \draw [semithick] (4,1) to [out=0,in=180] (6,1);
      \draw [semithick] (4,2) to [out=0,in=180] (6,2);
      % Twists
      \node at (2,2.5) {$\circ$};
      \node at (2,2.8) {\tiny $+$};
      \node at (2,1.5) {$\circ$};
      \node at (2,1.2) {\tiny $+$};
    \end{tikzpicture}
  };
  \node (E2) at (3,-12) {
    \begin{tikzpicture}[anchorbase,scale=.5]
      % Square
      \draw (0,0) rectangle (6,4);
      % Strands
      \draw [semithick] (0,2) -- (1.5,2);
      \draw [semithick] (1.5,2) to [out=0,in=180] (2,1) -- (4,1);
      \draw [semithick] (1.5,2) to [out=0,in=180] (2.5,2.5);
      \draw [semithick] (2.5,2.5) to [out=0,in=180]  (3,2);
      \draw [semithick] (2.5,2.5) to [out=0,in=180] (3,3);
      % Crossings
      \draw [white, line width=3] (3,2) to [out=0,in=180] (4,3);
      \draw [semithick] (3,2) to [out=0,in=180] (4,3);      
      \draw [white, line width=3] (3,3) to [out=0,in=180] (4,2);
      \draw [semithick] (3,3) to [out=0,in=180] (4,2);
      \draw [white, line width=3] (4,1) to [out=0,in=180] (5,3);
      \draw [semithick] (4,1) to [out=0,in=180] (5,3);
      \draw [white, line width=3] (4,3) to [out=0,in=180] (5,2);
      \draw [semithick] (4,3) to [out=0,in=180] (5,2);
      \draw [white, line width=3] (4,2) to [out=0,in=180] (5,1);
      \draw [semithick] (4,2) to [out=0,in=180] (5,1);
      % Free strand
      \draw [semithick] (5,3) to [out=0,in=180] (6,3);
      % Web strands
      \draw [semithick] (5,1) to [out=0,in=180] (6,1);
      \draw [semithick] (5,2) to [out=0,in=180] (6,2);
      % Twists
      \node at (2,1) {$\circ$};
      \node at (2,.7) {\tiny $+$};
      \node at (3,3) {$\circ$};
      \node at (3,3.3) {\tiny $+$};
      \node at (3,2) {$\circ$};
      \node at (3,1.7) {\tiny $+$};
    \end{tikzpicture}
  };
    \node (F) at (0,-15) {
    \begin{tikzpicture}[anchorbase,scale=.5]
      % Square
      \draw (0,0) rectangle (6,4);
      % Strands
      \draw [semithick] (0,2) -- (1.5,2);
      \draw [semithick] (1.5,2) to [out=0,in=180] (2,1) -- (3,1);
      \draw [semithick] (1.5,2) to [out=0,in=180] (2.5,2.5);
      \draw [semithick] (2.5,2.5) to [out=0,in=180]  (3,2);
      \draw [semithick] (2.5,2.5) to [out=0,in=180] (3,3) -- (4,3);
      % Crossings
      \draw [white, line width=3] (3,1) to [out=0,in=180] (4,2);
      \draw [semithick] (3,1) to [out=0,in=180] (4,2);
      \draw [white, line width=3] (3,2) to [out=0,in=180] (4,1);
      \draw [semithick] (3,2) to [out=0,in=180] (4,1);      
      \draw [white, line width=3] (4,1) to [out=0,in=180] (5,2);
      \draw [semithick] (4,1) to [out=0,in=180] (5,2);
      \draw [white, line width=3] (4,2) to [out=0,in=180] (5,3);
      \draw [semithick] (4,2) to [out=0,in=180] (5,3);
      \draw [white, line width=3] (4,3) to [out=0,in=180] (5,1);
      \draw [semithick] (4,3) to [out=0,in=180] (5,1);
      % Free strand
      \draw [semithick] (5,1) to [out=0,in=180] (6,1);
      % Web strands
      \draw [semithick] (5,3) to [out=0,in=180] (6,3);
      \draw [semithick] (5,2) to [out=0,in=180] (6,2);
      % Twists
      \node at (3,3) {$\circ$};
      \node at (3,3.3) {\tiny $+$};
      \node at (3,1) {$\circ$};
      \node at (3,.7) {\tiny $+$};
      \node at (3,2) {$\circ$};
      \node at (3,2.3) {\tiny $+$};
    \end{tikzpicture}
  };
  \draw [-] (A) -- (B1);
  \draw [-] (A) -- (B2);
  \draw [-] (B1) -- (C1);
  \draw [-] (B2) -- (C2);
  \draw [-] (C1) -- (D1);
  \draw [-] (C2) -- (D2);
  \draw [-] (D1) -- (F);
  \draw [-] (D2) -- (E2);
  \draw [-] (E2) -- (F);
\end{tikzpicture}
\end{equation}

%%%%%%%%%%%%%%%%%%%%%%%%%%%%%%%%%%%%%
\subsection{Main statement with half-twists}
%%%%%%%%%%%%%%%%%%%%%%%%%%%%%%%%%%%%%

\begin{theorem} \label{th:mainHT} 
  Framed foams between oriented framed web-tangles admit the following movie generators, in addition to identity movies over web-tangles (each picture is a representative of a family, obtained by crossing changes, changing half-twists types and signs, orientations, or taking planar symmetries):

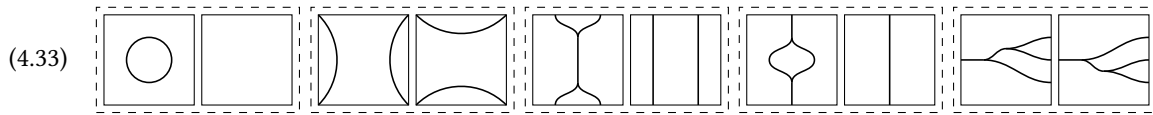
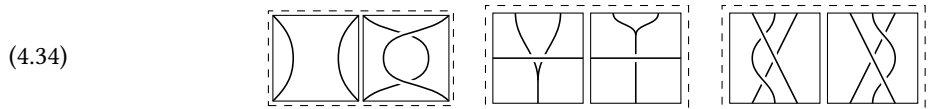
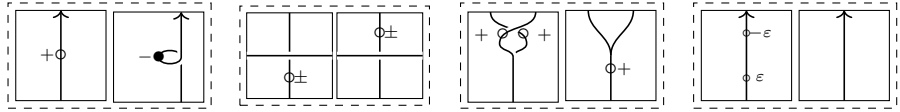
\begin{equation}

\end{equation}

Isotopies of framed foams induce sequences of movie moves from the following list:
\vspace{1cm}

\noindent {\bf Framed version of classical movie moves}
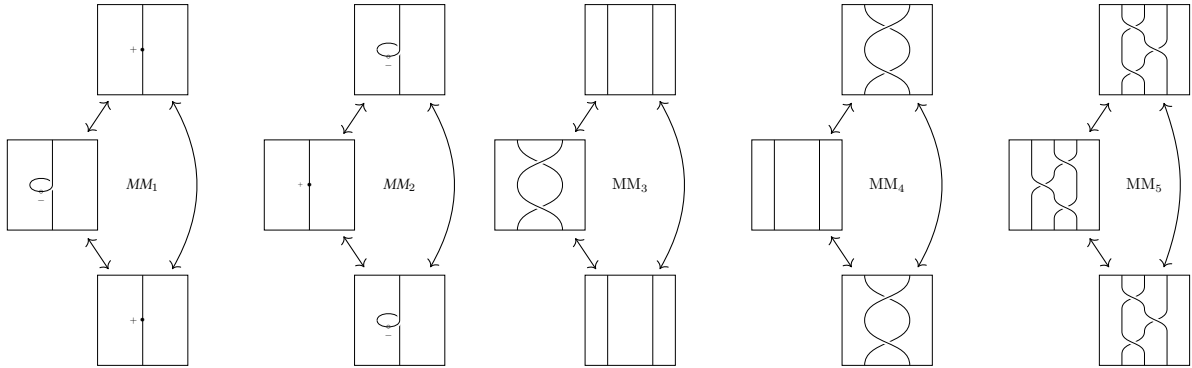
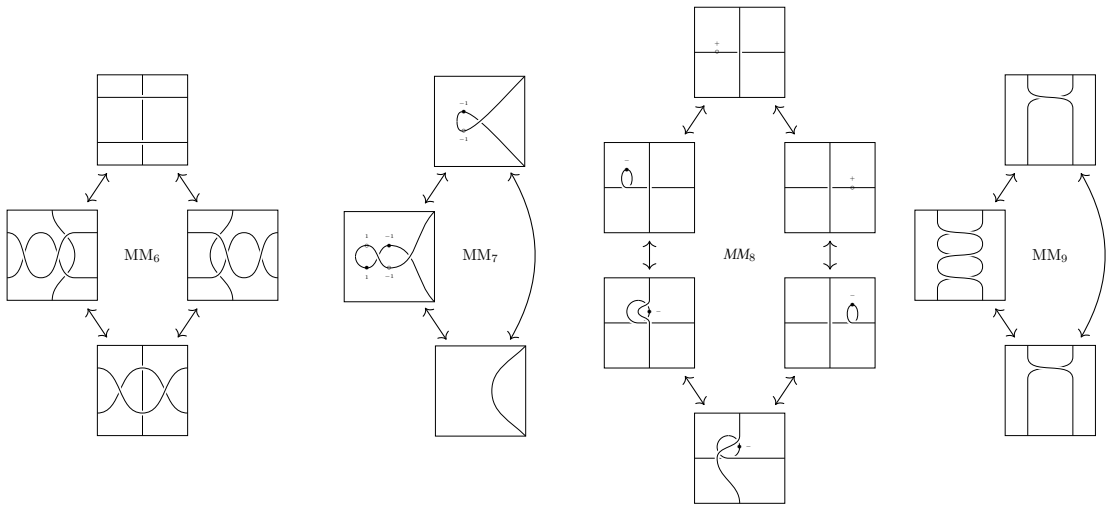
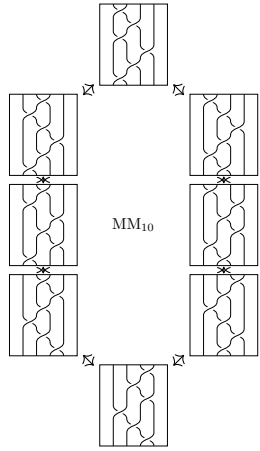
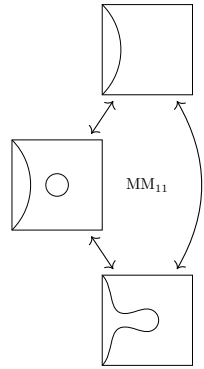
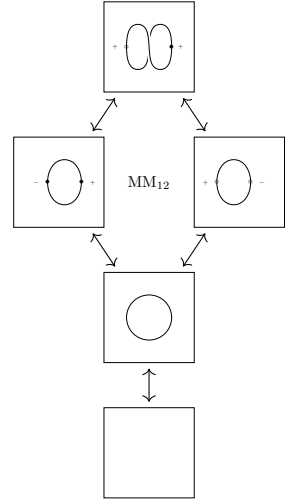
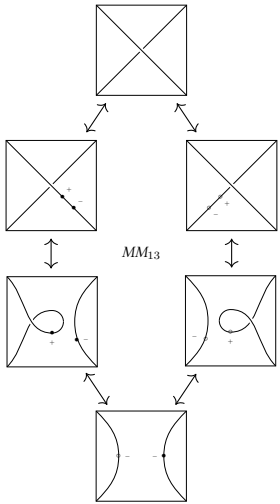
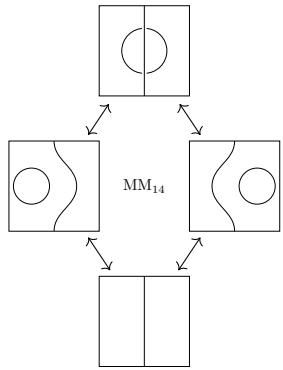
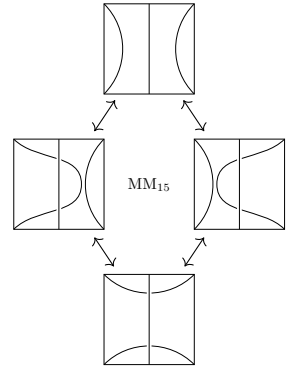
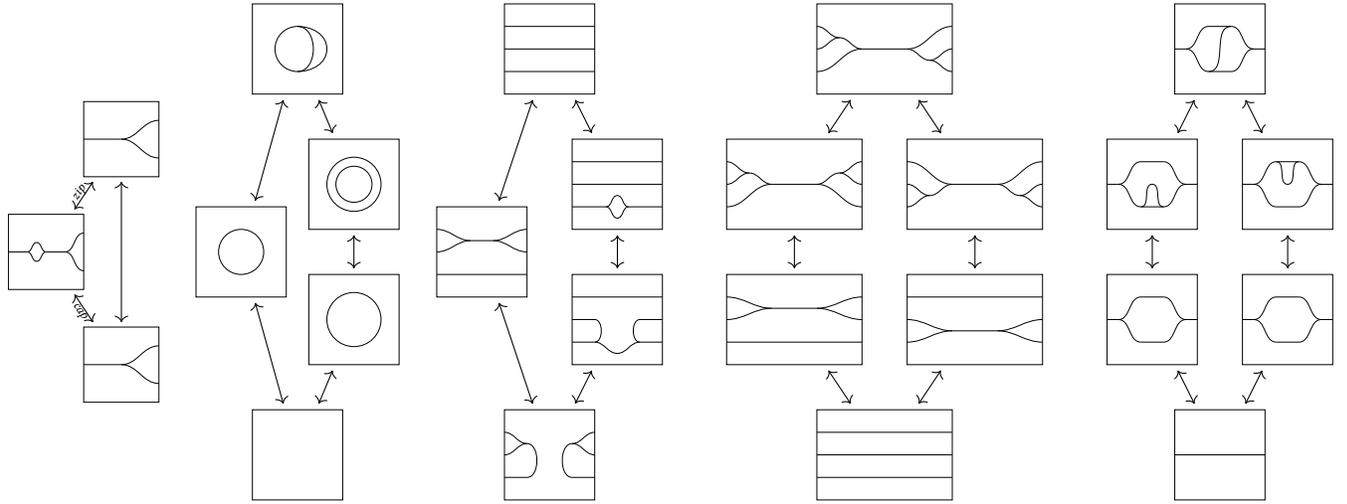
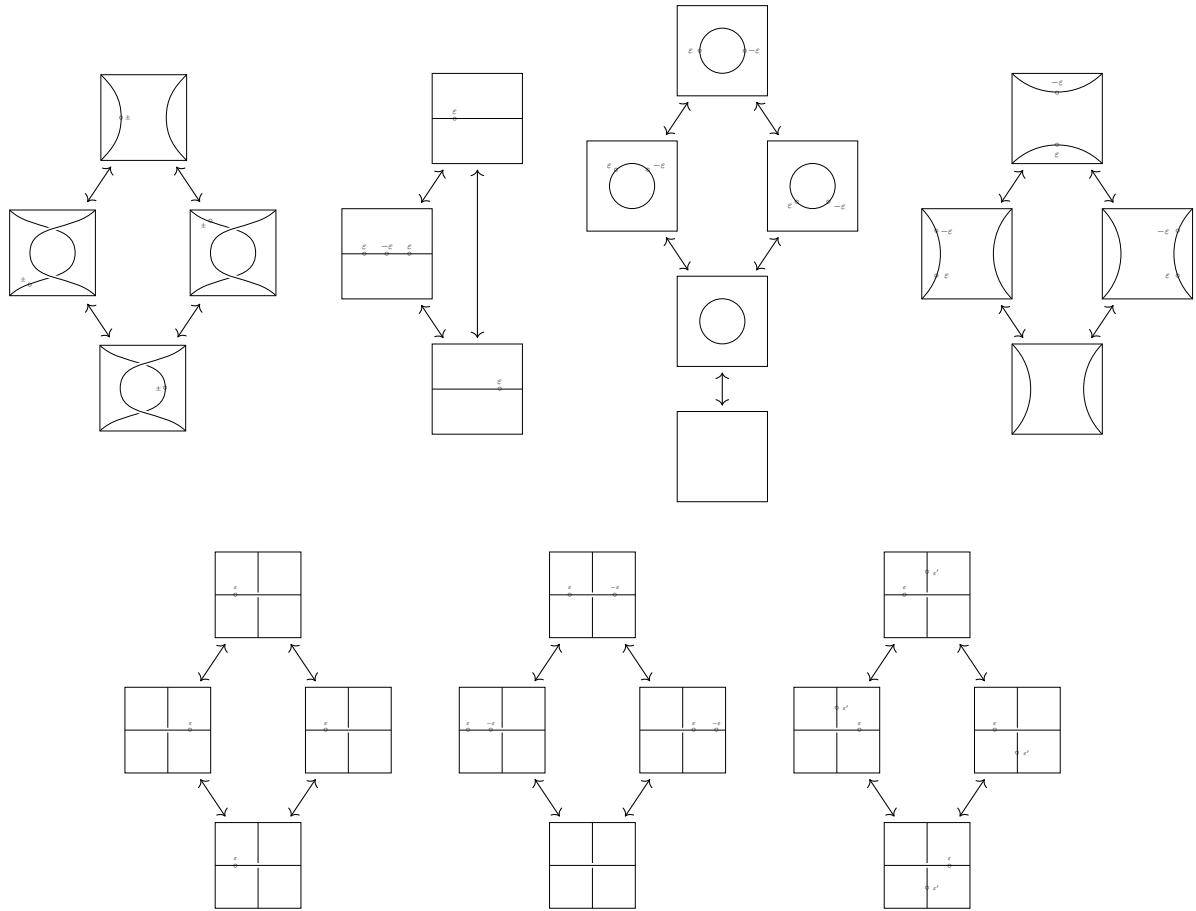
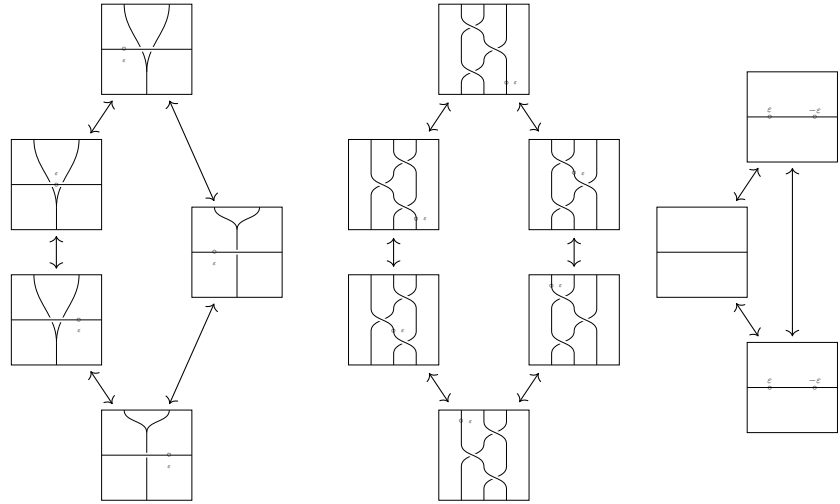
\[
  %
\]

\noindent {\bf Pure framing moves} 
\[
  %
\]

\noindent {\bf Invertibility of new generators}
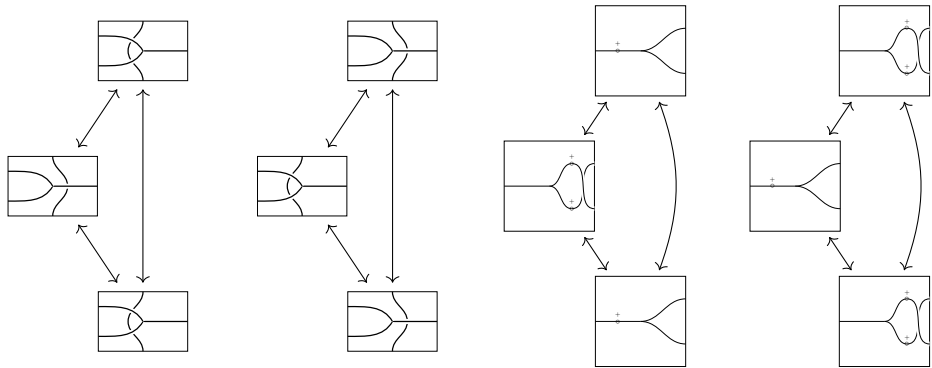
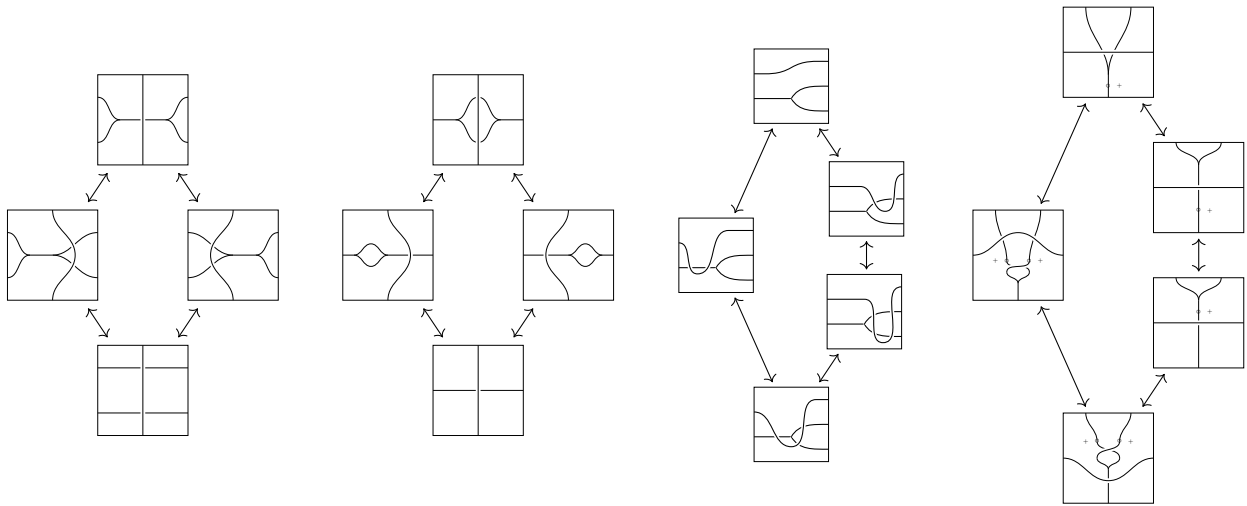
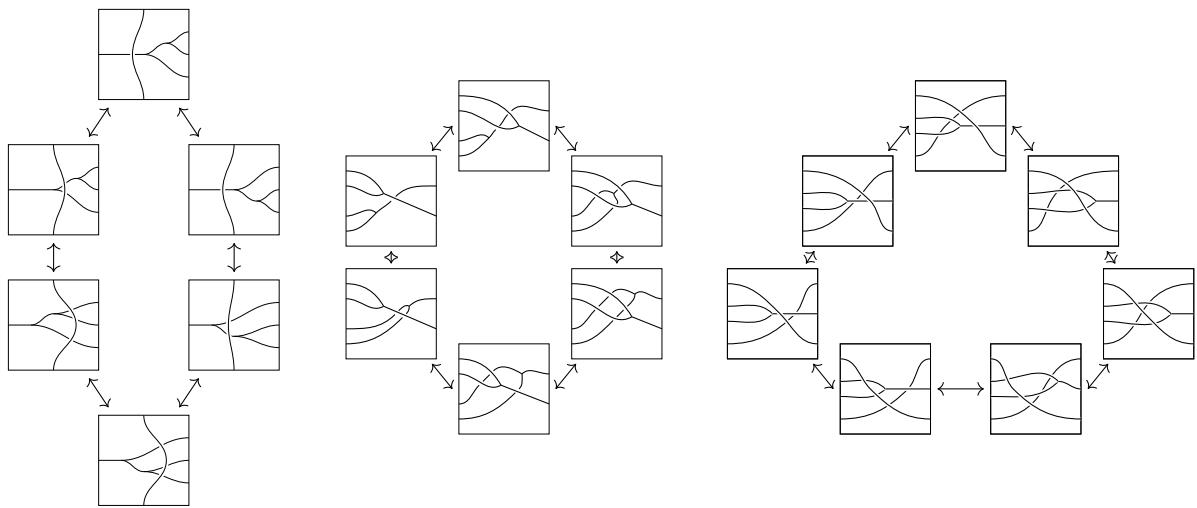
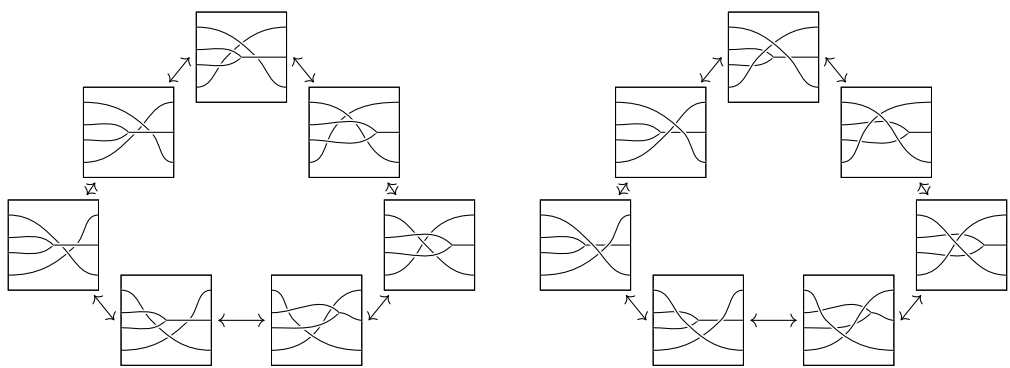
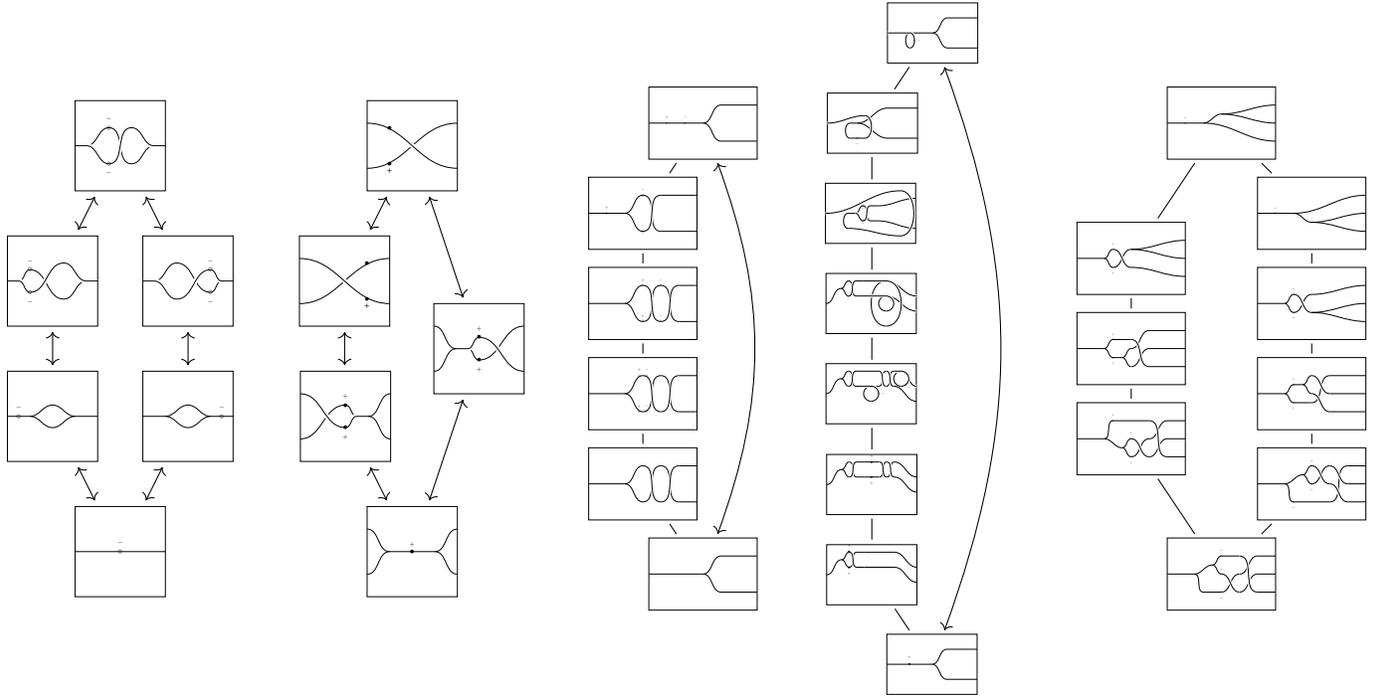
\[
  %
  };
  \draw [-] (A) -- (B1);
  \draw [-] (A) -- (B2);
  \draw [-] (B1) -- (C1);
  \draw [-] (B2) -- (C2);
  \draw [-] (C1) -- (D1);
  \draw [-] (C2) -- (D2);
  \draw [-] (D1) -- (F);
  \draw [-] (D2) -- (E2);
  \draw [-] (E2) -- (F);
\end{tikzpicture}
\]

\end{theorem}

%%%%%%%%%%%%%%%%%%%%%%%%%%%%%%%%%%%%%%%%%
\subsection{Main statement with full twists}
%%%%%%%%%%%%%%%%%%%%%%%%%%%%%%%%%%%%%%%%%%%

By restricting our attention to foams that generically bound webs with full twists as in the statement of Theorem~\ref{th:ReidFT}, one can deduce the following statement from the previous one. Indeed, one simply has to pre/post-compose generators by twisting the framing into the standard position, collapsing $\bullet$ and $\circ$ into {\tiny $\LEFTcircle$}, where the framing takes a full turn. 

\begin{theorem} \label{th:mainFT} 
  Framed foams between oriented framed web-tangles with preferred diagrams admit the following movie generators, in addition to identity movies over web-tangles (each picture is a representative of a family, obtained by crossing changes, changing half-twists types and signs, orientations, or taking planar symmetries):

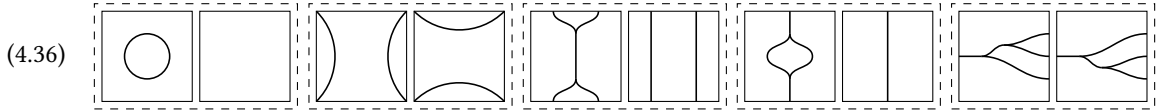
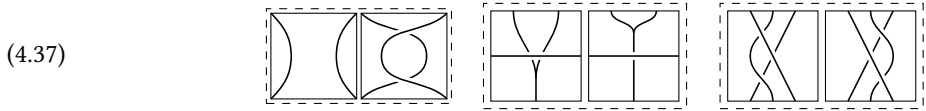
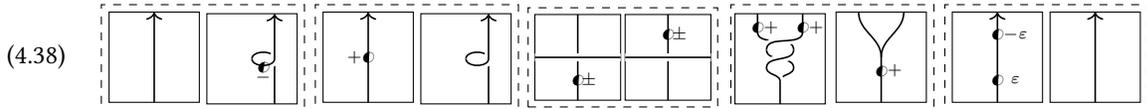
\begin{equation}

\end{equation}

Isotopies of framed foams induce sequences of movie moves that can be deduced from the list in Theorem~\ref{th:mainHT} by merging pairs $\circ$ and $\bullet$ into {\small $\LEFTcircle$} whenever an $\rm{R}_I$ move is involved, otherwise replacing half-twists by full-twists, with the following special cases:

\noindent {\bf Invertibility of new generators}
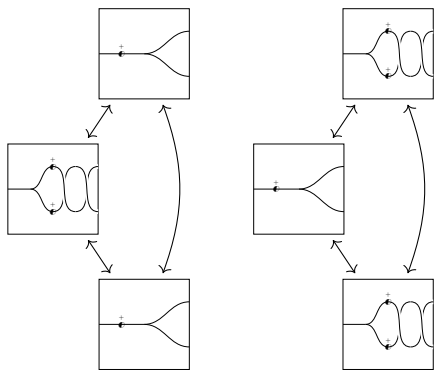
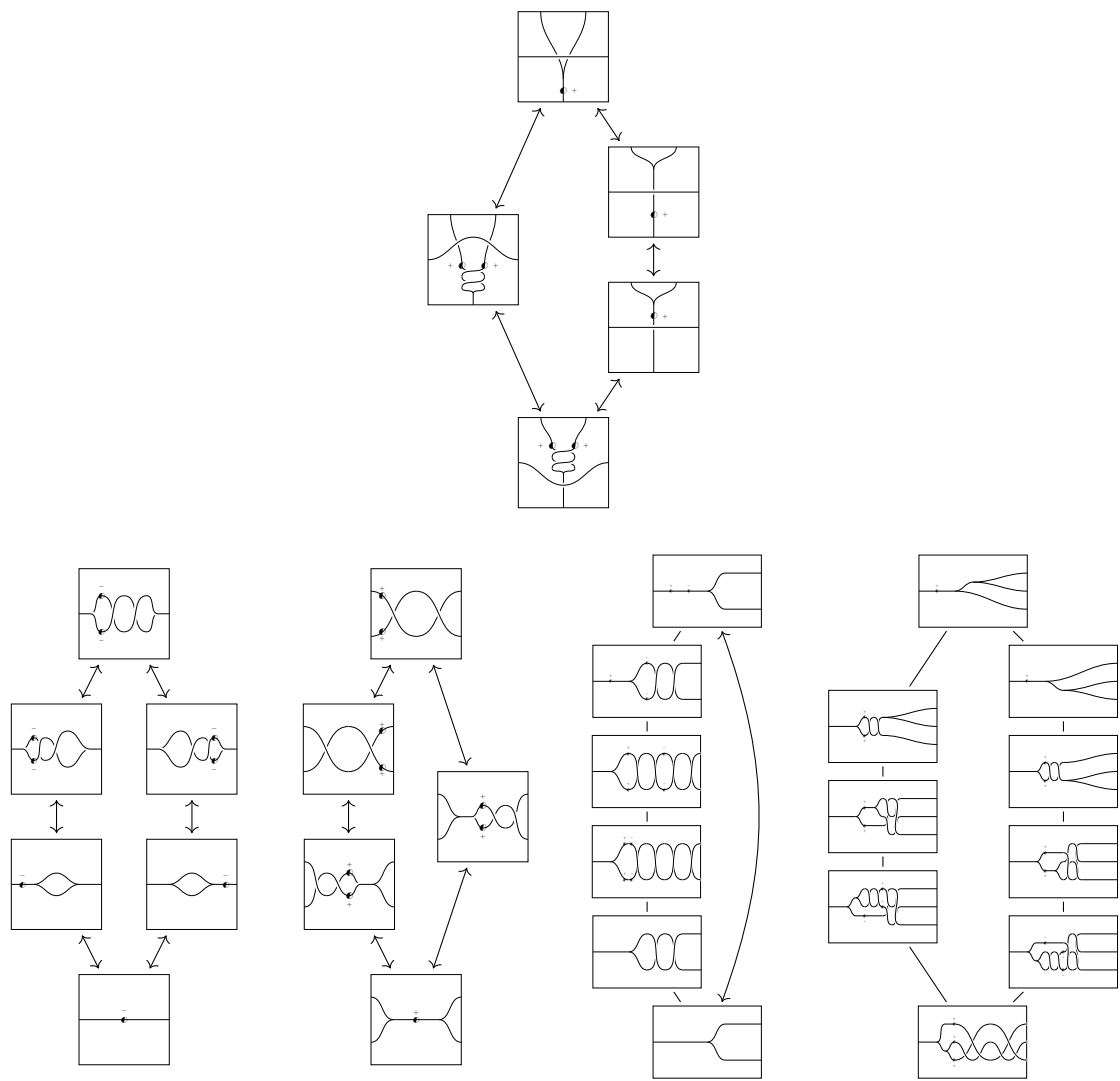
\[
  %
  };
  \draw [-] (A) -- (B1);
  \draw [-] (A) -- (B2);
  \draw [-] (B1) -- (C1);
  \draw [-] (B2) -- (C2);
  \draw [-] (C1) -- (D1);
  \draw [-] (C2) -- (D2);
  \draw [-] (D1) -- (F);
  \draw [-] (D2) -- (E2);
  \draw [-] (E2) -- (F);
\end{tikzpicture}
\]

\end{theorem}

\bibliographystyle{plain}

\begin{thebibliography}{10}

\bibitem{BN2}
D.~Bar-Natan.
\newblock Khovanov's homology for tangles and cobordisms.
\newblock {\em Geom. Topol.}, 9:1443--1499, 2005.
\newblock \href{http://arxiv.org/abs/math/0410495}{arXiv:math/0410495}.

\bibitem{Blan}
C.~Blanchet.
\newblock An oriented model for {K}hovanov homology.
\newblock {\em J. Knot Theory Ramifications}, 19(2):291--312, 2010.
\newblock \href{https://arxiv.org/abs/1405.7246}{arXiv:1405.7246}.

\bibitem{Cap}
C.L. Caprau.
\newblock {$\mathfrak{sl}(2)$} tangle homology with a parameter and singular
  cobordisms.
\newblock {\em Algebr. Geom. Topol.}, 8(2):729--756, 2008.
\newblock \href{https://arxiv.org/abs/0707.3051}{arXiv:0707.3051}.

\bibitem{Carter_foams}
J.~S. Carter.
\newblock Reidemeister/{R}oseman-type moves to embedded foams in 4-dimensional
  space.
\newblock In {\em New ideas in low dimensional topology}, volume~56 of {\em
  Ser. Knots Everything}, pages 1--30. World Sci. Publ., Hackensack, NJ, 2015.
\newblock \href{http://arxiv.org/abs/1210.3608}{arXiv:1210.3608}.

\bibitem{CarterKamada}
J~S Carter and S~Kamada.
\newblock {\em Diagrammatic Algebra}, volume 264.
\newblock American Mathematical Society, 2021.

\bibitem{CS_book}
J.~S. Carter and M.~Saito.
\newblock {\em Knotted surfaces and their diagrams}.
\newblock Number~55 in {M}athematical {S}urveys and {M}onographs. American
  Mathematical Soc., 1998.

\bibitem{CKM}
S.~Cautis, J.~Kamnitzer, and S.~Morrison.
\newblock Webs and quantum skew {H}owe duality.
\newblock {\em Math. Ann.}, 360(1-2):351--390, 2014.
\newblock \href{https://arxiv.org/abs/1210.6437}{arXiv:1210.6437}.

\bibitem{CMW}
D.~Clark, S.~Morrison, and K.~Walker.
\newblock Fixing the functoriality of {K}hovanov homology.
\newblock {\em Geom. Topol.}, 13(3):1499--1582, 2009.
\newblock \href{http://arxiv.org/abs/0701339}{arXiv:0701339}.

\bibitem{EST}
M.~Ehrig, C.~Stroppel, and D.~Tubbenhauer.
\newblock The {B}lanchet--{K}hovanov algebras.
\newblock {\em Categorification and higher representation theory},
  683:183--226, 2017.

\bibitem{ETW}
M.~Ehrig, D.~Tubbenhauer, and P.~Wedrich.
\newblock Functoriality of colored link homologies.
\newblock {\em Proceedings of the London Mathematical Society},
  117(5):996--1040, 2018.
\newblock \href{http://arxiv.org/abs/1703.06691}{arXiv:1703.06691}.

\bibitem{GoGui}
M.~Golubitsky and V.~Guillemin.
\newblock {\em Stable mappings and their singularities}, volume~14.
\newblock Springer Science \& Business Media, 2012.

\bibitem{Kh1}
M.~Khovanov.
\newblock A categorification of the {J}ones polynomial.
\newblock {\em Duke Math. J.}, 101(3):359--426, 2000.
\newblock \href{http://arxiv.org/abs/9908171}{arXiv:9908171}.

\bibitem{Kh2}
M.~Khovanov.
\newblock A functor-valued invariant of tangles.
\newblock {\em Algebr. Geom. Topol.}, 2:665--741, 2002.
\newblock \href{https://arxiv.org/abs/math/0103190}{arXiv:0103190}.

\bibitem{Kh5}
M.~Khovanov.
\newblock sl(3) link homology.
\newblock {\em Algebr. Geom. Topol.}, 4:1045--1081, 2004.
\newblock \href{https://arxiv.org/abs/math/0304375}{arXiv:0304375}.

\bibitem{KhR3}
M.~Khovanov and L.~Rozansky.
\newblock Topological {L}andau-{G}inzburg models on the world-sheet foam.
\newblock {\em Adv. Theor. Math. Phys.}, 11(2):233--259, 2007.
\newblock \href{http://arxiv.org/abs/0404189}{arXiv:0404189}.

\bibitem{Kup}
G.~Kuperberg.
\newblock Spiders for rank {$2$} {L}ie algebras.
\newblock {\em Comm. Math. Phys.}, 180(1):109--151, 1996.
\newblock \href{http://arxiv.org/abs/9712003}{arXiv:9712003}.

\bibitem{Laudenbach}
F.~Laudenbach.
\newblock Transversalit\'e de {T}hom et $h$-principe de {G}romov.
\newblock Lecture notes from lecture series at the {U}niversity of {O}uargla,
  2013.
\newblock
  \href{https://www.math.sciences.univ-nantes.fr/~laudenba/cours_ouargla3.pdf}{Available
  on the author website}.

\bibitem{MSV}
M.~Mackaay, M.~Sto{\v{s}}i{\'c}, and P.~Vaz.
\newblock {$\mf{sl}(N)$}-link homology {$(N\geq 4)$} using foams and the
  {K}apustin-{L}i formula.
\newblock {\em Geom. Topol.}, 13(2):1075--1128, 2009.
\newblock \href{http://arxiv.org/abs/0708.2228}{arXiv:0708.2228}.

\bibitem{MatherV}
J.~N. Mather.
\newblock Stability of {$C^{\infty}$} mappings: V, transversality.
\newblock {\em Advances in Mathematics}, 4(3):301--336, 1970.

\bibitem{Prz1}
J.~H. Przytycki.
\newblock Skein modules of {$3$}-manifolds.
\newblock {\em Bull. Polish Acad. Sci. Math.}, 39(1-2):91--100, 1991.
\newblock \href{https://arxiv.org/abs/math/0611797}{arXiv:math/0611797}.

\bibitem{Queff_PhD}
H.~{Queffelec}.
\newblock {Sur la cat\'egorification des invariants quantiques sln : \'etude
  alg\'ebrique et diagrammatique}.
\newblock {\em PhD thesis from the Universit\'e Paris 7}, December 2013.

\bibitem{QR}
H.~{Queffelec} and D.~E.~V. {Rose}.
\newblock {The $\mathfrak{sl}_n$ foam 2-category: a combinatorial formulation
  of Khovanov-Rozansky homology via categorical skew Howe duality}.
\newblock {\em Advances in Mathematics}, 302:1251--1339, 2016.
\newblock \href{http://arxiv.org/abs/1405.5920}{arXiv:1405.5920}.

\bibitem{QW_SkeinCat}
H.~Queffelec and P.~Wedrich.
\newblock Khovanov homology and categorification of skein modules.
\newblock {\em Quantum Topol.}, 12:129--209, 2021.
\newblock \href{http://arxiv.org/abs/1806.03416}{arXiv:1806.03416}.

\bibitem{RT}
N.~Yu. Reshetikhin and V.~G. Turaev.
\newblock Ribbon graphs and their invariants derived from quantum groups.
\newblock {\em Comm. Math. Phys.}, 127(1):1--26, 1990.

\bibitem{Thom1}
R.~Thom.
\newblock Quelques propri{\'e}t{\'e}s globales des vari{\'e}t{\'e}s
  diff{\'e}rentiables.
\newblock {\em Commentarii Mathematici Helvetici}, 28(1):17--86, 1954.

\bibitem{Thom2}
R.~Thom.
\newblock Les singularit{\'e}s des applications diff{\'e}rentiables.
\newblock In {\em Annales de l'institut Fourier}, volume~6, pages 43--87, 1956.

\bibitem{Tur}
V.~Turaev.
\newblock Skein quantization of {P}oisson algebras of loops on surfaces.
\newblock {\em Ann. Sci. \'Ecole Norm. Sup. (4)}, 24(6):635--704, 1991.

\bibitem{Vogel}
P.~Vogel.
\newblock Functoriality of {K}hovanov homology.
\newblock {\em Journal of Knot Theory and Its Ramifications}, 29(04):2050020,
  2020.
\newblock \href{https://arxiv.org/abs/1505.04545}{arXiv:1505.04545}.

\end{thebibliography}

\end{document}